\documentclass[11pt,reqno,a4paper]{amsart}
\usepackage[british]{babel}
\usepackage{amssymb,amstext,amsthm,eucal, float, upgreek}
\usepackage{latexsym,amsmath,amsfonts,amscd,mathrsfs, booktabs}
\usepackage{enumerate}
\usepackage{graphicx}
\usepackage{caption}
\usepackage{subcaption}
\usepackage{array,color}
\usepackage{enumitem}
\usepackage{tgpagella}
\usepackage{tikz, bbm}
\usepackage[utf8]{inputenc}
\usepackage[margin=2.8cm]{geometry}
\usepackage[small]{eulervm}
\usepackage[unicode]{hyperref}
\usepackage{framed}
\usepackage{xcolor}
\usetikzlibrary{positioning}
\tikzset{main node/.style={rectangle,rounded corners,fill=blue!20,draw,minimum size=1cm,inner sep=0pt},
            }
\tikzset{sec node/.style={rectangle,rounded corners,fill=green!20,draw,minimum size=1cm,inner sep=0pt},
            }
\usetikzlibrary{shapes.misc}
\tikzset{cross/.style={cross out, draw=black, minimum size=2*(#1-\pgflinewidth), inner sep=0pt, outer sep=0pt},
cross/.default={3pt}}

\definecolor{mygray}{gray}{0.7}

\usepackage[draft]{todonotes}   

\usetikzlibrary{arrows,decorations.markings,positioning}
\tikzset{inner sep=0pt, node distance=5mm,
  root/.style={circle,draw,minimum size=5pt,thick},
  broot/.style={circle,draw,minimum size=5pt,thick,fill},
  xroot/.style={circle,draw,minimum size=5pt,thick,label=below:$\times$},
  doublearrow/.style={postaction={decorate},   decoration={markings,mark=at position .6 with {\arrow[line width=1.2pt]{>}}},double distance=1.6pt,thick},
  rdoublearrow/.style={postaction={decorate},   decoration={markings,mark=at position .4 with {\arrowreversed[line width=1.2pt]{>}}},double distance=1.6pt,thick},
	rtriplearrow/.style={postaction={decorate},   decoration={markings,mark=at position .4 with {\arrowreversed[line width=1.2pt]{>}}},double distance=2.5pt,thick},
	ltriplearrow/.style={postaction={decorate},   decoration={markings,mark=at position .6 with {\arrow[line width=1.2pt]{>}}},double distance=2.5pt,thick},
  curvedline/.style={bend=right}
}
\hypersetup{
  pdftitle   = {Some notes on \texorpdfstring{$\mathrm{G}(3)$}{}},
  pdfkeywords = {},
  pdfauthor  = {Boris Kruglikov, Andrea Santi and Dennis The},
  pdfcreator = {\LaTeX\ with package \flqq hyperref\frqq}
}
\numberwithin{equation}{section}
\theoremstyle{plain}
\newtheorem{theorem}{Theorem}[section]
\newtheorem{proposition}[theorem]{Proposition}
\newtheorem{prop}[theorem]{Proposition}
\newtheorem{corollary}[theorem]{Corollary}
\newtheorem{cor}[theorem]{Corollary}
\newtheorem{lemma}[theorem]{Lemma}

\theoremstyle{definition}
\newtheorem{rem}[theorem]{Remark}
\newtheorem{rk}[theorem]{Remark}
\newtheorem{example}[theorem]{Example}
\newtheorem{definition}[theorem]{Definition}
\newcommand{\Hom}{\operatorname{Hom}}

\newcommand{\der}{\mathfrak{der}}
\newcommand{\fosp}{\mathfrak{osp}}
\newcommand{\fsl}{\mathfrak{sl}}
\newcommand{\fso}{\mathfrak{so}}
\newcommand{\fsp}{\mathfrak{sp}}

\newcommand{\fm}{\mathfrak{m}}
\newcommand{\fp}{\mathfrak{p}}

\newcommand{\fg}{\mathfrak{g}}
\newcommand{\fh}{\mathfrak{h}}
\newcommand{\fk}{\mathfrak{k}}

\newcommand{\1}{\mathbbm{1}}

\newcommand{\ZZ}{\mathbb{Z}}

\newcommand{\CC}{\mathbb{C}}

\newcommand{\be}{\boldsymbol{e}}

\newcommand{\gr}{\operatorname{gr}}
\newcommand{\bep}{\boldsymbol{\epsilon}}
\renewcommand{\a}{\alpha}
\renewcommand{\b}{\beta}
\newcommand{\res}{\operatorname{res}}

 \newcommand\cA{\mathcal{A}}
 \newcommand\cB{\mathcal{B}}
 \newcommand\bbC{\mathbb{C}}
 \newcommand\bbP{\mathbb{P}}
 \newcommand\sfZ{\mathsf{Z}}
 \newcommand{\fz}{\mathfrak{z}}
 \newcommand\tspan{\operatorname{span}}
 \newcommand\op{\oplus}
 \newcommand\fcspo{\mathfrak{cspo}}
 \newcommand\fspo{\mathfrak{spo}}
 \newcommand\id{\operatorname{id}}
 \newcommand\ptm[1]{{\tiny
 \left( \begin{array}{cccc|cccc} #1
 \end{array} \right)}}
 \newcommand{\fq}{\mathfrak{q}}
 \newcommand{\ff}{\mathfrak{f}}
 \newcommand\cU{\mathcal{U}}
 \newcommand\cV{\mathcal{V}}
 \newcommand{\fC}{\mathfrak{C}}
 \newcommand{\fG}{\mathfrak{G}}
 \newcommand\bbA{\mathbb{A}}
 
 \newcommand\cC{\mathcal{C}}
 \newcommand\cD{\mathcal{D}}
 \newcommand\cH{\mathcal{H}}
 \newcommand\cL{\mathcal{L}}
 \newcommand\bC{\mathbf{C}}
 \newcommand\bX{\mathbf{X}}
 \newcommand\bD{\mathbf{D}}
 \newcommand\bS{\mathbf{S}}
 \newcommand\bT{\mathbf{T}}
 \newcommand\bU{\mathbf{U}}
 \newcommand\bV{\mathbf{V}}
\newcommand\bB{\mathbf{B}}
 \newcommand\cE{\mathcal{E}}
 
 \newcommand{\Ch}{Ch}
 \renewcommand\ss{{\rm ss}}
 \newcommand\fms[1]{\fm_{\bar{#1}}}
 
 \newcommand\GA[3]{\raisebox{-0.05in}{\begin{tikzpicture}
 \draw (0,0) -- (1,0);
 \draw (0,0.05) -- (1,0.05);
 \draw (0,-0.05) -- (1,-0.05);
 \draw (0.6,0.15) -- (0.4,0) -- (0.6,-0.15);
 \node[draw,circle,inner sep=1.5pt,fill=white] at (0,0) {};
 \node[above] at (0,0.15) {$#1$};
 \node[draw,circle,inner sep=1.5pt,fill=white] at (1,0) {};
 \node[above] at (1,0.15) {$#2$};
 \node[draw,circle,inner sep=1.5pt,fill=white] at (2,0) {};
 \node[above] at (2,0.15) {$#3$};
 \end{tikzpicture}}}
 
 \newcommand\bbV{\mathbb{V}}
 \newcommand\bbU{\mathbb{U}}
 \newcommand\pr{\operatorname{pr}}
 \newcommand\fa{\mathfrak{a}}
 \newcommand\bbT{\mathbb{T}}

 \newcommand\g{\mathfrak{g}}
 \newcommand\p{\partial}
 
 \renewcommand\C{{\mathbb C}}
 \newcommand{\comm}[1]{}
 \newcommand\opp[1]{\mathop{\rm #1}\nolimits}

\newcommand{\FIXMEA}[1]{\textcolor{blue}{\textsf{FIXMEA: #1}}}

\allowdisplaybreaks
\begin{document}
\title[$G(3)$-supergeometry and the super Hilbert--Cartan equation]{$G(3)$-supergeometry and a supersymmetric extension\\ of the Hilbert--Cartan equation}
\author{Boris Kruglikov}
\author{Andrea Santi}
\author{Dennis The}

\address{(BK, DT) Department of Mathematics and Statistics, UiT The Arctic University of Norway, Troms\o{} 90-37, Norway}
\address{(AS) Department of Mathematics ``Tullio Levi-Civita'', University of Padova, 35121 Padova, Italy}
\thanks{}

 \begin{abstract}
We realize the simple Lie superalgebra $G(3)$ as supersymmetry of various geometric structures,
most importantly super-versions of the Hilbert--Cartan equation (SHC) and Cartan's involutive PDE 
system that exhibit $G(2)$ symmetry.
We provide the symmetries explicitly and compute, via the first Spencer cohomology groups, 
the Tanaka--Weisfeiler prolongation of the negatively graded Lie superalgebras associated 
with two particular choices of parabolics. We discuss non-holonomic superdistributions with 
growth vector $(2|4,1|2,2|0)$ deforming the flat model SHC, and prove that 
the second Spencer cohomology group gives a binary quadratic form, thereby indicating a 
``square-root'' of Cartan's classical binary quartic invariant for generic rank $2$ distributions 
in a $5$-dimensional space. Finally, we obtain super-extensions of Cartan's 
classical submaximally symmetric models, compute their symmetries and
observe a supersymmetry dimension gap phenomenon.
 \end{abstract}

\maketitle

\vskip-0.7cm\par\noindent
\tableofcontents
\vskip-0.5cm\par\noindent

\section{Introduction and the main results}\label{sec:introduction}
\vskip0.3cm\par

In the early days of Lie theory, W.\,Killing found all five exceptional simple Lie algebras,
yet without concrete geometric realizations. The simplest of these, the 14-dimensional Lie
algebra $G(2)$, discovered in 1887, was realized as the symmetry algebra of two different Klein geometries in 1893 by E.\,Cartan and F.\,Engel, in two successive papers in the same issue of Comptes Rendus
\cite{G2-Cartan,G2-Engel}.

Supersymmetry was brought to life in the context of quantum field theory and it
is based on the theory of Lie superalgebras.
The first simple (real) Lie superalgebra was computed by J.\,Wess and B.\,Zumino in 1974
as the symmetry superalgebra of $AdS^{5|8}$ \cite{MR0356830},
a superization of the anti de Sitter space playing a special role in general relativity. This was one
of the classical Lie superalgebras $\mathfrak{su}(2,2|1)=\mathfrak{osp}(4,4|2;\mathbb R)\cap\mathfrak{sl}(4|1;\CC)$.
The classification of simple complex Lie superalgebras was achieved by V.\,Kac in 1977 \cite{MR0486011},
and the simplest exceptional one, in the list of Lie superalgebras with a reductive even part, is $G(3)$ of dimension $(17|14)$. This Lie superalgebra is traditionally described by introducing the brackets on its even and odd parts and not as the symmetry superalgebra of some simple algebraic or geometric structure. (Arguably, one reason is that the smallest non-trivial representation of $G(3)$ is the adjoint representation \cite{MR1026702}.)

The goal of this paper is to realize $G(3)$ as the symmetry of a (Klein) supergeometry,
and then study the invariants and deformations of this supergeometry. We remark that $G(2)$ is a subalgebra of $G(3)$, and we extensively make use of this important fact. Indeed, we will establish super-analogs of the celebrated differential equations associated to $G(2)$. For simplicity, in this paper we restrict to Lie algebras and superalgebras over $\CC$,
the straightforward version over $\mathbb R$ corresponds to the split (normal) form. 

Let us briefly recall the classical results before we describe the super-models.
\subsection{History: realizations of $G(2)$ as symmetry}
In Cartan's realization, $G(2)$ is the symmetry of a rank 2 distribution in a 5-space.
This distribution is associated to the underdetermined ordinary differential equation
 \begin{equation}\label{HC-ODE}
z'=\tfrac12(u'')^2,
 \end{equation}
for the functions $u=u(x)$, $z=z(x)$. Equivalently, it is the 5-manifold $\Sigma = \{ z_1 = \frac{1}{2} (u_2)^2 \}$ in the mixed jet-space $J^{2,1}(\bbC,\bbC^2)=\bbC^6(x,u,u_1,u_2,z,z_1)$ equipped 
with the Pfaffian system
 \begin{align} \label{HC-Pfaff}
 \langle du-u_1dx, \,\, du_1-u_2dx,\,\, dz-\tfrac12 u_2^2dx\rangle,
 \end{align}
i.e., the pullback to $\Sigma$ of the Cartan system in $J^{2,1}(\bbC,\bbC^2)$.  
Symmetries of \eqref{HC-Pfaff} are often referred to as {\em internal} symmetries of \eqref{HC-ODE}.
In modern terms, this Pfaffian system is the Klein geometry encoded as a rank 2 distribution on
the flag variety $G(2)/P_1$, where $P_1$ is the parabolic subgroup
with marked Dynkin diagram \,\,
 \raisebox{-0.08in}{\begin{tikzpicture}
 \draw (0,0) -- (1,0);
 \draw (0,0.05) -- (1,0.05);
 \draw (0,-0.05) -- (1,-0.05);
 \draw (0.6,0.15) -- (0.4,0) -- (0.6,-0.15);
 \node[draw,circle,inner sep=1.5pt,fill=white] at (0,0) {};
 \node[draw,circle,inner sep=1.5pt,fill=white] at (1,0) {};
 \node[below] at (0,-0.1) {$\times$};
 \end{tikzpicture}}
\,\, of $G(2)$.\footnote{We will abuse the notation $G(2)$ for the Lie group and the corresponding Lie algebra, and similarly for $G(3)$ in the super-setting later on. For parabolics, $P$ will denote a Lie (super-)group with Lie (super-)algebra $\fp$, sometimes with extra ornamentation.}

In Engel's realization, $G(2)$ is the symmetry of a contact distribution $\cC$ on a 5-dimensional space $M$
equipped with a field of rational normal curves of degree $3$ (twisted cubics) $\cV\subset\bbP(\cC)$, in modern terms, a paraconformal or $GL(2)$-structure.
We recall that a twisted cubic is given by
$\lambda\mapsto[\lambda^3:\lambda^2:\lambda:1]$ in some (projective) frame of the distribution,
modulo projective reparametrizations of $\lambda$.

In 1910, E.\,Cartan \cite{MR1509120} realized $G(2)$ as the {\em contact} (or {\em external}) symmetry of an overdetermined system of differential equations
 \begin{equation}\label{C-PDE}
u_{xx}=\tfrac13\lambda^3,\quad u_{xy}=\tfrac12\lambda^2,\quad u_{yy}=\lambda.
 \end{equation}
Upon elimination of the parameter $\lambda$, one obtains an involutive PDE system that we call the {\em $G(2)$-contact PDE system}, namely
 \begin{equation} \label{G2-ct-PDE}
u_{xx}=\tfrac13 u_{yy}^3, \quad u_{xy}=\tfrac12 u_{yy}^2.
 \end{equation}
It should be mentioned that D.\,Hilbert in 1912 \cite{MR1511723} showed 
that \eqref{HC-ODE} (resp. \eqref{C-PDE}) do not allow integral curves
(resp. surfaces) to be expressed in closed form, that is, without quadratures, a phenomenon
explained by E.\,Cartan in a more general context in 1914. 
Henceforth \eqref{HC-ODE} is called the {\it Hilbert--Cartan
equation}. (Variations like the factor $\tfrac12$ in equation \eqref{HC-ODE} are inessential.
Later we will also have similar differences in notation 
for the super-versions.)

Since then many methods to compute symmetries of the HC equation \eqref{HC-ODE} have been developed,
in particular relating internal to generalized symmetries. 
For us the most important approach will be that of
N.\,Tanaka \cite{MR0266258} 
and B.\,Weisfeiler \cite{MR0232811}. This gives an upper bound on the symmetry algebra from 
the algebraic prolongation
of the symbol algebra (also known as the Carnot algebra in optimal control). We recall that associated to 
any distribution $\cD$ on a manifold, there is the weak derived flag
\begin{equation}
\label{eq:weakderivedflag}
\cD^{1}=\cD \subset \cD^{2}\subset\cdots\subset\cD^{k}\subset\cdots\;,\qquad 
\cD^{k+1} = [\cD, \cD^{k}].
\end{equation}
 A distribution $\cD$ is called {\em regular} if the ranks of $\cD^k|_x$ are constant in $x$ for all $k>0$, that is, the $\cD^k$ are distributions for all $k>0$.  Setting $\fg_{-i}(x)=\cD^{i}|_x/\cD^{i-1}|_x$, the symbol algebra at $x\in M$ is
$\fm_x=\bigoplus_{k<0}\fg_k(x)$. If we assume that $\cD$ is bracket-generating of depth $\mu$ ($\cD^{\mu}$ is the full tangent bundle and $\cD^{\mu-1}\subsetneqq \cD^{\mu}$) and 
strongly regular of type $\fm$ (all $\fm_x$ are isomorphic to a fixed negatively-graded Lie algebra $\fm$),
then $\fg_k=0$ precisely for all $k<-\mu$. The maximal prolongation of 
$\fm=\bigoplus_{-\mu\leq k<0}\fg_k$ is then defined as the unique (possibly infinite-dimensional) 
$\mathbb Z$-graded Lie algebra
 $$pr(\fm)=\bigoplus_{k=-\mu}^{+\infty} \g_k$$
that extends $\fm$, is transitive (for all $k\geq 0$, if $X\in\g_k$ is an element such that $[X,\g_{-1}]=0$, then $X=0$) and is maximal with these properties. In particular $\fg_0=\der_{gr}(\fm)$ is the Lie algebra of grade-preserving derivations of $\fm$. 

If $M=(M_o,\cA_M)$ is a supermanifold, with underlying topological space $M_o$ and sheaf of superfunctions $\cA_M$, then a distribution is defined as a (graded) $\cA_M$-subsheaf $\cD$ of 
the tangent sheaf $\mathcal TM=\operatorname{Der}(\cA_M)$ of $M$ that is locally a direct factor, 
see e.g. \cite[\S 4.7]{MR2069561}.  (We will often use ``superdistribution'' as a shortening of ``distribution on a supermanifold''.)  Any such $\cD$ induces a vector bundle 
$\cD|_{M_o}=\cup_{x\in M_o}\cD|_x$ on $M_o$ (where $\cD|_x$ is the evaluation of $\cD$ at $x\in M_o$), 
but we note that this bundle does not fully determine $\cD$. 
 The weak derived flag associated to $\cD$ is defined as in \eqref{eq:weakderivedflag}.  Similarly, $\cD$ is called {\em regular} if $\cD^k$ are superdistributions for all $k>0$, and bracket-generating of depth $\mu$ if $\cD^\mu=\mathcal TM$. Hence we obtain vector bundles $\cD^k|_{M_o}=\cup_{x\in M_o}\cD^k|_x$ on $M_o$ and a symbol $\fm_x$ at any $x\in M_o$ that is a (finite-dimensional) Lie superalgebra.

 We also consider the stalk $\cD_x^k$ of $\cD^k$ at $x\in M_o$ as a module over the local ring $(\cA_M)_x$ 
and set $gr(\mathcal T_xM)= \bigoplus_{k>0} gr(\mathcal T_xM)_{-k}$, where $gr(\mathcal T_xM)_{-k}=\cD_x^k/\cD_x^{k-1}$. This is naturally a graded Lie superalgebra free over $(\cA_M)_x$.  Since supervector fields are not determined by their values at the points of $M_o$, the correct generalization of the concept of strong regularity is given in terms of the stalks. 

\begin{definition}
Let $\cD$ be a regular distribution on a supermanifold $M=(M_o,\cA_M)$ that is 
bracket-generating of depth $\mu$. Then $\cD$ is strongly regular if there exists a negatively-graded Lie superalgebra $\fm=\bigoplus_{0<k\leq\mu}\fm_{-k}$ such that $gr(\mathcal T_xM)\cong (\cA_M)_x\otimes \fm$ at any $x\in M_o$, as graded Lie superalgebras over $(\cA_M)_x$.
\end{definition} 

Concretely, a strongly regular superdistribution admits a local basis of supervector fields adapted to the weak derived flag and whose brackets are
given by the structure constants of $\fm$, after the appropriate quotients have been taken.
The superdistributions considered in this paper will all tacitly be assumed strongly regular. 

The proof of the existence and uniqueness of $pr(\fm)$ given in \cite{MR0266258} extends verbatim to the Lie superalgebra case and the mild generalization $pr(\fm,\fg_0)$ that includes a reduction $\fg_0\subset\der_{gr}(\fm)$ of the structure Lie superalgebra  is straightforward. In the context of graded Lie superalgebras we will call it the {\em Tanaka--Weisfeiler prolongation}.

Let us come back to the classical case. The three realizations of $G(2)$ discussed above are conveniently related by the diagram of
Figure \ref{F:G2-twistor}.

 \begin{figure}[h]
 \begin{center}
 \begin{tikzpicture}
 \draw (0,0) -- (1,0);
 \draw (0,0.07) -- (1,0.07);
 \draw (0,-0.07) -- (1,-0.07);
 \draw (0.6,0.15) -- (0.4,0) -- (0.6,-0.15);
 \node[draw,circle,inner sep=2pt,fill=white] at (0,0) {};
 \node[draw,circle,inner sep=2pt,fill=white] at (1,0) {};
 \node[below] at (0,-0.1) {$\times$};
 \draw (3,0) -- (4,0);
 \draw (3,0.07) -- (4,0.07);
 \draw (3,-0.07) -- (4,-0.07);
 \draw (3.6,0.15) -- (3.4,0) -- (3.6,-0.15);
 \node[draw,circle,inner sep=2pt,fill=white] at (3,0) {};
 \node[draw,circle,inner sep=2pt,fill=white] at (4,0) {};
 \node[below] at (4,-0.1) {$\times$};
 \draw (1.5,1) -- (2.5,1);
 \draw (1.5,1.07) -- (2.5,1.07);
 \draw (1.5,0.93) -- (2.5,0.93);
 \draw (2.1,1.15) -- (1.9,1) -- (2.1,0.85);
 \node[draw,circle,inner sep=2pt,fill=white] at (1.5,1) {};
 \node[draw,circle,inner sep=2pt,fill=white] at (2.5,1) {};
 \node[below] at (1.5,0.9) {$\times$};
 \node[below] at (2.5,0.9) {$\times$};
\path[->,>=angle 90](1.9,0.75) edge (0.65,0.25);
\path[->,>=angle 90](2.1,0.75) edge (3.35,0.25);
 \end{tikzpicture}
\end{center}
\caption{$G(2)$-twistor correspondence as a relation between PDEs and ODEs.}
\label{F:G2-twistor}
\end{figure}
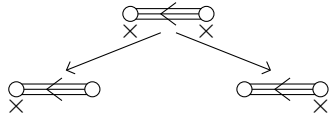\par\noindent
On the left is the 5-dimensional space $G(2)/P_1$ equipped with a
rank 2 distribution having growth vector $(2,1,2)$. (Here and in the following, the {\em growth vector} is the list of dimensions of the graded components
of the symbol algebra. Another convention is to list dimensions of the filtered components, in which case
it is called a $(2,3,5)$ distribution.) On the right is the 5-dimensional contact manifold $G(2)/P_2$
with a reduction of the structure group to $GL(2)\subset CSp(4)$.

Finally, on the top is the 6-dimensional space $G(2)/P_{12}$ equipped with a rank 2 distribution $\cD$ having growth vector $(2,1,1,1,1)$ and this geometry is indeed derived from the PDE model \eqref{G2-ct-PDE}. Namely, as a 6-space $\cE \subset J^2(\bbC^2,\bbC)$, it inherits:
 \begin{enumerate}
 \item[(i)] the Cartan distribution $\cD^{2}$ of rank 3, which coincides with the first derived of $\cD$. The
associated {\em Cauchy characteristic space} $\Ch(\cD^{2})$ consists of (internal) symmetries of $\cD^{2}$ that lie inside $\cD^{2}$ itself and is generated by $$\bC_1 = D_x - \lambda D_y\;,$$ where
\begin{equation}
D_x=\partial_x+u_x\partial_u+\frac{\lambda^3}{3}\partial_{u_x}+\frac{\lambda^2}{2}\partial_{u_y}\;,\quad
D_y=\partial_y+u_y\partial_u+\frac{\lambda^2}{2}\partial_{u_x}+\lambda\partial_{u_y}
\end{equation}
are truncated total derivatives.
 \item[(ii)] 1-dimensional fibres for $\cE \to J^1(\bbC^2,\bbC)$, with vertical bundle spanned by $\bC_2=\partial_\lambda$.  The latter generates $\Ch(\cD^{4})$.
 \end{enumerate}
 The rank 2 distribution $\cD$ can be recovered from (i)-(ii) as the span of $\bC_1$ and $\bC_2$.

The left arrow in Figure \ref{F:G2-twistor} is the quotient $\cE \to \Sigma\cong \cE / \bC_1$ of $\cE$ by $\bC_1$ and $\Sigma$ inherits the rank $2$ distribution $\cD^{2}/\bC_1$. The right arrow is the quotient $\cE \to M\cong\cE / \bC_2$ by $\bC_2$ and $M$ inherits the contact distribution $\cC = \cD^{4} / \bC_2$ equipped with the twisted cubic obtained as (the Zariski-closure of) the push-forward of $\bC_1$ through the projection.

Both quotients exist in general only locally and only for $G(2)$-invariant (flat) structures.
In fact, the second arrow does not exist for general curved parabolic geometries (the field of curves determined by the push-forward of a vector field in $\Ch(\cD^{2})$ may not be a rational normal
curve, therefore it does not give rise to the required reduction $GL(2)\subset CSp(4)$), while the first quotient exists universally, as the geometries are just determined  by the type of the associated distributions.
The latter gives a bijection between involutive PDE systems of the second order for $u = u(x,y)$ and
(by a theorem of E.\,Goursat) Monge equations, i.e., underdetermined ODEs of the type
 $$
z'=f(x,u,u',u'',z)\;,
 $$
for $u=u(x)$, $z=z(x)$.
Cartan \cite{MR1509120} showed that \eqref{G2-ct-PDE}  has maximal (contact)  symmetry dimension amongst all second order involutive PDE systems and \eqref{HC-ODE} maximal (internal) symmetry dimension amongst all Monge equations.

Generalizations of the $G(2) / P_{12} \to G(2) / P_1$ fibration to the other exceptional simple Lie groups were first investigated by K. Yamaguchi \cite{MR1699860}.  Recently, work of D. The \cite{MR3759350} gave the first explicit geometric generalizations of the Cartan--Engel $G(2)$-models to the exceptionals.  The idea is simply illustrated in the $G(2)$-case: the twisted cubic $\cV\subset\bbP(\cC)$ from Engel's $G(2)/P_2$-picture is a {\em Legendrian projective variety}, i.e., its osculations give a family $\widehat\cV$ of Lagrangian subspaces, and the fibres of $\cE \to M$ are modelled on $\widehat\cV$.  The Lagrange--Grassmann bundle $\widetilde{M}=LG(\cC) \to M$ (whose fibres are Lagrangian subspaces of the contact distribution $\cC$ on $M$) is locally isomorphic to $J^2(\bbC^2,\bbC)$, so $\cE \subset \widetilde{M}$ corresponds to a PDE.  The difference of perspective in \cite{MR3759350} is to view the PDE \eqref{G2-ct-PDE} as an equivalent description of $G(2)/P_2$-geometry, using the fact that $\widehat\cV$ provides the same reduction to $GL(2)\subset CSp(4)$ as $\cV$ does.

We now turn to our results in the super-setting.

\subsection{New results: realizations of $G(3)$ as supersymmetry}

We will demonstrate that a super-extension of the HC equation \eqref{HC-ODE} is given by the following system of
partial differential equations that we call {\it SHC} (for {\it super Hilbert-Cartan}):
 \begin{equation} \label{SHC-ODE}
z_x = \tfrac12 u_{xx}^2 + u_{x\nu} u_{x\tau} , 
\quad z_\nu = u_{xx} u_{x\nu}, \quad z_\tau = u_{xx} u_{x\tau}, \quad u_{\nu\tau} = -u_{xx}\;,
 \end{equation}
where $u=u(x,\nu,\tau)$ and $z=z(x,\nu, \tau)$. It is a submanifold $\Sigma$ of codimension $(2|2)$ in the mixed jet-superspace $J^{2,1}(\bbC^{1|2},\bbC^{2|0})$, equipped with the pullback of the Cartan system. Unlike \eqref{HC-ODE}, which has general solution depending on one arbitrary function of one variable, the space of solutions of \eqref{SHC-ODE} depends only on five arbitrary constants, see Section \ref{subsec:integralsubmnfds}. 

From the internal perspective, this corresponds to a superdistribution of rank $(2|4)$ in a $(5|6)$-dimensional supermanifold
with growth vector $(2|4,1|2,2|0)$, and the equation can be directly produced by integrating the graded nilpotent
Lie superalgebra associated to the parabolic $\mathfrak{p}_2^{\rm IV}\subset G(3)$
(see Sections \ref{subsec:2.2} and \ref{S:G3-gradings} for notations) via the super-version of the Baker--Campbell--Hausdorff
formula and then locally rectifying the corresponding Pfaffian system on $G(3)/P_2^{\rm IV}$.
This is how we obtained \eqref{SHC-ODE} initially; however in this paper we present a method closer to that used in \cite{MR3759350}.
We then prove that the internal symmetry superalgebra of \eqref{SHC-ODE} is $G(3)$ in Theorem \ref{T:G3-SHC-sym}.

Specifically, we begin with a contact grading.  From Table \ref{Table:19sgeo}, we note that there are two contact gradings on $G(3)$: the one associated to
$\mathfrak{p}^{\rm I}_1$ which corresponds to a purely odd distribution (i.e., which gives rise to a ``consistent'' $\mathbb Z$-grading in Kac's terminology),
and the other associated to $\mathfrak{p}^{\rm IV}_1$ corresponding to a distribution of mixed parity (i.e., an ``inconsistent'' $\mathbb Z$-grading).
Both have purely even normal bundle and we will explore the second option.

The homogeneous superspace $G(3)/P_1^{\rm IV}$ has dimension $(5|4)$ and it comes with
a contact superdistribution $\mathcal C$ of rank $(4|4)$. We first determine an invariant cone field in it,
characterizing the reduction of the structure group to $COSp(3|2)\subset CSpO(4|4)$;
note the order of letters.
In other words, at any fixed topological point $x\in G(3)/P_1^{\rm IV}$, the projectivization of $\mathcal C|_x$ contains a distinguished
subvariety $\cV|_x$ of dimension $(1|2)$. This supervariety is isomorphic to the unique irreducible flag manifold
of the simple Lie supergroup $OSp(3|2)$, namely
$$
\cV|_x\cong OSp(3|2)/P_1^{\rm II}\;,
$$
where $\mathfrak{p}_1^{\rm II}$ is the parabolic subalgebra
  \raisebox{-0.08in}{
	$
\begin{tikzpicture}
 \draw (1,0) -- (2,0);
 \node[draw,circle,inner sep=2pt,fill=mygray] at (1,0) {};
 \node[draw,circle,inner sep=2pt,fill=black] at (2,0) {};
 \node[below] at (1,-0.1) {$\times$};
 \end{tikzpicture}
 $
}.
We call it the $(1|2)$-twisted cubic, because its underlying classical manifold is a rational normal curve of degree 3 and it is super-deformed in 2 odd dimensions.

Lagrangian subspaces obtained as osculations of the field of $(1|2)$-twisted cubics $\mathcal V\subset\mathbb P(\mathcal C)$ determine
a submanifold $\cE$ in the superspace of 2-jets $J^{2}(\bbC^{2|2},\bbC^{1|0})$ and we obtain the following
extension of \eqref{G2-ct-PDE}, which we call {\em $G(3)$-contact super-PDE system}:
 \begin{equation}\label{SC-PDE}
 \begin{split}
u_{xx} &= \tfrac13 u_{yy}^3 + 2 u_{yy} u_{y\nu} u_{y\tau}, \quad
u_{xy} = \tfrac12 u_{yy}^2 + u_{y\nu} u_{y\tau}, \\
u_{x\nu}&= u_{yy} u_{y\nu}, \qquad\ u_{x\tau} = u_{yy} u_{y\tau}, \qquad\ u_{\nu\tau} = -u_{yy},
 \end{split}
 \end{equation}
where $u=u(x,y,\nu,\tau)$. See Theorem \ref{thm:superPDEG(3)}. Furthermore, we show that the contact symmetry algebra of this super-PDE system is exactly $G(3)$, see Theorem \ref{T:G3-symII}.

The models we obtain are related by the diagram of Figure \ref{F:G3-twistor}. On the left we see the $(5|4)$-dimensional superspace $G(3)/P_1^{\rm IV}$
with a field of $(1|2)$-twisted cubics in the projectivized contact distribution, that is, with a reduction of the structure group $CSpO(4|4)$ to $COSp(3|2)$.
On the right there is the $(5|6)$-dimensional superspace $G(3)/P_2^{\rm IV}$ endowed with a superdistribution of growth $(2|4,1|2,2|0)$, which corresponds to the SHC equation \eqref{SHC-ODE}.

Finally, on the top we see the $(6|6)$-dimensional superspace $G(3)/P_{12}^{\rm IV}$ equipped with the superdistribution $\cD$
of growth $(2|2,1|2,1|2,1|0,1|0)$. This is derived from the super-PDE model \eqref{SC-PDE}.
As in the $G(2)$-case, the superdistribution $\cD$ is not
the Cartan distribution of the super-PDE model considered as a submanifold $\cE$ of $J^2(\mathbb C^{2|2},\mathbb{C}^{1|0})$.
In fact, the Cartan distribution is the first derived $\cD^{2}$ of $\cD$ and has rank $(3|4)$;
we also note that $\cD^{3}$ has rank $(4|6)$, while
$\cD^{4}$ has rank $(5|6)$.

\begin{figure}[h]
 \begin{center}
 \begin{tikzpicture}
 \draw (0,0) -- (1,0);
 \draw (0,0.07) -- (1,0.07);
 \draw (0,-0.07) -- (1,-0.07);
 \draw (0.4,0.15) -- (0.6,0) -- (0.4,-0.15);
 \draw (1,0.05) -- (2,0.05);
 \draw (1,-0.05) -- (2,-0.05);
 \draw (1.6,0.15) -- (1.4,0) -- (1.6,-0.15);
 \node[draw,circle,inner sep=2pt,fill=white] at (0,0) {};
 \node[draw,circle,inner sep=2pt,fill=mygray] at (1,0) {};
 \node[draw,circle,inner sep=2pt,fill=black] at (2,0) {};
 \node[below] at (0,-0.1) {$\times$};
 \draw (4,0) -- (5,0);
 \draw (4,0.07) -- (5,0.07);
 \draw (4,-0.07) -- (5,-0.07);
 \draw (4.4,0.15) -- (4.6,0) -- (4.4,-0.15);
 \draw (5,0.05) -- (6,0.05);
 \draw (5,-0.05) -- (6,-0.05);
 \draw (5.6,0.15) -- (5.4,0) -- (5.6,-0.15);
 \node[draw,circle,inner sep=2pt,fill=white] at (4,0) {};
 \node[draw,circle,inner sep=2pt,fill=mygray] at (5,0) {};
 \node[draw,circle,inner sep=2pt,fill=black] at (6,0) {};
 \node[below] at (5,-0.1) {$\times$};
 \draw (2,1.5) -- (3,1.5);
 \draw (2,1.57) -- (3,1.57);
 \draw (2,1.43) -- (3,1.43);
 \draw (2.4,1.65) -- (2.6,1.5) -- (2.4,1.35);
 \draw (3,1.55) -- (4,1.55);
 \draw (3,1.45) -- (4,1.45);
 \draw (3.6,1.65) -- (3.4,1.5) -- (3.6,1.35);
 \node[draw,circle,inner sep=2pt,fill=white] at (2,1.5) {};
 \node[draw,circle,inner sep=2pt,fill=mygray] at (3,1.5) {};
 \node[draw,circle,inner sep=2pt,fill=black] at (4,1.5) {};
 \node[below] at (2,1.4) {$\times$};
 \node[below] at (3,1.4) {$\times$};
\path[->,>=angle 90] (2.8,1.25) edge (1,0.25);
\path[->,>=angle 90] (3.2,1.25) edge (5,0.25);
 \end{tikzpicture}
\end{center}
\caption{$G(3)$-twistor correspondence considered in this paper.}
\label{F:G3-twistor}
\end{figure}
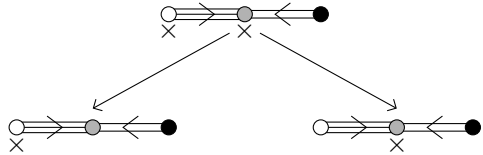

The superdistribution $\cD^{2}$ has a non-trivial Cauchy characteristic space $\Ch(\cD^{2})$, which we compute in Section \ref{CC-HC} (we remark that $\cD^2$ is simply denoted by the symbol $\cH$ in that section). It is spanned by the even supervector field
\begin{equation}
\label{eq:CCsupereven}
\bC=D_x-\lambda D_y-\theta D_\nu-\phi D_\tau\;,
\end{equation}
where $D_x,D_y, D_\tau, D_\nu$ are truncated total derivatives (explicit formulae are given in Section \ref{CC-HC}) and $\lambda=u_{yy}$, $\theta=-u_{y\tau}$, $\phi=u_{y\nu}$.
On the other hand $\cD^{4}$ has a $(1|2)$-dimensional Cauchy characteristic space
$\Ch(\cD^{4})=\langle \partial_\lambda|\partial_{\theta},\partial_{\phi}\rangle$.

The left arrow is the quotient by $\Ch(\cD^{4})$ and the contact superdistribution 
$\cC$ on $G(3)/P_1^{\rm IV}$ at the bottom left is $\cD^{4}/\Ch(\cD^{4})$.
The right arrow is the quotient by $\Ch(\cD^{2})$ and the superdistribution of rank $(2|4)$ 
on $G(3)/P_2^{\rm IV}$ at the bottom right is $\cD^{2}/\Ch(\cD^{2})$. Note that the left and right 
arrows are swapped w.r.t. the classical picture presented in Figure \ref{F:G2-twistor}; this is due to 
the reversal of the triple arrow in the Dynkin diagram of $G(3)$ w.r.t. that of $G(2)$.

The local quotient given by the left arrow exists in general only for special (in particular flat) structures.
However the local quotient to the right has no obstructions, and there is an equivalence of 
categories for the corresponding geometries, see Appendix \ref{A0}.

\subsection{Spencer cohomology of $G(3)$ and curved supergeometries}

Our proof that $\fg=G(3)$ is the internal symmetry superalgebra of \eqref{SHC-ODE}, 
resp.\ the contact symmetry superalgebra of \eqref{SC-PDE}, 
is based on the {\it explicit} determination of all the supersymmetries and on the computation 
of the Tanaka--Weisfeiler prolongation of the symbol algebra associated to the relevant distribution, 
resp.\ with a further orthosymplectic reduction $\mathfrak{cosp}(3|2)\subset\der_{gr}(\fm)$.
We emphasize that in \cite{MR3759350}, the proof that the given PDEs have symmetry realized by 
the exceptional Lie algebras was given independently from the explicit symmetry computation. 
That result was based on the theory of parabolic geometries, which is not available in the super-setting. 

Theorem \ref{thm:235H^1}, resp.\ Theorem \ref{prop:P1-cohom}, 
computes the first Spencer cohomology group $H^1(\fm,\fg)$  (see \eqref{eq:dn} for its decomposition into homogeneous components) of the negatively graded Lie superalgebra 
$\fm\subset\fg$ corresponding to the $\mathbb Z$-grading induced by $\mathfrak{p}^{\rm IV}_2$, 
resp.\ $\mathfrak{p}^{\rm IV}_1$. Vanishing in nonnegative, resp. positive, degrees 
amounts exactly to $pr(\fm)\cong G(3)$, resp. $pr(\fm,\fg_0)\cong G(3)$, 
that is, $G(3)$ is the maximal prolongation. 

In addition, we compute the second Spencer cohomology groups $H^2(\fm,\fg)$ that are classically identified
with the spaces of (fundamental) curvatures or structure functions \cite{Goncharov,MR0203626,MR1931997}.
Traditionally, this has two motivations. First, Spencer cohomology consists of compatibility constraints
of the Lie equation on symmetry and can be expressed via curvatures \cite{MR3253556}.
Second, in parabolic geometry $H^2(\fm,\fg)$ contains the complete obstructions to flatness of
the (regular, normal) Cartan connection, i.e., the so-called harmonic curvatures.
In this paper we take instead the
deformation approach, in which $H^2(\fm,\fg)$ classifies filtered deformations of graded subalgebras \cite{MR1688484};
note that the symmetry breaking mechanism of \cite{MR3604980} is also based on this deformation idea and
we apply it in the context of supergeometry. 

The relevant Lie algebra cohomologies are computed in the classical case
via Kostant's version of the Bott--Borel--Weil theorem \cite{MR0142696}, a result which does not hold in the general super-case. Cohomology groups have been known
for some irreducible (i.e, depth $\mu=1$) supergeometries \cite{MR1931997} and some
distinguished (i.e., corresponding to a Dynkin diagram with just one odd root) Borel subalgebras \cite{MR3531745}, but not for the parabolic subalgebras
of depth $\mu=2,3$  that we consider in this paper. (Note that $G(3)$ does not have $|1|$-gradings.) Our proofs use various techniques, such as the Hochschild--Serre spectral sequence for
$\mathfrak{p}^{\rm IV}_1$ and a combination of different exact sequences with the representation theory of $\mathfrak{osp}(1|2)$ for
$\mathfrak{p}^{\rm IV}_2$. 

In the latter case, it is intriguing that $H^{2,2}(\fm,\fg)\cong S^2\mathbb C^2$, which yields a ``square root'' of Cartan's classical  binary quartic invariant for $(2,3,5)$-distributions, see Theorem \ref{thm:SC2IV}. A similar phenomenon has been observed in the context of    supergravity \cite{MR3594366}. For the second cohomology group of the graded Lie superalgebra associated to $\mathfrak{p}^{\rm IV}_1$, see Theorem  \ref{prop:P1-cohom}. 

Note that comparing (super)dimensions $(p|q)$ has different meanings in the literature: sometimes
the maximal dimension is understood in the even sense (max $p$), sometimes in the odd sense (max $q$)
or in the total sense (max $p+q$). In this paper, we use the stronger notion of partial
order: $(p'|q')\leq(p|q)$ iff $p'\leq p$ and $q'\leq q$. In this sense, we prove in Theorem \ref{T:G3-sym} 
that $(17|14)$ is the maximal supersymmetry dimension for (locally transitive) $G(3)$-contact supergeometries. 

Finally, we discuss curved geometries modelled on the homogeneous superspace 
$G(3)/P_2^{\rm IV}$. We consider distributions in $(5|6)$-dimensional superspaces
with growth $(2|4,1|2,2|0)$ and prove
that the SHC symbol is rigid, i.e., given this growth (plus some mild non-degeneracy conditions),
the graded Lie superalgebra structure is unique. See Theorem \ref{RigThm} and Corollary \ref{Rigidity}.
Contrary to the case of rank $2$ distributions in $5$-dimensional spaces,
the generic rank $(2|4)$ distributions in $(5|6)$-superspaces have depth $\mu=2$, 
so the distributions with the indicated growth are {\it not} generic. 
We address the restrictions this puts on their even part.

We investigate integral submanifolds of general superdistributions with the growth vector $(2|4,1|2,2|0)$
in Theorem \ref{thm:cointegrals} and notice a difference it makes with the even case.

Then we observe a supersymmetry gap phenomenon in Theorem \ref{UpperBound}: 
The maximal supersymmetry dimension of (locally transitive) distributions with growth vector $(2|4,1|2,2|0)$
of SHC type is $(17|14)$, and among all such distributions
any symmetry superalgebra different from $G(3)$ has dimension at most $(10|8)$. 

Finally, we show in Theorem \ref{super-sub-max} that the following deformation of the SHC 
 $$
z_x = f(u_{xx}) + u_{x\nu}u_{x\tau},\quad
z_{\nu}  = f'(u_{xx}) u_{x\nu},\quad
z_{\tau} = f'(u_{xx}) u_{x\tau},\quad
u_{\nu\tau} = -f'(u_{xx}),
 $$
gives a realization of the above dimension bound whenever the function $f$ of one (even) variable  
is $f(s)=\int s^k\,ds$ with $k\neq-2,-\frac23,-\frac13,0,1$. 
These non-flat models can be considered as super-extensions of Cartan's classical 
submaximally symmetric $G(2)/P_1$ structures.

\subsection{Structure of the paper and future directions}

In Section \ref{S2} we recall the basics of $G(3)$,
its parabolic subalgebras and $\mathbb Z$-gradings.  Figure \ref{F:G3-supergeometries} gives all the $19$ generalized flag varieties of $G(3)$ and twistor correspondences, further discussed in Appendix \ref{A0}. 
(The diagram is complete in the flat case, while for curved geometries some arrows may disappear.) Associated to the $G(3)$-contact case, i.e., to the flag superspace $G(3)/P_1^{\rm IV}$, we compute the supervariety $\cV$, its osculations, and super-symmetric forms on a naturally associated Jordan superalgebra, leading to a collection $\widehat\cV$ of Lagrangian
subspaces along $\cV$.

Section \ref{S3} is devoted to cohomology -- this is an important ingredient in the proof that $G(3)$ is 
the symmetry superalgebra of the two main differential equations that we will derive in Section \ref{S4}. 
Some technical cohomological computations are postponed to the Appendix \ref{appendixA}.
We then explain in Section \ref{S4} the relation between the two differential equations and give 
the explicit expression of supersymmetries:
we encode them by the generating function of the
contact vector field in the case of the super-PDE \eqref{SC-PDE} and we delegate the formulae to
Appendix \ref{S:SHC-sym} in the case of the SHC equation \eqref{SHC-ODE}. 

Finally in Section \ref{S5} we discuss curved geometries of type $G(3)/P_2^{\rm IV}$:
their symbol, genericity and in which respects they differ from SHC.  
The submaximally symmetric models are derived at the end of this section, and their 
supersymmetries are explicitly given.

Throughout the paper, we will work 
with Lie superalgebras over the complex field  and freely exponentiate to the corresponding 
Lie supergroups and homogeneous superspaces via the functor of points 
(see Section \ref{S:super-V} and, e.g., \cite{MR2840967} for more details). 

In forthcoming works we will develop the theory of Cartan connections and parabolic geometries in the super-setting. In particular, this will allow us to deal with curved geometries with an intransitive symmetry superalgebra as well investigate the precise geometric relationship between our fundamental binary quadratic invariant and Cartan's classical binary quartic. 
Geometries modelled on other simple Lie superalgebras are also
of importance, e.g., the Lie superalgebra $F(4)$ is popular due to
its relation to conformal field theories. 
This work proposes $G(3)$ as the supersymmetry of differential equations.  The relation of our construction to the twistor spinors associated to Nurowski's conformal metrics 
will be discussed elsewhere.

 \section{Algebraic aspects and parabolic subalgebras of $G(3)$}\label{S2}

 \subsection{Root systems and Dynkin diagrams of $G(3)$}
 \label{S:G3-rts}

 The (complex) Lie superalgebra (LSA) $\fg = G(3)$ has dimension $(17|14)$, with even and odd parts:
 \begin{align}
 \fg_{\bar{0}} = G(2) \oplus A(1), \qquad
 \fg_{\bar{1}} = \bbC^7 \boxtimes \bbC^2. \label{G3}
 \end{align}
Here, we use the notation $G(2)$ and $A(1)\cong \fsp(2)$ to denote complex simple Lie algebras, while $\fg_{\bar{1}}$ is the $\fg_{\bar{0}}$-representation that is the (external) tensor product of the standard $G(2)$ and $A(1)$ representations. The somewhat unusual notation $\fsp(2)$ will be reserved specifically to the ideal $A(1) \subset \fg_{\bar{0}}$ throughout the whole paper, to avoid confusion with other $\fsl(2)$-subalgebras.

 A Cartan subalgebra $\fh$ of $\fg$ is by definition a Cartan subalgebra for $\fg_{\bar{0}}$.  All are conjugate, so we fix one such choice.  We adopt the root conventions in \cite[\S 2.19]{MR1773773}.  Fix vectors $\delta,\epsilon_1,\epsilon_2,\epsilon_3$ in $\fh^*$ with $\epsilon_1 + \epsilon_2 + \epsilon_3 = 0$ such that $\langle \epsilon_i, \epsilon_j \rangle = 1 - 3\delta_{ij}$, $\langle \delta, \delta \rangle = 2$ and $\langle \epsilon_i, \delta \rangle = 0$. The $G(3)$ root system $\Delta = \Delta_{\bar{0}} \cup \Delta_{\bar{1}} \subset \fh^* \backslash \{ 0 \}$ is given by:
  \[
 \Delta_{\bar{0}} = \{ \pm 2 \delta, \pm \epsilon_i, \epsilon_i - \epsilon_j\}, \quad
 \Delta_{\bar{1}} = \{ \pm \delta, \pm \delta \pm \epsilon_i  \},
 \]
where $1\leq i\neq j\leq 3$. Given any even root $\alpha$ one has $\langle\alpha,\alpha\rangle\neq0$, 
and the usual reflection
 \begin{equation}
\label{eq:evenreflection}
S_\alpha(\beta) =
 \beta - \frac{2\langle \beta, \alpha \rangle}{\langle \alpha, \alpha \rangle} \alpha
\end{equation}
on $\fh^*$ preserves each of $\Delta_{\bar{0}}$ and $\Delta_{\bar{1}}$. 
The Weyl group of $\fg$ is generated by all such even reflections.
For any fixed simple root system $\Pi$ and any odd isotropic root $\alpha\in\Pi$, we define the odd reflection (see \cite{MR2743764})
 \begin{equation}
\label{eq:oddreflection}
S_\alpha(\beta) = \begin{cases}
 \beta + \alpha, & \langle\alpha, \beta \rangle \neq 0;\\
  \beta, & \langle\alpha, \beta \rangle = 0, \, \beta \neq \alpha;\\
  -\alpha, & \beta = \alpha;
  \end{cases}
 \end{equation}
for any $\beta\in\Pi$.

Up to Weyl group equivalence, there are four inequivalent simple systems $\Pi = \{ \alpha_1,\alpha_2, \alpha_3 \}$, and each leads to a Cartan matrix and corresponding Dynkin diagram -- see e.g. \cite[\S 3]{MR2827478} for details.  These diagrams are given in Figure \ref{F:simple-rts}.
 \begin{itemize}
 \item Nodes are white, black, or grey according to whether the corresponding simple root $\alpha_i$ is even, odd with $\langle \alpha_i,\alpha_i \rangle \neq 0$, or odd with $\langle \alpha_i,\alpha_i \rangle = 0$;
 \item Dynkin labels $m_1,m_2,m_3$ inscribed above each node correspond to the highest root $\alpha_{high}=m_1\alpha_1 + m_2\alpha_2 + m_3\alpha_3$.
 \end{itemize}
One cannot extend the Weyl group to a larger group that includes reflections for isotropic odd roots, since the latter cannot in general be extended to linear transformations of $\fh^*$ that send roots into roots.  Nevertheless, applying \eqref{eq:oddreflection} to any $\beta\in\Pi$ transforms one simple root system to another, and red arrows in Figure \ref{F:simple-rts} indicate such transformations.

  \begin{figure}[h]
 \begin{center}
 \[
 \begin{array}{|c|c|c|c|} \hline
 I & II & III & IV\\ \hline
 \raisebox{-0.1in}{
 \begin{tikzpicture}[remember picture]
 \draw (0,0) -- (1,0);
 \draw (1,0) -- (2,0);
 \draw (1,0.07) -- (2,0.07);
 \draw (1,-0.07) -- (2,-0.07);
 \draw (1.6,0.15) -- (1.4,0) -- (1.6,-0.15);
 \node[draw,circle,inner sep=2pt,fill=mygray] (A1) at (0,0) {};
 \node[draw,circle,inner sep=2pt,fill=white] at (1,0) {};
 \node[draw,circle,inner sep=2pt,fill=white] at (2,0) {};
 \node[above] at (0,0.25) {2};
 \node[above] at (1,0.25) {4};
 \node[above] at (2,0.25) {2};
 \node[below] at (0,-0.25) {$\alpha_1$};
 \node[below] at (1,-0.25) {$\alpha_2$};
 \node[below] at (2,-0.25) {$\alpha_3$};
 \end{tikzpicture}}
 & \raisebox{-0.1in}{
 \begin{tikzpicture}[remember picture]
 \draw (0,0) -- (1,0);
 \draw (1,0) -- (2,0);
 \draw (1,0.07) -- (2,0.07);
 \draw (1,-0.07) -- (2,-0.07);
 \draw (1.6,0.15) -- (1.4,0) -- (1.6,-0.15);
 \node[draw,circle,inner sep=2pt,fill=mygray] (B1) at (0,0) {};
 \node[draw,circle,inner sep=2pt,fill=mygray] (B2) at (1,0) {};
 \node[draw,circle,inner sep=2pt,fill=white] at (2,0) {};
 \node[above] at (0,0.25) {3};
 \node[above] at (1,0.25) {4};
 \node[above] at (2,0.25) {2};
 \node[below] at (0,-0.25) {$\alpha_1$};
 \node[below] at (1,-0.25) {$\alpha_2$};
 \node[below] at (2,-0.25) {$\alpha_3$};
 \end{tikzpicture}}
 & \raisebox{-0.3in}{
 \begin{tikzpicture}[remember picture]
 \draw (0.5,0) -- (1.5,0);
 \draw (0.5,0.07) -- (1.5,0.07);
 \draw (0.5,-0.07) -- (1.5,-0.07);
 \draw (0.5,0) -- (1,1) -- (1.5,0);
 \draw (1.52,0.11) -- (1.02,1.1);
 \node[draw,circle,inner sep=2pt,fill=mygray] (C1) at (0.5,0) {};
 \node[draw,circle,inner sep=2pt,fill=mygray] (C2) at (1.5,0) {};
 \node[draw,circle,inner sep=2pt,fill=white] at (1,1) {};
 \node[above] at (0.5,0.2) {2};
 \node[above] at (1.5,0.2) {2};
 \node[above] at (1,1.2) {3};
 \node[below] at (0.5,-0.2) {$\alpha_1$};
 \node[below] at (1.5,-0.2) {$\alpha_2$};
 \node[below] at (1,0.8) {$\alpha_3$};
 \end{tikzpicture}}
 & \raisebox{-0.1in}{
 \begin{tikzpicture}[remember picture]
 \draw (0,0) -- (1,0);
 \draw (0,0.07) -- (1,0.07);
 \draw (0,-0.07) -- (1,-0.07);
 \draw (0.4,0.15) -- (0.6,0) -- (0.4,-0.15);
 \draw (1,0.05) -- (2,0.05);
 \draw (1,-0.05) -- (2,-0.05);
 \draw (1.6,0.15) -- (1.4,0) -- (1.6,-0.15);
 \node[draw,circle,inner sep=2pt,fill=white] at (0,0) {};
 \node[draw,circle,inner sep=2pt,fill=mygray] (D2) at (1,0) {};
 \node[draw,circle,inner sep=2pt,fill=black] at (2,0) {};
 \node[above] at (0,0.15) {2};
 \node[above] at (1,0.15) {3};
 \node[above] at (2,0.15) {3};
 \node[below] at (0,-0.25) {$\alpha_1$};
 \node[below] at (1,-0.25) {$\alpha_2$};
 \node[below] at (2,-0.25) {$\alpha_3$};
 \end{tikzpicture}} \\ \hline
 \begin{array}{l}
 \alpha_1 = \delta + \epsilon_3\\
 \alpha_2 = \epsilon_1\\
 \alpha_3 = \epsilon_2 - \epsilon_1
 \end{array}
 & \begin{array}{l}
 \alpha_1 = -\delta - \epsilon_3\\
 \alpha_2 = \delta - \epsilon_2\\
 \alpha_3 = \epsilon_2 - \epsilon_1
 \end{array}
 &
 \begin{array}{l}
 \alpha_1 = -\delta + \epsilon_2\\
 \alpha_2 = \delta - \epsilon_1\\
 \alpha_3 = \epsilon_1
 \end{array}
 &
 \begin{array}{l}
 \alpha_1 = \epsilon_2 - \epsilon_1\\
 \alpha_2 = \epsilon_1 - \delta\\
 \alpha_3 = \delta
 \end{array}\\ \hline
 \end{array}
 \]
 \begin{tikzpicture}[overlay, red, remember picture]
\draw[<->, ultra thick, densely dotted] (A1) to[in=-130, out = -50] (B1);
\draw[<->, ultra thick, densely dotted] (B2) to[in=-150, out = -50] (C1);
\draw[<->, ultra thick, densely dotted] (C2) to[in=-130, out = -30] (D2);
 \end{tikzpicture}
 \end{center}
 \caption{Inequivalent simple root systems for $G(3)$}
 \label{F:simple-rts}

 \end{figure}

 \subsection{Parabolic subalgebras and a map of $G(3)$-supergeometries}
\label{subsec:2.2}

  A {\em $\ZZ$-grading} of $\fg$ is a decomposition $\fg = \bigoplus_{k \in \ZZ} \fg_k$ satisfying $[\fg_i,\fg_j] \subset \fg_{i+j}$ for all $i,j \in \ZZ$.  In particular, $\fg_0$ is a LSA and each $\fg_k$ is a $\fg_0$-module.  Since the Killing form on $G(3)$ is non-degenerate, then $\fg_k = (\fg_{-k})^*$ as $\fg_0$-modules.
 The corresponding parabolic subalgebra is $\fp = \fg_{\geq 0}= \bigoplus_{k\geq 0} \fg_k$, and $\fg_-$ is the associated {\em symbol algebra}, a nilpotent graded LSA.  Letting $\der_{gr}(\fg_-)$ denote the LSA of (super-)derivations of $\fg_-$ of zero degree, we have $\fg_0 \subset \der_{gr}(\fg_-)$.  Moreover, for these gradings $\fg_{-1}$ is bracket-generating, i.e., \ $\fg_{-1}$ generates all of $\fg_-$ by iteratively bracketing with $\fg_{-1}$, so $\der_{gr}(\fg_-) \hookrightarrow \mathfrak{gl}(\fg_{-1})$.

Such $\ZZ$-gradings are obtained from a choice of {\em grading element $\sfZ \in \fh$}. 
 \comm{
Namely, define $\fg_k = \bigoplus_{\alpha(\sfZ)=k} \fg_\alpha \oplus(\delta_{0,k}\!\cdot\!\fh)$, 
where $\delta_{0,k}$ is the Kronecker delta and $\fg_\alpha$ is the root space for $\alpha \in \Delta$. 
 }
Namely, define $\fg_k = \{ x \in \fg : [\sfZ, x] = k x \}$ for any $k\in\mathbb Z$.  Note that $\sfZ \in\fh \subset\fg_0$ and 
the root space $\fg_\alpha \subset \fg_k$ for $\alpha \in \Delta$ such that
$\alpha(\sfZ)=k$. It follows that $\sfZ \in \fz(\fg_0)$. 
If $\mu = \max\{ k : \fg_k \neq 0 \}$, then 
$\mu = \alpha_{high}(\sfZ)$ and $\fg$ is said to have a {\em $|\mu|$-grading}.
Given a simple root system $\{ \alpha_1,\alpha_2, \alpha_3 \}$, let $\{ \sfZ_1, \sfZ_2, \sfZ_3 \} \subset \fh$ be its dual basis.  Then $\sfZ = \sum_{i \in \cA} \sfZ_i$ specifies a grading element for any nonempty subset $\cA \subset \{ 1,2,3 \}$, and in turn a parabolic subalgebra $\fp_\cA$.  All non-trivial gradings of $G(3)$ are of this form, varying the choice of simple root system  (labelled $I$ to $IV$ as in Figure  \ref{F:simple-rts}) \cite{MR0498755}.

We refer to $M_{\cA} = G / P_\cA$ as a {\em $G(3)$-supergeometry}, where $G$ and $P_{\cA}$ are (connected) Lie supergroups corresponding to $\fg$ and $\fp_\cA$ respectively. If $\cB \subset \cA$, there are natural fibrations $M_\cA \to M_\cB$.  Considering each simple root system from Figure \ref{F:simple-rts} leads to the map of $G(3)$-supergeometries given in Figure \ref{F:G3-supergeometries},
see Appendix \ref{A0} for further details.

\begin{figure}[h]
 \begin{center}
 \begin{tikzpicture}
    \node[main node] (T1) {$M_{123}^{\rm I}$};
    \node[main node] (T2) [right = 2cm of T1] {$M_{123}^{\rm II}$};
    \node[main node] (T3) [right = 2cm of T2] {$M_{123}^{\rm III}$};
    \node[main node] (T4) [right = 2cm of T3] {$M_{123}^{\rm IV}$};

    \node[main node] (M1) [below left = 1.3cm and 0.05cm of T1]  {$M_{12}^{\rm I}$};
    \node[main node] (M2) [below = 1.3cm of T1] {$M_{13}^{\rm I}$};
    \node[main node] (M3) [below right = 1.3cm and 0.05cm of T1] {$M_{23}^I = M_{23}^{\rm II}$};
    \node[main node] (M4) [below = 1.3cm of T2]  {$M_{12}^{\rm II}$};
    \node[main node] (M5) [below right = 1.3cm and 0.075cm of T2]  {$M_{13}^{\rm II} = M_{23}^{\rm III}$};
    \node[main node] (M6) [below = 1.3cm of T3]  {$M_{12}^{\rm III}$};
    \node[main node] (M7) [below right = 1.3cm and 0.05cm of T3]  {$M_{13}^{\rm III}=M_{13}^{\rm IV}$};
    \node[sec node] (M8) [below = 1.3cm of T4]  {$M_{12}^{\rm IV}$};
    \node[main node] (M9) [below right = 1.3cm and 0.1cm of T4]  {$M_{23}^{\rm IV}$};

    \node[main node] (B1) [below = 2cm of M1]  {$M_1^{\rm I}$};
    \node[main node] (B2) [below right = 2cm and -0.75cm of M2]  {$M_2^{\rm I} = M_2^{\rm II}$};
    \node[main node] (B3) [below = 2cm of M4]  {$M_3^{\rm I} = M_3^{\rm II} = M_2^{\rm III}$};
    \node[main node] (B4) [below = 2cm of M6]  {$M_1^{\rm II} = M_3^{\rm III} = M_3^{\rm IV}$};
    \node[sec node] (B5) [below left = 2cm and 0.1cm of M9]  {$M_1^{\rm III} = M_1^{\rm IV}$};
    \node[sec node] (B6) [below = 2cm of M9]  {$M_2^{\rm IV}$};

    \node [below=0.1cm of B5] {{\tiny {\color{red} $G(3)$-contact}}};
    \node [below=0.1cm of B6] {{\tiny {\color{red} SHC}}};

    \path[draw,thick]
    (T1) edge node {} (M1)
    (T1) edge node {} (M2)
    (T1) edge node {} (M3)
    (T2) edge node {} (M3)
    (T2) edge node {} (M4)
    (T2) edge node {} (M5)
    (T3) edge node {} (M5)
    (T3) edge node {} (M6)
    (T3) edge node {} (M7)
    (T4) edge node {} (M7)
    (T4) edge node {} (M8)
    (T4) edge node {} (M9)
    (M1) edge node {} (B1)
    (M1) edge node {} (B2)
    (M2) edge node {} (B1)
    (M2) edge node {} (B3)
    (M3) edge node {} (B2)
    (M3) edge node {} (B3)
    (M4) edge node {} (B2)
    (M4) edge node {} (B4)
    (M5) edge node {} (B3)
    (M5) edge node {} (B4)
    (M6) edge node {} (B3)
    (M6) edge node {} (B5)
    (M7) edge node {} (B4)
    (M7) edge node {} (B5)
    (M8) edge node {} (B5)
    (M8) edge node {} (B6)
    (M9) edge node {} (B4)
    (M9) edge node {} (B6);

\end{tikzpicture}
\end{center}
\caption{Map of $G(3)$-supergeometries. Green nodes are the focus of this article.}
\label{F:G3-supergeometries}
\end{figure}
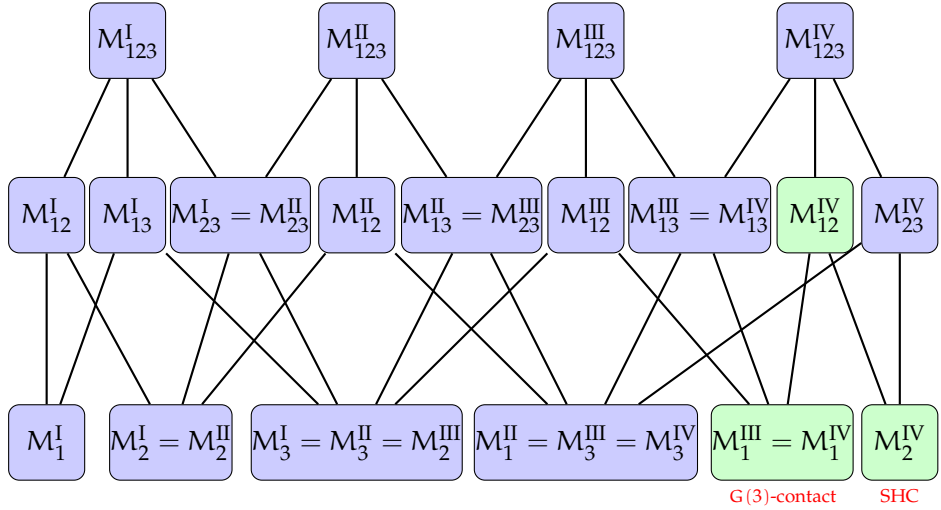

For each $G(3)$-supergeometry, it is natural to search for some explicit geometric structures whose symmetry algebra is precisely $G(3)$.  Our goal is to carry out this program in two cases, labelled in Figure \ref{F:G3-supergeometries} as $G(3)$-contact and SHC.

 \subsection{SHC and contact gradings}
 \label{S:G3-gradings}

 We will work with $\Pi = \{ \alpha_1, \alpha_2, \alpha_3 \}$ labelled IV in Figure \ref{F:simple-rts}.  This gives an associated positive $\Delta^+$ (resp. negative $\Delta^-$) system of roots. Explicitly,
 \begin{align*}
 \Delta_{\bar{0}}^+: & \quad
\alpha_1,\ 2\alpha_3,\ \alpha_2 + \alpha_3,\ \alpha_1 + \alpha_2 + \alpha_3,\ 
\alpha_1 + 2 \alpha_2 + 2 \alpha_3,\ \alpha_1 + 3 \alpha_2 + 3\alpha_3,\
2\alpha_1 + 3 \alpha_2 + 3\alpha_3;\\
 \Delta_{\bar{1}}^+: & \quad
\alpha_2,\ \alpha_3,\ \alpha_1 + \alpha_2,\ \alpha_2 + 2\alpha_3,\ \alpha_1 + \alpha_2 + 2\alpha_3,\ 
\alpha_1 + 2\alpha_2 + \alpha_3,\ \alpha_1 + 2\alpha_2 + 3\alpha_3;
 \end{align*}
with $\Delta_{i}^- =-\Delta_{i}^+$, $i\in\mathbb Z_2$.
 While there are many possibilities for $\sfZ$, we will focus on two choices:
 \begin{itemize}
 \item {\em contact grading} defined by $\sfZ = \sfZ_1$, which is a $|2|$-grading;
 \item {\em SHC (``super Hilbert--Cartan'') grading} defined by $\sfZ = \sfZ_2$, which is a $|3|$-grading.
 \end{itemize}
 These gradings closely parallel those considered in the classical cases in \cite{MR3759350, MR1699860}, and indeed
the even parts $(\fg_-)_{\bar{0}}$ give precisely the symbol algebras considered there. The above choices of $\Pi$ and $\sfZ$ were motivated from the following distinguishing feature in the classical cases:
\vskip0.3cm\par\noindent
{\it Let $\fg$ be a complex simple Lie algebra not of type $A$ or $C$.  The adjoint representation of $\fg$ has highest weight (root) that is a fundamental weight $\lambda_k = \alpha_{high} = \sum_i m_i \alpha_i$ with $m_k = 2$ and $\sum_{i \in \mathcal{N}_k} m_i = 3$, where $\mathcal{N}_k$ are the neighbours to the $k^{\text{th}}$-node (excluding the $k^{\text{th}}$-node) in the Dynkin diagram. The element $\sfZ = \sfZ_k$ defines a contact grading, while $\sfZ = \sum_{i \in \mathcal{N}_k} \sfZ_i$ defines a $|3|$-grading  generalizing the one for classical $(2,3,5)$-geometries (in particular $\dim\fg_{-3}=2$ as for the Hilbert--Cartan grading).}
\vskip0.3cm\par
For both gradings above, $\fz(\fg_0) = \tspan\{ \sfZ \} \subset (\fg_0)_{\bar{0}}$.  Below are the roots organized by parity and grading (with $\Delta_{i}(k) = \{ \alpha \in \Delta_{i} : \alpha(\sfZ) = k \}$ for $i\in\mathbb Z_2$), along with the module structure for the semisimple part $(\fg_0)_{\bar{0}}^{\ss}$ of $(\fg_0)_{\bar{0}}$.
 \vskip0.2cm\par\noindent
{\underline{\it Contact grading}: $\mathfrak{p}_1^{\rm IV}$}
 \begin{align} \label{ct-grading}
 \begin{array}{|c|c|c|} \hline
 k & \Delta_{\bar{0}}(k)  &  \Delta_{\bar{1}}(k) \\ \hline\hline
 0 & \pm (\alpha_2 + \alpha_3),\,\, \pm 2\alpha_3 & \pm \alpha_2, \,\, \pm \alpha_3, \,\, \pm (\alpha_2 + 2\alpha_3)\\ \hline
 1 & \begin{array}{l} \alpha_1, \,\, \alpha_1 + \alpha_2 + \alpha_3, \\ \alpha_1 + 2\alpha_2 + 2\alpha_3, \,\, \alpha_1 + 3\alpha_2 + 3\alpha_3 \end{array} & \begin{array}{l} \alpha_1 + \alpha_2, \quad \alpha_1 + \alpha_2 + 2\alpha_3, \\  \alpha_1 + 2\alpha_2 + \alpha_3,\,\, \alpha_1 + 2\alpha_2 + 3\alpha_3 \end{array}\\ \hline
 2 & 2\alpha_1 + 3\alpha_2 + 3\alpha_3 & \\ \hline
 \end{array}
 \end{align}
 \begin{align}  \label{E:contact-m}
 \begin{array}{|c|c|c|c|} \hline
 k & (\fg_k)_{\bar{0}} & (\fg_k)_{\bar{1}} & \mbox{dim}\\ \hline\hline
 0 & \bbC \op \fsl(2) \op \fsp(2) & S^2 \bbC^2 \boxtimes \bbC^2 & 7|6\\ \hline
 -1 & S^3 \bbC^2 \boxtimes \bbC & \bbC^2 \boxtimes \bbC^2 & 4|4\\ \hline
 -2 & \bbC \boxtimes \bbC &  & 1|0 \\ \hline
 \end{array}
 \end{align}
 Note that $\alpha_2 + \alpha_3 = \epsilon_1$ and $2\alpha_3 = 2\delta$ are the positive roots of $\fsl(2)$ and $\fsp(2)$, respectively.  The bracket $\Lambda^2\fg_{-1} \to \fg_{-2}$ yields a $\fg_0$-invariant conformal symplectic-orthogonal structure on $\fg_{-1}$, so we can naturally view $\fg_0 \subset \fcspo(\fg_{-1})\cong \fcspo(4|4)$.  We will make this explicit in Section \ref{S:eta}.  Also, $\fg_0=\bbC \sfZ_1\oplus\ff\cong\mathfrak{cosp}(3|2)$, where
\begin{equation}
\ff=(\fg_0)_{\bar 0}^{\ss}\oplus(\fg_0)_{\bar 1}\cong\mathfrak{osp}(3|2)
\label{eq:LSAff}
\end{equation}
is the semisimple part of $\fg_0$.

 \vskip0.2cm\par\noindent
{\underline{\it SHC grading:} $\mathfrak{p}_2^{\rm IV}$}
  \begin{align} \label{HC-grading}
 \begin{array}{|c|c|c|} \hline
 k & \Delta_{\bar{0}}(k) &  \Delta_{\bar{1}}(k) \\ \hline\hline
 0 & \pm\alpha_1, \pm 2\alpha_3 & \pm \alpha_3\\\hline
 1 & \begin{array}{l} \alpha_2 + \alpha_3, \,\, \alpha_1 + \alpha_2 + \alpha_3 \end{array} &
 \begin{array}{l}
 \alpha_2, \,\, \alpha_1 + \alpha_2,\,\,
 \alpha_2 + 2\alpha_3, \,\, \alpha_1 + \alpha_2 + 2\alpha_3\end{array}\\ \hline
 2 & \alpha_1 + 2 \alpha_2 + 2 \alpha_3 & \alpha_1 + 2\alpha_2 + \alpha_3,\,\, \alpha_1 + 2\alpha_2 + 3\alpha_3\\ \hline
 3 & \alpha_1 + 3 \alpha_2 + 3\alpha_3,\,\,2\alpha_1 + 3 \alpha_2 + 3\alpha_3 & \\ \hline
 \end{array}
 \end{align}
 \begin{align}
 \begin{array}{|c|c|c|c|} \hline
 k & (\fg_k)_{\bar{0}} & (\fg_k)_{\bar{1}} & \mbox{dim}\\ \hline\hline
 0 & \bbC \op \fsl(2) \op \fsp(2) & \bbC \boxtimes \bbC^2 & 7|2\\ \hline
 -1 & \bbC^2 \boxtimes \bbC & \bbC^2 \boxtimes \bbC^2 & 2|4\\ \hline
 -2 & \bbC \boxtimes \bbC & \bbC \boxtimes \bbC^2 & 1|2 \\ \hline
 -3 & \bbC^2 \boxtimes \bbC &  & 2|0\\ \hline
 \end{array}
 \end{align}
 Note that $\alpha_1= \epsilon_2-\epsilon_1$ and $2\alpha_3 = 2\delta$ are the positive roots of $\fsl(2)$ and $\fsp(2)$ respectively, and $\fg_0 \cong \bbC \op \fsl(2) \op \fosp(1|2)$.

 \begin{prop} \label{P:max-subalg}
 For the contact grading of $G(3)$, the subalgebra $\fg_0 \subset \fcspo(\fg_{-1})$ is a maximal subalgebra.
 For the SHC grading of $G(3)$, we have $\fg_0 \cong \der_{gr}(\fg_-)$.
 \end{prop}

 \begin{proof}
To establish the first claim, it suffices to show that $\ff = \fosp(3|2) \subset \fk := \fspo(4|4)$ is a maximal subalgebra.  The decompositions of $\ff$ and $\fk$  into even and odd parts are\footnote{We use the symbol $\boxtimes$ to denote the external tensor product of $\fsl(2)$ and $\fsp(2)$ representations.}
 \begin{align}
 \ff &= \ff_{\bar{0}} \op \ff_{\bar{1}} \cong (\fsl(2) \op \fsp(2)) \op (S^2 \bbC^2 \boxtimes \bbC^2),\\
 \fk &= \fk_{\bar{0}} \op \fk_{\bar{1}} \cong (\fsp(4) \op \fso(4)) \op (\bbC^4 \boxtimes \bbC^4),
 \end{align}
and $\ff_{\bar{0}} \hookrightarrow \fk_{\bar{0}}$ via the action on $\fg_{-1}=(\fg_{-1})_{\bar 0}\oplus(\fg_{-1})_{\bar 1} \cong (S^3\bbC^2 \boxtimes \bbC) \op (\bbC^2 \boxtimes \bbC^2)$. Note that  $\fso(4) \cong \fsl(2) \op \fsp(2)$, with $\fsp(2) \subset \ff_{\bar{0}}$ embedded purely in the latter factor $\fsp(2)$. On the other hand,
$\fsl(2) \subset \ff_{\bar{0}}$ is diagonally embedded in $\fsp(4) \op \fsl(2)$ and we denote its image by $\fsl(2)_{diag}$.
 \vskip0.1cm\par\noindent
{\underline{\it Step 1.}}
We first claim that the unique subalgebra $\widetilde\fh$ properly contained between $\ff_{\bar{0}}$ and $\fk_{\bar{0}}$ is
$\widetilde\fh=\fsl(2)\oplus \fso(4)$, where $\fsl(2)\subset \fsp(4)$ via the irreducible action on $S^{3}\bbC^2$.

Indeed, if $\widetilde\fh$ is a subalgebra such that $\ff_{\bar{0}} \subsetneq \widetilde\fh\subsetneq \fk_{\bar{0}}$ then $\widetilde\fh=\fh\oplus\fsp(2)$ for some subalgebra
$$\fsl(2)_{diag}\subsetneq\fh\subsetneq \fsp(4)\oplus\fsl(2)\;,$$
where $\fsl(2)$ is the first factor of $\fso(4)$. However $\fsp(4)\op\fsl(2)\cong (\fsl(2)\oplus S^6\bbC^2)\oplus\fsl(2)_{diag}$ as an $\fsl(2)_{diag}$-module, where $\fsp(4)\cong \fsl(2)\oplus S^6\bbC^2$. Since $S^6\bbC^2$ is an $\fsl(2)$-module and $\fsp(4)$ is a simple Lie algebra,
we immediately see that $S^{6}\bbC^2$ is not a subalgebra of $\fsp(4)$.

By the previous discussion
$$
\fh=\begin{cases}\fsl(2)\oplus\fsl(2)_{diag}=\fsl(2)\oplus\fsl(2),\;\text{or}\\
S^{6}\bbC^2\oplus\fsl(2)_{diag},
\end{cases}
$$
but the second case does not correspond to any subalgebra of $\fsp(4)\oplus\fsl(2)$, hence the claim.

%
 \vskip0.2cm\par\noindent
{\underline{\it Step 2.}}
As $\ff_{\bar{0}}$-modules,
 \[
\fk_{\bar{1}} \cong (S^3 \bbC^2 \boxtimes \bbC) \otimes (\bbC^2 \boxtimes \bbC^2) \cong
(S^4 \bbC^2 \boxtimes \bbC^2 ) \op (S^2 \bbC^2 \boxtimes \bbC^2),
 \]
 where the last isomorphism gives the decomposition into irreducibles.

 Assume $\widetilde\ff$ is a subalgebra such that $\ff \subsetneq \widetilde\ff \subsetneq \fk$, so we get corresponding inclusions of their even and odd parts. Since $\widetilde\ff_{\bar{1}}$ is an $\ff_{\bar{0}}$-module, then $\widetilde\ff_{\bar{1}} = \fk_{\bar{1}}$ or $\widetilde\ff_{\bar{1}} = \ff_{\bar{1}}$ by Schur's lemma and we consider the two possibilities separately:
 \begin{enumerate}
 \item[(a)] $\widetilde\ff_{\bar{1}} = \fk_{\bar{1}}$. In this case $\widetilde\ff_{\bar{0}} = \ff_{\bar{0}}$ or $\widetilde\ff_{\bar{0}} = \widetilde\fh$, thus
$[\fk_{\bar{1}},\fk_{\bar{1}}] \subset \widetilde\fh$.  However, $[\fk_{\bar{1}},\fk_{\bar{1}}] = \fk_{\bar{0}}$ since $\fk= \fspo(4|4)$ is simple.
 \item[(b)] $\widetilde\ff_{\bar{1}} = \ff_{\bar{1}}$. Here $\widetilde\ff_{\bar{0}} = \fk_{\bar{0}}$ or $\widetilde\ff_{\bar{0}} = \widetilde\fh$
and in both cases $\ff_{\bar{1}}\subset \fk_{\bar{1}}$ is $\widetilde\fh$-invariant. However $\fk_{\bar{1}}$ is $\widetilde\fh$-irreducible.
\end{enumerate}
 In both cases, we obtained contradictions, so $\ff \subset \fk$ is maximal, hence the first claim is proven.

The second claim follows from our Theorem \ref{thm:235H^1} (in particular, $H^{0,1}(\fg_-,\fg) = 0$).
 \end{proof}

Despite tensor fields on supermanifolds $M=(M_o,\mathcal A_M)$ not being determined by their values at points $x\in M_o$, the $G$-invariant geometric structures on the homogeneous supermanifold $M=G/P$ correspond bijectively to $\fp$-invariant data on $\fg/\fp$, cf. \cite{MR2640006, MR2541343}. For the contact and SHC cases above, we will describe geometric structures given by:
 \begin{itemize}
 \item a superdistribution corresponding to $(\fg_{-1}\oplus\fp)/ \fp$;
 \item a reduction of the structure group $\mathrm{Aut}_{gr}(\fg_-)$ with associated LSA $\der_{gr}(\fg_-)$ to a connected (super-)subgroup $G_0$ with subalgebra $\fg_0$. 
 \end{itemize}
Since $\fg_+=\bigoplus_{k>0} \fg_k$ acts trivially on $(\fg_{-1}\oplus\fp)/ \fp$, it suffices to consider the $\fg_0$-action on $\fg_{-1}$.

 As no reduction is required for SHC, the rest of Section \ref{S2} is devoted to considering the contact case and giving explicit descriptions of various $\fg_0$-invariant structures on $\fg_{-1}$. More precisely, we will introduce the $(1|2)$-twisted cubic $\mathcal V|_x\subset \mathbb P(\mathcal C|_x)$, which gives rise to the maximal reduction $G_0=COSp(3|2)\subset CSpO(4|4)$ of Proposition \ref{P:max-subalg}. For simplicity of notation, we will denote $\mathcal V|_x$  just by $\mathcal V$ in the rest of Section \ref{S2}.

\subsection{Structures associated with the contact grading of $G(3)$}
\label{subsec:structurescontact}

Let $\fg = \fg_{-2} \op \fg_{-1} \op \fg_0 \op \fg_1 \op \fg_2$ be the contact grading of $G(3)$ associated with the parabolic subalgebra $\mathfrak p_{1}^{\rm IV}$. In this and in the following subsection $\mathfrak f\cong\mathfrak{osp}(3|2)$ is as in \eqref{eq:LSAff}.
\subsubsection{The $\fg_0$-invariant CSpO-structure $[\eta]$ on $\fg_{-1}$}
\label{S:eta}
The Lie bracket $\Lambda^2 \fg_{-1} \to \fg_{-2}$ induces a $\fg_0$-invariant {\em conformal symplectic-orthogonal} (CSpO) structure on $V=\fg_{-1}$, which we make explicit in this section. Now $$V_{\bar{0}} = S^3 \bbC^2 \boxtimes \bbC\;,\qquad V_{\bar{1}} = \bbC^2 \boxtimes \bbC^2\;,$$ as modules for $(\fg_0)_{\bar{0}}^{\ss} \cong \fsl(2) \op \fsp(2)$.
We fix standard bases $\{ x,y \}$ and $\{ e, f \}$ on the two copies of $\bbC^2$, along with dual bases $\{ \partial_x, \partial_y \}$ and $\{ \partial_e, \partial_f \}$ of $(\bbC^2)^*$. We denote the invariant symplectic forms on the respective copies of $\bbC^2$ both by $\omega$ and normalize them by requiring $\omega(x,y) = \omega(e,f) = 1$. Finally, we let $\flat : S^\bullet(\bbC^2) \to S^\bullet(\bbC^2)^*$ be the duality induced by $\omega$, that is, the natural extension to an algebra automorphism of the identifications $\bbC^2\cong (\bbC^2)^*$ given by $x \mapsto \partial_y$, $y \mapsto -\partial_x$, and $e \mapsto \partial_f$, $f \mapsto -\partial_e$.

The Lie algebra $(\fg_0)_{\bar{0}}^{\ss} $ is spanned by the two $\fsl(2)$-triples
 \begin{align} \label{E:sl2}
 \begin{array}{ll}
 H_1 = x \partial_x - y \partial_y, \quad
 X_1 = x\partial_y, \quad
 Y_1 = y\partial_x,
\end{array}
 \end{align}
and
 \begin{align} \label{E:sp2}
 \begin{array}{ll}
H_2 = e\partial_e - f\partial_f, \quad
 \,\,X_2 = e\partial_f, \quad
 \,Y_2 = f\partial_e,
 \end{array}
 \end{align}
while the odd part $(\fg_0)_{\bar{1}} \cong S^2 \bbC^2 \boxtimes \bbC^2$ has a basis
 $\{ x^2 \otimes e, \, x^2 \otimes f, \, xy \otimes e, \, xy \otimes f, \, y^2 \otimes e,\, y^2 \otimes f \}$. The restriction of the adjoint action to $(\fg_0)_{\bar{1}}$ and $V$ is determined by $(\fg_0)_{\bar{0}}^{\ss}$-equivariant maps, unique up to an overall scale (note that we use $\odot$ for symmetric tensor product below):
 \begin{align*}
 \begin{array}{lll}
(\fg_0)_{\bar{1}} \cdot V_{\bar{0}} \subset V_{\bar{1}}: 
& (t \otimes w', g) \mapsto c_1 g(t^\flat) \otimes w'\;,\\
 (\fg_0)_{\bar{1}} \cdot V_{\bar{1}} \subset V_{\bar{0}}: 
 & (t \otimes w', u \otimes w) \mapsto  \omega(w',w)t\odot u\;,
\end{array}
\end{align*}
where $t\otimes w'\in (\fg_0)_{\bar{1}}$, $g\in V_{\bar{0}}$, $u\otimes w\in V_{\bar{1}}$, and $c_1$ is a constant that we shall soon fix.  There are also $(\fg_0)_{\bar{0}}^{\ss}$-invariant skewsymmetric and symmetric bilinear forms:
 \begin{align}
\label{eq:skewsymmetric-1}
 \mbox{on $V_{\bar{0}}$}: &\quad (g,h) \mapsto g_{xxx} h_{yyy} - 3 g_{xxy} h_{yyx} + 3 g_{xyy} h_{yxx} - g_{yyy} h_{xxx}\;,\\
\label{eq:symmetric-1}
 \mbox{on $V_{\bar{1}}$}: &\quad (u_1 \otimes w_1, u_2 \otimes w_2) \mapsto c_2\, \omega(u_1,u_2) \omega(w_1,w_2)\;,
 \end{align}
where $g,h\in V_{\bar{0}}$ and $u_1 \otimes w_1, u_2 \otimes w_2\in V_{\bar{1}}$. Here $c_2$ is also a constant to be determined.
\begin{lemma}
\label{lem:reductionCSPO}
 Let $\eta$ be the symplectic-orthogonal structure on $V=V_{\bar 0}\oplus V_{\bar 1}$ defined by $\eta(V_{\bar{0}},V_{\bar{1}}) = 0$ and the bilinear forms \eqref{eq:skewsymmetric-1} and \eqref{eq:symmetric-1} on $V_{\bar{0}}$ and $V_{\bar{1}}$, respectively. Then $\eta$ is $\ff$-invariant if and only if $c_1 c_2=-6$ and, in this case, its conformal class $[\eta]$ is $\fg_0$-invariant.
\end{lemma}
\begin{proof}
We have already seen invariance by $(\fg_{0})_{\bar 0}^{\ss}$. Invariance of $\eta$ under $T = y^2 \otimes e \in (\fg_0)_{\bar{1}}$ implies
 \[
 0 = \eta(T(x^3), y \otimes f) + \eta(x^3, T(y\otimes f)) = \eta( 6c_1 x\otimes e, y \otimes f) + \eta(x^3, y^3) = 6 c_1 c_2 + 36,
 \]
 which forces $c_1 c_2 = -6$.  A longer check shows that this condition is also sufficient.
\end{proof}

\subsubsection{The $\fg_0$-action on $\fg_{-1}$}
 \label{S:g0-mat}
Let $\fspo(4|4)$ be the LSA of linear transformations preserving $\eta$, and $\fcspo(4|4)$ its 1-dimensional central extension. By Lemma \ref{lem:reductionCSPO}, the conformal class $[\eta]$ is $\fg_0$-invariant, and this yields a reduction to $\fg_0 \subset \fcspo(4|4)$. We normalize $c_2 = -6^3$ (so $c_1 = \frac{1}{6^2}$) and consider the basis
\begin{align} \label{spo-basis}
 \{ v_1,\ldots,v_4\;|\; v_5,\ldots, v_8 \}=\{ x^3, \,\, -3x^2 y, \,\, -6y^3,\,\, -6xy^2 \;|\; x \otimes e,\,\, x \otimes f,\,\, y \otimes f,\,\, -y \otimes e \}\;
\end{align}
of $V$.
In this case $\eta$ is represented by ($-6^3$ times) the matrix
 \begin{align} \label{spo-structure}
 \left( \begin{array}{cc|ccc}
 0 & \id_2 \\
 -\id_2 & 0 \\ \hline
 && 0 &\id_2 \\
 && \id_2 & 0
 \end{array} \right).
 \end{align}
With respect to the basis \eqref{spo-basis}, the semisimple part $\ff\subset \fspo(4|4)$ of $\fg_0$ is spanned by the following matrices, which act as usual on column vectors: 
  \vskip0.2cm\par\noindent
{\underline{\it Even part}}
 \[
 \begin{array}{ll}
 H_1 = \ptm{3 &&&\\ &1 &&\\ &&-3&\\ &&&-1&\\ \hline &&&&1\\&&&&&1\\ &&&&&&-1\\&&&&&&&-1}, &
 H_2 = \ptm{&&&&\\&&&&\\&&&&\\&&&&\\ \hline &&&&1\\ &&&&&-1\\ &&&&&&-1\\ &&&&&&&1},\\
 X_1 = \ptm{
 0 & -3 & 0 & 0\\
 0 & 0 & 0 & 4\\
 0 & 0 & 0 & 0\\
 0 & 0 & 3 & 0\\ \hline
 &&&& 0 & 0 & 0 & -1\\
 &&&& 0 & 0 & 1 & 0\\
 &&&& 0 & 0 & 0 & 0\\
 &&&& 0 & 0 & 0 & 0}, &
 X_2 = \ptm{
 &&&\\
 &&&\\
 &&&\\
 &&&\\ \hline
 &&&& 0 & 1 & 0 & 0\\
 &&&& 0 & 0 & 0 & 0\\
 &&&& 0 & 0 & 0 & 0\\
 &&&& 0 & 0 & -1 & 0},\\
 Y_1 = \ptm{
 0 & 0 & 0 & 0\\
 -1 & 0 & 0 & 0\\
 0 & 0 & 0 & 1\\
 0 & 1 & 0 & 0\\ \hline
 &&&& 0 & 0 & 0 & 0\\
 &&&& 0 & 0 & 0 & 0\\
 &&&& 0 & 1 & 0 & 0\\
 &&&& -1 & 0 & 0 & 0}, &
 Y_2 = \ptm{
 &&&\\
 &&&\\
 &&&\\
 &&&\\ \hline
 &&&& 0 & 0 & 0 & 0\\
 &&&& 1 & 0 & 0 & 0\\
 &&&& 0 & 0 & 0 & -1\\
 &&&& 0 & 0 & 0 & 0}.
 \end{array}
 \]
  \vskip0.2cm\par\noindent
{\underline{\it Odd part}} 
 \[
 \begin{array}{l@{\,}ll@{\,}l}
 A_1 &= \ptm{
 &&&& 0 & 3 & 0 & 0\\
 &&&& 0 & 0 & -1 & 0\\
 &&&& 0 & 0 & 0 & 0\\
 &&&& 0 & 0 & 0 & 0\\ \hline
 0 & 0 & 0 & -1 \\
 0 & 0 & 0 & 0 \\
 0 & 0 & 0 & 0 \\
 0 & 0 & 3 & 0}, &
 A_2 &= \ptm{
 &&&& -3 & 0 & 0 & 0\\
 &&&& 0 & 0 & 0 & -1\\
 &&&& 0 & 0 & 0 & 0\\
 &&&& 0 & 0 & 0 & 0\\ \hline
 0 & 0 & 0 & 0 \\
 0 & 0 & 0 & -1 \\
 0 & 0 & -3 & 0 \\
 0 & 0 & 0 & 0}, \\
 A_3 &= \ptm{
 &&&& 0 & 0 & 0 & 0\\
 &&&& 0 & -2 & 0 & 0\\
 &&&& 0 & 0 & 0 & 0\\
 &&&& 0 & 0 & -1 & 0\\ \hline
 0 & 1 & 0 & 0 \\
 0 & 0 & 0 & 0 \\
 0 & 0 & 0 & 0 \\
 0 & 0 & 0 & -2}, &
 A_4 &= \ptm{
 &&&& 0 & 0 & 0 & 0\\
 &&&& 2 & 0 & 0 & 0\\
 &&&& 0 & 0 & 0 & 0\\
 &&&& 0 & 0 & 0 & -1\\ \hline
 0 & 0 & 0 & 0 \\
 0 & 1 & 0 & 0 \\
 0 & 0 & 0 & 2 \\
 0 & 0 & 0 & 0}, \\
 A_5 &= \ptm{
 &&&& 0 & 0 & 0 & 0\\
 &&&& 0 & 0 & 0 & 0\\
 &&&& 0 & 0 & -1 & 0\\
 &&&& 0 & -1 & 0 & 0\\ \hline
 1 & 0 & 0 & 0 \\
 0 & 0 & 0 & 0 \\
 0 & 0 & 0 & 0 \\
 0 & 1 & 0 & 0}, &
 A_6 &= \ptm{
 &&&& 0 & 0 & 0 & 0\\
 &&&& 0 & 0 & 0 & 0\\
 &&&& 0 & 0 & 0 & -1\\
 &&&& 1 & 0 & 0 & 0\\ \hline
 0 & 0 & 0 & 0 \\
 1 & 0 & 0 & 0 \\
 0 & -1 & 0 & 0 \\
 0 & 0 & 0 & 0}.
 \end{array}
 \]
 These odd matrices correspond to $A_1 = 3x^2\otimes e$, $A_2 = 3x^2 \otimes f$, $A_3 = 6xy \otimes e$, $A_4 = 6xy \otimes f$, $A_5 = 6y^2 \otimes e$, and $A_6 = 6y^2 \otimes f$.

 \subsubsection{A distinguished supervariety}
 \label{S:super-V}

Let $\bbP(V)=\mathrm{Gr}(1|0;4|4)\cong \bbP^{3|4}$ be the projective superspace corresponding to the linear supermanifold $V=V_{\bar 0}\oplus V_{\bar 1}\cong \bbC^{4|4}$, see \cite[\S 4.3]{MR1632008}, with its natural action of the connected Lie supergroup $G_0=COSp(3|2)\subset CSpO(4|4)$ generated by $\fg_0 \subset \fcspo(4|4)$ (resp. $F\subset\mathrm{SpO}(4|4)$ generated by $\ff\subset \fspo(4|4)$). The underlying topological manifold of $\bbP(V)$ is the $3$-dimensional classical projective space $\bbP(V_{\bar 0})\cong \bbP^{3}$.

It is convenient to consider $\bbP(V)$ and the Lie supergroup $G_0$ in the sense of their {\it functor of points} $\bbA\mapsto \bbP(V)(\bbA)$ and $\bbA\mapsto G_0(\bbA)$ \cite{MR2840967}. Concretely, for any finite-dimensional supercommutative superalgebra $\bbA=\bbA_{\bar 0}\oplus\bbA_{\bar 1}$ (e.g., any  exterior algebra with a finite number of generators), we consider the $\bbA$-module
$V\otimes\bbA$ and set $$\bbP(V)(\bbA):=\bbP^{1|0}(V\otimes \bbA)\;,$$ that is the collection of free $\bbA$-modules in $V\otimes\bbA$ of rank $(1|0)$ \cite[Prop. 1.7.7]{MR3328668}.

We note that $V\otimes\bbA=(V\otimes \bbA)_{\bar 0}\oplus (V\otimes\bbA)_{\bar 1}$, where
$$ (V\otimes\bbA)_{\bar 0}:=(V_{\bar 0}\otimes \bbA_{\bar 0})\oplus (V_{\bar 1}\otimes \bbA_{\bar 1})\;,\qquad (V\otimes\bbA)_{\bar 1}:=(V_{\bar 0}\otimes \bbA_{\bar 1})\oplus (V_{\bar 1}\otimes \bbA_{\bar 0})\;,
$$
and set $V(\bbA):=(V \otimes \bbA)_{\bar 0}$. The correspondence $\bbA\mapsto V(\bbA)$ is the functor of points associated to the linear supermanifold $V$. The (set-theoretic) group $G_0(\bbA)$ acts on $V(\bbA)$ by means of even transformations with coefficients in $\bbA$ \cite{Rittenberg:1977eg}, thus giving an action of $G_0$ on $V$ in the sense of \cite[Def. 11.7.2]{MR2840967}. Clearly, we also have an action of $G_0(\bbA)$ on the full $V\otimes\bbA$ and therefore an induced action of $G_0$ on $\bbP(V)$ via the associated functors of points. A similar construction holds for the Lie supergroup $F$.

We consider the topological point $o := [v_1] = [x^3]\in \bbP(V_{\bar 0})$ and its isotropy subalgebra $\fq \subset \ff$, which is parabolic. More precisely $\ff$ admits a $|1|$-grading
 \begin{align} \label{f-decomp}
 \ff = \ff_{-1} \op \ff_0 \op \ff_1
 \end{align}
such that $\fq = \ff_{\geq 0}$. We note that $\ff_{\pm 1}$ are abelian and that the grading
comes from the grading element $\sfZ_2 \in \fh$, that is:
 \begin{align*}
 \begin{array}{c|c|c|c|c}
 k & (\ff_k)_{\bar{0}} & (\ff_k)_{\bar{1}} & \mbox{Even roots} & \mbox{Odd roots}\\ \hline
 1 & X_1 & A_1, A_2 & \alpha_2 + \alpha_3 & \alpha_2,\, \alpha_2 + 2\alpha_3\\
 0 & H_1, H_2, X_2, Y_2 & A_3, A_4 & \pm 2\alpha_3 & \pm \alpha_3\\
 -1 & Y_1 & A_5, A_6 & -\alpha_2 - \alpha_3 & -\alpha_2,\, -\alpha_2 - 2\alpha_3\\
 \end{array}
 \end{align*}
Let $\cV \subset \bbP(V)$ be the $G_0$-orbit through $o$. We explicitly describe $\cV$ near $o$ by exponentiating the action of $\ff_{-1} = \tspan\{ Y_1, A_5, A_6 \}$ applied to $o$, thought as an element of $V(\bbA)$. The relevant 1-parameter subgroups of $G_0$ (in the sense of \cite[\S 7.5]{MR2840967}) are as follows:

 \begin{itemize}
 \item  The curve $\lambda \mapsto (x+\lambda y)^3$ in $V(\bbA)$ w.r.t. the even parameter $\lambda\in\bbA_{\bar 0}$ has derivative equal to $Y_1 \cdot x^3 = 3x^2 y$ at $\lambda=0$ and  indeed it coincides with $\exp(\lambda Y_1) \cdot v_1$. So we calculate the latter via the components of $(x+\lambda y)^3$ in the basis \eqref{spo-basis};
 \item We use odd $\theta,\phi\in\bbA_{\bar 1}$ for the 1-parameter subgroups generated by elements of the odd part $(\fg_0)_{\bar{1}}$, e.g., $\exp(\theta A_5) = 1 + \theta A_5$ since $\theta^2 = 0$.
 \end{itemize}

We obtain
 \begin{align} \label{cV}
 \ptm{1\\0\\0\\0\\\hline 0\\0\\0\\0}
 \stackrel{\exp(\lambda Y_1)}{\longmapsto}
 \ptm{1\\-\lambda\\-\frac{\lambda^3}{6}\\-\frac{\lambda^2}{2}\\\hline0\\0\\0\\0}
 \stackrel{\exp(\theta A_5)}{\longmapsto}
 \ptm{1\\-\lambda\\-\frac{\lambda^3}{6}\\-\frac{\lambda^2}{2}\\\hline \theta \\0\\0\\-\theta \lambda}
 \stackrel{\exp(\phi A_6)}{\longmapsto}
 \ptm{1\\-\lambda\\-\frac{\lambda^3}{6} + \phi \theta \lambda\\-\frac{\lambda^2}{2}+ \phi\theta \\\hline \theta\\ \phi \\ \phi \lambda\\-\theta \lambda}\;,
 \end{align}
as an element of $V(\mathbb A)$. The projectivization of \eqref{cV} yields our local parametrization of $\cV$ near $o$ with parameters  $\lambda\in\bbA_{\bar 0}$ and $\theta,\phi\in\bbA_{\bar 1}$.  Its Zariski-closure is the full $G_0$-orbit $\cV$, given by adding the super-point at infinity, i.e., the projectivization of  $(0,0,1,0\,|\,0,0,0,0)^\top$.

\begin{definition} \label{D:cV}
The supervariety $\cV \subset \bbP(V)$ is called the {\it $(1|2)$-twisted cubic} and the {\it super-point} $\ell=\ell(\lambda,\theta,\phi)$ of $\cV$ is the free $\bbA$-module of rank $1|0$ generated by \eqref{cV} or the super-point at infinity.
\end{definition}

For later use in Section \ref{sec:lagrangiansubspacealong}, it will be convenient to re-write the result \eqref{cV} by re-ordering the basis of $V$ and using the canonical isomorphism
$V \otimes \bbA\cong \bbA\otimes V$ that interchanges right with left coordinates via the ``sign rule'', namely:
 \begin{align} \label{cV-left}
v_1 - \lambda v_2 - \theta v_5 - \phi v_6 - \left(\frac{\lambda^3}{6} + \lambda\theta\phi \right) v_3 - \left(\frac{\lambda^2}{2} + \theta \phi\right) v_4 - \lambda \phi v_7 + \lambda \theta v_8,
 \end{align}
where $\lambda\in\bbA_{\bar 0}$ and $\theta,\phi\in\bbA_{\bar 1}$.
 \subsubsection{Osculations of $\cV$}
 \label{S:based-osc}

Let $\cU(\fg_0)$ be the universal enveloping algebra of $\fg_0$, naturally filtered by degree, 
and let $\cU_k(\fg_0)$ denote the $k$-th filtrand, where $k \in \ZZ_{\geq 0}$.  Define the map $\cU_k(\fg_0) \to V$ given by $t \mapsto t \cdot v_1$, and call its image $\widehat{T}_o^{(k)} \cV$ the {\em $k$-th osculating space} of $\cV$ at $o = [v_1]$.  (The space $\widehat{T}_o^{(1)} \cV$ is also called the {\em affine tangent space} of $\cV$ at $o$.) Since the semisimple part $\ff_0^{\ss}$ of $\ff_0$ is contained in the annihilator of $v_1$, the aforementioned map is $\ff_0^{\ss}$-equivariant.

 Recall from \eqref{f-decomp} that $\fq$ preserves the line $o$, so we obtain a $\fq$-invariant filtration of $V$ by osculating spaces
$$
\cdots\subset\widehat{T}_o^{(k-1)} \cV\subset\widehat{T}_o^{(k)} \cV\subset\widehat{T}_o^{(k+1)} \cV\subset\cdots\;,
$$
with associated graded $$\gr(V) = \bigoplus_{k\in \ZZ_{\geq 0}} N_k$$ given by the sum of {\em normal spaces} $N_k = \widehat{T}_o^{(k)} \cV / \widehat{T}_o^{(k-1)} \cV$.  The induced surjection $\cU_k(\fg_0) \to N_k$ has $\cU_{k-1}(\fg_0)$ in its kernel, so it descends to a $\ff_0^{\ss}$-equivariant map $\cU_k(\fg_0) / \cU_{k-1}(\fg_0)  \to N_k$. 

Since $\fq$ preserves $o$ and $\cU_k(\fg_0) / \cU_{k-1}(\fg_0) \cong S^k(\fg_0)$ as $\fg_0$-modules, so as $\ff_0^{\ss}$-modules as well, the restriction
$$\varphi_k : S^k(\ff_{-1}) \to N_k$$
of our map to $S^k(\ff_{-1})$ is still surjective and $\ff_0^{\ss}$-equivariant. Hence
 \begin{align} \label{Nk}
 N_k \cong S^k(\ff_{-1}) / \ker(\varphi_k),
 \end{align}
 as $\ff_0^{\ss}$-modules. One can also easily see that $N_k$ has degree $-k$ w.r.t. the grading element $\sfZ_2$.  

For any $i,j$, the full symmetrization of $S^i(\ff_{-1}) \otimes \ker(\varphi_j)$ sits inside $S^{i+j}(\ff_{-1})$ and acting on $v_1$ lands in $\widehat{T}^{(i+j-1)}_o \cV$. Hence $\gr(V)$ inherits from the product in $S^\bullet(\ff_{-1})$ a (supercommutative, associative) $\ZZ$-graded  superalgebra structure
 \begin{align} \label{grV-alg}
 N_i \otimes N_j \to N_{i+j},
 \end{align}
which is $\ff_0^{\ss}$-equivariant.

We organize the roots \eqref{ct-grading} of $V = \fg_{-1}$ using $\sfZ_2$ to obtain
$
 \gr(V) = N_0 \op N_1 \op N_2 \op N_3
$
 with:
 \begin{align} \label{E:grV}
 \begin{array}{|c|c|c|c|c|c} \hline
  \gr(V) & \mbox{Grading} & \mbox{Even part} & \mbox{Odd part}\\ \hline\hline
 N_0 & 0 & -\alpha_1 & \\ \hline
 N_1 & -1 & -\alpha_1 - \alpha_2 - \alpha_3 & -\alpha_1 - \alpha_2, \,\, -\alpha_1 - \alpha_2 - 2\alpha_3\\ \hline
 N_2 & -2 & -\alpha_1 - 2\alpha_2 - 2\alpha_3 & -\alpha_1 - 2\alpha_2 - 3\alpha_3, \,\, -\alpha_1 - 2\alpha_2 - \alpha_3\\ \hline
 N_3 & -3 & -\alpha_1 - 3\alpha_2 - 3\alpha_3 & \\ \hline
 \end{array}
 \end{align}
 The identification \eqref{Nk} yields the dictionary below, where $Y_1, A_5, A_6$ refer to the {\em equivalence classes} of these elements modulo $\ker(\varphi_1)$, and similarly for the elements in $S^{k}(\ff_{-1})$, $k=2,3$. The action of a representative in a class on $v_1$   is illustrated next to it on the right.
 \[
 \begin{array}{|c|c|c|c|c|} \hline
  & \multicolumn{2}{c|}{\mbox{Even representatives}} & \multicolumn{2}{c|}{\mbox{Odd representatives}}\\ \hline\hline
 N_0 & 1 & v_1 & & \\ \hline
 N_1 &  Y_1 & -v_2 & \begin{array}{l} A_5\\ A_6 \end{array} & \begin{array}{l} v_5\\ v_6 \end{array} \\ \hline
 N_2 & (Y_1)^2\equiv A_5 A_6 & -v_4  & \begin{array}{l} Y_1 A_5\\ Y_1 A_6 \end{array} & \begin{array}{l}-v_8\\ v_7 \end{array}\\ \hline
 N_3 & (Y_1)^3\equiv Y_1 A_5 A_6 & -v_3  & & \\ \hline
 \end{array}
 \]
Note that one generating relation $(Y_1)^2 \equiv A_5 A_6$ arises.  (The relations $(A_5)^2 = (A_6)^2 = 0$ and $A_5 A_6 = -A_6 A_5$ are automatic since $A_5,A_6$ are odd.)

 \subsubsection{Super-symmetric cubic and quadratic forms on a Jordan superalgebra}
\label{subsec:Super-symmetric cubic and quadratic forms on a Jordan superalgebra}

 Let $W = N_1$.  Note that $N_1 \otimes N_2 \to N_3$ is a non-degenerate $\ff_0^{\ss}$-equivariant pairing and that $\ff_0^{\ss}$ acts trivially on the $1$-dimensional module $N_3$. Hence $N_2 \cong W^*$ as $\ff_0^{\ss}$-modules and $N_1 \otimes N_1 \to N_2$ yields a map $W \otimes W \to W^*$ that is $\ff_0^{\ss}$-invariant. The associated cubic form $\fC \in S^3 W^*$ is clearly supersymmetric, as the algebra structure $\eqref{grV-alg}$ is induced from the product in $S^{\bullet}(\ff_{-1})$.

Concretely, we fix $(Y_1)^3\in N_3$ and consider the dual bases:
 \begin{align} \label{W-basis}
 \begin{array}{l@{\quad}l@{\quad}l@{\quad}l}
 \mbox{Basis of $W$\;\,}: & w_1 = Y_1 , & w_2 = A_5 , & w_3 = A_6 .\\
 \mbox{Basis of $W^*$}: & w^1 = (Y_1)^2 , & w^2 = Y_1 A_6 ,& w^3 = -Y_1 A_5 .
 \end{array}
 \end{align}
\begin{proposition} The following supersymmetric forms are $\ff_0^{\ss}$-invariant:
 \begin{align*}
 \begin{array}{ll}
 \fC\!\!\!\!\!&=\frac{1}{3} (w^1)^3 - 2 w^1 w^2 w^3 \in S^3 W^*,\\
\mathfrak{G}\!\!\!\!\!&=(w^1)^2 - 4 w^2 w^3 \in S^2 W^*,\\
 \fC^* \!\!\!\!\!&= \frac{4}{9} (w_1)^3 - \frac{2}{3} w_1 w_2 w_3 \in S^3 W, \\
\mathfrak{G}^*\!\!\!\!\!&= (w_1)^2 - w_2 w_3 \in S^2 W,
 \end{array}
 \end{align*}
 and they are the unique $\ff_0^{\ss}$-invariant forms in these spaces up to scale.
 \end{proposition}
\begin{proof}
Since $\ff_0^{\ss}$ is semisimple, it is generated by its odd part $(\ff_0)_{\bar{1}}$ and it is sufficient to check invariance under the latter. The claim follows then from a
direct calculation using the explicit expressions w.r.t. the basis \eqref{W-basis} of the action of the elements of $(\ff_0)_{\bar{1}}$ on $W$,   
 \[
 A_3 = {\tiny
 \left( \begin{array}{c|cc}
 0 & 0 & 2\\ \hline -1 & 0 & 0\\ 0 & 0 & 0
 \end{array} \right)}, \quad
 A_4 = {\tiny
 \left( \begin{array}{c|cc}
 0 & -2 & 0\\ \hline 0 & 0 & 0\\ -1 & 0 & 0
 \end{array} \right)},
 \]
 and the action
\[
 -A_3^{st} = {\tiny
 \left( \begin{array}{c|cc}
 0 & 1 & 0\\ \hline 0 & 0 & 0\\ 2 & 0 & 0
 \end{array} \right)}, \quad
 -A_4^{st} = {\tiny
 \left( \begin{array}{c|cc}
 0 & 0 & 1\\ \hline -2 & 0 & 0\\ 0 & 0 & 0
 \end{array} \right)},
 \]
on $W^*$ given by the negative of the supertranspose.
\end{proof} 

We note that post-composing $W \otimes W \to W^*$ with the $\ff_0^{\ss}$-equivariant duality
\[
 w^1 \mapsto w_1, \quad
 w^3 \mapsto -\frac{1}{2}w_2, \quad
 w^2 \mapsto \frac{1}{2}w_3\;,
 \]
from $W^*$ to $W$ induced by $\fG$ gives an $\ff_0^{\ss}$-equivariant (supercommutative, not associative) superalgebra structure $\circ$ on $W$. Its multiplication table is
 \[
 \begin{array}{c|ccc}
 \circ & w_1 & w_2 & w_3\\ \hline
 w_1 & w_1 & \frac{1}{2} w_2 & \frac{1}{2} w_3\\
 w_2 & \frac{1}{2} w_2 & 0 & w_1 \\
 w_3 & \frac{1}{2} w_3 & -w_1 & 0
 \end{array}
 \]
and we recognize a non-unital simple Jordan superalgebra, called the {\em Kaplansky superalgebra} \cite[p.1381]{MR0498755}.

 \subsubsection{A key identity}
Let $W=N_1$ be as in Section \ref{subsec:Super-symmetric cubic and quadratic forms on a Jordan superalgebra} and let us recall that for any finite-dimensional supercommutative superalgebra $\bbA=\bbA_{\bar 0}\oplus\bbA_{\bar 1}$, we may define $W(\bbA):= (W \otimes \bbA)_{\bar 0}\cong (\bbA \otimes W)_{\bar 0}$, where the isomorphism interchanges right with left coordinates via the ``sign rule''. In this section, we shall work exclusively with left coordinates; accordingly $W^*(\bbA) := (\bbA \otimes W^*)_{\bar 0}$. We will also freely use Einstein's summation convention by summing over repeated indices. 

For any $T \in W(\bbA)$, we write
\begin{equation}
T = t^a w_a = \lambda w_1 + \theta w_2 + \phi w_3\;,
\label{eq:parameterT}
\end{equation}
where the index $a=1,2,3$ and the parameters $t^1=\lambda \in \bbA_{\bar{0}}$, $t^2=\theta\in \bbA_{\bar{1}}$ and $t^3=\phi\in \bbA_{\bar{1}}$.  We extend the definition of $\fC$ and $\fG$ from $W$ to $W(\bbA)$ using the non-trivial components of $\fC = \fC_{abc} w^a w^b w^c$ given by \footnote{Our definition of super-symmetric forms includes a weight in their expression as a sum of tensor products, e.g. $w^2 w^3 = \frac{1}{2}( w^2 \otimes w^3 - w^3 \otimes w^2)$.}
\begin{align} \label{fC-coeffs}
 \fC_{111} = \frac{1}{3}, \quad \fC_{123}  = -\frac{1}{3} = -\fC_{132} = \fC_{231}=\cdots,
 \end{align}
 and $\bbA$-linearity (in the super-sense) on the left, i.e., we have
 \begin{align} \label{E:CG}
 \fC(T^3) := t^c t^b t^a \fC_{abc} = \frac{\lambda^3}{3} + 2\lambda\theta\phi, \quad
 \fG(T^2) := t^b t^a \fG_{ab} = \lambda^2 + 4\theta\phi,
 \end{align}
for any $T\in W(\bbA)$ as in \eqref{eq:parameterT}. Here and in the following, the terms $T^2$ and $T^3$ inside the parentheses are meant to suggest $T$ inserted twice or three times. We will also use notation such as
 \begin{align} \label{E:CG2}
 \fC_c(T^2) := \frac{1}{3} \partial_{t^c} (\fC(T^3)), \quad
 \fC_{bc}(T) := \frac{1}{2} \partial_{t^b} \fC_c(T^2), \quad
 \fC_{abc} := \partial_{t^a} \fC_{bc}(T),
 \end{align}
 so that $\fC(T^3) = t^c \fC_c(T^2) = t^c t^b \fC_{bc}(T) = t^c t^b t^a \fC_{abc}$, as expected. Then a straightforward computation shows that
 \begin{align*}
 3\fC_c(T^2) &=
 \begin{pmatrix}
 \lambda^2 + 2\theta\phi, & 2\lambda\phi, & -2\lambda\theta
 \end{pmatrix},\\
 3\fC_{bc}(T) & =
 \begin{pmatrix}
 \lambda & \phi & -\theta\\
\phi & 0 & -\lambda\\
 -\theta & \lambda & 0
 \end{pmatrix}.
 \end{align*}

In a similar way we consider $T^*\in W^*(\bbA)$ and write $T^* = t^*_a w^a = \mu w^1 + \delta w^2 + \epsilon w^3$,
where $\mu\in \bbA_{\bar{0}}$ and $\delta, \epsilon\in \bbA_{\bar{1}}$. 
We define
 \[
 \fC^*((T^*)^3) = \frac{4}{9} \mu^3 + \frac{2}{3} \mu\delta\epsilon, \quad
 \fG^*((T^*)^2) = \mu^2 + \delta\epsilon
 \]
and likewise introduce tensors of the forms $(\fC^*)^c((T^*)^2)$ etc, explicitly given by
 \begin{align*}
3(\fC^*)^c((T^*)^2) &= \begin{pmatrix}
 \frac{4}{3} \mu^2 + \frac{2}{3} \delta \epsilon \\
 \frac{2}{3} \mu \epsilon \\
 -\frac{2}{3} \mu \delta
 \end{pmatrix},\\
 3 (\fC^*)^{bc}(T^*) &=
 \begin{pmatrix}
 \frac{4}{3} \mu & \frac{1}{3} \epsilon & -\frac{1}{3} \delta\\
 \frac{1}{3} \epsilon & 0 & -\frac{1}{3} \mu\\
 -\frac{1}{3} \delta & \frac{1}{3} \mu & 0
 \end{pmatrix}.
 \end{align*}
 A straightforward verification now yields a supersymmetric version of some key identities for $\fC$ and $\fC^*$ used in \cite{MR3759350}.
 They will be crucial in the symmetry computation in Section \ref{S:G3-PDE-syms}.
  \begin{proposition} \label{P:C-ids} For any $T \in W(\bbA)$ and $T^* \in W^*(\bbA)$, the following identity holds:
 \begin{align}
 &\fC_c(T^2) \fC_a(T^2) (\fC^*)^{ac}(T^*) = \frac{4}{27} \fC(T^3) t^c t^*_c \label{id1}
 \end{align}
 Via differentiation, we obtain:
 \begin{align}
 &\fC_{bc}(T) \fC_a(T^2) (\fC^*)^{ac}(T^*) = \frac{1}{27} \left( 3\fC_b(T^2) t^c t^*_c + \fC(T^3) t^*_b \right) \label{id2}
 \end{align}
 and
 \begin{align}
 &\fC_{dbc} \fC_a(T^2) (\fC^*)^{ac}(T^*) + 2 \fC_{da}(T) (\fC^*)^{ac}(T^*) \fC_{cb}(T) (-1)^{|c|} \nonumber \\
 &\qquad= \frac{1}{9} \left( 2\fC_{db}(T) t^c t^*_c + t^*_d \fC_b(T^2) + \fC_d(T^2) t^*_b \right) \label{id3}\;,
 \end{align}
where $|c|\in \mathbb{Z}_2$ is the parity of $t^c\in\mathbb A$.
\end{proposition}

\subsection{Lagrangian subspaces along $\cV$}
  \label{sec:lagrangiansubspacealong}
	
 In Section \ref{S:super-V} we constructed the supervariety $\mathcal V\subset \mathbb P(V)$,
which we called $(1|2)$-twisted cubic, and in Section \ref{S:based-osc} we carried out the osculations of $\cV$ at the point $o=[v_1]=[x^3]$ of its  underlying topological manifold, the (classical) twisted cubic.  Using the explicit expressions of the matrices of $\ff$ in Section \ref{S:g0-mat} it is clear that the affine tangent space $$\widehat{T}^{(1)}_o \cV = \tspan\{ v_1, v_2, v_5, v_6 \}$$ is {\em Lagrangian} with respect to $\eta$, i.e., it is $\eta$-null and of maximal dimension $(2|2)$.  

Let $LG(V)$ be the Lagrangian--Grassmannian of the linear supermanifold $V$, considered, as usual, via its functor of points $\bbA\mapsto LG(V)(\bbA):= LG(V\otimes\bbA)$. Here $LG(V\otimes\bbA)$ is the collection of Lagrangian free $\bbA$-modules in $V\otimes\bbA\cong \bbA\otimes V$ of rank $(2|2)$, see \cite{MR849339}. In particular we may regard the $\mathbb A$-module $\widehat{T}^{(1)}_o \cV(\bbA):=\bbA\otimes\widehat{T}^{(1)}_o \cV$ as an element of $LG(V)(\bbA)$.

Since $\cV$ is $G_0$-invariant and the CSpO-structure $[\eta]$ is $G_0$-invariant too, then the collection  of Lagrangian $\bbA$-modules 
$$\widehat\cV := \{ \widehat{T}^{(1)}_\ell \cV \mid \ell=\text{super-point of}\;\cV \} \subset LG(V)$$
given by the affine tangent spaces along $\cV$ is also $G_0$-invariant.  To describe $\widehat{\cV}$ concretely, we consider the local parametrization \eqref{cV-left} of $\cV$ and osculate at any super-point $\ell=\ell(\lambda,\theta,\phi)$.
Furthermore, to facilitate the passage to the super-PDE picture considered in Section \ref{S:G3-PDE}, we will describe both $\cV$ and $\widehat\cV$ in a {\em $\mathrm{CSpO}$-basis}, i.e., a basis of $V$ w.r.t. which $\eta$ is represented by a multiple of
 \begin{align} \label{E:CSpO-eta}
 \left( \begin{array}{cccc|cccc}
 &&&&1 \\
 &&&&&1 \\
 &&&&&&1 \\
 &&&&&&&1 \\ \hline
 -1&&&\\
 &-1&&\\
 &&1&\\
 &&&1
 \end{array} \right).
 \end{align}
 In terms of \eqref{spo-basis}, one such $\mathrm{CSpO}$-basis is $
 \{b_0,\ldots,b_3,b^0,\ldots,b^3\}  =  \{v_1,v_2,v_5,v_6,v_3,v_4,v_7,v_8 \}$.

 First of all, we observe that the local parametrization \eqref{cV-left} can be efficiently re-written using the cubic form $\fC$ as
 \begin{align} \label{E:cV}
 b_0 - t^a b_a - \frac{1}{2} \fC(T^3) b^0 - \frac{3}{2} \fC_a(T^2) b^a,
 \end{align}
with $T=t^a w_a = \lambda w_1 + \theta w_2 + \phi w_3\in W(\bbA)$ as in \eqref{eq:parameterT}. The affine tangent spaces $\widehat{T}^{(1)}_\ell \cV$ are then obtained by supplementing
the first derivatives
\begin{align*}
B_a := b_a + \frac{3}{2} \fC_a(T^2) b^0 + 3 \fC_{ac}(T) b^c
\end{align*}
of \eqref{E:cV} w.r.t. the parameters $t^a$, $a=1,2,3$, with \eqref{E:cV} itself, or, alternatively, with
$$
 B_0:= b_0 + \fC(T^3) b^0 + \frac{3}{2} \fC_a(T^2) b^a\;.
$$
In other words, we have proved the following.
\begin{proposition} \label{P:hV} A local description of $\widehat\cV\subset LG(V)$ near the topological basepoint $o=[v_1]=[x^3]$ is given by the affine tangent spaces $\widehat{T}^{(1)}_\ell \cV = \tspan_{\bbA}\{B_0, B_1, B_2,B_3 \}$ varying the super-point $\ell=\ell(\lambda,\theta,\phi)$ of $\cV$.
\end{proposition}

\begin{rem} \label{R:cV-hV} Since $\fg_0 \subset \fcspo(\fg_{-1})$ is maximal by Proposition \ref{P:max-subalg}, then both $\cV$ and $\widehat\cV$ reduce the structure algebra $\fcspo(\fg_{-1})$ to precisely $\fg_0$.  We now explain how $\cV$ is recoverable from $\widehat\cV$.  We osculate the latter to get the subspaces $\widehat{T}^{(2)}_\ell \cV$.  From \eqref{E:grV}, the second osculating space $\widehat{T}^{(2)}_o \cV$ has associated graded vector space $N_0 \op N_1 \op N_2$, so has codimension one in $V$.  Noting that $\fg_{-2}$ is the root space for the root $-2\alpha_1 - 3\alpha_2 - 3\alpha_3$, then \eqref{E:grV} also implies that the orthogonal w.r.t. $\eta$ of $\widehat{T}^{(2)}_o \cV$ is precisely the line $o$ itself.  By $G_0$-invariance, the orthogonal w.r.t. $\eta$ of $\widehat{T}^{(2)}_\ell \cV$ is $\ell$ itself, and in this way all of $\cV$ is recovered.
 \end{rem}

 In Section \ref{S4}, we further discuss $\cV$ and $\widehat\cV$ and derive our main super-PDE of study.  Proving that these have $G(3)$ supersymmetry relies on key cohomological facts that we establish next.


 \section{Computation of the Spencer cohomology}\label{S3}

 \subsection{Hochschild--Serre spectral sequence}
\label{subsec:specseq}
Let $\fm=\fm_{\bar 0}\oplus\fm_{\bar 1}$ be a finite-dimensional Lie superalgebra (LSA) and $M$ an $\fm$-module. We recall that the
$n$-cochains of the associated Chevalley--Eilenberg complex are the linear maps $\Lambda^n\fm\to\ M$ or, equivalently, the elements of $M\otimes\Lambda^n\fm^*$, where $\Lambda^\bullet$ is meant in the super-sense.

The space of $n$-cochains $C^{n}=C^n(\fm,M)=M\otimes\Lambda^n\fm^*$ has a natural descending filtration
$$
C^n=F^0C^n\supset F^1 C^{n}\supset\cdots\supset F^{n}C^n\supset F^{n+1}C^n=0\;,
$$
where
\begin{equation*}
\begin{aligned}
F^pC^n&=\Big\{c\in C^n\mid c\;\text{vanishes upon insertion of at least}\;n+1-p\;\text{elements of}\;\fm_{\bar 0}\Big\}
\end{aligned}
\end{equation*}
for all $0\leq p\leq n+1$. For simplicity, we will allow the filtration index $p$ to ``go out of bounds'', with the
understanding that the objects in question are either the full object or
zero (that is, $F^{-i}C^{n}=C^n$ and $F^{n+1+i}C^{n}=0$ for all $i>0$).

The Chevalley--Eilenberg differential $\partial:C^{n}(\fm,M) \to
C^{n+1}(\fm,M)$ is compatible with the filtration; this means that $\partial(F^pC^n)\subset F^pC^{n+1}$.
For $n=0,1$ and $2$ it is
explicitly given  by the following expressions:
\begin{align}
  \begin{split}\label{eq:CE0}
    &\partial : C^{0}(\fm,M) \to C^{1}(\fm,M)\\
    &\partial\varphi(X) = (-1)^{x|\varphi|}X\cdot\varphi\,,
  \end{split}
  \\
  \begin{split}\label{eq:CE1}
    &\partial : C^{1}(\fm,M)\to C^{2}(\fm,M)\\
    &\partial\varphi(X,Y) = (-1)^{x|\varphi|}X\cdot\varphi(Y) - (-1)^{y(x+|\varphi|)}Y\cdot\varphi(X) - \varphi([X,Y])\,,
  \end{split}
\\
  \begin{split}\label{eq:CE2}
    &\partial:C^{2}(\fm,M) \to C^{3}(\fm,M)\\
    &\partial\varphi(X,Y,Z) = (-1)^{x|\varphi|}X\cdot\varphi(Y,Z)-(-1)^{y(x+|\varphi|)}Y\cdot\varphi(X,Z) + (-1)^{z(x+y+|\varphi|)} Z\cdot\varphi(X,Y) \\
    & {} \qquad\qquad\qquad - \varphi([X,Y],Z) - (-1)^{x(y+z)} \varphi([Y,Z],X) -(-1)^{z(x+y)} \varphi([Z,X],Y)\,,
  \end{split}
\end{align}
where $x,y,\dots$ denote the parity of elements $X,Y,\dots$ of $\fm$
and $|\varphi|$ the parity of $\varphi\in C^{n}(\fm,M)$, with $n=0,1,2$ respectively.

The Hochschild--Serre spectral sequence  is the spectral sequence $(E_r,\partial_r)_{r\geq 0}$ associated to the cochain complex $C^\bullet$ together with the filtration of subcomplexes $F^{p}C^{\bullet}$. More formally, we have
\begin{equation*}
\begin{aligned}
Z_{r}^{p,q}&=\left\{c\in F^{p}C^{p+q}\mid \partial c\in F^{p+r}C^{p+q+1}\right\}\;,\\
B_{r}^{p,q}&=\partial Z_{r-1}^{p-r+1,q+r-2}\;,
\end{aligned}
\end{equation*}
for all nonnegative integers $p,q,r$, where $q=n-p$ is the
complementary index and $r$ the index of the spectral sequence.
The $r^{\text{th}}$-page of the spectral sequence is bigraded $E_r=\bigoplus E_r^{p,q}$ with the components given by
\begin{equation*}
E_r^{p,q}=\frac{Z^{p,q}_r}{Z_{r-1}^{p+1,q-1}+B^{p,q}_r}
\end{equation*}
and with the differential
$$
\partial_r:E_r^{p,q}\to E_r^{p+r,q+1-r}
$$
that is induced by the action of $\partial$ on $Z^{p,q}_r$. We have $\partial_r^2=0$ and $E_{r+1}\cong H^\bullet(E_r,d_r)$.

Since the chain complex and the filtration are both bounded below,
the spectral sequence converges $E_r^{p,q}\Rightarrow H^{n}(\fm,M)$ to the Chevalley--Eilenberg cohomology. It is not difficult to see that the $0^{\text{th}}$-page of the spectral sequence has components $E_0^{p,q}=M\otimes\Lambda^p(\fm_{\bar 1})^*\otimes\Lambda^q (\fm_{\bar 0})^*$.
We will be interested in
$n=0,1,2$, in which cases
the spectral sequence degenerates at the $4^{\text{th}}$ page. More precisely, we have:
\begin{proposition}
\label{prop:H012}
$H^{0}(\fm,M)\cong E^{0,0}_2$, $H^{1}(\fm,M)\cong E^{1,0}_2\oplus E^{0,1}_3$ and $H^{2}(\fm,M)\cong E^{2,0}_3\oplus E^{1,1}_3\oplus E^{0,2}_4$.
\end{proposition}
The following lemma is straightforward.
\begin{lemma}
\label{lem:H^0}
The groups $E_1^{0,0}$ and $E_2^{0,0}$ consist of the trivial modules in $M$ under the action of $\fm_{\bar 0}$ and, respectively, $\fm$.
\end{lemma}
The following identifications will be useful to compute $E_1$ and $E_2$ in Section \ref{S:ct-cohom}.

\begin{proposition}\cite[Theorem 1.5.1]{MR874337}
\label{lem:E1}
We have
$E_{1}^{p,q}=H^q(\fm_{\bar 0},M\otimes\Lambda^p(\fm_{\bar 1}^*))$ 
for all $p,q\geq 0$. If
$\fm_{\bar 0}$ is an ideal of $\fm$, then ${\fm/\fm_{\bar 0}}\cong\fm_{\bar 1}$ is a purely odd abelian LSA and
$$E_2^{p,q}=H^p({\fm/\fm_{\bar 0}},H^{q}(\fm_{\bar 0},M))$$ for all $p,q\geq 0$. Here the action of $\fm/\fm_{\bar 0}$ on $H^{q}(\fm_{\bar 0},M)$ is given by the natural action of $\fm_{\bar 1}$ on $M$.
\end{proposition}
For more details on the Hochschild--Serre spectral sequence, we refer the reader to \cite{MR874337}.

 \subsection{An exact sequence in cohomology}
\label{subsec:es}
Let $\fg=\fg_{-\mu}\oplus\cdots\oplus\fg_{\mu}$ be a $\mathbb Z$-grading of $\fg=\mathrm{G}(3)$ with depth $\mu>0$ and associated parabolic subalgebra $\fp=\fg_{\geq 0}=\fg_{0}\oplus\cdots\oplus\fg_{\mu}$. We are interested in the Spencer cohomology of $\fg$, i.e., the cohomology associated to the complex $C^{\bullet}(\fm,\fg)$ where the negatively graded part $\fm=\fg_{-\mu}\oplus\cdots\oplus\fg_{-1}$ of $\fg$ acts on $\fg$ via the adjoint representation.

Note that the $\mathbb Z$-degree in $\fg$ extends to the space of cochains by declaring that $\fg_d^*$ has degree $-d$ and that the differential $\partial$ has the degree zero. In particular, the complex breaks up
into the direct sum of complexes for each degree and the group
\begin{equation}
\label{eq:dn}
H^n(\fm,\fg)=\bigoplus_{d\in\mathbb Z} H^{d,n}(\fm,\fg)
\end{equation}
into the direct sum of its homogeneous components.
Any page $E_{r}^{p,q}$ of the spectral sequence decomposes in a similar way.

The space $C^{d,n}(\fm,\fg)$ of $n$-cochains of degree $d$ has a natural $\fg_0$-module structure and the same is true for the
spaces of cocycles and coboundaries, as $\partial$ is $\fg_0$-equivariant;
this implies that any $H^{d,n}(\fm,\fg)$ has a representation of $\fg_0$ and hence of the (purely even) Lie algebra $(\fg_0)_{\bar 0}$.
This equivariance is very useful in calculations, as we will have ample opportunity to demonstrate. Any page of the spectral sequence is also a $(\fg_0)_{\bar 0}$-module.
\begin{lemma}
\label{lem:centrI}
The centralizer $\xi_{\fg}(\fm)$ of $\fm$ in $\fg$ coincides with the center of $\fm$.
Hence $H^{d,0}(\fm,\fg)=0$ for all $d\geq 0$.
\end{lemma}
\begin{proof}
The ideal of $\fg$ that is generated by the centralizer of $\fm$ in $\fp$ is easily seen to be contained in $\fp$, hence it is trivial by simplicity of $\fg$. The last claim follows readily from Proposition \ref{prop:H012} and Lemma \ref{lem:H^0}.
\end{proof}
There is an interesting and very useful relationship between the Spencer groups \eqref{eq:dn} for the LSA $\fg$ and the classical Spencer groups.
Let
\begin{equation*}
0\longrightarrow K^{n}\stackrel{\imath}{\longrightarrow}\Lambda^n \fm^*\stackrel{\res}{\longrightarrow} \Lambda^n \fm_{\bar 0}^*\longrightarrow 0
\end{equation*}
be the short exact sequence given by the natural restriction map $\res:\Lambda^n \fm^*\to\Lambda^n \fm_{\bar 0}^*$ with kernel
$$K^0=0\;,\qquad K^n=\sum_{1\leq i\leq n}\Lambda^{n-i}\fm_{\bar 0}^*\otimes\Lambda^i\fm_{\bar 1}^*	\qquad\text{for}\qquad n>0\;,$$ and let
\begin{equation}
\label{eq:shortexact}
0\longrightarrow \fg\otimes K^\bullet\stackrel{\imath}{\longrightarrow} C^{\bullet}(\fm,\fg)\stackrel{\res}{\longrightarrow} C^{\bullet}(\fm_{\bar0},\fg)\longrightarrow 0
\end{equation}
be the associated short exact sequence of differential complexes. With some abuse of notation, we give the following.
\begin{definition}
\label{def:complexM1}
The differential complex $C^{\bullet}(\fm_{\bar 1},\fg)=\fg\otimes K^\bullet$ is the subcomplex of $C^{\bullet}(\fm,\fg)$ given by 
$C^{0}(\fm_{\bar 1},\fg)=0$ and the $n$-cochains, $n\geq 1$, that vanish when all entries are in $\fm_{\bar 0}$. 
\end{definition}
It is not difficult to see that every morphism in the sequence \eqref{eq:shortexact} is $(\fg_0)_{\bar 0}$-equivariant. The associated long exact sequence in cohomology, together with Lemma \ref{lem:centrI} and the discussion above it, gives the following general result, which we will extensively use in Section \ref{sec:cohomology}.
\begin{proposition}
\label{prop:longexact}
For all $d\geq 0$, there exists a long exact sequence of vector spaces
\begin{equation}
\begin{split}
\label{sequenza}
0\longrightarrow \xi^d_{\fg}(\fm_{\bar 0})\longrightarrow H^{d,1}(\fm_{\bar 1},\fg)\longrightarrow H^{d,1}(\fm, \fg)\longrightarrow\phantom{cccccccccccccccccccccccc}\,\\
 \phantom{ccc}\longrightarrow H^{d,1}(\fm_{\bar 0},\fg)\longrightarrow H^{d,2}(\fm_{\bar 1},\fg)\longrightarrow H^{d,2}(\fm,\fg)\longrightarrow H^{d,2}(\fm_{\bar 0},\fg)
\end{split}
\end{equation}
where $\xi^d_{\fg}(\fm_{\bar 0})$ is the component of degree $d$ of the centralizer of $\fm_{\bar 0}$ in $\fg$. The morphisms in the sequence are all $(\fg_0)_{\bar 0}$-equivariant.
\end{proposition}

\subsection{Spencer cohomology for \texorpdfstring{$\mathfrak{p}^{\rm IV}_{1}\subset\mathrm{G}(3)$}{} }
 \label{S:ct-cohom}

 Given the contact grading of $\fg = G(3)$, we let $\fp= \fp_1^{\rm IV} = \fg_{\geq 0}$ and
 $\fm = \fg_-$, and consider the space of cocycles $C^\bullet(\fm,\fg)$ and the cohomology groups $H^\bullet(\fm,\fg)$ as in Section \ref{subsec:es}.  We first note that $H^0(\fm,\fg)= \fg_{-2}$ by \eqref{E:contact-m} and Lemma \ref{lem:centrI}, and then turn
to compute $H^{d,1}(\fm,\fg)$ for $d \geq 0$ and $H^{d,2}(\fm,\fg)$ for $d > 0$ using spectral sequences.

We recall that
 \[
  E_0^{p,q}=\fg\otimes\Lambda^p(\fms{1})^* \otimes\Lambda^q (\fm_{\bar 0})^*,
 \]
 where $\Lambda^p(\fms{1})^*$ is meant in the super-sense, i.e., it refers to the symmetric $p$-th product of elements.

 \subsubsection{The $E_1$-page}
Since $[\fms{0},\fms{1}] \subset (\fm_{-2})_{\bar{1}} = 0$, an immediate strong property is that
 \begin{align} \label{E:ct-property}
 \mbox{{\em $\fms{1}$ and $(\fms{1})^*$ are trivial as $\fms{0}$-modules}}.
 \end{align}
This fact together with Proposition \ref{lem:E1} directly implies that
 \begin{equation}
\label{Epq1}
\begin{split}
E^{p,q}_1 = H^q(\fm_{\bar{0}},\fg \otimes \Lambda^p (\fms{1})^*)
	\cong \Lambda^p (\fms{1})^* \otimes H^q(\fm_{\bar{0}},\fg)\;.
 \end{split}
\end{equation}
 Since $\fms{0} \subset \fg_{\bar{0}} = G(2)\oplus \fsp(2)$ is the negative part of the contact grading of $G(2)$ and $\fsp(2) \subset \fg_0$, we can use Kostant's version of the Bott--Borel--Weil theorem to compute the Lie algebra cohomology group $H^q(\fms{0},\fg)$ as a module for $(\fg_0)_{\bar{0}} \cong \bbC \oplus \fsl(2) \oplus \fsp(2)$.
 
 {\em Let $\bbV_{k,\ell}[r]$ be the even irreducible representation with highest weight $(k,\ell)$ for $\fsl(2) \oplus \fsp(2)$ and with degree $r$ w.r.t. the grading element $\sfZ_1$. We let $\bbU_{k,\ell}[r]$ be the same, but regarded as an odd module.}

Kostant's theorem gives representative lowest weight vectors for each irreducible component of $H^q(\fms{0},\fg)$ and we use square brackets around a given element to denote its cohomology class.  We use the notation $e_\alpha$ to denote the root vector associated to $\alpha\in\Delta$ and attach a parity to $e_\alpha$ in the natural way, e.g., since $\alpha_1+\alpha_2$ is odd (see \eqref{ct-grading}), then $e_{\alpha_1 + \alpha_2}$ is odd too and $e_{\alpha_1 + \alpha_2} \wedge e_{\alpha_1 + \alpha_2}$ refers to a symmetric tensor product. Finally, we let $h_\alpha\in\fh$ be the coroot corresponding to any $\alpha\in\Delta$.

 Note that $G(2)\oplus\fsp(2)$ has simple roots and Cartan matrix given by
 \[
 \begin{cases}
 \widetilde\alpha_1 = \alpha_2 + \alpha_3, \\
 \widetilde\alpha_2 = \alpha_1, \\
 \widetilde\alpha_3 = 2\alpha_3
 \end{cases}, \quad \langle \widetilde\alpha_i, \widetilde\alpha_j^\vee \rangle = \begin{pmatrix} 2 & -1 & 0\\ -3 & 2 & 0\\ 0 & 0 & 2\end{pmatrix},
 \]
so that $\fg \cong G(2)  \op A(1) \op (\bbC^7 \boxtimes \bbC^2)$ as $\fg_{\bar{0}}$-modules, with respective lowest weights (see \eqref{ct-grading}):
 \[
 -2\alpha_1 - 3\alpha_2 - 3\alpha_3 = - 3\widetilde\alpha_1 - 2\widetilde\alpha_2, \quad -2\alpha_3 = -\widetilde\alpha_3, \quad -\alpha_1 - 2\alpha_2 - 3\alpha_3 = - 2\widetilde\alpha_1 -\widetilde\alpha_2 - \frac{1}{2} \widetilde\alpha_3.
 \]
Applying Kostant, we get the groups $H^k(\fms{0},\fg)$ given in Table \ref{F:lwv}.

 \begin{example} Note that $\bbC^7 \boxtimes \bbC^2 = \GA{1}{0}{1}$, i.e., (minus lowest) weight
$$\lambda = \widetilde\lambda_1 + \widetilde\lambda_3 = 2\widetilde\alpha_1 + \widetilde\alpha_2 + \frac{1}{2} \widetilde\alpha_3\;,$$
where $\widetilde\lambda_1, \widetilde\lambda_2, \widetilde\lambda_3$ are the fundamental weights.  Because we work with the contact grading of $G(2)$, we use the Weyl group element $w = (\widetilde{2} \widetilde{1})$.  The affine action of the Weyl group gives
 \[
 w \cdot \lambda = \GA{6}{-4}{1},
 \]
 which is (minus) the lowest weight of $H^2(\fms{0},\bbC^7 \boxtimes \bbC^2)$. A lowest weight vector is
 \[
 \phi= [e_{\widetilde\alpha_2} \wedge e_{\sigma_{\widetilde{2}}(\widetilde\alpha_1)} \otimes e_{w(-\lambda)}] = [e_{\alpha_1} \wedge e_{\alpha_1 + \alpha_2 + \alpha_3} \otimes e_{-\alpha_2 - 2\alpha_3}],
 \]
 where $\sigma_{\widetilde{2}}$ is the simple reflection corresponding to $\widetilde\alpha_2$.
 Hence, $\sfZ_1$ gives $\phi$ degree 2, and thus $H^2(\fms{0},\bbC^7 \boxtimes \bbC^2) \cong \bbU_{6,1}[2]$.
 \end{example}

 \begin{table}[h]
  \[
 \begin{array}{|c|c|l|c|l|} \hline
 k & \Lambda^k (\fms{1})^* & \mbox{Lowest weight vector} & H^k(\fms{0},\fg) & \mbox{Lowest weight vector $\phi$} \\ \hline\hline
 3 & \bbU_{3,3}[3] & e_{\alpha_1 + \alpha_2} \wedge e_{\alpha_1 + \alpha_2} \wedge e_{\alpha_1 + \alpha_2} & \multicolumn{2}{l|}{\mbox{not relevant for our computations}}\\
 & \bbU_{1,1}[3] & \mbox{not relevant for our computations}& \multicolumn{2}{l|}{\mbox{not relevant for our computations}}\\ \hline
 2 & \bbV_{2,2}[2] & e_{\alpha_1 + \alpha_2} \wedge e_{\alpha_1 + \alpha_2} & \bbV_{7,0}[1] & [e_{\alpha_1} \wedge e_{\alpha_1 + \alpha_2 + \alpha_3} \otimes e_{-\alpha_1 - 3\alpha_2 - 3\alpha_3}]\\
 &\bbV_{0,0}[2] & \begin{array}{@{}l} e_{\alpha_1 + \alpha_2} \wedge e_{\alpha_1 + 2\alpha_2 + 3\alpha_3} \\
  \quad + c e_{\alpha_1 + \alpha_2 + 2\alpha_3} \wedge e_{\alpha_1 + 2\alpha_2 + \alpha_3}
 \end{array} & \bbV_{4,2}[2] & [e_{\alpha_1} \wedge e_{\alpha_1 + \alpha_2 + \alpha_3} \otimes e_{-2\alpha_3}]\\
 &&& \bbU_{6,1}[2] & [e_{\alpha_1} \wedge e_{\alpha_1 + \alpha_2 + \alpha_3} \otimes e_{-\alpha_2 - 2\alpha_3}]\\ \hline
 1 & \bbU_{1,1}[1] & e_{\alpha_1 + \alpha_2} & \bbV_{6,0}[0] & [e_{\alpha_1} \otimes e_{-\alpha_1-3\alpha_2-3\alpha_3}] \\
 & && \bbV_{3,2}[1] & [e_{\alpha_1} \otimes e_{-2\alpha_3}]\\
 &&& \bbU_{4,1}[0] & [e_{\alpha_1} \otimes e_{-\alpha_1-2\alpha_2-3\alpha_3}]\\ \hline
 0 & \bbC & 1 & \bbV_{0,0}[-2] & [e_{-2\alpha_1 - 3\alpha_2 - 3\alpha_3}]\\
 &&& \bbV_{0,2}[0] & [e_{-2\alpha_3}]\\
 &&& \bbU_{1,1}[-1] & [e_{-\alpha_1 - 2\alpha_2 - 3\alpha_3}] \\ \hline
 \end{array}
 \]
 \caption{Irreducible $(\fg_0)_{\bar{0}}$-module decompositions and lowest weight vectors}
 \label{F:lwv}
 \end{table}

Taking tensor products as in \eqref{Epq1}, we obtain the $E_1$-page of the spectral sequence as in Table  \ref{F:E1-P1IV}.  Since for $p+q=k \leq 2$, $E_1^{p,q}$ only has degrees $\leq 2$, and we are only interested in $H^{d,k}(\fm,\fg)$ for nonnegative degrees $d \geq 0$, we only display those terms with degrees $0,1,2$.

  \begin{table}[H]
  \begin{tiny}
 \[
 \begin{array}{|@{}c@{}|@{}c@{}|@{}c@{}|@{}c@{}|c|c|}\hline
 {\color{blue} \bbV_{4,2}[2]} + \bbU_{6,1}[2] + {\color{blue} \bbV_{7,0}[1]}
 & \begin{array}{cccc}
 \mbox{Degree $\leq 2$}:\\
 \bbU_{6,1}[2] + \bbU_{8,1}[2]
 \end{array} & *
 & *
 \\ \hline
 \bbV_{3,2}[1] + {\color{blue} \bbU_{4,1}[0] + \bbV_{6,0}[0]} &
 \begin{array}{cc}
 \bbU_{2,1}[2] + \bbU_{2,3}[2] + \bbU_{4,1}[2] + \bbU_{4,3}[2]\\
 {\color{blue} \bbV_{3,0}[1]} + \bbV_{3,2}[1] + {\color{blue} \bbV_{5,0}[1] + \bbV_{5,2}[1]}\\
{\color{blue} \bbU_{5,1}[1] + \bbU_{7,1}[1]}
 \end{array} &
 \begin{array}{c}
 \mbox{Degree $\leq 2$}:\\
 2\bbU_{4,1}[2] +
 \bbU_{2,1}[2] +\bbU_{2,3}[2] \\
 +\bbU_{4,3}[2] +\bbU_{6,1}[2] +\bbU_{6,3}[2]\\
 + \bbV_{6,0}[2] + \bbV_{4,2}[2] +\bbV_{6,2}[2] +\bbV_{8,2}[2]
 \end{array} & *
 \\ \hline
 \bbV_{0,2}[0] &
 \begin{array}{c}
 \bbU_{1,1}[1] + \bbU_{1,3}[1] \\
 + \bbV_{0,0}[0] +  \bbV_{0,2}[0] + {\color{blue}  \bbV_{2,0}[0]} + \bbV_{2,2}[0]
 \end{array}&
 \begin{array}{c}
 \bbV_{0,2}[2] + \bbV_{2,0}[2] + \bbV_{2,2}[2] + \bbV_{2,4}[2]\\
 + 2\bbU_{1,1}[1] + \bbU_{1,3}[1] + {\color{blue} \bbU_{3,1}[1]} + \bbU_{3,3}[1]\\
 + \bbV_{0,0}[0] + \bbV_{2,2}[0]
 \end{array} &
 \begin{array}{c}
 \mbox{Degree $\leq 2$}:\\
  \bbV_{0,0}[2] + \bbV_{0,2}[2] + \bbV_{2,0}[2] + 2\bbV_{2,2}[2]\\
+ \bbV_{2,4}[2] + \bbV_{4,2}[2] + \bbV_{4,4}[2]\\
 + \bbU_{1,1}[1] + \bbU_{3,3}[1]
 \end{array}\\ \hline
 \end{array}
 \]
 \end{tiny}
 \caption{$E_1$-page (in degrees $0,1,2$) in the $G(3)$-contact case.  Modules in blue clearly survive to the $E_2$-page by $(\fg_0)_{\bar{0}}$-equivariance.}
 \label{F:E1-P1IV}
 \end{table}
The differential $\partial_1 : E_1^{p,q} \to E_1^{p+1,q}$ is $(\fg_0)_{\bar{0}}$-equivariant, so for $p+q \leq 2$, we immediately see that some modules survive to the $E_2$-page by Schur's lemma, e.g., $\bbU_{7,1}[1] \subset E_1^{1,1}$ lies in $\ker(\partial_1)$ but not in $\operatorname{im}(\partial_1)$ (since $\bbU_{7,1}[1]$ does not appear in either $E_1^{2,1}$ or $E_1^{0,1}$), hence $\bbU_{7,1}[1] \subset E_2^{1,1}$. All modules that survive to the $E_2$-page directly by Schur's lemma are colored in blue in Table~\ref{F:E1-P1IV}.

 \subsubsection{The $E_2$-page}

 For the $E_2$-page, we may
work directly with
 \[
 E_2^{p,q}=H^p({\fm/\fm_{\bar 0}},H^{q}(\fm_{\bar 0},\fg))
 \]
due to the second part of Proposition \ref{lem:E1}. (Remember that $[\fm_{\bar 0},\fm_{\bar 1}]=0$.)
 Here, $\fm/\fms{0} \cong \fms{1}$ is an abelian LSA, whose action on $H^q(\fms{0},\fg)$ is the one induced just on the coefficients $\fg$ of representative cocycles.  We denote $\partial_1$ simply by
$$
\partial:E_{1}^{p,q}\cong C^p(\fm/\fms{0}, H^q(\fms{0},\fg))\longrightarrow E_{1}^{p+1,q}\cong C^{p+1}(\fm/\fms{0}, H^q(\fms{0},\fg))
$$
and heavily rely on \eqref{ct-grading} and Table \ref{F:lwv}. We also note that $E_{1}^{0,q}=H^q(\fms{0},\fg)$ coincides precisely with the first column of Table  \ref{F:E1-P1IV}.
\begin{proposition}
\label{prop:E2pagecontact}
The $E_2$-page is given by the following Table:
\begin{table}[H]
  \begin{tiny}
 \[
 \begin{array}{|@{}c@{}|@{}c@{}|@{}c@{}|@{}c@{}|c|c|}\hline
 \bbV_{4,2}[2] + {\color{blue} \bbV_{7,0}[1]}
 & * & *
 & *
 \\ \hline
  {\color{blue} \bbU_{4,1}[0] + \bbV_{6,0}[0]} &
 \begin{array}{cc}
 {\color{blue} \bbV_{3,0}[1]} + {\color{blue} \bbV_{5,0}[1] + \bbV_{5,2}[1]}\\
{\color{blue} \bbU_{5,1}[1] + \bbU_{7,1}[1]}
 \end{array} &
 \begin{array}{c}
\mbox{Degree $\leq 1$}:\\
-
 \end{array} & *
 \\ \hline
 - &
 \begin{array}{c}
 {\color{blue}  \bbV_{2,0}[0]}
 \end{array}&
 \begin{array}{c}
 {\color{blue} \bbU_{3,1}[1]} \\
 \end{array} &
 \begin{array}{c}
 \mbox{Degree $\leq 1$}:\\
 -
 \end{array}\\ \hline
 \end{array}
 \]
 \end{tiny}
 \caption{$E_2$-page (in degrees $0,1,2$) in the $G(3)$-contact case.}
 \label{F:E2-P1IV}
 \end{table}
\noindent
Modules in blue automatically survive to the $E_3$-page by $(\fg_0)_{\bar{0}}$-equivariance.
\end{proposition}
\begin{proof}
We start with
$$E_2^{0,q} = ( H^{q}(\fm_{\bar 0},\fg) )^{\fm/\fms{0}}\;.$$
The module $\bbU_{6,1}[2] \subset H^2(\fms{0},\fg)$ has the l.w.v. $\phi=[e_{\alpha_1} \wedge e_{\alpha_1 + \alpha_2 + \alpha_3} \otimes e_{-\alpha_2 - 2\alpha_3}]$ and, letting $u= e_{-\alpha_1 - 2\alpha_2 - \alpha_3} \in \fms{1}$, we see that $(\partial \phi)(u) = -u\cdot \phi$ is a nonzero multiple of $$[e_{\alpha_1} \wedge e_{\alpha_1 + \alpha_2 + \alpha_3} \otimes e_{-\alpha_1-3\alpha_2 - 3\alpha_3}]\;,$$ that is, the l.w.v. for $\bbV_{7,0}[1] \subset H^2(\fms{0},\fg)$.  Thus, $\bbU_{6,1}[2] \not\subset E_2^{0,2}$. The modules $\bbV_{3,2}[1] \subset H^1(\fms{0},\fg)$ and $\bbV_{0,2}[0] \subset H^0(\fms{0},\fg)$ have l.w.v. $[e_{\alpha_1} \otimes e_{-2\alpha_3}]$ and $[e_{-2\alpha_3}]$ respectively, and we may similarly conclude $\bbV_{3,2}[1] \not\subset E_2^{0,1}$, $\bbV_{0,2}[0] \not\subset E_2^{0,0}$. The left-most column of Table \ref{F:E2-P1IV} is so established, as all other modules in the left-most column of Table \ref{F:E1-P1IV} are blue and survive to the $E_2$-page.

Furthermore, we also immediately infer that $\bbU_{6,1}[2]\subset E_{1}^{1,2}$, $\bbV_{3,2}[1]\subset E_1^{1,1}$ and $\bbV_{0,2}[0]\subset E_1^{1,0}$
are in the image of $\partial$ and therefore do not survive to $E_2$.


We consider now the cases $(p,q) = (1,0), (1,1)$ and:
 \begin{enumerate}
 \item[(i)] For each $(\fg_0)_{\bar{0}}$-irreducible representation $\bbT$ in $E_1^{p,q}$ that is not yet known whether it survives to $E_2$ or not (they are displayed in Table  \ref{F:ct-E2-lwv}), we write the associated l.w.v. $\phi$.
 \item[(ii)] For each such $\phi$, we examine $\partial\phi \in E_1^{p+1,q}$ using \eqref{eq:CE0}-\eqref{eq:CE2}.  If $\partial \phi \neq 0$ then $\bbT \not\subset E_2^{p,q}$ by Schur's lemma, since $\bbT$ always occurs with multiplicity one in $E_1^{p,q}$.
 \end{enumerate}
The results of these computations are summarized in Table \ref{F:ct-E2-lwv} and, in turn, this establishes $E_2^{p,q}$ in Table \ref{F:E2-P1IV} for $(p,q) = (1,0), (1,1)$. We also see that the modules $\bbU_{1,3}[1]$, $\bbV_{2,2}[0]$, $\bbV_{0,0}[0]$ and (one copy
of) $\bbU_{1,1}[1]$ in $E_1^{2,0}$ are in the image of $\partial$ and therefore do not survive to $E_2$.

We turn to the case $(p,q)=(2,0)$. By the previous discussion, the irreducible modules in $E_1^{2,0}$ that we do not yet know whether they survive to $E_2$ or not are in Table  \ref{F:ct-E2-lwvII}, together with the associated l.w.v. The proof is similar to previous cases and we omit most details. Table \ref{F:ct-E2-lwvII} also implies that $\bbU_{1,1}[1]$ and $\bbU_{3,3}[1]$ in $E_1^{3,0}$
are in the image of $\partial$ and do not survive to $E_2$.

Let us however give some details for verifying $(\partial\phi)(u,u,u_1) \neq 0$ on $2\bbU_{1,1}[1] \subset E^{2,0}_1$.  As $(\fg_0)_{\bar{0}}^{\ss}$-modules, $\fm_{\bar{1}} \cong \bbC^2 \boxtimes \bbC^2 \cong \bbU_{1,1}[-1] \subset H^0(\fm_{\bar{0}},\fg)$.  Take the basis
 \begin{align*}
 & e_1 = x_1 x_2 \leftrightarrow e_{-\alpha_1 - \alpha_2},
 && e_2 = y_1 x_2 \leftrightarrow e_{-\alpha_1 - 2\alpha_2 - \alpha_3}, \\
 & e_3 = x_1 y_2 \leftrightarrow e_{-\alpha_1 - \alpha_2 - 2\alpha_3}, 
 && e_4 = y_1 y_2 \leftrightarrow e_{-\alpha_1 - 2\alpha_2 - 3\alpha_3}, \quad
 \end{align*}
 and let $(\omega^1,\omega^2,\omega^3,\omega^4)$ be the dual basis. The lowering operators are $Y_1 = y_1\partial_{x_1}$ and $Y_2 = y_2 \partial_{x_2}$, which we use to verify that any l.w.v. $\phi \in 2\bbU_{1,1}[1] \subset \bigwedge^2 (\fm_{\bar{1}})^* \otimes \fm_{\bar{1}}$ is of the form:
 \[
 \phi = a \left(\omega^1 \wedge \omega^1 \otimes e_1 
 + \omega^1 \wedge \omega^2 \otimes e_2 
 + \omega^1 \wedge \omega^3 \otimes e_3\right) 
 + \left((a - b)\omega^1 \wedge \omega^4 
 + b \omega^2 \wedge \omega^3 \right) \otimes e_4.
 \]
 Since $[e_1,e_4] \neq 0$, then $ (\partial\phi)(e_1,e_1,e_4) = -2 [e_1, \phi(e_1,e_4)] - [e_4, \phi(e_1,e_1)] = (-3 a + 2b) [e_1, e_4]$ is generically nonzero, while $b=\frac{3}{2} a$ yields a l.w.v. for $\bbU_{1,1}[1] \subset E^{2,0}_1$ in the image of $\partial|_{E^{1,0}_1}$.
\end{proof}
\begin{table}[H]
 \begin{small}
 \[
 \begin{array}{|c|c|l|l|} \hline
 (p,q) & \mbox{$(\fg_{0})_{\bar 0}$-module} & \mbox{Form of l.w.v. $\phi \in E_1^{p,q}$} & \mbox{Remarks}\\ \hline
 (1,0) & \bbU_{1,3}[1] & e_{\alpha_1 + \alpha_2} \otimes [e_{-2\alpha_3}] & (\partial \phi)(u,u) \neq 0\\ \cline{2-4}
 & \bbU_{1,1}[1] & e_{\alpha_1 + \alpha_2} \otimes [h_{2\alpha_3}] + c e_{\alpha_1 + \alpha_2 + 2\alpha_3} \otimes [e_{-2\alpha_3}] & (\partial \phi)(u,u) \neq 0\\ \cline{2-4}
 & \bbV_{2,2}[0] & e_{\alpha_1 + \alpha_2} \otimes [e_{-\alpha_1 - 2\alpha_2 - 3\alpha_3}] & (\partial \phi)(u,u) \neq 0\\ \cline{2-4}
 & \bbV_{0,0}[0] & \begin{array}{@{}l} 
 e_{\alpha_1 + \alpha_2} \otimes [e_{-\alpha_1 - 2\alpha_2 - 3\alpha_3}]\\
 \quad + c_1 e_{\alpha_1 + \alpha_2 + 2\alpha_3} \otimes [e_{-\alpha_1 - 2\alpha_2 - \alpha_3}]\\
 \quad + c_2 e_{\alpha_1 + 2\alpha_2 + \alpha_3} \otimes [e_{-\alpha_1 - \alpha_2 - 2\alpha_3}]\\
 \quad + c_3 e_{\alpha_1 + 2\alpha_2 + 3\alpha_3} \otimes [e_{-\alpha_1 - \alpha_2}]\\
 \end{array} 
 & (\partial \phi)(u,u) \neq 0\\ \hline
 (1,1) & \bbU_{4,3}[2] & e_{\alpha_1 + \alpha_2} \otimes [e_{\alpha_1} \otimes e_{-2\alpha_3}] & (\partial \phi)(u,u) \neq 0\\ \cline{2-4}
 & \bbU_{4,1}[2] & \begin{array}{@{}l} e_{\alpha_1 + \alpha_2} \otimes [e_{\alpha_1} \otimes h_{2\alpha_3}] \\
 \quad + c e_{\alpha_1 + \alpha_2+2\alpha_3} \otimes [e_{\alpha_1} \otimes e_{-2\alpha_3}]
 \end{array} & (\partial \phi)(u,u) \neq 0\\ \cline{2-4}
 & \bbU_{2,3}[2] & \begin{array}{@{}l} e_{\alpha_1 + \alpha_2} \otimes [e_{\alpha_1 + \alpha_2 + \alpha_3} \otimes e_{-2\alpha_3}] \\
 \quad + c e_{\alpha_1 + 2\alpha_2 + \alpha_3} \otimes [e_{\alpha_1} \otimes e_{-2\alpha_3}]
 \end{array} & (\partial \phi)(u,u) \neq 0\\ \cline{2-4}
 & \bbU_{2,1}[2] & \begin{array}{@{}l} 
 e_{\alpha_1 + \alpha_2} \otimes [e_{\alpha_1 + \alpha_2 + \alpha_3} \otimes h_{2\alpha_3}]
  \\
 \quad + c_1 e_{\alpha_1 + \alpha_2 + 2\alpha_3} \otimes [e_{\alpha_1 + \alpha_2 + \alpha_3} \otimes e_{-2\alpha_3}]\\
 \quad + c_2 e_{\alpha_1 + 2\alpha_2 + 3\alpha_3} \otimes [e_{\alpha_1} \otimes e_{-2\alpha_3}]\\
 \quad + c_3 e_{\alpha_1 + 2\alpha_2 + \alpha_3} \otimes [e_{\alpha_1} \otimes h_{2\alpha_3}]
 \end{array} & (\partial \phi)(u,u) \neq 0\\ \hline
 \end{array}
 \]
 \end{small}
 \caption{The remaining modules in $E_{1}^{p,q}=\Lambda^p (\fms{1})^* \otimes H^q(\fm_{\bar{0}},\fg)$ for the values $(p,q)=(1,0)$ and $(1,1)$. Here $u = e_{-\alpha_1 - \alpha_2}\in\fm_{\bar 1}$ and $c,c_i$ are constants.}
 \label{F:ct-E2-lwv}
 \end{table}
\begin{table}[H]
 \begin{small}
 \[
 \begin{array}{|c|c|l|l|} \hline
 (p,q) & \mbox{$(\fg_{0})_{\bar 0}$-module} & \mbox{Form of l.w.v. $\phi \in E_1^{p,q}$} & \mbox{Remarks}\\ \hline
 (2,0) & \bbV_{2,4}[2] & e_{\alpha_1 + \alpha_2} \wedge e_{\alpha_1 + \alpha_2} \otimes [e_{-2\alpha_3}] & (\partial \phi)(u,u,u) \neq 0\\ \cline{2-4}
   & \bbV_{2,2}[2] &
   \begin{array}{@{}l}
   e_{\alpha_1 + \alpha_2} \wedge e_{\alpha_1 + \alpha_2} \otimes [h_{2\alpha_3}] \\
   \quad + c e_{\alpha_1 + \alpha_2} \wedge e_{\alpha_1 + \alpha_2+2\alpha_3} \otimes [e_{-2\alpha_3}] \\
   \end{array} & (\partial \phi)(u,u,u) \neq 0\\ \cline{2-4}
   & \bbV_{2,0}[2] &
   \begin{array}{@{}l}
   e_{\alpha_1 + \alpha_2 } \wedge e_{\alpha_1 + \alpha_2} \otimes [e_{2\alpha_3}]\\
   \quad + c_1 e_{\alpha_1 + \alpha_2} \wedge e_{\alpha_1 + \alpha_2 + 2\alpha_3} \otimes [h_{2\alpha_3}] \\
   \quad + c_2 e_{\alpha_1 + \alpha_2 + 2\alpha_3} \wedge e_{\alpha_1 + \alpha_2 + 2\alpha_3} \otimes [e_{-2\alpha_3}] \\
   \end{array} & (\partial \phi)(u,u,u_1) \neq 0\\ \cline{2-4}
 & \bbV_{0,2}[2] & \begin{array}{@{}l} (e_{\alpha_1 + \alpha_2} \wedge e_{\alpha_1 + 2\alpha_2 + 3\alpha_3} \\
  \quad + c e_{\alpha_1 + \alpha_2 + 2\alpha_3} \wedge e_{\alpha_1 + 2\alpha_2 + \alpha_3}) \otimes [e_{-2\alpha_3}]  \end{array} & (\partial \phi)(u,u,u_1) \neq 0\\ \cline{2-4}
 & \bbU_{3,3}[1] & e_{\alpha_1 + \alpha_2} \wedge e_{\alpha_1 + \alpha_2} \otimes [e_{-\alpha_1 - 2\alpha_2 - 3\alpha_3}] & (\partial \phi)(u,u,u) \neq 0\\ \cline{2-4}
 & 2\bbU_{1,1}[1] &
 \begin{array}{@{}l}
 a e_{\alpha_1 + \alpha_2} \wedge e_{\alpha_1 + \alpha_2} \otimes [e_{-\alpha_1 - \alpha_2}] \\
 \quad + a e_{\alpha_1 + \alpha_2} \wedge e_{\alpha_1 + \alpha_2 + 2\alpha_3} \otimes [e_{-\alpha_1 - \alpha_2 - 2\alpha_3}] \\
 \quad + a e_{\alpha_1 + \alpha_2} \wedge e_{\alpha_1 + 2\alpha_2 + \alpha_3} \otimes [e_{-\alpha_1 - 2\alpha_2 - \alpha_3}] \\
 \quad + (a-b) e_{\alpha_1 + \alpha_2} \wedge e_{\alpha_1 + 2\alpha_2 + 3\alpha_3} \otimes [e_{-\alpha_1 - 2\alpha_2 - 3\alpha_3}]\\
 \quad + b e_{\alpha_1 + \alpha_2 + 2\alpha_3} \wedge e_{\alpha_1 + 2\alpha_2 + \alpha_3} \otimes [e_{-\alpha_1 - 2\alpha_2 - 3\alpha_3}]
 \end{array} & \begin{array}{c} (\partial\phi)(u,u,u_1) \neq 0 \\ \mbox{for generic $a,b$} \end{array} \\ \hline
 \end{array}
 \]
 \end{small}
 \caption{The remaining modules in $E_{1}^{2,0}=\Lambda^2 (\fms{1})^* \otimes H^0(\fm_{\bar{0}},\fg)$. The elements $u= e_{-\alpha_1 - \alpha_2}$ and $u_1 = e_{-\alpha_1 - 2\alpha_2 - 3\alpha_3}$ are in $\fm_{\bar 1}$ and $a,b,c,c_i$ are constants.} 
 \label{F:ct-E2-lwvII}
 \end{table}

 \subsubsection{The main result}
By Proposition \ref{prop:H012} and Proposition \ref{prop:E2pagecontact}, it only remains to understand whether $\bbV_{4,2}[2] \subset E_2^{0,2}$ survives to the $E_4$-page.
We note that this is the only possibly non-trivial contribution to the cohomology group $H^{2,2}(\fm,\fg)$, and that the latter is a representation for the semisimple part $\mathfrak{osp}(3|2)$ of $\fg_0$.
However, a simple LSA does not admit any non-trivial representation that is purely even, hence $\bbV_{4,2}[2]$ does not survive
to the $E_4$-page.

 \begin{theorem}
 \label{prop:P1-cohom}
 Let $\fg=\fg_{-2}\oplus\cdots\oplus\fg_{2}$ be the contact grading of $\fg = G(3)$ with associated parabolic subalgebra $\fp_1^{\rm IV}$. Then:
 \[
  H^{d,1}(\fm,\fg) = \begin{cases}
  0, & d > 0;\\
  \bbV_{6,0}[0] \op \bbV_{2,0}[0]\op\bbU_{4,1}[0], & d=0;
  \end{cases}
 \]
 and
 \[
  H^{d,2}(\fm,\fg) =
  \begin{cases}
  0, & d = 0 \mbox{ or } d > 1;\\
  \bbV_{7,0}[1] \op \bbV_{3,0}[1] \op \bbV_{5,0}[1] \op \bbV_{5,2}[1] \op \bbU_{7,1}[1] \op \bbU_{5,1}[1]  \op \bbU_{3,1}[1], & d=1;
  \end{cases}
 \]
where $\bbV_{k,\ell}[r]$ is the even irreducible representation with highest weight $(k,\ell)$ for $\fsl(2) \oplus \fsp(2)$ and having degree $r$ w.r.t. the grading element $\sfZ_1$ of $\fg$, and $\bbU_{k,\ell}[r]$ is the same, but regarded as an odd module. The cohomology groups
$H^{0,1}(\fm,\fg)$ and $H^{1,2}(\fm\,\fg)$ are irreducible representations for $\mathfrak{osp}(3|2)$, their highest weights w.r.t. a choice of distinguished Borel subalgebra (i.e., corresponding to a Dynkin diagram with just one odd root) are $(4,1)$ and $(5,2)$, respectively.
 \end{theorem}
\begin{proof}
We already proved the first claim.

Consider now a composition series for the $\mathfrak{osp}(3|2)$-module $H^{0,1}(\fm,\fg)$ and note, by dimension reasons, that its irreducible constituents are given by some of the $\mathfrak{osp}(3|2)$-modules of small dimension displayed in \cite[Table 3.65]{MR1773773} (possibly up to a parity change). Under the finer decomposition for the action of $\fsl(2) \oplus \fsp(2)\subset\mathfrak{osp}(3|2)$,
the group $H^{0,1}(\fm,\fg)$ is the direct sum of all the irreducible $\fsl(2) \oplus \fsp(2)$-submodules that appear in its $\mathfrak{osp}(3|2)$-constituents.
A look at such decompositions for the $\mathfrak{osp}(3|2)$-modules in \cite[Table 3.65]{MR1773773} of dimension less than or equal to $\dim H^{0,1}(\fm,\fg)=(10|10)$ immediately implies that the composition series consists of just one module of highest weight $(4,1)$, i.e., $H^{0,1}(\fm,\fg)$ is $\mathfrak{osp}(3|2)$-irreducible. (Note that the labels in \cite[Table 3.65]{MR1773773} are in the opposite order than ours and that $\mathfrak{osp}(3|2)_{\bar 0}$-modules are indicated by their dimensions.)

The argument for $H^{1,2}(\fm,\fg)$ is completely analogous.
\end{proof}

 As $\fg_0$-modules, $H^{0,1}(\fm,\fg) \cong \der_{gr}(\fg_-) / \fg_0 \cong \fcspo(\fg_{-1}) / \fg_0$, so our result $H^{0,1}(\fm,\fg)\cong\bbV_{6,0}[0] \op \bbV_{2,0}[0] \op \bbU_{4,1}[0]$ could have been directly obtained from the modules $S^6 \bbC^2 \op \fsl(2) \op S^4 \bbC^2 \boxtimes \bbC^2$ appearing in the proof of Proposition \ref{P:max-subalg}. Vanishing of $H^{d,1}(\fm,\fg)$ for $d > 0$ implies that for the Tanaka--Weisfeiler prolongation $\pr(\fm,\fg_0)$, we have \cite{MR0266258}:

 \begin{cor} \label{C:pr-ct}
  Let $\fg=\fg_{-2}\oplus\cdots\oplus\fg_{2}$ be the contact grading of $\fg = G(3)$ with associated parabolic subalgebra $\fp_1^{\rm IV}$.  Then $\fg \cong \pr(\fm,\fg_0)$.
 \end{cor}


\subsection{Spencer cohomology for \texorpdfstring{$\mathfrak{p}^{\rm IV}_{2}\subset\mathrm{G}(3)$}{} }
\label{sec:cohomology}

\subsubsection{The cochain complex for \texorpdfstring{$\mathfrak{p}^{\rm IV}_{2}\subset\mathrm{G}(3)$}{} and the groups \texorpdfstring{$H^1(\fm,\fg)$}{}}
\label{subsec:cohp2}
Theorem \ref{thm:SC2IV} deals with the Spencer cohomology of $\mathrm{G}(3)$ w.r.t. the grading with associated parabolic subalgebra $\mathfrak{p}_2^{\rm IV}$. \\
Proofs will be given in this and next sections, and in Appendix \ref{appendixA}. We depart here with the description of the relevant cochain complex and some intermediate but important results.

Table \ref{tab:p_2^IV} below recollects the components of the $\mathbb Z$-grading $\fg=\fg_{-3}\oplus\cdots\oplus\fg_{3}$ of $\fg=\mathrm{G}(3)$, emphasizing their structure as modules for the (semisimple part of the) reductive LSA $$\fg_0=\mathbb{C} Z\oplus \fsl(2)\oplus\mathfrak{osp}(1|2)\;,$$ where $Z=Z_2$ is the grading element, together with branchings w.r.t. the purely even subalgebra $(\fg_0)_{\bar 0}=\mathbb {C}Z\oplus \fsl(2)\oplus\fsp(2)$ (as already discussed in Section \ref{S:G3-gradings}). Note that the grading is compatible with the decomposition
 \begin{equation}\label{eq:decomp}
\fg = \fg_{\bar 0}\oplus\fg_{\bar 1}
      = (\mathrm{G}(2)\oplus\fsp(2))\oplus (\CC^7\boxtimes\CC^2)\;;
 \end{equation}
more precisely the grading induces the $(2,3,5)$-grading on $\mathrm{G}(2)$, $\fsp(2)$ sits all in degree zero while the odd part $\fg_{\bar 1}$ has no graded components in degrees $\pm 3$. In particular, $Z$ is an element of $\mathrm{G}(2)$ and it precisely coincides with the grading element of the $(2,3,5)$-grading.
\begin{small}
\begin{table}[h]
\begin{center}
\begin{tabular}{c||*{1}{c}|*{1}{c}|*{1}{c}|*{1}{c}||}
\toprule
  {\bfseries Graded Components} & ${\bf \fg_0}$ & ${\bf\fg_{\pm 1}}$ & ${\bf\fg_{\pm 2}}$ & ${\bf\fg_{\pm 3}}$  \\
\midrule[0.02em]\midrule[0.02em]
& $\CC^{1|0}\boxtimes\CC^{1|0}$ &  &  & \\
{As $\fg_0$-module} & $\fsl(2)\boxtimes\CC^{1|0}$\; & $\CC^{2|0}\boxtimes \CC^{1|2}$ & $\CC^{1|0}\boxtimes \CC^{1|2}$ & $\CC^{2|0}\boxtimes \CC^{1|0}$ \\
& $\;\;\;\;\;\;\CC^{1|0}\boxtimes\mathfrak{osp}(1|2)$ & & & \\
\midrule[0.02em]
& $\CC\boxtimes\CC$  &  &  & \\
{Even part as $(\fg_0)_{\bar 0}$-module} & $\fsl(2)\boxtimes\CC$\;\;\;\;\, & $\CC^{2}\boxtimes\CC$\,\,  & $\CC\boxtimes\CC$ & $\CC^{2}\boxtimes\CC$\,\,\\
{}  & \;\;\;\;\;$\CC\boxtimes\fsp(2)$  &  &  & \\
\midrule[0.02em]
& &  & & \\
{Odd part as $(\fg_0)_{\bar 0}$-module}& \,\,$\CC\boxtimes\CC^{2}$ & $\CC^{2}\boxtimes\CC^{2}$ & \,\,$\CC\boxtimes\CC^{2}$ & \\
&  &  &  & \\
\midrule[0.02em]\midrule[0.02em]
{Dimension} & \multicolumn{1}{c|}{$7|2$} & \multicolumn{1}{c|}{$2|4$} & \multicolumn{1}{c|}{$1|2$} & \multicolumn{1}{c||}{$2|0$} \\
\bottomrule
\end{tabular}
\end{center}
\caption{The decomposition of the graded components of $\fg=\mathrm{G(3)}$.}
\label{tab:p_2^IV}
\end{table}
\end{small}
\par
Some of the components obtained by restriction of the bracket of $\mathrm{G(3)}$ to the irreducible $(\fg_0)_{\bar 0}$-modules of Table \ref{tab:p_2^IV} are automatically zero, by $(\fg_0)_{\bar 0}$-equivariance, parity and $\mathbb Z$-degree.
 It is a straightforward matter  using the root system of $\mathrm{G}(3)$ to verify that all other components have ``full rank'' -- i.e., image as large as permitted by Schur's lemma, parity and $\mathbb Z$-degree -- with the sole exception
of the Lie brackets between the irreducible $(\fg_0)_{\bar 0}$-components of $\fg_0$.

This implies the following lemma.
\begin{lemma}
\label{lem:centralizers}
The centralizer of $(\fg_{-2})_{\bar 0}\oplus(\fg_{-3})_{\bar 0}$ in $\fg$ is given by $\fsp(2)\oplus(\fg_{-1})_{\bar 1}\oplus\fg_{-2}\oplus\fg_{-3}$ and the centralizer of $\fm_{\bar 0}$ by $\fsp(2)\oplus(\fg_{-2})_{\bar 1}\oplus \fg_{-3}$.
\end{lemma}
Now, a direct application of Kostant's version of the Bott--Borel--Weil theorem tells us that, as $\fsl(2)$-modules, $H^{d,1}(\fm_{\bar 0},\mathrm{G}(2))=0$ for all $d\geq 0$ and
\begin{equation*}
\begin{aligned}
H^{d,1}(\fm_{\bar 0},\CC)&\cong\begin{cases} 0\;\;\,\,\text{for all}\; d\geq 0, d\neq 1\;\\
\CC^2\;\text{if}\; d=1\end{cases}\\
H^{d,1}(\fm_{\bar 0},\CC^7)&\cong\begin{cases} 0\;\;\,\,\text{for all}\; d>0\;\\
S^2\CC^2\;\;\,\,\text{if}\; d=0\end{cases}
\end{aligned}
\end{equation*}
so that the only non-trivial homogeneous components of $H^{1}(\fm_{\bar 0},\fg)$ are $H^{0,1}(\fm_{\bar 0},\fg)\cong S^2\CC^2\boxtimes \CC^2$
and $H^{1,1}(\fm_{\bar 0},\fg)\cong \CC^2\boxtimes \fsp(2)$. Explicitly:
\begin{equation}
\begin{aligned}
H^{0,1}(\fm_{\bar 0},\fg)&=\Big\{\varphi:(\fg_{-1})_{\bar 0}\to(\fg_{-1})_{\bar 1}\mid [X,\varphi(Y)]=[Y,\varphi(X)]\;\;\text{for all}\;\;X,Y\in(\fg_{-1})_{\bar 0}\Big\}\;,\\
H^{1,1}(\fm_{\bar 0},\fg)&=\fsp(2)\otimes(\fg_{-1})_{\bar 0}^*\;.
\end{aligned}
\end{equation}
The following is a consequence of the above discussion, Proposition \ref{prop:longexact} and Lemma \ref{lem:centralizers}.
\begin{proposition}
\label{prop:3exse}
There exist long exact sequences of $(\fg_0)_{\bar 0}$-modules
\begin{equation}
\begin{split}
\label{sequenzaII}
0\longrightarrow \fsp(2)\longrightarrow H^{0,1}(\fm_{\bar 1},\fg)\longrightarrow H^{0,1}(\fm, \fg)\longrightarrow S^2\CC^2\boxtimes \CC^2
\end{split}
\end{equation}
\begin{equation}
\begin{split}
\label{sequenzaIII}
0\longrightarrow H^{1,1}(\fm_{\bar 1},\fg)\longrightarrow H^{1,1}(\fm, \fg)\longrightarrow \CC^2\boxtimes \fsp(2)
\end{split}
\end{equation}
and
\begin{equation}
\begin{split}
\label{sequenzaIV}
0\longrightarrow H^{d,1}(\fm_{\bar 1},\fg)\longrightarrow H^{d,1}(\fm, \fg)\longrightarrow 0
\end{split}
\end{equation}
for all $d>1$.
\end{proposition}

To proceed further, we shall need the explicit form of some of the Lie brackets of $\fg$.
We fix a symplectic basis
$(\be_1,\be_2)$ for $\CC^{2}\boxtimes\CC$ and $(\bep_1,\bep_2)$ for $\CC\boxtimes\CC^{2}$, normalised to $\omega_{12}=\omega^{12} = 1$, where $\omega$ denotes the symplectic structures on both spaces. We will use small latin indices for $\CC^{2}\boxtimes\CC$, employ Einstein's summation convention and
the Northeast convention
to raise and lower
indices, when working in components:
\begin{equation*}
\begin{aligned}
X^a&=X_b\omega^{ba}\qquad\text{and}\qquad X_a=\omega_{ab}X^b\;,
\end{aligned}
\end{equation*}
from where it follows that $\omega_{ab} \omega^{ac} = \delta_b^c$. We will work on $\CC\boxtimes\CC^{2}$ likewise  with greek indices.
Finally, we fix a basis $\1$ of $\CC\boxtimes\CC$
and use the short-cut $\bep_{a\alpha}=\be_{a}\boxtimes\bep_{\alpha}$ for elements of $\CC^2\boxtimes\CC^2$.
The proof of the following result is omitted for the sake of brevity.
\begin{lemma}
The non-trivial components of the Lie brackets of $\fm$ are:
\begin{itemize}
\item[(i)] For all $\be_a,\be_b\in (\fg_{-1})_{\bar 0}$ and $\bep_{a\alpha},\bep_{b\beta}\in (\fg_{-1})_{\bar 1}$ we have
\begin{equation}
\begin{aligned}
{}[\be_a,\be_b]&=\omega_{ab}\1\;,\\
[\be_a,\bep_{b\beta}]&=\omega_{ab}\bep_\beta\;,\\
[\bep_{a\alpha},\bep_{b\beta}]&=\omega_{ab}\omega_{\alpha\beta}\1\;,
\end{aligned}
\end{equation}
where all R.H.S. are in $\fg_{-2}$;
\item[(ii)] For all $\be_a\in (\fg_{-1})_{\bar 0}$, $\bep_{a\alpha}\in (\fg_{-1})_{\bar 1}$ and $\bep_\beta\in(\fg_{-2})_{\bar 1}$ we have
\begin{equation}
\begin{aligned}
{}[\be_a,\1]&=\be_a\;,\\
[\bep_{a\alpha},\bep_\beta]&=\omega_{\alpha\beta}\be_a\;,
\end{aligned}
\end{equation}
where all R.H.S. are in $\fg_{-3}$.
\end{itemize}
The LSA $\fg_{0}$ acts by derivations on $\fm$ via the natural action of $(\fg_0)_{\bar 0}$ on the irreducible modules of Table \ref{tab:p_2^IV} and the action of any $\bep_\gamma\in(\fg_0)_{\bar 1}$, whose non-trivial components are given by:
\begin{equation}
\begin{aligned}
{}[\bep_\gamma,\be_a]&=\bep_{a\gamma}\;,\\
[\bep_\gamma,\bep_{a\alpha}]&=-2\omega_{\gamma\a}\be_a\;,\\
[\bep_\gamma,\1]&=2\bep_\gamma\;,\\
[\bep_{\gamma},\bep_\beta]&=-\omega_{\gamma\beta}\1\;,\\
\end{aligned}
\end{equation}
where $\be_a\in (\fg_{-1})_{\bar 0}$, $\bep_{a\alpha}\in (\fg_{-1})_{\bar 1}$ and $\bep_\beta\in(\fg_{-2})_{\bar 1}$.
\end{lemma}
 We depart by directly computing some of the cohomology groups of Definition \ref{def:complexM1}.
\begin{lemma}
\label{lem:cohomM1}
The cohomology group $H^{0,1}(\fm_{\bar 1},\fg)\cong\fsp(2)$ whereas $H^{d,1}(\fm_{\bar 1},\fg)=0$ for all $d>0$.
\end{lemma}
\begin{proof}
The groups in question consist just of the cocycles $\varphi\in\Hom(\fm_{\bar 1},\fg)$ (remember that $C^{0}(\fm_{\bar 1},\fg)=0$ by definition).
We first note that
$0=\partial\varphi(X,Y) 
=[X,\varphi(Y)]$ whenever
\begin{itemize}
	\item[(i)] $X\in\fm_{\bar 0}$ has degree $-2,-3$ and $Y\in\fm_{\bar 1}$; or
	\item[(ii)] $X\in(\fg_{-1})_{\bar 0}$ and $Y\in(\fg_{-2})_{\bar 1}$.
\end{itemize}
We then have
\begin{equation}
\label{eq:imagesubset}
\begin{aligned}
\varphi(\fg_{-1})_{\bar 1}&\subset\fsp(2)\oplus(\fg_{-1})_{\bar 1}\oplus\fg_{-2}\oplus\fg_{-3}\;,\\
\varphi(\fg_{-2})_{\bar 1}&\subset\fsp(2)\oplus(\fg_{-2})_{\bar 1}\oplus\fg_{-3}\;,
\end{aligned}
\end{equation}
by Lemma \ref{lem:centralizers}, in particular $H^{d,1}(\fm_{\bar 1},\fg)=0$ for all $d>2$.
We now decompose
$\varphi=\varphi_{\bar 0}+\varphi_{\bar 1}$ into its even and odd components and study them separately and for each degree $d=0,1,2$.
\vskip0.03cm\par\noindent
{\underline{Case $d=2$}} In this case $\varphi=\varphi_{\bar 1}:(\fg_{-2})_{\bar 1}\longrightarrow\fsp(2)$ by \eqref{eq:imagesubset} and
$
0=\varphi([X,Y])
$
for all $X\in(\fg_{-1})_{\bar 0}$ and $Y\in(\fg_{-1})_{\bar 1}$, whence $\varphi=0$.
\vskip0.03cm\par\noindent
{\underline{Case $d=1$}}
Similarly $\varphi=\varphi_{\bar 1}:(\fg_{-1})_{\bar 1}\longrightarrow\fsp(2)$ and
$
0=[Y,\varphi(X)]
$
for $X\in(\fg_{-1})_{\bar 1}$ and $Y\in(\fg_{-2})_{\bar 1}$, whence $\varphi=0$.
\vskip0.05cm\par\noindent
{\underline{Case $d=0$}} In this case $\varphi_{\bar 1}=0$ by \eqref{eq:imagesubset}. We work in components and write
$
\varphi=\varphi^{b\beta}{}_{a\alpha}+\varphi^{\beta}{}_{\alpha}
$,
where
\begin{align}
\label{eq:d0-1}
\varphi^{b\beta}{}_{a\alpha}&:(\fg_{-1})_{\bar 1}\longrightarrow (\fg_{-1})_{\bar 1}\;,\\
\label{eq:d0-2}
\varphi^{\beta}{}_{\alpha}&:(\fg_{-2})_{\bar 1}\longrightarrow (\fg_{-2})_{\bar 1}\;.
\end{align}
Now
\begin{align*}
0&=\partial\varphi(\be_a,\bep_{b\beta})=[\be_a,\varphi(\bep_{b\beta})]-\varphi[\be_a,\bep_{b\beta}]\\
&=(\varphi_{a}{}^{\gamma}{}_{b\beta}-\omega_{ab}\varphi^\gamma{}_\beta)\bep_{\gamma}
\end{align*}
for all $\be_a\in(\fg_{-1})_{\bar 0}$ and $\bep_{b\beta}\in(\fg_{-1})_{\bar 1}$. It follows that $\varphi_{a\gamma b\beta}=\omega_{ab}\varphi_{\gamma\beta}$, with the component \eqref{eq:d0-1} completely determined by \eqref{eq:d0-2}. Furthermore
\begin{align*}
0&=\partial\varphi(\bep_{a\alpha},\bep_{b\beta})=[\bep_{a\alpha},\varphi(\bep_{b\beta})]+[\bep_{b\beta},\varphi(\bep_{a\alpha})]\\
&=(\varphi_{a\alpha b\beta}+\varphi_{b\beta a\alpha})\1
\end{align*}
for all $\bep_{a\alpha},\bep_{b\beta}\in(\fg_{-1})_{\bar 1}$, which readily implies $\varphi_{\alpha\beta}=\varphi_{\beta\alpha}$.

The group $H^{0,1}(\fm_{\bar 1},\fg)$ is a $(\fg_0)_{\bar 0}$-module and we have just seen that it either vanishes or it is isomorphic to
$\fsp(2)$. However, it cannot vanish due to the exact sequence \eqref{sequenzaII}.
\end{proof}
\begin{corollary}
\label{cor:esd=0,1}
$H^{d,1}(\fm, \fg)=0$ for all $d>1$ and there exist short exact sequences of $(\fg_0)_{\bar 0}$-modules
\begin{equation}
\begin{split}
\label{sequenzaIIbis}
0\longrightarrow  H^{0,1}(\fm, \fg)\longrightarrow S^2\CC^2\boxtimes \CC^2\;,
\end{split}
\end{equation}
\begin{equation}
\begin{split}
\label{sequenzaIIIbis}
0\longrightarrow H^{1,1}(\fm, \fg)\longrightarrow \CC^2\boxtimes \fsp(2)\;.
\end{split}
\end{equation}
\end{corollary}
We are now ready to prove our first main result, namely the vanishing of $H^{1}(\fm, \fg)$ in non-negative degrees. In view of the preceeding corollary, this amounts to understanding the images of the embeddings in \eqref{sequenzaIIbis} and \eqref{sequenzaIIIbis}.

Interestingly, this can be done without any computation.
Recall that any group $H^{d,1}(\fm,\fg)$ has a representation of the LSA $\fg_0$ and not just of $(\fg_0)_{\bar 0}$. In particular, it carries a representation of the ``di-spin'' algebra $\mathfrak{osp}(1|2)$ \cite{MR0438925}. By a result of 
Djokovi\'{c} and
Hochschild, all (finite-dimensional) representations of $\mathfrak{osp}(1|2)$ are completely reducible \cite{MR0387363}.
Moreover, any irreducible
representation different from the $1$-dimensional trivial one decomposes, under $\mathfrak{osp}(1|2)_{\bar 0}\cong\fsp(2)$, into the direct sum of {\it two} irreducible
$\fsp(2)$-modules with highest weights which are {\it different} and {\it consecutive} \cite{MR0438925}. 

By Corollary \ref{cor:esd=0,1}
\begin{align}
H^{0,1}(\fm, \fg)&\subset 3\CC^2\\
H^{1,1}(\fm, \fg)&\subset 2S^2\CC^2
\end{align}
as $\fsp(2)$-modules. The above discussion tells us immediately that $H^{d,1}(\fm, \fg)=0$ for $d=0,1$ and we have proved the following.
\begin{theorem}
\label{thm:235H^1}
Let $\fg=\fg_{-3}\oplus\cdots\oplus\fg_{3}$ be the $\mathbb Z$-grading of $\fg=\mathrm{G}(3)$ with the associated parabolic subalgebra $\fp_2^{\rm IV}$. Then
$H^{d,1}(\fm, \fg)=0$ for all $d\geq 0$.
\end{theorem}
By \cite{MR0266258} we then have:
 \begin{cor} \label{C:pr-235}
  Let $\fg$ be as in Theorem \ref{thm:235H^1}, then $\fg \cong \pr(\fm)$.
 \end{cor}

\subsubsection{The groups \texorpdfstring{$H^2(\fm,\fg)$}{} and the main result}
\label{subsec:spencoh}
We now turn to the second cohomology group, whose determination is more involved.
The classical Kostant's theorem tells us that
\begin{equation*}
\begin{aligned}
H^{4,2}(\fm_{\bar 0},\mathrm{G}(2))&\cong S^4\CC^2\;,\\
H^{4,2}(\fm_{\bar 0},\CC)&\cong S^2\CC^2\;,\\
H^{3,2}(\fm_{\bar 0},\CC^7)&\cong S^3\CC^2\;,
\end{aligned}
\end{equation*}
as $\fsl(2)$-modules, and this gives, together with the knowledge of the lowest weight vectors, all the non-trivial components of $H^{2}(\fm_{\bar 0},\fg)$:
\begin{equation}
\label{eq:BBWcoh2}
\begin{aligned}
H^{4,2}(\fm_{\bar 0},\fg)&=\Big\{\varphi_{ab}{}^{d}{}_{c}:(\fg_{-1})_{\bar 0}\otimes (\fg_{-3})_{\bar 0}\to\fsl(2)\mid \varphi_{abdc}=\varphi_{(abdc)}
\Big\}\\
&\;\;\;\;+\Big\{\varphi_{ab}{}^{\beta}{}_{\alpha}:(\fg_{-1})_{\bar 0}\otimes (\fg_{-3})_{\bar 0}\to\fsp(2)\mid \varphi_{ab}{}^{\beta}{}_{\alpha}=\varphi_{(ab)}{}^{\beta}{}_{\alpha}\Big\}\\
&\cong S^4\CC^2\boxtimes \CC+ S^2\CC^2\boxtimes \fsp(2)\;,\\
H^{3,2}(\fm_{\bar 0},\fg)&=\Big\{\varphi_{ab}{}^{c\gamma}:(\fg_{-1})_{\bar 0}\otimes (\fg_{-3})_{\bar 0}\to(\fg_{-1})_{\bar 1}\mid \varphi_{abc}{}^{\gamma}=\varphi_{(abc)}{}^{\gamma}
\Big\}\\
&\cong S^3\CC^2\boxtimes \CC^2
\;.
\end{aligned}
\end{equation}
By Propositions \ref{prop:longexact}, \ref{prop:3exse} and Theorem \ref{thm:235H^1} we then have:
\begin{proposition}
\label{prop:4spectralsequences}
There exist long exact sequences of $(\fg_0)_{\bar 0}$-modules
\begin{equation}
\begin{split}
\label{sequenzaH^2I}
0\longrightarrow\fsp(2)\otimes(\fg_{-1})_{\bar 0}^*\longrightarrow H^{1,2}(\fm_{\bar 1},\fg)\longrightarrow H^{1,2}(\fm,\fg)\longrightarrow 0
\end{split}
\end{equation}
\begin{equation}
\begin{split}
\label{sequenzaH^2II}
0\longrightarrow H^{3,2}(\fm_{\bar 1},\fg)\longrightarrow H^{3,2}(\fm,\fg)\longrightarrow S^3\CC^2\boxtimes \CC^2
\end{split}
\end{equation}
\begin{equation}
\begin{split}
\label{sequenzaH^2III}
0\longrightarrow H^{4,2}(\fm_{\bar 1},\fg)\longrightarrow H^{4,2}(\fm,\fg)\longrightarrow S^4\CC^2\boxtimes \CC+ S^2\CC^2\boxtimes \fsp(2)
\end{split}
\end{equation}
and
\begin{equation}
\begin{split}
\label{sequenzaH^2VI}
0\longrightarrow H^{d,2}(\fm_{\bar 1},\fg)\longrightarrow H^{d,2}(\fm,\fg)\longrightarrow 0
\end{split}
\end{equation}
for $d=2$ and all $d\geq 5$.
\end{proposition}
Sequence \eqref{sequenzaH^2III} can be immediately improved as follows. We first note that $C^{4,2}(\fm,\fg)$, as a $(\fg_0)_{\bar 0}$-module, has a unique irreducible submodule of type $S^4\CC^2\boxtimes\CC$, which consists of the maps as in the first set described in
\eqref{eq:BBWcoh2}. Its elements are closed in the classical complex $C^{\bullet}(\fm_{\bar 0},\fg)$, but they are {\it not} closed in the full complex $C^{\bullet}(\fm,\fg)$: the condition
\begin{align*}
\partial\varphi(\be_a,\be_b,\bep_{c\gamma})=-\varphi_{ab}{}^{d}{}_{c}\bep_{d\gamma}=0
\end{align*}
for all $\be_a\in(\fg_{-1})_{\bar 0}$, $\be_b\in(\fg_{-3})_{\bar 0}$ and $\bep_{c\gamma}\in(\fg_{-1})_{\bar 1}$ yields $\varphi=0$. It follows that $H^{4,2}(\fm,\fg)$ does not contain any irreducible $(\fg_0)_{\bar 0}$-submodule of type $S^4\CC^2\boxtimes\CC$ and that the image of
the restriction map $$\res:H^{4,2}(\fm,\fg)\longrightarrow H^{4,2}(\fm_{\bar0},\fg)$$ is actually included in $S^2\CC^2\boxtimes \fsp(2)$. We proved:
\begin{proposition}
\label{cor:improvement}
There exists a long exact sequence of $(\fg_0)_{\bar 0}$-modules
\begin{equation}
\label{sequenzaH^2IIIimprovement}
\begin{split}
0\longrightarrow H^{4,2}(\fm_{\bar 1},\fg)\longrightarrow H^{4,2}(\fm,\fg)\longrightarrow S^2\CC^2\boxtimes \fsp(2)\;.
\end{split}
\end{equation}
\end{proposition}
The next one is our main result.
\begin{theorem}
\label{thm:SC2IV}
Let $\fg=\fg_{-3}\oplus\cdots\oplus\fg_{3}$ be the $\mathbb Z$-grading of $\fg=\mathrm{G}(3)$ with the associated parabolic subalgebra $\fp_2^{\rm IV}$. Then:
\begin{equation*}
\begin{aligned}
H^{d,1}(\fm,\fg)&=0\;\text{for all}\;d\geq 0\;,\\
H^{d,2}(\fm,\fg)_{\bar 0}&\cong\begin{cases} 0\;\;\,\,\;\;\,\,\,\;\;\;\;\;\text{for all}\;d>0, d\neq 2,\\
S^2\CC^2\boxtimes\Lambda^2\CC^2\;\text{if}\; d=2\;,\end{cases}\!\!\!\!\!\\
H^{d,2}(\fm,\fg)_{\bar 1}&=0\;\text{for all}\;d>0\;,\\
\end{aligned}
\end{equation*}
Moreover, any cohomology class in $H^{2,2}(\fm,\fg)_{\bar 0}$ admits a canonical representative $\varphi$ with components
\begin{align*}
\varphi_{a b\b}{}^\gamma&:(\fg_{-1})_{\bar 0}\otimes(\fg_{-1})_{\bar 1}\to(\fg_{0})_{\bar 1}\;,\\
\varphi_{a\a b\b}{}^{\fsl(2)}+\varphi_{a\a b\b}{}^{\fsp(2)}&:\Lambda^2(\fg_{-1})_{\bar 1}\to\fsl(2)\oplus\fsp(2)\subset (\fg_{0})_{\bar 0}\;,\\
\varphi_{\a a}{}^{b\b}&:(\fg_{-2})_{\bar 1}\otimes(\fg_{-1})_{\bar 0}\longrightarrow (\fg_{-1})_{\bar 1}\;,\\
\varphi_{a\a b}{}^{\gamma}&:(\fg_{-1})_{\bar 1}\otimes\fg_{-3}\longrightarrow (\fg_{-2})_{\bar 1}\;,
\end{align*}
which are symmetric (resp. skewsymmetric) in lowered latin indices (resp. greek indices) and which all depend on the first component $\varphi_{ab\beta}{}^{\gamma}$ via the cocycle conditions
\begin{equation}
\label{eq:cocyleconditions}
\begin{aligned}
\varphi_{b\b a}{}^\gamma&=-2\varphi_{a b\b}{}^\gamma\;,\\
\varphi_{\a cb\b}&=\varphi_{c b\b\a}\;,\\
\varphi_{a\a b\b}{}^{\fsl(2)}(\be_c)&=
-2\varphi_{cb\b\a}\be_{a}
-2\varphi_{ca\a\b}\be_{b}\;,\\
\varphi_{a\a b\b}{}^{\fsp(2)}(\bep_{\gamma})&=-2\omega_{\a\gamma}\varphi_{a b\b}{}^\delta\bep_{\delta}
-2\omega_{\b\gamma}\varphi_{ba\a}{}^\delta\bep_\delta\;,
\end{aligned}
\end{equation}
where $\be_c\in\CC^2\boxtimes\CC$ and $\bep_\gamma\in\CC\boxtimes\CC^2$.
\end{theorem}
\begin{proof}
The claim on first cohomology groups has already been established in Theorem \ref{thm:235H^1} whereas the vanishing of $H^{d,2}(\fm,\fg)_{\bar 1}$ in positive degrees uses Djokovi\'{c}--Hochschild theorem and it is postponed to Proposition \ref{prop:SC2IVodd}. It remains to compute the even part of the second cohomology groups, which we do via Propositions \ref{prop:4spectralsequences} and \ref{cor:improvement}
and an explicit description of the groups $H^{d,2}(\fm_{\bar 1},\fg)_{\bar 0}$.

In Appendix \ref{appendixA}, we show that $H^{d,2}(\fm_{\bar 1},\fg)_{\bar 0}=0$
for all $d\geq 3$, so in particular $H^{d,2}(\fm,\fg)_{\bar 0}=0$ for all $d \geq 5$ by \eqref{sequenzaH^2VI}. The even part
of \eqref{sequenzaH^2II} is
\begin{equation*}
\begin{split}
0\longrightarrow H^{3,2}(\fm_{\bar 1},\fg)_{\bar 0}\longrightarrow H^{3,2}(\fm,\fg)_{\bar 0}\longrightarrow 0\;,
\end{split}
\end{equation*}
whence $H^{3,2}(\fm,\fg)_{\bar 0}=0$ as well.

By the first part of Proposition \ref{thm:cohomology42} in Appendix \ref{appendixA} and Proposition \ref{cor:improvement}
we have an exact sequence
$$
0\longrightarrow H^{4,2}(\fm,\fg)_{\bar 0}\longrightarrow S^2\CC^2\boxtimes \fsp(2)\;,
$$
so that $H^{4,2}(\fm,\fg)_{\bar 0}=0$, since $H^{4,2}(\fm,\fg)_{\bar 0}$ does not contain any $(\fg_0)_{\bar 0}$-submodule isomorphic to $S^2\CC^2\boxtimes S^2\CC^2$ by the second part of Proposition \ref{thm:cohomology42}.

It remains to deal with the cases $d=1,2$, which we now study separately.
\vskip0.1cm\par\noindent
{\underline{Case $d=2$}} We have
$H^{2,2}(\fm,\fg)_{\bar 0}\cong H^{2,2}(\fm_{\bar 1},\fg)_{\bar 0}$ by Proposition \ref{prop:4spectralsequences} and we compute the latter.

We first note that there is a non-trivial space
$(\fg_{1})_{\bar 1}\otimes (\fg_{-1})_{\bar 1}^*+(\fg_{0})_{\bar 1}\otimes (\fg_{-2})_{\bar 1}^*$
of $1$-cochains, on which the differential acts faithfully. Indeed, we may use the associated $2$-coboundaries to accomodate any
$2$-cocycle $\varphi\in Z^{2,2}(\fm_{\bar 1},\fg)_{\bar 0}$ to
vanish on $(\fg_{-2})_{\bar 0}\otimes (\fg_{-1})_{\bar 1}$ and $(\fg_{-2})_{\bar 0}\otimes (\fg_{-2})_{\bar 1}$.
Any such ``normalized'' $2$-cocycle has therefore components
\begin{align*}
\varphi_{a b\b}{}^\gamma&:(\fg_{-1})_{\bar 0}\otimes(\fg_{-1})_{\bar 1}\to(\fg_{0})_{\bar 1}\;,\\
\varphi_{a\a b\b}{}^Z+\varphi_{a\a b\b}{}^{\fsl(2)}+\varphi_{a\a b\b}{}^{\fsp(2)}&:\Lambda^2(\fg_{-1})_{\bar 1}\to(\fg_{0})_{\bar 0}=\mathbb {C}Z\oplus \fsl(2)\oplus\fsp(2)\;,\\
\varphi_{a\a\b}{}^{b}&:(\fg_{-1})_{\bar 1}\otimes (\fg_{-2})_{\bar 1}\longrightarrow (\fg_{-1})_{\bar 0}\;,\\
\varphi_{\a a}{}^{b\b}&:(\fg_{-2})_{\bar 1}\otimes(\fg_{-1})_{\bar 0}\longrightarrow (\fg_{-1})_{\bar 1}\;,\\
\varphi_{a\a b}{}^{\gamma}&:(\fg_{-1})_{\bar 1}\otimes\fg_{-3}\longrightarrow (\fg_{-2})_{\bar 1}\;,\\
\varphi_{\a\b}{}^{\1}&:\Lambda^2 (\fg_{-2})_{\bar 1}\longrightarrow (\fg_{-2})_{\bar 0}\;.
\end{align*}
First of all we note
\begin{align*}
\partial\varphi|_{(\fg_{-2})_{\bar 0}\otimes(\fg_{-1})_{\bar 1}\otimes(\fg_{-2})_{\bar 1}}
&=0\Longrightarrow \varphi_{a\a\b}{}^{b}\be_b=0\Longrightarrow \varphi_{a\a\b}{}^{b}=0\;,\\
\partial\varphi|_{(\fg_{-2})_{\bar 0}\otimes(\fg_{-1})_{\bar 1}\otimes(\fg_{-1})_{\bar 1}}
&=0\Longrightarrow \varphi_{a\a b\b}{}^Z=0\;.
\end{align*}
Now
\begin{align}
\notag
\partial\varphi|_{(\fg_{-2})_{\bar 0}\otimes(\fg_{-1})_{\bar 1}\otimes(\fg_{-1})_{\bar 0}}
&=0\Longrightarrow [\varphi_{a b\b}{}^\gamma\bep_\gamma,\1]+\varphi(\bep_{b\b},[\be_a, \1])=0\\
&\;\;\;\;\;\;\,\notag\Longrightarrow 2\varphi_{a b\b}{}^\gamma\bep_\gamma+\varphi_{b\b a}{}^\gamma\bep_\gamma=0\\
&\;\;\;\;\;\;\,\label{eq:2,2I}\Longrightarrow \varphi_{b\b a}{}^\gamma=-2\varphi_{a b\b}{}^\gamma
\end{align}
and
\begin{align}
\notag
\partial\varphi|_{(\fg_{-1})_{\bar 0}\otimes(\fg_{-1})_{\bar 0}\otimes(\fg_{-1})_{\bar 1}}
&\notag=0\Longrightarrow [\be_a,\varphi_{b c\gamma}{}^\delta \bep_\delta]-[\be_b,\varphi_{a c\gamma}{}^\delta \bep_\delta]\\
&\notag\;\;\;\;\;\;\;\;\;\;\;\;\;\;-\varphi([\be_b,\bep_{c\gamma}],\be_a)+\varphi([\be_a,\bep_{c\gamma}],\be_b)=0
\\
&\;\;\;\;\;\;\,\notag\Longrightarrow \varphi_{b c\gamma}{}^\delta \bep_{a\delta}-\varphi_{a c\gamma}{}^\delta\bep_{b\delta}=
\omega_{ac}\varphi_{\gamma b}{}^{d\delta}\bep_{d\delta}
-\omega_{bc}\varphi_{\gamma a}{}^{d\delta}\bep_{d\delta}
\\
&\;\;\;\;\;\;\,\label{eq:2,2II}\Longrightarrow \varphi_{b c\gamma}{}^\delta \delta^d_a
-\varphi_{a c\gamma}{}^\delta \delta^{d}_b=\omega_{ac}\varphi_{\gamma b}{}^{d\delta}
-\omega_{bc}\varphi_{\gamma a}{}^{d\delta}\;.
\end{align}
In a similar way, one gets
\begin{align}
\label{eq:2,2III}
\partial\varphi|_{(\fg_{-1})_{\bar 0}\otimes(\fg_{-2})_{\bar 1}\otimes(\fg_{-2})_{\bar 1}}
&=0
\Longrightarrow
\omega_{ba}\varphi_{\a\b}{}^\1=\varphi_{\b ab}{}_{\a}+\varphi_{\a a b}{}_{\beta}
\;,\\
\label{eq:2,2IV}
\partial\varphi|_{(\fg_{-2})_{\bar 1}\otimes(\fg_{-1})_{\bar 0}\otimes(\fg_{-1})_{\bar 0}}
&=0\Longrightarrow
\varphi_{\a [a b]\gamma}=0
\;,\\
\label{eq:2,2V}
\partial\varphi|_{(\fg_{-2})_{\bar 1}\otimes(\fg_{-1})_{\bar 1}\otimes(\fg_{-1})_{\bar 0}}
&=0\Longrightarrow
\varphi_{c b\b\a}-\varphi_{\a cb\b}=\omega_{bc}\varphi_{\a\b}{}^{\1}
\;.
\end{align}
It is clear by equations \eqref{eq:2,2III}-\eqref{eq:2,2V} that
\begin{align}
\varphi_{\a\b}{}^\1&=0\;,\\
\varphi_{\b ab}{}_{\a}&=-\varphi_{\a a b}{}_{\beta}\;,\\
\label{eq:relation}
\varphi_{\a cb\b}&=\varphi_{c b\b\a}\;,
\end{align}
whence
\begin{equation}
\label{eq:fundsymm}
\varphi_{c b\b\a}=-\varphi_{c b\a\b}=\varphi_{bc\b\a}\;,
\end{equation}
i.e., the first component of the normalized cocyle defines an element of $S^2\CC^2\boxtimes\Lambda^2\CC^2$.

We want to show that \eqref{eq:2,2II} is automatically satisfied. We first plug \eqref{eq:relation} back into \eqref{eq:2,2II}, lower all the indices, use skew-symmetry \eqref{eq:fundsymm} in the greek indices and rearrange terms to arrive at the equation
\begin{align*}
\omega_{da}\varphi_{b c\gamma\delta}-\omega_{ca}\varphi_{bd\gamma\delta}-\omega_{db}\varphi_{a c\gamma\delta}
+\omega_{cb}\varphi_{ad\gamma\delta}=0\;,
\end{align*}
which, suppressing greek indices, is a linear equation on symmetric bilinear forms on $\CC^2$ (w.r.t. latin indices). By $\fsl(2)$-equivariance, this equation is either solved only by  the zero form or the whole $S^2\CC^2$. It is not difficult to check that at least one non-zero form solves the equation, hence all do.

The cocyle conditions
\begin{align}
\notag
\partial\varphi|_{\fg_{-3}\otimes(\fg_{-1})_{\bar 1}\otimes(\fg_{-1})_{\bar 1}}
&=0\notag\Longrightarrow [\varphi_{a\a b\b}{}^{\fsl(2)},\be_c]=\omega_{\a\gamma}\varphi_{b\b c}{}^{\gamma}\be_{a}+
\omega_{\b\gamma}\varphi_{a\a c}{}^{\gamma}\be_{b}\\
&\notag\;\;\;\;\;\;\,\Longrightarrow
[\varphi_{a\a b\b}{}^{\fsl(2)},\be_c]=\varphi_{b\b c\a}\be_{a}+
\varphi_{a\a c\b}\be_{b}\\
&\label{eq:2,2VI}\;\;\;\;\;\;\,\Longrightarrow
[\varphi_{a\a b\b}{}^{\fsl(2)},\be_c]=
-2\varphi_{cb\b\a}\be_{a}
-2\varphi_{ca\a\b}\be_{b}
\;,\\
\notag
\partial\varphi|_{(\fg_{-2})_{\bar 1}\otimes(\fg_{-1})_{\bar 1}\otimes(\fg_{-1})_{\bar 1}}
&=0\Longrightarrow [\varphi_{a\a b\b}{}^{\fsp(2)},\bep_{\gamma}]=\omega_{\a\gamma}\varphi_{b\b a}{}^{\delta}\bep_{\delta}
+\omega_{\b\gamma}\varphi_{a\a b}{}^{\delta}\bep_\delta\\
&\;\;\;\;\;\;\,\label{eq:2,2VII}\Longrightarrow [\varphi_{a\a b\b}{}^{\fsp(2)},\bep_{\gamma}]=-2\omega_{\a\gamma}\varphi_{a b\b}{}^\delta\bep_{\delta}
-2\omega_{\b\gamma}\varphi_{ba\a}{}^\delta\bep_\delta
\;,
\end{align}
determine the remaining components taking values in $\fsl(2)$ and $\fsp(2)$ in terms, again, of the first component of the normalized cocyle. It is a direct matter to verify that the remaining cocycle conditions on $(\fg_{-1})_{i}\otimes(\fg_{-1})_{\bar 1}\otimes(\fg_{-1})_{\bar 1}$, $i\in\mathbb Z_2$, are all automatically satisfied.

In summary, all non-trivial components of the normalized cocyle are determined by the first component
$$\varphi_{ab\b\gamma}\in S^2\CC^2\boxtimes\Lambda^2\CC^2$$ via \eqref{eq:2,2I}, \eqref{eq:relation}, \eqref{eq:2,2VI} and \eqref{eq:2,2VII}.
Hence $H^{2,2}(\fm,\fg)_{\bar 0}\cong S^2\CC^2\boxtimes\Lambda^2\CC^2$.
\vskip0.1cm\par\noindent
{\underline{Case $d=1$}}
We recall that $H^{1,2}(\fm,\fg)_{\bar 0}\cong H^{1,2}(\fm_{\bar 1},\fg)_{\bar 0}/\partial(\fsp(2)\otimes (\fg_{-1})_{\bar 0}^*)$ by Proposition \ref{prop:4spectralsequences} and note that any $2$-cocycle $\varphi\in Z^{1,2}(\fm_{\bar 1},\fg)_{\bar 0}$ has components
\begin{align*}
\varphi_{a}{}^{c\gamma}{}_{b\beta}&:(\fg_{-1})_{\bar 0}\otimes(\fg_{-1})_{\bar 1}\to(\fg_{-1})_{\bar 1}\;,\\
\varphi_{a\alpha b\beta}{}^{c}&:\Lambda^2(\fg_{-1})_{\bar 1}\to(\fg_{-1})_{\bar 0}\;,\\
\varphi_{a\alpha\1}{}^{\beta}&:(\fg_{-1})_{\bar 1}\otimes (\fg_{-2})_{\bar 0}\longrightarrow (\fg_{-2})_{\bar 1}\;,\\
\varphi_{a\alpha \b}{}^{\1}&:(\fg_{-1})_{\bar 1}\otimes (\fg_{-2})_{\bar 1}\longrightarrow (\fg_{-2})_{\bar 0}\;,\\
\varphi_{a}{}^{\gamma}{}_{\beta}&:(\fg_{-1})_{\bar 0}\otimes(\fg_{-2})_{\bar 1}\longrightarrow (\fg_{-2})_{\bar 1}\;,\\
\varphi_{\alpha\beta}{}^a&:\Lambda^2 (\fg_{-2})_{\bar 1}\longrightarrow \fg_{-3}\;.
\end{align*}
Now $H^{1,1}(\fm_{\bar 1},\fg)_{\bar 0}=0$ by Lemma \ref{lem:cohomM1}, whence $Z^{1,1}(\fm_{\bar 1},\fg)_{\bar 0}=B^{1,1}(\fm_{\bar 1},\fg)_{\bar 0}=0$ and the differential $\partial$ acts faithfully on the space of $1$-cochains
$
C^{1,1}(\fm_{\bar 1},\fg)_{\bar 0}=(\fg_{0})_{\bar 1}\otimes (\fg_{-1})_{\bar 1}^*+(\fg_{-1})_{\bar 1}\otimes (\fg_{-2})_{\bar 1}^*
$.
We may use the associated $2$-coboundaries to modify $\varphi\in Z^{1,2}(\fm_{\bar 1},\fg)_{\bar 0}$ in such a way that
\begin{equation}
\begin{split}
\label{eq:twocondnorm}
\varphi_{a\alpha\1}{}^{\beta}&=0\;,\\
\varphi_{a\alpha \b}{}^{\1}&=0\;,
\end{split}
\end{equation}
and, when working in $H^{1,2}(\fm,\fg)_{\bar 0}\cong H^{1,2}(\fm_{\bar 1},\fg)_{\bar 0}/\partial(\fsp(2)\otimes (\fg_{-1})_{\bar 0}^*)$, we may quotient also with $\partial(\fsp(2)\otimes (\fg_{-1})_{\bar 0}^*)$ and arrange for
\begin{equation}
\begin{split}
\label{eq:onecondnorm}
\varphi_{a}{}_{\gamma\beta}=-\varphi_{a}{}_{\beta\gamma}\;.
\end{split}
\end{equation}
We have proved that $H^{1,2}(\fm,\fg)_{\bar 0}$ is isomorphic with the space of $2$-cocyles $\varphi\in Z^{1,2}(\fm_{\bar 1},\fg)_{\bar 0}$
which satisfy \eqref{eq:twocondnorm}-\eqref{eq:onecondnorm}. We now show that this space is trivial.

The non-trivial cocyle conditions are:
\begin{align}
\label{eq:ICC}
\partial\varphi|_{(\fg_{-2})_{\bar 0}\otimes(\fg_{-1})_{\bar 1}\otimes(\fg_{-1})_{\bar 1}}
&=0\Longrightarrow \varphi_{a\a b\b}{}^c=0\;,\\
\label{eq:IICC}
\partial\varphi|_{(\fg_{-2})_{\bar 1}\otimes(\fg_{-1})_{\bar 1}\otimes(\fg_{-1})_{\bar 0}}
&=0\Longrightarrow
 \varphi_{a c\gamma b\beta}-\omega_{cb}\varphi_{a\beta\gamma}-\omega_{ab}\varphi_{\beta\gamma c}=0
\;,\\
\label{eq:IIICC}
\partial\varphi|_{(\fg_{-1})_{\bar 0}\otimes(\fg_{-1})_{\bar 1}\otimes(\fg_{-1})_{\bar 1}}
&=0\Longrightarrow \varphi_{ac\gamma b\beta}+\varphi_{ab\beta c\gamma}=0\;,\\
\label{eq:VICC}
\partial\varphi|_{(\fg_{-1})_{\bar 0}\otimes(\fg_{-1})_{\bar 0}\otimes(\fg_{-1})_{\bar 1}}
&=0\Longrightarrow \varphi_{ba\delta c\gamma}-\varphi_{ab\delta c\gamma}
=\omega_{cb}\varphi_{a\delta\gamma}-\omega_{ca}\varphi_{b\delta\gamma}\;.
\end{align}
It follows that
\begin{align*}
\varphi_{ba\delta c\gamma}-\varphi_{ab\delta c\gamma}
&=\omega_{cb}\varphi_{a\delta\gamma}-\omega_{ca}\varphi_{b\delta\gamma}\\
&=\varphi_{a c\gamma b\delta}-\omega_{ab}\varphi_{\delta\gamma c}
-\varphi_{b c\gamma a\delta}+\omega_{ba}\varphi_{\delta\gamma c}\\
&=\varphi_{a c\gamma b\delta}-\varphi_{b c\gamma a\delta}-2\omega_{ab}\varphi_{\delta\gamma c}\\
&=-\varphi_{a b\delta c\gamma}+\varphi_{ba\delta c\gamma}-2\omega_{ab}\varphi_{\delta\gamma c}\;,
\end{align*}
whence $\varphi_{\alpha\beta}{}^{a}=0$. Furthermore
\begin{align*}
\varphi_{a c\gamma b\beta}&=\omega_{cb}\varphi_{a\beta\gamma}\\
&=\omega_{bc}\varphi_{a\gamma\beta}\\
&=\varphi_{a b\beta c\gamma}
\end{align*}
so that $\varphi_{a}{}^{c\gamma}{}_{b\beta}=0$ and $\varphi_{a}{}^{\gamma}{}_{\beta}=0$.
\end{proof}
Recall that $H^{d,2}(\fm,\fg)$ is naturally an  $\mathfrak{osp}(1|2)$-module and that
its odd part has highest weights under $\mathfrak{osp}(1|2)_{\bar 0}\cong\fsp(2)$ which are necessarily odd integers. By Djokovi\'{c}--Hochschild theorem and the results on $H^{d,2}(\fm,\fg)_{\bar 0}$ of Theorem \ref{thm:SC2IV}, it follows that
$H^{d,2}(\fm,\fg)_{\bar 1}=0$ for all positive $d\neq 2$, and that $H^{2,2}(\fm,\fg)_{\bar 1}$ is either vanishing or $S^2\CC^2\boxtimes\CC^2$.

\begin{proposition}
\label{prop:SC2IVodd}
Let $\fg=\fg_{-3}\oplus\cdots\oplus\fg_{3}$ be the $\mathbb Z$-grading of $\fg=\mathrm{G}(3)$ with the associated parabolic subalgebra $\fp_2^{\rm IV}$. Then $H^{d,2}(\fm,\fg)_{\bar 1}=0$ for all $d>0$.
\end{proposition}
\begin{proof}
We only need to show $H^{2,2}(\fm,\fg)_{\bar 1}=0$ and by the exact sequence \eqref{sequenzaH^2VI} this is the same as $H^{2,2}(\fm_{\bar 1},\fg)_{\bar 1}=0$. A cocyle $\varphi\in Z^{2,2}(\fm_{\bar 1},\fg)_{\bar 1}$ has the components
\begin{align*}
\varphi_{a b\b}{}^{Z}+\varphi_{a b\b}{}^{\fsl(2)}+\varphi_{a b\b}{}^{\fsp(2)}&:(\fg_{-1})_{\bar 0}\otimes(\fg_{-1})_{\bar 1}\to(\fg_{0})_{\bar 0}=\mathbb {C}Z\oplus \fsl(2)\oplus\fsp(2)\;,\\
\varphi_{a\a b\b}{}^{\gamma}&:\Lambda^2(\fg_{-1})_{\bar 1}\to(\fg_{0})_{\bar 1}\;,\\
\varphi_{\a}{}^{b\b}{}_{c\gamma}&:(\fg_{-2})_{\bar 1}\otimes(\fg_{-1})_{\bar 1} \longrightarrow (\fg_{-1})_{\bar 1}\;,\\
\varphi_{\1a\a}{}^{b}&:(\fg_{-2})_{\bar 0}\otimes (\fg_{-1})_{\bar 1}\longrightarrow (\fg_{-1})_{\bar 0}\;,\\
\varphi_{\a}{}^b{}_a&:(\fg_{-2})_{\bar 1}\otimes(\fg_{-1})_{\bar 0} \longrightarrow (\fg_{-1})_{\bar 0}\;,\\
\varphi_{a b\b}{}^\1&:\fg_{-3}\otimes (\fg_{-1})_{\bar 1}\longrightarrow (\fg_{-2})_{\bar 0}\;,\\
\varphi_{\a\b}{}^\gamma&:\Lambda^2 (\fg_{-2})_{\bar 1}\longrightarrow (\fg_{-2})_{\bar 1}\;,\\
\varphi_{\1\a}{}^{\1}&:(\fg_{-2})_{\bar 0}\otimes (\fg_{-2})_{\bar 1}\longrightarrow (\fg_{-2})_{\bar 0}\;,\\
\varphi_{a\a}{}^{b}&: \fg_{-3}\otimes(\fg_{-2})_{\bar 1}\longrightarrow \fg_{-3}\;.
\end{align*}
In this case, we may use $2$-coboundaries in $\partial((\fg_{1})_{\bar 0}\otimes (\fg_{-1})_{\bar 1}^*)$
and $\partial((\mathbb {C}Z\oplus \fsl(2))\otimes (\fg_{-2})_{\bar 1}^*)$ to accomodate first for
\begin{align}
\label{eq:B1}
\varphi_{\1a\a}{}^{b}&=0\;,\\
\label{eq:B2}
\varphi_{\a}{}^b{}_a&=0\;.
\end{align}
Using that the first prolongation of $\fsp(2)$ acting on the purely odd space $\mathbb{C}^2$ is trivial (see, e.g. \cite{MR2552682}), one sees that the component $\varphi_{\a\b}{}^\gamma$ of the Spencer differential on $\fsp(2)\otimes (\fg_{-2})_{\bar 1}^*$ is injective, hence an isomorphism and
\begin{align}
\label{eq:B3}
\varphi_{\a\b}{}^\gamma&=0
\end{align}
too. It follows that $H^{2,2}(\fm_{\bar 1},\fg)_{\bar 1}$ is isomorphic to the space of
$2$-cocycles satisfying \eqref{eq:B1}-\eqref{eq:B3}.

We first note
\begin{align*}
\partial\varphi|_{(\fg_{-2})_{\bar 0}\otimes(\fg_{-1})_{\bar 1}\otimes(\fg_{-1})_{\bar 1}}
&=0\Longrightarrow \varphi_{a\a b\b}{}^{\gamma}=0\;,\\
\partial\varphi|_{(\fg_{-2})_{\bar 1}\otimes(\fg_{-1})_{\bar 0}\otimes(\fg_{-1})_{\bar 0}}
&=0\Longrightarrow \varphi_{\1\a}{}^{\1}=0\;,\\
\partial\varphi|_{(\fg_{-2})_{\bar 1}\otimes(\fg_{-2})_{\bar 0}\otimes(\fg_{-1})_{\bar 0}}
&=0\Longrightarrow \varphi_{a\a}{}^{b}=0\;,\\
\partial\varphi|_{(\fg_{-2})_{\bar 0}\otimes(\fg_{-1})_{\bar 0}\otimes(\fg_{-1})_{\bar 1}}
&=0\Longrightarrow \varphi_{a b\b}{}^\1=-2\varphi_{a b\b}{}^{Z}\;,\\
\partial\varphi|_{(\fg_{-2})_{\bar 1}\otimes(\fg_{-2})_{\bar 1}\otimes(\fg_{-1})_{\bar 1}}
&=0\Longrightarrow \varphi_{\b b\alpha c\gamma}=-\varphi_{\a b\b c\gamma} \;,
\end{align*}
and that the identity with $3$ odd elements of degree $-1$ is automatically satisfied.

Now
\begin{align*}
\partial\varphi|_{(\fg_{-3})_{\bar 0}\otimes(\fg_{-1})_{\bar 0}\otimes(\fg_{-1})_{\bar 1}}
&=0\Longrightarrow [\varphi_{bc\gamma}{}^{\fsl(2)},\be_a]=3\varphi_{bc\gamma}{}^{Z}\be_a -\varphi_{ac\gamma}{}^\1\be_b\\
&\;\;\;\;\;\;\,\Longrightarrow [\varphi_{bc\gamma}{}^{\fsl(2)},\be_a]=3\varphi_{bc\gamma}{}^{Z}\be_a+2\varphi_{a c\gamma}{}^{Z}
\be_b
\end{align*}
whose pure trace as an endomorphism of $(\fg_{-3})_{\bar 0}$ is $0=3\omega_{af}\varphi_{bc\gamma}{}^{Z}+\omega_{bf}\varphi_{a c\gamma}{}^{Z}-\omega_{ba}\varphi_{f c\gamma}{}^{Z}$. Multiplying by $\omega^{af}$ yields $8\varphi_{bc\gamma}{}^{Z}=0$, so that
\begin{align*}
\varphi_{b c\gamma}{}^{Z}&=0\;,\\
\varphi_{bc\gamma}{}^{\fsl(2)}&=0\;,\\
\varphi_{b c\gamma}{}^{\1}&=0\;.
\end{align*}
Finally
\begin{align*}
\partial\varphi|_{(\fg_{-2})_{\bar 1}\otimes(\fg_{-1})_{\bar 1}\otimes(\fg_{-1})_{\bar 0}}
&=0\Longrightarrow [\varphi_{ba\a}{}^{\fsp(2)},\bep_\beta]=\varphi_{\b b}{}^{\delta}{}_{a\a}\bep_{\delta}\\
&\;\;\;\;\;\;\,\Longrightarrow  \varphi_{ba\a\delta\beta}=\varphi_{\b b\delta a\a}\;,
\end{align*}
which implies the vanishing of the two terms separately, due to symmetry and, respectively, skew-symmetry in the indices $\delta$ and $\beta$.
\end{proof}

 \section{$G(3)$ as the supersymmetry of differential equations}\label{S4}

 \subsection{Super jet-spaces and equation supermanifolds}
 \label{S:jet}

 A {\em contact supermanifold} $(M,\cC)$ is a supermanifold $M$ of dimension $(2p+1|n)$
equipped with a distribution $\cC$ of corank $(1|0)$ that is maximally non-integrable, i.e., locally, $\cC = \ker(\sigma)$ for some (even) $1$-form $\sigma\in\Omega^1(M)$ such that $\eta =\, d\sigma|_\cC$ is non-degenerate.  The latter is a super-skewsymmetric bilinear form on $\cC$, namely, it is skew-symmetric on $\cC_{\bar 0}$ and symmetric on $\cC_{\bar 1}$, and since $\sigma$ is only well-defined up to scale, we refer to the conformal class $[\eta]$ of $\eta$ as a {\it $CSpO$-structure} on $\cC$.

We will be interested in the case where $n=2q$ is even. In this case, there exist Darboux coordinates $(x^i, u, u_i)$ on $M$, i.e., coordinates
w.r.t. which
$$\sigma = du - \sum_{i=1}^{p+q}(dx^i) u_i\;.$$
We remark that $u$ is an even coordinate, whereas $(x^i)$ and  $(u_i)$ consists each of $p$ even and $q$ odd variables.
For any fixed $i=1,\ldots, p+q$, the variables $x^i$ and $u_i$ have the same parity, which we denote by $|i|\in\mathbb{Z}_2$.

Letting $D_{x^i} = \partial_{x^i} + u_i \partial_u$, we have $\iota_{D_{x^i}} \sigma =\iota_{\partial_{u_i}} \sigma = 0$,  where $\iota$ denotes insertion from the left, so the distribution
\begin{equation}
\cC = \langle D_{x^i}, \partial_{u_i}\mid i=1,\ldots,p+q \rangle\;.
\end{equation}
We will often use the shortcuts $\bD_i=D_{x^i}$, $\bU^i= \partial_{u_i}$, and $\cC=\langle \bD_i, \bU^i \rangle$.
It is now clear that $d\sigma = \sum_{i=1}^{p+q}dx^i \wedge du_i$ is non-degenerate on $\cC$. Locally, we refer to $M$ as the {\em first jet-(super)space} $J^1(\bbC^{p|q},\bbC^{1|0})$, where $\{ x^i\}$ are the independent variables.
\begin{definition}
A frame $\mathcal{F}=\{D_i,U^i\mid i=1,\ldots,p+q\}$ of the superdistribution $\cC$ is called a $CSpO$-{\it frame} if $\eta$ is represented w.r.t. $\mathcal F$ by a multiple of \eqref{E:CSpO-eta}.
\end{definition}
\begin{example}
\label{E:flat-CSpO}
The ``flat frame'' $\mathcal{F}_{flat}=\{\bD_i,\bU^i\mid i=1,\ldots,p+q\}$ is always a $CSpO$-frame.
\end{example}

The setting for 2nd order super-PDE is the Lagrange--Grassmann bundle $\widetilde{M}=LG(\cC) \stackrel{\pi}{\to} M$ consisting of the collection of all Lagrangian subspaces in $(\cC,[d\sigma|_\cC])$. As usual, we introduce $\widetilde M$ via its super-points, i.e., 
using the definition of the Lagrangian--Grassmannian $LG(V)$ via the functor $\bbA\mapsto LG(V)(\bbA)$ given in Section \ref{sec:lagrangiansubspacealong} and the flat $CSpO$-frame of Example \ref{E:flat-CSpO}. 

Explicitly, we take bundle-adapted local coordinates $(x^i,u,u_i,u_{ij})$ on $\widetilde{M}$, where the $\{ u_{ij}\}$ correspond to the Lagrangian subspace $L = \tspan_{\bbA}\{ \widetilde{D}_{x^i} \}$ generated by the supervector fields
 \begin{equation}
\label{wD-1}
\begin{split}
\widetilde{D}_{x^i} = \partial_{x^i} + u_i \partial_u + \sum_{j=1}^{p+q}u_{ij} \partial_{u_j}
=\bD_i + \sum_{j=1}^{p+q} u_{ij} \bU^j
\;,
 \end{split}
\end{equation}
for $i=1,\ldots,p+q$. Here $L$ is thought as a super-point of the Lagrange--Grassmann bundle, 
in particular the local coordinates on $\widetilde M$ are thought as elements of $\mathbb A$ of the appropriate parity. We stress that $u_{ij} = (-1)^{|i||j|} u_{ji}$ and that $\widetilde{D}_{x^i}(u_j) = u_{ij}$.

The supermanifold $\widetilde{M}$ inherits a canonical differential system $\widetilde\cC$. 
In terms of the functor of points, we have $\widetilde\cC|_L = (\pi_*)^{-1}(L)$ at any super-point $L$ 
of $\widetilde{M}$. Equivalently, we let
 $$
\sigma_k = du_k -\sum_{i=1}^{p+q}(dx^i) u_{ik}
 $$ 
and get 
 $$
\widetilde{\cC} = \ker\{ \sigma, \sigma_1,\dots, \sigma_{p+q} \} = 
\langle \widetilde{D}_{x^i}, \partial_{u_{ij}}\mid i,j=1,\ldots,p+q \rangle\;.
 $$
Locally, we refer to $\widetilde{M}$ as the {\em second  jet-(super)space} $J^2(\bbC^{p|q},\bbC^{1|0})$.

A vector field $\bS$ on $M$ is a contact vector field if it preserves $\cC$ via the Lie derivative, 
i.e., $\cL_{\bS} \cC \subset \cC$.
Vector fields on $\widetilde M$ that preserve $\widetilde \cC$ are also called contact. 
Any even, resp.\ odd, contact vector field $\bS$ on $M$ canonically prolongs to a unique even, 
resp.\ odd, contact vector field $\widetilde\bS$ on $\widetilde{M}$.
Conversely, we have the (super) Lie-B\"acklund theorem: any contact vector field on $\widetilde{M}$ is 
the (unique) prolongation of a contact vector field on $M$. Indeed, the derived distribution
 $$
\widetilde\cC^2 = [\widetilde\cC, \widetilde\cC] = 
\langle \partial_{x^i} + u_i \partial_u, \partial_{u_i}, \partial_{u_{ij}} \rangle
 $$ 
has associated Cauchy characteristic space 
$\Ch(\widetilde\cC^2) =\langle \partial_{u_{ij}}\mid i,j=1,\ldots, p+q\rangle$ and 
any contact vector field on $\widetilde{M}$ preserves this space, hence it is projectable over $M$.

Fixed a (local) defining $1$-form $\sigma$ of $\cC$, any (local) contact vector field $\bS$ on $M$ is uniquely determined by the {\em generating superfunction}
$f=\iota_{\bS} \sigma$, conversely any (local) superfunction on $M$ determines a contact vector field. The explicit expression of $\bS$ in terms of $f$ is given in Proposition \ref{prop:generating}.
We will write $\bS = \bS_f$ and induce the {\em Lagrange bracket} on superfunctions from the Lie bracket of vector fields via $\bS_{[f,g]}=[\bS_f,\bS_g]$, where $f$ and $g$ are two superfunctions.

 \begin{proposition}
\label{prop:generating}
The contact vector field associated to a superfunction $f = f(x^i,u,u_i)$ on $M$ is given by
 \begin{align}
\label{contact-vf}
 \bS_f &=f \partial_u -\sum_{i=1}^{p+q}(-1)^{|i|(|f|+1)} (\partial_{u_i} f ) D_{x^i} + \sum_{i=1}^{p+q}(-1)^{|i| |f|} (D_{x^i} f) \partial_{u_i}
 \end{align}
and the Lagrange bracket by
\begin{align} \label{lagrange-bracket}
 [f,g] = f \partial_u g - (-1)^{|f||g|} g \partial_u f + \sum_{i=1}^{p+q}(-1)^{|i||f|} (D_{x^i} f ) \partial_{u_i} g -\sum_{i=1}^{p+q} (-1)^{|g|(|f| + |i|)} (D_{x^i} g) \partial_{u_i} f\;,
 \end{align}
where $f$ and $g$ are two superfunctions. Finally, the canonical prolongation $\widetilde\bS_f$ of the contact vector field $\bS_f$ is given by
\begin{align} \label{contact-vf-pr}
\widetilde\bS_f= \bS_f +\sum_{j,k=1}^{p+q} h_{jk} \partial_{u_{jk}}\;,\;\;\text{where}\;\;h_{jk} = (-1)^{(|j|+|k|)|f|} \widetilde{D}_{x^j} \widetilde{D}_{x^k} f\;.
\end{align}
 \end{proposition}

 \begin{proof}
We will use Einstein's summation convention by summing over repeated indices, and write $\bS = a^i \partial_{x^i} + b \partial_u + c_i \partial_{u_i}$ and $f= \iota_{\bS}\sigma = b - a^i u_i$.  Computing modulo $\sigma$ yields
  \begin{align*}
 0 \equiv \cL_{\bS} \sigma &= d(\iota_{\bS}\sigma) + i_{\bS} d\sigma
 = dx^i (\partial_{x^i} f) + du(\partial_u f) + du_i(\partial_{u_i} f) + a^i du_i - (-1)^{|i|} c_i dx^i \\
 &\equiv dx^i \bigl(\partial_{x^i}f + u_i \partial_u f -(-1)^{|i| |f|} c_i\bigr) + 
 du_i\bigl(\partial_{u_i}f +(-1)^{|i|(|f|+1)} a^i\bigr) 
\end{align*}
whence
 \begin{align*}
 \bS&= -(-1)^{|i|(|f|+1)} \Big(\partial_{u_i} f \Big) \partial_{x^i} + \Big(f -(-1)^{|i|(|f|+1)} (\partial_{u_i} f ) u_i\Big) \partial_u + \Big( (-1)^{|i| |f|}(\partial_{x^i} f + u_i \partial_u f)\Big) \partial_{u_i}\;,
 \end{align*}
which coincides with \eqref{contact-vf}. 
Then a direct computation yields \eqref{lagrange-bracket}.

Finally, we find that $\iota_{\widetilde{\bS}_f} \sigma_k = (-1)^{|k||f|} \widetilde{D}_{x^k} f$, and hence
 \begin{align*}
 0 \equiv \cL_{\widetilde{\bS}_f} \sigma_k &= d(\iota_{\widetilde{\bS}_f} \sigma_k)+\iota_{\widetilde{\bS}_f} d \sigma_k  \\
 &= (-1)^{|k| |f|}d(\widetilde{D}_{x^k} f)- (-1)^{|k||f|} du_{ki} (\partial_{u_i} f )  - (-1)^{|j||f|}dx^j (h_{jk}) \\
 &\equiv - (-1)^{|j||f|}dx^j (h_{jk}) + (-1)^{|k| |f|}\Big(dx^i (\widetilde{D}_{x^i} \widetilde{D}_{x^k} f) + du_{ij} (\partial_{u_{ij}} \widetilde{D}_{x^k} f)- du_{ki} (\partial_{u_i} f )\Big)\\
 &\equiv dx^j \Big(-(-1)^{|j||f|}h_{jk} + (-1)^{|k| |f|} (\widetilde{D}_{x^j} \widetilde{D}_{x^k} f) \Big),
 \end{align*}
where we used that the parity $|h_{jk}|=|f|+|j|+|k|$. This immediately gives \eqref{contact-vf-pr}.
 \end{proof}
 
 \begin{definition}
A {\em 2nd order super-PDE} is a sub-supermanifold $\cE\subset \widetilde{M}$. 
A {\em contact} (or {\em external}) {\em symmetry}  of $\cE$ is a contact vector field $\widetilde\bS$ 
on $\widetilde{M}$ that 
is tangent to $\cE$.
 \end{definition}

We recall that a vector field on a supermanifold is not determined by its values at points. The definition of the restriction of $\widetilde\bS$  to $\cE$ may easily be given via the action on superfunctions, see e.g. \cite[\S 1.5]{MR2640006}.

 \subsection{The $G(3)$-contact super-PDE}\label{S:G3-PDE}
 
From now on and until the end of Section \ref{S4}, we restrict ourselves to the case where the 
contact supermanifold $(M,\cC)$ has dimension $\dim M=(5|4)$, so the contact distribution $\cC$ 
has rank $(4|4)$. On $M$ we consider coordinates $(x^i,u,u_i)=(x,y,\nu,\tau,u,u_x,u_y,u_\nu,u_\tau)$,
where $x,y,u,u_x,u_y$ are even and $\nu,\tau,u_\nu,u_\tau$ odd, and similarly for the coordinates on 
$\widetilde M$, which is $(9|8)$-dimensional. In the rest of Section \ref{S4}, it is convenient to let the index $i$ run from $0$ to $3$, instead of $1\leq i\leq 4$ as per Section \ref{S:jet}.
%

Given {\it any} CSpO-frame 
\begin{equation}
\label{E:nonflat-CSpO}
\mathcal{F}=\{D_i,U^i\mid i=0,\ldots,3\}
\end{equation}
of $\cC$,
we may introduce a subvariety $\cV \subset \bbP(\cC)$ of projectivized supervector fields according to the (Zariski-closure of the) parametrization \eqref{E:cV}.  See also the note before Definition \ref{D:cV}.  
Namely, for any fixed $T = t^a w_a=\lambda w_1 + \theta w_2 + \phi w_3\in W(\bbA)$ as in \eqref{eq:parameterT}, we consider the (even) supervector field
\begin{align} \label{bV}
 \bV(T) = D_0 - t^a D_a - \frac{1}{2} \fC(T^3) U^0 - \frac{3}{2} \fC_c(T^2) U^a
 \end{align}
and the $\mathbb A$-submodule $[\bV(T)]=\tspan_{\bbA}\{\bV(T)\}$ of rank $(1|0)$ that it generates. The analogous definition can be given for the super-point at infinity.
\begin{definition}
The subvariety $\cV\subset\bbP(\cC)$ with functor of points $\mathbb A\mapsto\cV(\mathbb A)=\bigcup_{T\in W(\bbA)}[\bV(T)]$ is called the {\it field of $(1|2)$-twisted cubics} associated to the frame \eqref{E:nonflat-CSpO}.
\end{definition}
One can easily check (using Remark \ref{R:cV-hV} and the functorial description of Lie supergroups) 
that the field of $(1|2)$-twisted cubics depends only on the orbit of a CSpO-frame under the natural right action of $G_0=COSp(3|2)$, the connected subgroup of $CSpO(4|4)$ with LSA $\fg_0=\mathfrak{cosp}(3|2)\subset\fcspo(4|4)$.

\begin{definition} \label{G(3)-ct-geo}
A {\em $G(3)$-contact supergeometry} $(M,\cC,\mathcal V)$ is the datum of a contact supermanifold $(M,\cC)$ of dimension $(5|4)$
equipped with a field of $(1|2)$-twisted cubics $\cV \subset \bbP(\cC)$ associated to a $G_0$-stable family of CSpO-frames.
\end{definition}
To any $G(3)$-contact supergeometry, we may also associate the collection 
$$\widehat\cV= \{ \widehat{T}^{(1)}_\ell \cV \mid \ell=\text{super-point of}\;\cV \} \subset \widetilde M=LG(\cC)$$
of the affine tangent spaces along $\cV$.
From Proposition \ref{P:hV}, we locally have $\widehat{T}^{(1)}_{[\bV(T)]} \cV = \tspan_{\bbA}\{B_0, B_a \}$, where $a=1,2,3$ and
\begin{equation}
\label{TV}
 \begin{aligned}
 B_0 &= D_0 + \fC(T^3) U^0 + \frac{3}{2} \fC_a(T^2) U^a,\\ 
 B_a &= D_a + \frac{3}{2} \fC_a(T^2) U^0 + 3 \fC_{ac}(T) U^c\;. %
 \end{aligned}
\end{equation}
We recall that  $\cV$ and $\widehat\cV$ provide the same reduction  $G_0\subset CSpO(4|4)$ of structure 
group and that $\cV$ can be canonically recovered from $\widehat\cV$ (see Remark \ref{R:cV-hV}). 
From this and the Lie-B\"acklund theorem, it follows that $(M,\cC,\cV)$ and 
$(\widetilde M,\widetilde\cC,\widehat\cV)$ have the same contact symmetries.

We will now focus on the {\em flat} $G(3)$-contact supergeometry, i.e., the supergeometry $(M,\cC,\cV)$ for which an admissible CSpO-frame is $\mathcal{F}_{flat}=\{\bD_i,\bU^i\mid i=1,\ldots,p+q\}$ (see Example \ref{E:flat-CSpO}).

\begin{theorem}
\label{thm:superPDEG(3)}
The 2nd order super-PDE $\cE=\widehat\cV\subset \widetilde M$ naturally associated to the flat $G(3)$-contact supergeometry is $(6|6)$-dimensional and it is given by the $G(3)$-contact super-PDE system, i.e., the following generalization of the classical $G(2)$-contact PDE system:
 \begin{equation}
\label{G3-PDE}
 \begin{split}
 u_{xx} &= \displaystyle \frac{1}{3} u_{yy}^3 + 2 u_{yy} u_{y\nu} u_{y\tau}, \quad
 u_{xy} = \frac{1}{2} u_{yy}^2 + u_{y\nu} u_{y\tau}, \\
 u_{x\nu} &= u_{yy} u_{y\nu}, \quad u_{x\tau} = u_{yy} u_{y\tau}, \quad u_{\nu\tau} = -u_{yy}.
 \end{split}
\end{equation}
\end{theorem}

\begin{proof}
 Comparing the expressions for $\widehat{T}^{(1)}_{[\bV(T)]} \cV = \tspan_{\bbA}\{ \bB_0, \bB_a \}$ with the bundle-adapted local coordinates of $\widetilde{M}$ defined by \eqref{wD-1}, we observe that $\widehat\cV$ is simply defined by the relations
 \begin{align} \label{G-PDE}
 \begin{pmatrix}
 u_{00} & u_{0a} \\
 u_{a0} & u_{ab}
 \end{pmatrix}
 =
 \begin{pmatrix}
 \fC(T^3) & \frac{3}{2} \fC_a(T^2)\\
 \frac{3}{2} \fC_a(T^2) & 3\fC_{ab}(T)
 \end{pmatrix}\;,
 \end{align}
 or, using the explicit components of $\fC$ from \eqref{E:CG}, by
 \begin{align} \label{PDE-sys}
 \begin{pmatrix}
 u_{xx} & u_{xy} & u_{x\nu} & u_{x\tau}\\
 u_{yx} & u_{yy} & u_{y\nu} & u_{y\tau}\\
 u_{\nu x} & u_{\nu y} & u_{\nu\nu} & u_{\nu\tau}\\
 u_{\tau x} & u_{\tau y} & u_{\tau\nu} & u_{\tau\tau}\\
 \end{pmatrix} =
 \begin{pmatrix}
 \frac{\lambda^3}{3} + 2\lambda \theta \phi & \frac{\lambda^2}{2} + \theta \phi & \lambda\phi & -\lambda\theta\\
 \frac{\lambda^2}{2} + \theta \phi & \lambda & \phi & -\theta\\
 \lambda\phi & \phi & 0 & -\lambda\\
 -\lambda\theta & -\theta & \lambda & 0
 \end{pmatrix}.
 \end{align}
 Eliminating the parameters $\lambda,\theta,\phi$, yields the desired result.
\end{proof}

In the next sections, we will show that the space of contact symmetries of \eqref{G3-PDE} has the maximal possible dimension. Even more, we will see that it is isomorphic to $G(3)$.

 \begin{rem} 
The so-called {\em Goursat PDE} play an important role in Cartan's classical $G(2)$-story \cite{MR1509120}, 
and Yamaguchi \cite{MR1699860} considered their generalization to the other exceptional simple 
Lie algebras $\g$.  In \cite{MR3759350}, a uniform and explicit parametric description for such 
equations with symmetry $\g$ was found.  Remarkably, this further generalizes to the $G(3)$-case as:
 \begin{align}\label{Goursat-param}
u_{00} = t^a t^b u_{ba} - 2\fC(T^3), \quad u_{0a} = t^b u_{ba} - \frac{3}{2} \fC_a(T^2),\quad
  1\leq a,b\leq 3.
 \end{align}
Analogous to \cite{MR3759350}, this system is geometrically obtained by considering the family 
$\widetilde\cV$ of {\em all} Lagrangian subspaces, locally described by \eqref{wD-1}, 
which contain some super-point $\ell$ of $\cV$, locally described by \eqref{bV} w.r.t.\ 
the flat frame $\mathcal{F}_{flat}$ over $M$. In particular, $\widehat\cV \subset \widetilde\cV$.  
Imposing the incidence condition quickly leads to the given expressions. 
 
Equations \eqref{Goursat-param} define a submanifold in 
the second jet-space $\widetilde M$,
parametrized by $\lambda\in\mathbb A_{\bar 0}$ and $\theta,\phi\in\mathbb A_{\bar 1}$. 
Elimination from the above $(2|2)$ equations 
these $(1|2)$ variables yields the following single $(1|0)$ super-PDE
(we denote $r=u_{xx}, s=u_{xy}, t=u_{yy}, q=u_{\nu\tau}$):
 \begin{equation*}\label{Goursat-explicit}
(12rt^3-12s^2t^2-36rst+32s^3+9r^2)\,(2t^3-6st+3r)^3(2qt^2-4qs+2st-3r)^4\equiv0\!\!\!\mod
\{a_k\} ,
 \end{equation*}
where $0\leq k\leq3$ and $a_0=u_{x\nu}u_{x\tau}$, $a_1=u_{x\nu}u_{y\tau}$,
$a_2=u_{x\tau}u_{y\nu}$, $a_3=u_{y\nu}u_{y\tau}$ define a nilpotent ideal:
$a_k^2=0$, $a_0a_3=a_1a_2$ in the ring of even functions. 
The coefficients of $a_0,a_1,a_2,a_3,a_1a_2$ can be found in the Maple supplement accompanying the arXiv posting of this article.

The arguments in Section \ref{S:G3-PDE-syms} simultaneously establish that
\eqref{Goursat-param} has symmetry superalgebra $G(3)$. 
Evaluation of this super-PDE gives the classical Goursat second order PDE 
$12rt^3-12s^2t^2-36rst+32s^3+9r^2=0$ invariant with respect to $G(2)$.
 \end{rem}

\subsection{Symmetries of the $G(3)$-contact super-PDE} \label{S:G3-PDE-syms}

 In this section, we show:

 \begin{theorem} \label{T:G3-sym} 
The upper bound on the dimension of the contact symmetry algebra $\mathfrak{inf}(M,\cC,\cV)$ of
a locally transitive\footnote{The transitivity hypothesis here can be removed to obtain a stronger result, analogous to that in the (classical) parabolic geometry setting.  Details for achieving this generalization 
will be given in a future work.} $G(3)$-contact supergeometry $(M,\cC,\cV)$ is $(17|14)$. Among these, 
it is reached only for the contact symmetry algebra of the flat $G(3)$-contact supergeometry.
\end{theorem}

\begin{theorem}
\label{T:G3-symII}
The contact symmetry algebra $\mathfrak{inf}(\widetilde M,\widetilde\cC,\widehat\cV)$ of 
the $G(3)$-contact super-PDE system \eqref{G3-PDE} is isomorphic to $\fg=G(3)$ and 
it is spanned by the following $(17|14)$ symmetries:
 \begin{table}[H]
 \[
 \begin{array}{|c|c|l|} \hline
 \fg_2 && \begin{array}{@{}l} u(u - x u_x - y u_y - \nu u_\nu - \tau u_\tau) - \frac{1}{2} \left(\frac{y^3}{3} + 2y\nu\tau\right) u_x + \frac{1}{2} \left( \frac{4}{9} u_y^3 + \frac{2}{3} u_y u_\nu u_\tau \right) x \nonumber \\
 \qquad + \frac{1}{4} \left(y^2 + 2\nu\tau \right)\left( \frac{4}{3} u_y^2 + \frac{2}{3} u_\nu u_\tau\right) + \frac{1}{3} y\tau u_y u_\tau + \frac{1}{3} y\nu u_y u_\nu
 \end{array} \\ \hline
 \fg_1 && x(u - x u_x - y u_y - \nu u_\nu - \tau u_\tau) - \frac{y^3}{6} - y\nu\tau\\
 && y(u - x u_x - \frac{1}{3} y u_y - \frac{2}{3} \nu u_\nu - \frac{2}{3} \tau u_\tau) + \frac{2}{3} x u_y^2 + \frac{1}{3} x u_\nu u_\tau + \frac{4}{3} \nu\tau u_y \\
 && \nu(u - x u_x - \frac{2}{3} y u_y - \frac{4}{3} \tau u_\tau) - \frac{1}{3} x u_y u_\tau -  \frac{1}{6} y^2 u_\tau \\
 && \tau(u - x u_x - \frac{2}{3} y u_y - \frac{4}{3} \nu u_\nu ) + \frac{1}{3} u_y u_\nu x + \frac{1}{6} y^2 u_\nu \\
 && uu_x - \frac{2}{9} u_y^3 - \frac{1}{3} u_y u_\nu u_\tau \\
 && uu_y + \frac{1}{2} y^2 u_x + \nu\tau u_x - \frac{2}{3} y u_y^2 - \frac{1}{3} y u_\nu u_\tau - \frac{1}{3} \nu u_y u_\nu - \frac{1}{3} \tau u_y u_\tau \\
 && uu_\nu + y\tau u_x - \frac{2}{3} \tau u_y^2 - \frac{1}{3} \tau u_\nu u_\tau - \frac{1}{3} y u_y u_\nu\\
 && uu_\tau - y\nu u_x + \frac{2}{3} \nu u_y^2 + \frac{1}{3} \nu u_\nu u_\tau - \frac{1}{3} y u_y u_\tau \\ \hline
 \fz(\fg_0) && \sfZ := 2u - x u_x - y u_y - \nu u_\nu - \tau u_\tau\\ \hline
 \fg_0^{\ss} & \ff_1 & y u_x - \frac{2}{3} u_y^2 - \frac{1}{3} u_\nu u_\tau, \quad
 \nu u_x + \frac{1}{3} u_y u_\tau, \quad \tau u_x - \frac{1}{3} u_y u_\nu \\ \hline
 & \ff_0 & \sfZ_0 := \frac{3}{2} x u_x + \frac{1}{2} (y u_y + \nu u_\nu + \tau u_\tau),\\
 & &\nu u_\nu - \tau u_\tau, \quad \nu u_\tau, \quad \tau u_\nu, \quad y u_\nu - 2\tau u_y, \quad y u_\tau + 2\nu u_y\\ \hline
 & \ff_{-1} & xu_y + \frac{y^2}{2} + \nu\tau, \quad x u_\nu + y\tau, \quad xu_\tau - y\nu\\ \hline
 \fg_{-1} & & x, y, \nu, \tau, u_x, u_y, u_\nu, u_\tau\\ \hline
 \fg_{-2} & & 1\\ \hline
 \end{array}
 \]
 \caption{Generating functions of the symmetries of the $G(3)$-contact super-PDE system \eqref{G3-PDE}}
 \label{G3-PDE-syms}
 \end{table}
\end{theorem}

We now turn to the proof of the theorems. We shall work locally and consider Darboux coordinates 
$(x^i, u, u_i)$ on $M$, where $0\leq i\leq 3$. We set
\begin{align*}
\label{eq:degreecoordinates}
\deg(x^i)&=\deg(u_i)=1\;,\quad \deg(\partial_{x^i})=\deg(\partial_{u_i})=-1\;,\\
\deg(u)&=2\;,\quad \deg(\partial_{u})=-2\;,
\end{align*}
and extend the definition of degree to any weight-homogeneous polynomial vector field on $M$ 
in the obvious way. In particular, we may decompose the space $\mathfrak{inf}(M,\cC)$ of all  
contact vector fields on $M$ into the direct sum
$$
\mathfrak{inf}(M,\cC)=\bigoplus_{k\geq -2}\mathfrak{inf}(M,\cC)_k\;,\qquad \mathfrak{inf}(M,\cC)_k=\left\{\bS\in\mathfrak{inf}(M,\cC)\mid\deg(\bS)=k\right\}\;,
$$
of its homogeneous subspaces and we have an associated filtration
$
\mathfrak{inf}(M,\cC)=\mathfrak{inf}(M,\cC)^{-2}\subset \mathfrak{inf}(M,\cC)^{-1}\subset \mathfrak{inf}(M,\cC)^{0}\subset\cdots
$ of $\mathfrak{inf}(M,\cC)$, compatible with the Lie bracket of vector fields.

Now the contact symmetry algebra $\mathfrak{inf}(M,\cC,\cV)$ of a supergeometry $(M,\cC,\cV)$ is not graded in general but it is naturally filtered by the filtration induced as a subspace of $\mathfrak{inf}(M,\cC)$. The associated graded LSA $\fa = \gr(\mathfrak{inf}(M,\cC,\cV))$ has non-positive part $\fa_- \op \fa_0$ contained in $\fg_- \op \fg_0$, where
$\fg=\fg_{-2}\oplus\cdots\oplus\fg_{2}$ is the contact grading of $\fg = G(3)$ with parabolic subalgebra $\fp_1^{\rm IV}$.
By local transitivity, we have $\fa_{-}=\fg_{-}$, so $[\fa_k,\fg_{-1}] \subset \fa_{k-1}$ for all $k>0$, hence $\fa \subset \pr(\fg_-,\fa_0) \subset \pr(\fg_-,\fg_0) \cong \fg$, where the last isomorphism is due to Corollary \ref{C:pr-ct}. So
\begin{align*}
\dim\mathfrak{inf}(M,\cC,\cV)_{\bar 0}&=\dim\fa_{\bar 0}\leq \dim\fg_{\bar 0}=17\;,\\
\dim\mathfrak{inf}(M,\cC,\cV)_{\bar 1}&=\dim\fa_{\bar 1}\leq \dim\fg_{\bar 1}=14\;.
\end{align*}
This proves the first claim of Theorem \ref{T:G3-sym}. 

Assume now $\dim\mathfrak{inf}(M,\cC,\cV)=(17|14)$. Then $\mathfrak{inf}(M,\cC,\cV)$ is a filtered deformation of the graded LSA $\fa = \gr(\mathfrak{inf}(M,\cC,\cV))\subset\fg$, which by dimension reasons is necessarily $\fg=G(3)$. Since the grading element $\sfZ=\sfZ_1$ of the contact grading $\fg=\fg_{-2}\oplus\cdots\oplus\fg_{2}$ of $G(3)$ belongs to $\fg$, any filtered deformation of $\fg$ is actually isomorphic (as a filtered algebra) to $\fg=\fg_{-2}\oplus\cdots\oplus\fg_{2}$.
This proves the uniqueness statement of Theorem \ref{T:G3-sym}.

In view of the discussion above Theorem \ref{thm:superPDEG(3)}, the contact symmetry algebra
$\mathfrak{inf}(\widetilde M,\widetilde\cC,\widehat\cV)$ of the 2nd order super-PDE \eqref{G3-PDE}
is isomorphic to the contact symmetry algebra $\mathfrak{inf}(M,\cC,\cV)$ of the flat $G(3)$-contact supergeometry. Hence, we conclude the proofs of
Theorems \ref{T:G3-sym}--\ref{T:G3-symII} if we verify that the
supervector fields described in Table \ref{G3-PDE-syms} are contact symmetries of \eqref{G3-PDE}. (In particular, this shows that the upper bound on symmetry dimension is realized.)

We now turn to the explicit computation of the contact symmetries of \eqref{G3-PDE}, but we will in fact carry out a more conceptual calculation of the symmetries of \eqref{G-PDE}, with the cubic form $\fC$ on $W(\bbA)$ only required to satisfy the key identities of Proposition \ref{P:C-ids}.  Similar identities hold for the exceptional Lie algebras, and our calculation is a generalization of \cite[\S 3.4]{MR3759350} to the $G(3)$-case. We will use an index notation with $0 \leq i,j,k \leq 3$, while $1 \leq a,b,c \leq 3$.

Using $\fC$ given in \eqref{E:CG}, it is straightforward to check that Table \ref{G3-PDE-syms} is a specialization of the formulas in Table \ref{all-syms}.  Here we use the notation $X = x^a w_a \in W(\bbA)$ and $P = u_a w^a \in W^*(\bbA)$, where $\{ w_a \}$ is a basis of $W$ and $\{ w^a \}$ is the corresponding dual basis in $W^*$.

\begin{table}[h]
 \[
 \begin{array}{|c|c|l|} \hline
 \fg_2 & & u(u - x^i u_i) - \frac{1}{2} \fC(X^3) u_0 + \frac{1}{2} \fC^*(P^3) x^0 + \frac{9}{4} \fC_c(X^2) (\fC^*)^c(P^2)\\ \hline
 \fg_1 & & x^0(u - x^i u_i) - \frac{1}{2} \fC(X^3) \\
 & & x^a(u - x^i u_i) + (-1)^{|a|} \left( \frac{3}{2} (\fC^*)^a(P^2) x^0 + \frac{9}{2} \fC_b(X^2) (\fC^*)^{ba}(P)\right)\\
 & & uu_0 - \frac{1}{2} \fC^*(P^3)\\
 & & uu_a + \frac{3}{2} \fC_a(X^2) u_0 - \frac{9}{2} \fC_{ab}(X) (\fC^*)^b(P^2)\\ \hline
 \fz(\fg_0) & & \sfZ := 2u - x^i u_i\\ \hline
 \fg_0^{\ss} & \ff_1 & x^a u_0 -  \frac{3}{2} (-1)^{|a|} (\fC^*)^a(P^2) \\
 & \fz(\ff_0) & \sfZ_0 := \frac{3}{2} x^0 u_0 + \frac{1}{2} x^c u_c \\
 & \ff_0^{\ss} & \psi^a{}_b := x^a u_b + (-1)^{|a|} ( \frac{1}{3} \delta^a{}_b x^c u_c - 9(-1)^{|a||b|}\fC_{bc}(X) (\fC^*)^{ca}(P) )\\
 & \ff_{-1} & u_a x^0 + \frac{3}{2} \fC_a(X^2) \\ \hline
 \fg_{-1} & & x^i, u_i\\ \hline
 \fg_{-2} & & 1 \\ \hline
 \end{array}
 \]
 \caption{Generating functions of the symmetries of the $G(3)$-contact super-PDE system \eqref{G3-PDE} expressed in terms of the cubic form $\fC$.}
 \label{all-syms}
 \end{table}
The functions listed in $\fg_{-1}, \fg_{-2}$ and $\fz(\fg_0)$ in Table \ref{all-syms} generate the vector fields
 \[
 x^i \partial_u + (-1)^{|i|} \partial_{u_i}, \quad -\partial_{x^i}, \quad \partial_u, \quad 2u \partial_u+x^i \partial_{x^i}  + u_i \partial_{u_i}\;,
 \]
whose prolongations have a trivial action on all the coordinates $u_{ij}$. Clearly, these prolonged vector fields are symmetries of \eqref{G-PDE}. The last vector field acts as a grading element $\sfZ$.  

Let us assume for the moment that the generating function $f$ in $\fg_2$ is a symmetry. Then, by repeatedly
applying $\fg_{-1}$ to $f$ via \eqref{lagrange-bracket} we obtain all the symmetries in $\fg_1$ and $\fg_0$. (In particular,
 $\ff_1$ and $\ff_{-1}$ are spanned by $[x^a,[u_0,f]]$ and $[u_a,[x^0,f]]$ respectively, while $\ff_0$ is spanned by the grading element $\sfZ_0 = [x^0,[u_0,f]] - \frac{1}{2} \sfZ$ of $\ff$ and by $\psi^a{}_b := [x^a,[u_b,f]] - \delta^a{}_b (-1)^{|a|} (\frac{1}{2} \sfZ + \frac{1}{3} \sfZ_0)$.)

We are then left to prove that $f$ is a symmetry of \eqref{G-PDE}, but verifying that directly would require calculating the prolonged vector field via \eqref{contact-vf-pr} and checking infinitesimal invariance, an approach that is computationally very involved.  Instead, we recall that $\cV$ and $\widehat\cV$ have the same contact symmetries and turn to prove that $f$ is a symmetry of the flat $G(3)$-contact supergeometry $(M,\cC,\cV)$:

\begin{proposition} The even function
 \begin{align}
 f &= u(u - x^i u_i) - \frac{1}{2} \fC(X^3) u_0 + \frac{1}{2} \fC^*(P^3) x^0 + \frac{9}{4} \fC_c(X^2) (\fC^*)^c(P^2) \label{top-sym}
  \end{align}
 generates a contact symmetry $\bS_f$ of the flat $G(3)$-contact supergeometry $(M,\cC,\cV)$.
 \end{proposition}

 \begin{proof}
Let us first explain how the general form of $f$ has been obtained.  Since $[\fg_2,\fg_{-2}] \subset \fz(\fg_0)$, we may impose the normalization condition $\sfZ = [1,f] = f_u$, which implies $$f = u(u - x^i u_i) + g(x^i,u_i)\;.$$
The condition $[\sfZ,f] = 2f$ is equivalent to $x^i \partial_{x^i} g + u_i \partial_{u_i} g = 4g$, i.e., $g$ is even and homogeneous of degree $4$. We will now show that the vector field
\[
 \bS_f = \Big(x^i u -(-1)^{|i|} \partial_{u_i} g \Big) \partial_{x^i} + \Big(u^2 + g  - u_i (\partial_{u_i} g ) \Big) \partial_u + \Big(u_i u - u_i x^j u_j + \partial_{x^i} g\Big) \partial_{u_i}
 \]
is a contact symmetry of $\cV$ for $f$ as in \eqref{top-sym}. In other words, we consider
$\bV = \bV(T)$ as in \eqref{bV} (for the flat frame) and study the equation
 \begin{align} \label{inf-sym}
 [\bS_f, \bV] \in \widehat{T}^{(1)}_{[\bV]} \cV = \tspan\{ \bV, \bB_a \}, \quad \forall\, T \in W(\bbA),
 \end{align}
 where $\bB_a$ was defined in \eqref{TV} (again, for the flat frame).

 We depart with
 \begin{align}
 [\bS_f,\bD_k] &= -\left(\delta^i_k u + u_k x^i - (-1)^{|i|} \partial_{x^k} \partial_{u_i} g\right) \bD_i - \Big(\partial_{x^k} \partial_{x^i} g\Big) \bU^i,\\
 [\bS_f,\bU^k] &= \Big((-1)^{|i|} \partial_{u_k} \partial_{u_i} g \Big) \bD_i + \Big(-\delta_i^k (u - x^j u_j) + (-1)^{|k|} x^k u_i  - \partial_{u_k} \partial_{x^i} g\Big) \bU^i,
 \end{align}
 and specializing to $f$ as in \eqref{top-sym}, we obtain:
 \begin{align*}
 [\bS_f,\bD_0] &= -\left( u + u_0 x^0 \right) \bD_0 +\left(-u_0 x^c + (-1)^{|c|} \frac{3}{2} (\fC^*)^c(P^2) \right) \bD_c, \\
 [\bS_f,\bD_a] &= -\left( u_a x^0 + \frac{3}{2} \fC_a(X^2)\right) \bD_0
 +\left(-\delta^c_a u - u_a x^c + (-1)^{|c|}
 9 \fC_{ab}(X) (\fC^*)^{bc}(P) \right) \bD_c \\
 &\qquad + \left( 3\fC_{ac}(X) u_0 - \frac{9}{2} \fC_{acb} (\fC^*)^b(P^2) \right)\bU^c,\\
 [\bS_f,\bU^0] &= \left( - u + x^j u_j + x^0 u_0\right) \bU^0 + \left( x^0 u_c  + \frac{3}{2} \fC_c(X^2)\right) \bU^c,\\
 [\bS_f,\bU^a]
 &=
 (-1)^{|c|}\left(  3 (\fC^*)^{ac}(P) x^0 + \frac{9}{2} \fC_b(X^2) (\fC^*)^{bac} \right) \bD_c
 + \left((-1)^{|a|} x^a u_0  - \frac{3}{2} (\fC^*)^a(P^2) \right) \bU^0\\
 &\qquad  + \left(-\delta_c^a (u - x^j u_j) + (-1)^{|a|} x^a u_c  - 9 (-1)^{|a||c|} \fC_{cb}(X) (\fC^*)^{ba}(P)\right) \bU^c.
 \end{align*}
 In view of \eqref{bV}, we have that $[\bS_f,\bV] = \rho^i \bD_i + \mu_i \bU^i$, where:
 \begin{align*}
 \rho^0 &= -\left( u + u_0 x^0 \right) + t^a\left( u_a x^0 + \frac{3}{2} \fC_a(X^2)\right),\\
 \rho^c &= -u_0 x^c + (-1)^{|c|} \frac{3}{2} (\fC^*)^c(P^2) - t^a\left(-\delta^c_a u - u_a x^c + (-1)^{|c|}
 9 \fC_{ab}(X) (\fC^*)^{bc}(P) \right)\\
 &\qquad -\frac{3}{2} \fC_a(T^2) (-1)^{|c|} \left( 3(\fC^*)^{ac}(P) x^0 +  \frac{9}{2} \fC_b(X^2) (\fC^*)^{bac} \right),\\
 \mu_0 &= -\frac{1}{2} \fC(T^3) \left( - u + x^j u_j + x^0 u_0\right)  -\frac{3}{2} \fC_a(T^2)
 \left((-1)^{|a|} x^a u_0  - \frac{3}{2} (\fC^*)^a(P^2) \right), \\
 \mu_c &= -\frac{1}{2} \fC(T^3) \left( x^0 u_c  + \frac{3}{2} \fC_c(X^2)\right) + t^a \left( - 3\fC_{ac}(X) u_0 + \frac{9}{2} \fC_{acb} (\fC^*)^b(P^2) \right) \\
 &\qquad -\frac{3}{2} \fC_a(T^2) \Big(-\delta_c^a (u - x^j u_j) + (-1)^{|a|} x^a u_c  - 9 (-1)^{|a||c|} \fC_{cb}(X) (\fC^*)^{ba}(P)\Big).
 \end{align*}
 Using \eqref{TV}, we then find that $[\bS_f,\bV] - \rho^0 \bV - (\rho^c + \rho^0 t^c) \bB_c =\sigma_i \bU^i$, where
 \begin{align*}
 \sigma_0 = \mu_0 - \rho^0 \fC(T^3) - \frac{3}{2} \rho^c \fC_c(T^2), \qquad
 \sigma_c = \mu_c - \frac{3}{2} \rho^0 \fC_c(T^2)  - 3 \rho^b \fC_{bc}(T).
 \end{align*}
 Thus, \eqref{inf-sym} holds if and only if $\sigma_0 =\sigma_c=0$.  Let us first examine $\sigma_0$:
 \begin{align*}
 \sigma_0
 &= -\frac{1}{2} \fC(T^3) \left( - u + x^j u_j + x^0 u_0\right)  -\frac{3}{2} \fC_c(T^2)
 \left((-1)^{|c|} x^c u_0  - \frac{3}{2} (\fC^*)^c(P^2) \right) \\
 &\qquad - \left(-\left( u + u_0 x^0 \right) + t^c\left( u_c x^0 + \frac{3}{2} \fC_c(X^2)\right)\right) \fC(T^3)\\
 &\quad - \frac{3}{2} \left( -u_0 x^c + (-1)^{|c|} \frac{3}{2} (\fC^*)^c(P^2) - t^a\left(-\delta^c_a u - u_a x^c + (-1)^{|c|}  9 \fC_{ab}(X) (\fC^*)^{bc}(P) \right) \right.\\
 &\qquad \left. -\frac{3}{2} \fC_a(T^2) (-1)^{|c|} \left( 3(\fC^*)^{ac}(P) x^0 + \frac{9}{2} \fC_b(X^2) (\fC^*)^{bac} \right) \right) \fC_c(T^2)
\\
&= x^0 \left( \frac{27}{4} \fC_c(T^2) \fC_a(T^2) (\fC^*)^{ac}(P) - \fC(T^3) t^c u_c \right)\\
 &\qquad + \frac{27}{2} t^a \fC_{ab}(X) \fC_c(T^2) (\fC^*)^{cb}(P) -\frac{1}{2} \fC(T^3) x^c u_c - \frac{3}{2} t^a u_a x^c \fC_c(T^2) \\
 &\qquad + \frac{3}{2} \left( \frac{27}{4} \fC_c(T^2) \fC_a(T^2) \fC_b(X^2) (\fC^*)^{bac} - \fC(T^3) t^c \fC_c(X^2) \right)\;,
 \end{align*}
where the last equation follows from the vanishing of all terms which are either quadratic in $t$ or involve $u$ or $u_0$. The identity $\sigma_0=0$ follows then directly from equations \eqref{id1}-\eqref{id2}.

Finally, we turn to $\sigma_c$:
 \begin{align*}
 \sigma_c
 &= -\frac{1}{2} \fC(T^3) \left( x^0 u_c  + \frac{3}{2} \fC_c(X^2)\right) + t^a \left( -3\fC_{ac}(X) u_0 + \frac{9}{2} \fC_{acb} (\fC^*)^b(P^2) \right) \\
 &\quad -\frac{3}{2} \fC_a(T^2) \left(-\delta_c^a (u - x^j u_j) + (-1)^{|a|} x^a u_c  - 9 (-1)^{|a||c|} \fC_{cb}(X) (\fC^*)^{ba}(P)\right) \\
 &\quad - \frac{3}{2} \left(-\left( u + u_0 x^0 \right) + t^a\left( u_a x^0 + \frac{3}{2} \fC_a(X^2)\right)\right)\fC_c(T^2)\\
 &\quad - 3\left(
 -u_0 x^b + (-1)^{|b|} \frac{3}{2} (\fC^*)^b(P^2) - t^a\left(-\delta^b_a u - u_a x^b + (-1)^{|b|}
 9 \fC_{ad}(X) (\fC^*)^{db}(P) \right) \right.\\
 &\qquad\qquad \left. -\frac{3}{2} \fC_a(T^2) (-1)^{|b|} \left( 3(\fC^*)^{ab}(P) x^0 + \frac{9}{2} \fC_d(X^2) (\fC^*)^{dab}\right)
 \right) \fC_{bc}(T)\;.
 \end{align*}
 The terms linear in $t$ and those involving $u$ or $u_0$ are easily seen to vanish.  The
terms cubic in $t$ and those involving $x^0$ vanish both by \eqref{id2}.  Only the terms quadratic  in $t$ remain, i.e.,
 \begin{align*}
 \frac{2}{3} \sigma_c
 &= -\fC_c(T^2) x^c u_c - x^a \fC_a(T^2)  u_c - 2 t^a u_a x^b \fC_{bc}(T) \\
 &\qquad  + 9 \left( \fC_{cb}(X) \fC_a(T^2) (\fC^*)^{ab}(P)
 + 2 t^d (-1)^{|b|}
 \fC_{de}(X) (\fC^*)^{eb}(P) \fC_{bc}(T)  \right)
 \end{align*}
 and vanishes as a consequence of \eqref{id3}.
 \end{proof}

 \subsection{Cauchy characteristic reduction and the super Hilbert--Cartan equation}\label{CC-HC}
We recall that $\widehat\cV \subset \widetilde{M}$ defined by \eqref{G3-PDE} admits local coordinates
$
 (x,y,u,u_x,u_y, \lambda \,\,|\,\, \nu,\tau, u_\nu,u_\tau, \theta, \phi)
 $. We let $\cE=\widehat\cV$ and pass to a certain (local) quotient $\overline\cE$ of $\cE$ equipped with a distribution $\overline\cH$.

 The tautological system $\widetilde{\cC} = \tspan\{ \widetilde{D}_{x^i}, \partial_{u_{ij}} \}$ on $\widetilde{M}$ induces on $\cE$
 the rank $(3|4)$ distribution
 \begin{align} \label{system-dist}
 \cH = \langle D_x, D_y, \partial_\lambda \,\,|\,\, D_\nu, D_\tau, \partial_{\theta}, \partial_{\phi} \rangle,
 \end{align}
 where the vector fields $D_x,D_y,D_\nu,D_\tau$ are the restrictions of the (truncated) total derivatives of $\widetilde M$ to $\cE$, namely:
 \begin{equation}
\label{eq:totalderivativesxynutau}
\begin{cases}
 D_x = \partial_x + u_x \partial_u + (\frac{\lambda^3}{3} + 2\theta \phi \lambda) \partial_{u_x} + (\frac{\lambda^2}{2} + \theta \phi) \partial_{u_y} + \lambda\phi \partial_{u_\nu} - \lambda\theta \partial_{u_\tau}\;,\\
 D_y = \partial_y + u_y \partial_u + (\frac{\lambda^2}{2} + \theta \phi) \partial_{u_x} + \lambda \partial_{u_y} + \phi \partial_{u_\nu} - \theta \partial_{u_\tau}\;,\\
 D_\nu = \partial_\nu + u_\nu \partial_u + \lambda\phi\partial_{u_x} + \phi \partial_{u_y} - \lambda \partial_{u_\tau}\;,\\
 D_\tau = \partial_\tau + u_\tau \partial_u - \lambda\theta\partial_{u_x} - \theta \partial_{u_y} + \lambda \partial_{u_\nu}.
 \end{cases}
\end{equation}
Note that $D_x,D_y,D_\nu,D_\tau$ supercommute. Let 
 $$
\mathfrak{inf}(\cH)= \{ \bX \in \mathfrak{X}(\cE) : \cL_\bX \cH \subset \cH \}
 $$ 
be the algebra of internal symmetries of $\cH$. It is clear that all contact symmetries of \eqref{G3-PDE} (namely, $G(3)$ as stated in Theorem \ref{T:G3-symII}) induce symmetries of $\cH$, but $\mathfrak{inf}(\cH)$ is in fact larger.  Indeed, the {\em Cauchy characteristic space}
 \[
 \Ch(\cH) = \{ \bX \in \Gamma(\cH) : \cL_\bX \cH \subset \cH \}
 \]
is contained in $\mathfrak{inf}(\cH)$ and it is a module for the space of superfunctions of $\cE$. So, if $\Ch(\cH) \neq 0$, then $\Ch(\cH)$ and $\mathfrak{inf}(\cH)$ are infinite-dimensional. This is the case here:

 \begin{proposition} \label{CC}
 $\Ch(\cH)$ is spanned by $\bC = D_x - \lambda D_y - \theta D_\nu - \phi D_\tau$.
 \end{proposition}

 \begin{proof}
We let $$\bC = a_1 D_x + a_2 D_y + a_3 \partial_\lambda + a_4 D_\nu + a_5 D_\tau + a_6 \partial_{\theta} + a_7 \partial_{\phi}$$ and compute Lie brackets modulo $\cH$. We depart with
 \begin{align*}
 0 \equiv [D_\nu,\bC] \equiv \widetilde{a_5}[D_\nu,\partial_\lambda] + \widetilde{a_6} [D_\nu,\partial_{\theta}] + \widetilde{a_7} [D_\nu,\partial_{\phi}]
 \equiv (- \widetilde{a_5}\phi + \widetilde{a_7} \lambda) \partial_{u_x} + \widetilde{a_5} \partial_{u_\tau} + \widetilde{a_7} \partial_{u_y},
 \end{align*}
 where $\widetilde{a}_i= (a_i)_{\bar{0}} - (a_i)_{\bar{1}}$ for all $a_i = (a_i)_{\bar{0}} + (a_i)_{\bar{1}}$, $i=1,\ldots,7$.  Thus, $a_5 = a_7 = 0$ and similarly $[D_\tau,\bC] \equiv 0$ implies $a_6=0$.  Moving on, we have
 \begin{align*}
 0 &\equiv [\partial_{\theta},\bC] 
  \equiv (2 \widetilde{a_1} \phi\lambda + \widetilde{a_2} \phi - \widetilde{a_4} \lambda) \partial_{u_x} + (\widetilde{a_1} \phi - \widetilde{a_4})\partial_{u_y} -  (\widetilde{a_1}\lambda + \widetilde{a_2})\partial_{u_\tau},
 \end{align*}
 so that $a_2 = -a_1\lambda$ and $a_4 = -a_1 \phi$, since $\lambda$ is even and $\phi$ is odd.
 Finally, we have
 \begin{align*}
 0 &\equiv [\partial_{\phi}, \bC] \equiv (-\widetilde{a_1} \theta \lambda + \widetilde{a_3} \lambda) \partial_{u_x} + (-\widetilde{a_1} \theta + \widetilde{a_3}) \partial_{u_y}
 \end{align*}
 so that $a_3 = -a_1\theta$. The remaining conditions $[D_x,\bC]\equiv[D_y,\bC]\equiv[\partial_\lambda, \bC]\equiv 0$ are all easily verified.
 \end{proof}

 Consider the (local) quotient $\overline\cE = \cE / \Ch(\cH)$, namely the space of integral curves of $\Ch(\cH)$. We notice that $\{x=0\}$ defines a hypersurface transverse to $\Ch(\cH)$, so we can locally identify it with $\overline\cE$ and consider in there the $(5|6)$-coordinates given by
 $
 (y,u,z, u_y, \lambda \,\,|\,\, \nu,\tau, u_\nu, u_\tau, \theta, \phi)
 $, where $z:=u_x$.
 On $\overline\cE$, we have the rank $(2|4)$-distribution induced from \eqref{system-dist}, namely:
 \begin{align} \label{D-bar}
 \overline\cH = \langle D_y, \partial_\lambda \,\,|\,\, D_\nu,D_\tau,\partial_{\theta}, \partial_{\phi} \rangle,
 \end{align}
 where $D_y,D_\nu$ and $D_\tau$ are as in \eqref{eq:totalderivativesxynutau} (with $\partial_{u_x}$ replaced by $\partial_{z}$).  A direct computation shows that the growth vector of $\overline\cH$ is $(2|4, 1|2, 2|0)$ and that the symbol algebra is isomorphic to the negatively graded part of the SHC grading of $G(3)$, as presented in Section \ref{S:G3-gradings}.

 Alternatively, $\overline\cH$ is determined by the Pfaffian system
 \begin{align}
\label{eq:Pfaffiansystem}
 \begin{cases}
  du - (dy) u_y - (d\nu) u_\nu - (d\tau) u_\tau \\
  dz - (dy) \left(\frac{\lambda^2}{2} + \theta \phi\right) - (d\nu)  \lambda \phi + (d\tau) \lambda \theta \\
 du_y - (dy) \lambda - (d\nu) \phi + (d\tau) \theta \\
 du_\nu - (dy) \phi - (d\tau) \lambda \\
 du_\tau + (dy) \theta + (d\nu) \lambda
 \end{cases}\;,
 \end{align}
 which is the pullback of the canonical system
 \[
 \begin{cases}
 du - (dy) u_y - (d\nu) u_\nu - (d\tau) u_\tau \\
 dz - (dy) z_y - (d\nu) z_\nu - (d\tau) z_\tau \\
 du_y - (dy) u_{yy} - (d\nu) u_{y\nu} - (d\tau) u_{y\tau}\\
 du_\nu - (dy) u_{y \nu} + (d\tau) u_{\nu\tau}\\
 du_\tau - (dy) u_{y\tau} - (d\nu) u_{\nu\tau}
 \end{cases}
 \]
 on the mixed jet-space $J^{1,2}(\bbC^{1|2},\bbC^{2|0})$ to the sub-supermanifold
 \begin{align} \label{SHC-param}
 \begin{cases}
 z_y = \frac{\lambda^2}{2} + \theta \phi, \quad z_\nu = \lambda \phi, \quad z_\tau = -\lambda \theta,\\
 u_{yy} = \lambda, \quad u_{y\nu} = \phi, \quad u_{y\tau} = -\theta, \quad u_{\nu\tau} = -\lambda\;.
 \end{cases}
 \end{align}
Eliminating the parameters $\lambda, \theta, \phi$ and relabelling $y$ to $x$, 
we obtain the SHC equation \eqref{SHC-ODE}.
\comm{
what we refer to as the {\em SHC equation}:
 \begin{align} \label{SHC}
 z_x = \frac{(u_{xx})^2}{2} + u_{x\nu} u_{x\tau}, \quad
 z_\nu = u_{xx} u_{x\nu}, \quad
 z_\tau = u_{xx} u_{x\tau}, \quad
 u_{\nu\tau} = -u_{xx},
 \end{align}
 which is a supersymmetric generalization of the classical Hilbert--Cartan equation \eqref{HC-ODE} with (internal) symmetry algebra $G(2)$.
 }
 
\begin{theorem}
\label{T:G3-SHC-sym} The internal symmetry algebra $\mathfrak{inf}(\overline\cE, \overline{\cH})$ of the SHC equation is isomorphic to $G(3)$ and is spanned by the $(17|14)$ symmetries given in Appendix \ref{S:SHC-sym}.
 \end{theorem}

 \begin{proof} All contact symmetries of the $G(3)$-contact super-PDE system \eqref{G3-PDE} preserve $(\cE,\cH)$ and hence $\Ch(\cH)$.
Therefore, they project to a subalgebra of $\mathfrak{inf}(\overline\cE, \overline{\cH})$.

We claim that they project isomorphically. First note that all contact symmetries of \eqref{G3-PDE}
preserve the vertical bundle $\operatorname{Vert}(\cE) = \langle \partial_\lambda, \partial_{\theta}, \partial_{\phi} \rangle$ over the contact supermanifold $(M,\cC)$. On the other hand, we check using Proposition \ref{CC} that $\Ch(\cH)$ does not preserve $\operatorname{Vert}(\cE)$. It follows that the contact symmetries of \eqref{G3-PDE} are transverse to $\Ch(\cH)$, hence our claim. Note that the subalgebra of $\mathfrak{inf}(\overline\cE,\overline{\cH})$ so determined is isomorphic to $G(3)$ by Theorem \ref{T:G3-symII}. 

Finally, by Theorem \ref{thm:235H^1} and the arguments from the proof of Theorem \ref{T:G3-sym}, 
we obtain that $\dim(\mathfrak{inf}(\overline\cE, \overline{\cH})) \leq (17|14)$.  
Thus, $\mathfrak{inf}(\overline\cE,\overline{\cH}) \cong G(3)$.
 \end{proof}

 \section{Curved supergeometries from $G(3)$-symmetric models}\label{S5}

\subsection{Rigidity of the symbol}

Consider a bracket-generating superdistribution $\cD\subset TM$ without Cauchy characteristics.
Then its symbol superalgebra $\fm=\fg_-$ is {\it fundamental} (i.e., it is generated by $\g_{-1}$) and {\it non-degenerate}  (i.e., it has no central elements in $\g_{-1}$).
In this section we specialize to the case $\dim M=(5|6)$, with the growth vector of $\cD$ given by $(2|4,1|2,2|0)$.

The Lie brackets on the even part $\fm_{\bar 0}$ of $\fm$ consist of the skew-form $\omega:\Lambda^2(\g_{-1})_{\bar0}\to(\g_{-2})_{\bar0}$
as well as the map $\beta:(\g_{-1})_{\bar0}\otimes(\g_{-2})_{\bar0}\to(\g_{-3})_{\bar0}$ that, 
when non-degenerate, serves for a (conformal) identification $(\g_{-1})_{\bar0}\cong(\g_{-3})_{\bar0}$.
We denote the remaining Lie brackets on $\fm$ by
 \begin{gather*}
q:\Lambda^2(\g_{-1})_{\bar1}\to(\g_{-2})_{\bar0}, \quad
\Xi:(\g_{-1})_{\bar0}\otimes(\g_{-1})_{\bar1}\to(\g_{-2})_{\bar1},\quad
\Theta:(\g_{-1})_{\bar1}\otimes(\g_{-2})_{\bar1}\to(\g_{-3})_{\bar0}.
 \end{gather*}
We recall that $\Lambda^\bullet$ is meant in the super-sense, in particular $q$ is a quadratic form. The only non-trivial Jacobi identity is
 \begin{equation}
\label{eq:J.I.}
\Theta(\theta_1,\Xi(e,\theta_2))+\Theta(\theta_2,\Xi(e,\theta_1))=\beta(e,q(\theta_1,\theta_2)),
\end{equation}
for all $e\in(\fg_{-1})_{\bar 0}$ and $\theta_1,\theta_2\in(\fg_{-1})_{\bar 1}$, and the fundamental and non-degeneracy properties are equivalent to:
 \begin{align*}
 &(F1): \omega\neq0 \mbox{ or } q\neq0; \quad
 (F2): \opp{Im}\Xi= (\fg_{-2})_{\bar{1}}; \quad
 (F3): \opp{Im}\beta+\opp{Im}\Theta= (\fg_{-3})_{\bar{0}};\\
 &(N1): \omega(e,\cdot)=0,\, \Xi(e,\cdot)=0 \Rightarrow e=0; \quad
 (N2): q(\theta,\cdot)=0,\, \Xi(\cdot,\theta)=0 \Rightarrow \theta=0.
 \end{align*}
Let us remark that since $\omega$ does not appear in the Jacobi identity \eqref{eq:J.I.}, its value is important
only via properties (F1) and (N1).

Using the following uniform notation for the basis of $\fm$,
\begin{gather} 
(\g_{-1})_{\bar0}=\langle e_1, e_2 \rangle,\quad
(\g_{-2})_{\bar0}=\langle h \rangle,\quad 
(\g_{-3})_{\bar0}=\langle f_1, f_2 \rangle, \label{m-basis1}\\
(\g_{-1})_{\bar1}=\langle \theta_1', \theta_1'', \theta_2', \theta_2'' \rangle, \quad
(\g_{-2})_{\bar1}=\langle \varrho_1,\varrho_2\rangle, \label{m-basis2}
 \end{gather}
we obtain such fundamental, non-degenerate symbols with growth $(2|4,1|2,2|0)$:

\begin{enumerate}
\item[(M1)] SHC symbol algebra:
\begin{align*} \label{SHC-m-brackets}
 \begin{cases}
 \omega(e_1,e_2)=h, \quad \beta(e_1,h)=f_1, \quad \beta(e_2,h)=f_2,\\\\
 \begin{array}{c|cccc}
 q & \theta_1' & \theta_1'' & \theta_2' & \theta_2'' \\ \hline
 \theta_1' & 0 & 0 & h & 0\\
 \theta_1'' & 0 & 0 & 0 & h\\
 \theta_2' & h & 0 & 0 & 0\\
 \theta_2'' & 0 & h & 0 & 0\\
 \end{array} \qquad
 \begin{array}[b]{c|cccc}
 \Xi & \theta_1' & \theta_1'' & \theta_2' & \theta_2''\\
 \hline
 e_1 & 0 & 0 & \varrho_1 & \varrho_2 \\
 e_2 & -\varrho_2 & \varrho_1 & 0 & 0
 \end{array} \qquad
 \begin{array}{c|cc}
 \Theta & \varrho_1 & \varrho_2 \\
 \hline
 \theta_1' & f_1 & 0 \\
\theta_1'' & 0 & f_1\\
 \theta_2' & 0 & -f_2 \\
\theta_2'' & f_2 & 0
\end{array}
\end{cases}
 \end{align*}
 
\item[(M2)] $\opp{rank}(\beta) = 1$:
\begin{align*}
\begin{cases}
 \omega(e_1,e_2)=h, \quad \beta(e_1,h)=f_1, \quad \beta(e_2,h)=0,\\ \\
 \begin{array}{c|cccc}
 q & \theta_1' & \theta_1'' & \theta_2' & \theta_2'' \\ \hline
 \theta_1' & 0 & 0 & h & 0\\
 \theta_1'' & 0 & 0 & 0 & h\\
 \theta_2' & h & 0 & 0 & 0\\
 \theta_2'' & 0 & h & 0 & 0\\
 \end{array} \qquad
 \begin{array}[b]{c|cccc}
 \Xi & \theta_1' & \theta_1'' & \theta_2' & \theta_2''\\
 \hline
 e_1 & \rho_1 & \rho_2 & 0 & 0\\
 e_2 & 0 & 0 & 0 & 0
 \end{array} \qquad
 \begin{array}{c|cc}
 \Theta & \varrho_1 & \varrho_2 \\
 \hline
 \theta_1' & 0 & f_2 \\
\theta_1'' & -f_2 & 0\\
 \theta_2' & f_1 & 0 \\
\theta_2'' & 0 & f_1
\end{array}
\end{cases}
 \end{align*}

\item[(M3)] $q=0$ and $\Theta = 0$:
 \begin{align*}
 \begin{cases}
 \omega(e_1,e_2) = h, \quad \beta(e_1,h) = f_1, \quad \beta(e_2,h) = f_2, \\ \\
 \begin{array}[b]{c|cccc}
 \Xi & \theta_1' & \theta_1'' & \theta_2' & \theta_2''\\
 \hline
 e_1 & \rho_1 & \rho_2 & 0 & 0\\
 e_2 & 0 & 0 & \rho_1 & \rho_2
 \end{array} \qquad
 \end{cases}
 \end{align*}
 
\item[(M4)] $\omega = 0$: $q$, $\beta$, $\Xi$, $\Theta$ are the same as for the SHC symbol algebra.
\end{enumerate}

The SHC symbol algebra is so-called because it is isomorphic to the negatively graded part of the $\mathbb Z$-grading of $G(3)$ associated to $\mathfrak{p}_2^{\rm IV}\subset G(3)$.
In the notation of Appendix \ref{S:SHC-sym}, the basis elements are explicitly given as supervector fields by
 \begin{gather*}
e_1=-D_x,\quad e_2=\p_{u_{xx}}\;,\quad h=\p_{u_x}+u_{xx}\p_z\;,\quad f_1=\p_u\;,\quad f_2=\p_z\;,\\
\theta_1'=D_{\nu}\;,\quad \theta_1''=D_{\tau}\;,\quad
 \theta_2'=\p_{u_{x\nu}}\;,\quad \theta_2''=\p_{u_{x\tau}}\;,\quad\\ 
\varrho_1=\p_{u_\nu}+u_{x\tau}\p_z\;,\quad\varrho_2=\p_{u_\tau}-u_{x\nu}\p_z\;,
 \end{gather*}
where $(x,u,u_x,u_{xx},z|\tau,\nu,u_\tau,u_\nu,u_{x\tau},u_{x\nu})$ are coordinates on a $(5|6)$-supermanifold and the supervector fields $D_x$, $D_\nu$, $D_\tau$ are as in Appendix \ref{S:SHC-sym}.
\begin{theorem}\label{RigThm}
Any fundamental, non-degenerate symbol superalgebra $\fm=\g_{-1}\oplus\g_{-2}\oplus\g_{-3}$ of growth $(2|4,1|2,2|0)$ is isomorphic to one of the models (M1)--(M4).
 \end{theorem}

 \begin{proof}
The proof splits into three main steps. Throughout the proof, we will use the notation $K_{e}=\opp{Ker}\Xi(e,\cdot)\subset (\fg_{-1})_{\bar 1}$ for any $e\in(\fg_{-1})_{\bar 0}$.

\medskip\par
\noindent\underline{Step 1}: 
Either (i) $\opp{rank}(\beta)=1$ and $q,\omega$ are non-degenerate, or (ii) $\beta$ is an isomorphism
and $\Xi(\cdot,\theta)$ is an isomorphism for generic $\theta$.
\smallskip\par
Suppose that $\Xi(\cdot,\theta)$ is degenerate for all $\theta\in(\fg_{-1})_{\bar 1}$. 
Then property (F2) 
implies that the linear map 
 $$
\Xi:(\fg_{-1})_{\bar 1}\to \operatorname{Hom}((\fg_{-1})_{\bar 0},(\fg_{-2})_{\bar 1})
\cong\operatorname{End}(\C^2)
 $$ 
has rank equal to $2$ and that there is $0\neq e_2\in(\fg_{-1})_{\bar 0}$ 
such that $\Xi(e_2,\theta)=0$ for all $\theta\in(\fg_{-1})_{\bar 1}$.
Indeed, consider the matrix 
of this map 
 $$
\Xi(\cdot,\theta)=\begin{pmatrix} \xi_{11}(\theta) & \xi_{12}(\theta) \\ 
\xi_{21}(\theta) & \xi_{22}(\theta) \end{pmatrix}\;,
 $$
with linear functionals $\xi_{ij}\in(\fg_{-1})_{\bar 1}^*$
as entries.
Since $\Xi\not\equiv0$ at least one entry is non-zero, as a polynomial in $\theta$.
Assume, e.g., $\xi_{12}\neq0$. Then $\det\Xi(\cdot,\theta)\equiv0$ implies that $\xi_{12}$ 
divides $\xi_{11}\!\cdot\!\xi_{22}$. Hence $\xi_{12}$ divides either $\xi_{11}$ or $\xi_{22}$.
In the second case the rows of $\Xi(\cdot,\theta)$ are proportional and consequently 
$\mathop{Im}\Xi$ is a 1-dimensional subspace in $(\fg_{-2})_{\bar 1}$, contradicting (F2).
Therefore we get the first case, in which the columns of $\Xi(\cdot,\theta)$ are proportional,
whence the claim.

By (N1), this claim forces $\omega\neq0$. Moreover, for any $e_1\in(\fg_{-1})_{\bar 1}\setminus\C e_2$ 
we have that $\Xi(e_1,\cdot):(\fg_{-1})_{\bar 1}\to(\fg_{-2})_{\bar 1}$ is an epimorphism  by (F2), so 
$\dim K_{e_1}=2$. By $(N2)$, any $0\neq\theta\in K_{e_1}$
satisfies $q(\theta,\cdot)\neq0$. Since $\dim((\fg_{-2})_{\bar 0})=1$, then \eqref{eq:J.I.} implies 
$\beta(e_2,h)=0$ for every $0\neq h\in(\fg_{-2})_{\bar 0}$. 

We claim  $\beta(e_1,h)\neq 0$. If not, then $\beta=0$, so by \eqref{eq:J.I.}, $\Theta(\cdot,\Xi(e_1,\cdot))$ is skew and this descends to 
$(\fg_{-1})_{\bar 1} \mod K_{e_1}$, which is $2$-dimensional. Thus $\dim\opp{Im}(\Theta)\leq1$, which contradicts (F3). 

Since $K_{e_1}$ is $q$-isotropic by \eqref{eq:J.I.} and 
$q(\theta,\cdot)\neq0$ for $0\neq \theta\in K_{e_1}$, then $q$ is nondegenerate. 
The corresponding normal form of all brackets is that of the model (M2).
\smallskip

If $\Xi(\cdot,\theta)$ is an isomorphism for some $\theta$ then it is an isomorphism for all $\theta$ in a Zariski-open subset $\mathcal U\subset (\fg_{-1})_{\bar 1}$. In this case, identity \eqref{eq:J.I.}
with $\theta_1=\theta_2=\theta\in\mathcal U$ implies $\opp{Im}\Theta\subset\opp{Im}\beta$, so $\opp{Im}\beta= (\fg_{-3})_{\bar{0}}$ by (F3) and
$\beta$ {\it is an isomorphism}. {\it From now on we consider only this case.}

\medskip\par\noindent
\underline{Step 2}: (i) $\dim K_e=2$ for all $0\neq e\in(\fg_{-1})_{\bar 0}$ and (ii) $K_{e_1}\cap K_{e_2}=\{0\}$ for any two linearly independent vectors $e_1,e_2\in(\fg_{-1})_{\bar 0}$.
\smallskip\par

 Let $0 \neq e\in(\fg_{-1})_{\bar 0}$ be arbitrary.  Clearly $\dim K_e\ge2$, since $\dim (\g_{-1})_{\bar1}=4$ and $\dim(\g_{-2})_{\bar1}=2$.  Moreover, $K_e$ is $q$-isotropic (using
 \eqref{eq:J.I.} and that $\beta$ is an isomorphism).
 
 We first establish (ii).  If (ii) fails, then there exists $0\neq \theta\in (\g_{-1})_{\bar1}$ such that $\Xi(\cdot,\theta)=0$, so $q(\theta,\cdot)\not=0$ by (N2).  By \eqref{eq:J.I.}, for any $\theta_2 \in (\fg_{-1})_{\bar{1}}$, we have the identity
 \[
 \Theta(\theta,\Xi(e,\theta_2)) = \beta(e,q(\theta,\theta_2))
 \]
 of elements $(\fg_{-3})_{\bar{0}}$. Since $\beta$ is an isomorphism, then: 
 \begin{enumerate}
 \item[(a)] fixing $\theta_2 \in (\fg_{-1})_{\bar{1}}$ with $q(\theta,\theta_2) \neq 0$, we see that $\Theta(\theta,\cdot):(\g_{-2})_{\bar1}\to(\g_{-3})_{\bar0}$ is an epimorphism with $\dim(\g_{-2})_{\bar1} = \dim(\g_{-3})_{\bar0}$, so is an isomorphism.
 \item[(b)] $K_e = \{ \theta_2 : q(\theta,\theta_2) = 0 \}$ for any $0 \neq e \in(\fg_{-1})_{\bar 0}$.  Since $q(\theta,\cdot) \neq 0$, then $\dim(K_e) = 3$.
 \end{enumerate}
By \eqref{eq:J.I.}, $K_e$ is a 3-dimensional $q$-isotropic subspace of the 4-dimensional space $(\fg_{-1})_{\bar{1}}$.  Hence, there exists $0 \neq \theta' \in (\fg_{-1})_{\bar{1}}$ with $q(\theta',\cdot) = 0$.  Setting $\theta_1=\theta'$ and $\theta_2=\theta$ in \eqref{eq:J.I.} yields $\Theta(\theta,\Xi(e,\theta')) = 0$ for all $e\in(\fg_{-1})_{\bar 0}$.  By (a), we get $\Xi(\cdot,\theta')=0$, which contradicts (N2). 

This implies claim (ii), and then claim (i) follows easily.

\smallskip

In summary, we may decompose $(\fg_{-1})_{\bar 1}=K_{e_1}\oplus K_{e_2}$
into the direct sum of complementary $q$-isotropic planes, hence $\opp{Ker}(q)$ is even-dimensional.  Moreover, the decomposition implies that $\Xi(e_1,\cdot) : (\fg_{-1})_{\bar{1}} \to (\fg_{-2})_{\bar{1}}$ is injective on $K_{e_2}$, and similarly for $\Xi(e_2,\cdot)$, so
 \begin{align} \label{Xi-rel}
 \Xi(e_1,K_{e_2})=\Xi(e_2,K_{e_1})=(\g_{-2})_{\bar1}.
 \end{align}
(Note that the collection of $q$-isotropic planes $K_e$ determined by $[e]\in\mathbb{P}((\g_{-1})_{\bar0})$ 
is nothing but the  projective line of the so-called $\alpha$-planes, for some choice of orientation.)

\medskip\par\noindent
\underline{Step 3}: Either (i) $q=0$ and $\Theta=0$, or (ii) $q$ is non-degenerate.
\smallskip\par
If $\Theta=0$, then \eqref{eq:J.I.} implies $q=0$ (since $\beta$ is an isomorphism), and moreover (F1) implies $\omega \neq 0$.  Conversely, if $q=0$, then take $e_1,e_2\in(\fg_{-1})_{\bar 0}$ linearly independent as above, and set $e=e_2$, $\theta_2\in K_{e_2}$, and $\theta_1\in K_{e_1}$ in \eqref{eq:J.I.} to get $\Theta(\theta_2,\Xi(e_2,\theta_1))=0$.  By \eqref{Xi-rel}, we have $\Theta(\theta_2,\cdot)=0$. Similarly $\Theta(\theta_1,\cdot)=0$, implying $\Theta=0$. Using the decomposition $(\fg_{-1})_{\bar 1}=K_{e_1}\oplus K_{e_2}$ and
\eqref{Xi-rel}, it is easy to see that the corresponding normal form is that of the model (M3).

Now suppose $q\neq0$. If $q$ is degenerate, we may choose $\theta_2''\in K_{e_2}$ such that 
$q(\theta_2'',\cdot)=0$. Then \eqref{eq:J.I.} with $e=e_2$ implies 
$\Theta(\theta_2'',\Xi(e_2,\theta_1))=0$ for all $\theta_1\in K_{e_1}$, hence
$\Theta(\theta_2'',\cdot)=0$ by \eqref{Xi-rel}. 

Let $\theta_2'$ be any element in $K_{e_2}$ that is not proportional to $\theta_2''$. 
Note that $\Theta(\theta_2',\Xi(e_1,\theta_2''))=0$ by \eqref{eq:J.I.} and the previous argument, 
and that $\Theta(\theta_2',\Xi(e_1,\theta_2'))=0$ since $K_{e_2}$ is $q$-isotropic. It follows from \eqref{Xi-rel} that
$\Theta(\theta_2',\cdot)=0$, thus $\Theta(K_{e_2},\cdot)=0$. One similarly shows $\Theta(K_{e_1},\cdot)=0$, 
 yielding $\Theta=0$, which contradicts the assumption $q\neq0$. Hence $q$ is non-degenerate.

\smallskip

We now turn to the normal form of the Lie brackets when $q$ is non-degenerate.  Fix bases as in \eqref{m-basis1}-\eqref{m-basis2} such that: (1) $\beta(e_i,h) = f_i$, (2) we have a $q$-Witt basis of $(\fg_{-1})_{\bar{1}}$, i.e.,\ $K_{e_1} = \langle \theta_1',\theta_1'' \rangle$ and 
$K_{e_2} = \langle \theta_2',\theta_2'' \rangle$ with $q(\theta_i^\alpha,\theta_j^\beta)=(1-\delta_{ij})\delta^{\alpha\beta}h$, (3) $\rho_\alpha=\Xi(e_1,\theta_2^\alpha)$ (recall equation \eqref{Xi-rel}.) 

 Using \eqref{eq:J.I.} with the inputs indicated below, we get the relations
 \begin{align}
 \theta_1^\alpha,\theta_2^\beta,e_1 \quad\Rightarrow\quad & \Theta(\theta_1^\alpha, \rho_\beta) = \delta^\alpha_\beta f_1\;, \label{JI1}\\
 \theta_2^\alpha,\theta_2^\beta,e_1 \quad\Rightarrow\quad & \Theta(\theta_2^\alpha,\rho_\beta) +  \Theta(\theta_2^\beta,\rho_\alpha) = 0\;, \label{JI2}\\ 
 \theta_1^\alpha, \theta_1^\beta,e_2 \quad\Rightarrow\quad & \Theta(\theta_1^\alpha,\Xi(e_2,\theta_1^\beta)) +  \Theta(\theta_1^\beta,\Xi(e_2,\theta_1^\alpha)) = 0\;, \label{JI3}\\ 
 \theta_1^\alpha, \theta_2^\beta,e_2 \quad\Rightarrow\quad & \Theta(\theta_2^\beta,\Xi(e_2,\theta_1^\alpha)) = \delta^{\alpha\beta} f_2\;. \label{JI4}
 \end{align}
 Equations \eqref{JI1}-\eqref{JI2} force $\Theta$ to agree with the corresponding component of the SHC symbol (M1), except for the element $\Theta(\theta_2'',\rho_1) = -\Theta(\theta_2',\rho_2)$.  Then, \eqref{JI1}-\eqref{JI3} imply $\Xi(e_2,\theta_1') = -c\rho_2$ and $\Xi(e_2,\theta_1'') = c\rho_1$ with $c \neq 0$ since $\Xi(e_2,\cdot)|_{K_{e_1}}$ is injective.  Finally, the relations \eqref{JI4} imply that $\Theta(\theta_2',\rho_2) = -\frac{1}{c} f_2 = -\Theta(\theta_2'',\rho_1)$.  
 
 Rescaling $e_2$ and $f_2$, we arrive at the canonical normal form for $q,\beta,\Xi,\Theta$ as in the SHC symbol algebra. The map $\omega$ is either vanishing, in which case we are led to the model (M4), or non-degenerate, in which case
$\omega(e_1,e_2)=\lambda h$ for some $0\neq\lambda\in\C$. Rescaling the generators $e_1,e_2,f_1,f_2,\rho_1,\rho_2$ by $\lambda^{-1/2}$, we precisely get the SHC symbol algebra. 
\end{proof}


A fundamental, non-degenerate symbol superalgebra $\fm=\fm_{\bar 0}\oplus\fm_{\bar 1}$ of growth $(2|4,1|2,2|0)$ is called a {\it super-extension of the HC symbol} if it has even part $\fm_{\bar 0}$ isomorphic to the unique fundamental graded nilpotent Lie algebra of growth $(2,1,2)$.  
By Theorem \ref{RigThm}, such a super-extension $\fm$ is isomorphic either to the model (M1) or (M3). 
\begin{definition} 
\label{def:super-extensiondef}
A (strongly regular) superdistribution $\cD$ on a supermanifold $M=(M_o,\cA_M)$ with the growth vector $(2|4,1|2,2|0)$ 
is called:
\begin{itemize}
\item[(i)] a {\em super-extension} of a generic rank 2 distribution on the 5-dimensional space $M_o$ if its symbol superalgebra is a super-extension of the HC symbol;
\item[(ii)] of {\em SHC type} if its symbol superalgebra is isomorphic to (M1).
\end{itemize} 
\end{definition}
Clearly, any SHC type superdistribution is a super-extension of a generic rank 2 distribution on $M_o$. Since a super-extension $\fm$ of the HC symbol with $q \neq 0$ is isomorphic to (M1), we have:
\begin{cor}\label{Rigidity}
The symbol superalgebra of any superdistribution of SHC type is rigid w.r.t. small deformations of the superdistribution preserving the growth vector.
\end{cor}

\subsection{Rank $(2|4)$ distributions in a $(5|6)$-dimensional superspace}
\subsubsection{Generic rank $(2|4)$ distributions}
In the classical case, a generic rank 2 distribution on a 5-dimensional space has growth vector $(2,1,2)$. 
Such distributions are equivalent to Monge normal form given by the Cartan distribution of the ODE $z_x=f(x,u,u_x,u_{xx},z)$ with $\p_{u_{xx}}^2 f\neq 0$,
so they are parametrized by $1$ function of $5$ variables. 

On the other hand, a generic rank $(2|4)$ distribution on a $(5|6)$-dimensional supermanifold 
has depth 2. More precisely, the growth is $(2|4,3|2)$ as the brackets
$[\cdot,\cdot]:\Lambda^2\cD_{\bar1}\to(\mathcal TM/\cD)_{\bar0}$
and $[\cdot,\cdot]:\cD_{\bar0}\otimes\cD_{\bar1}\to(\mathcal TM/\cD)_{\bar1}$
are both surjective in general, by a direct counting of the rank of the stalks  of the sheaves. 
We conclude that superdistributions of SHC type are {\it not} generic among all
rank $(2|4)$ distributions on a $(5|6)$-dimensional supermanifold.

Let us compute the functional dimension of a generic rank $(2|4)$ distribution on the superspace $M=\C^{5|6}(x^k|\theta^c)$, where $1\leq k\leq 5$ and $1\leq c\leq 6$.
The distribution has generators
 \begin{equation}
\label{eq:generatorsgeneric24}
\begin{aligned}
w_i&=\partial_{x^i}+\sum_{j=3}^5A_i^j(x,\theta)\partial_{x^j}+\sum_{b=5}^6B_i^b(x,\theta)\partial_{\theta^b}
\quad (1\leq i\leq2)\;,\\
\zeta_a&=\partial_{\theta^a}+\sum_{j=3}^5C_a^j(x,\theta)\partial_{x^j}+\sum_{b=5}^6D_a^b(x,\theta)\partial_{\theta^b}
\quad (1\leq a\leq4)\;,
\end{aligned}
 \end{equation}
where $A_i^j,D_a^b$ are even and $B_i^b,C_a^j$ odd superfunctions.
Note that any superfunction on $M$ has a Taylor expansion in the 6 odd coordinates, 
with coefficients being ordinary functions of the 5 even coordinates. 
The total number of the coefficients in a superfunction is $2^6=64$, 
and for  an even or odd superfunction is $\frac12 64=32$.
Consequently the space of distributions of the form \eqref{eq:generatorsgeneric24} is parametrized by $6\cdot5\cdot32=960$ ordinary functions of 5 variables.

A general (parity-preserving) change of coordinates involves $(5+6)\cdot 32=352$ ordinary functions,
and its action on the space of distributions \eqref{eq:generatorsgeneric24} is generically free:
this is the absence of symmetries for generic distributions of the given type. 
Therefore the moduli space of such distributions is parametrized by $960-352=608$ 
ordinary functions of 5 variables. 

\subsubsection{Distributions of SHC type}
\label{subsec:distributionsofSHCtype}
The moduli space of rank $(2|4)$ distributions of SHC type is smaller, but it is still 
quite difficult to parametrize.
Here, we restrict to the following system of 4 differential equations that generalizes the SHC equation \eqref{SHC-ODE}:
 \begin{equation}\label{FGHK}
z_x=F\;,\quad z_\nu=G\;,\quad z_\tau=H\;,\quad u_{\nu\tau}=K\;,
 \end{equation}
where $F,K$ are even superfunctions (resp., $G,H$ odd superfunctions) on the supermanifold $M=\C^{5|6}(x,u,u_x,u_{xx},z|\nu,\tau,u_\nu,u_\tau,u_{x\nu},u_{x\tau})$. The associated Cartan superdistribution $\cD$ has even generators
 \begin{align} 
\label{FGHK-deformeven}
D_x    &= \p_x+u_x\p_u+u_{xx}\p_{u_x}+F\p_z+u_{x\nu}\p_{u_\nu}+u_{x\tau}\p_{u_\tau},\quad \p_{u_{xx}},
 \end{align}
and odd generators
\begin{equation}
\label{FGHK-deformodd}
 \begin{aligned} 
D_\nu  &= \p_\nu+u_\nu\p_u+u_{x\nu}\p_{u_x}+G\p_z+K\p_{u_\tau},\quad   \p_{u_{x\nu}},\\
D_\tau &= \p_\tau+u_\tau\p_u+u_{x\tau}\p_{u_x}+H\p_z-K\p_{u_\nu},\quad  \p_{u_{x\tau}}.
 \end{aligned}
\end{equation}

Naively, we would expect that such superdistributions are always of SHC type, but this is not the case in general. For instance, the distribution associated to the system
$$
z_x=f(x,u,u_x,u_{xx},z)\;,\quad z_\nu=0\;,\quad z_\tau=0\;,\quad u_{\nu\tau}=0\;,
$$
has growth vector $(2|4,2|2,1|0)$ provided $\p_{u_{xx}} f\neq 0$, and clearly it is  not 
the super-extension of a Monge equation in the specified sense.
The superfunctions $F,G,H,K$ which give rise to superdistributions of SHC type 
are constrained as follows.

\begin{prop} \label{FGHK-constraints}
The superdistribution $\cD$ with generators \eqref{FGHK-deformeven} 
and \eqref{FGHK-deformodd} is of SHC-type 
if and only if $\p_{u_{xx}}^2 F$ is invertible (as a superfunction, i.e., its evaluation to the underlying classical manifold is nowhere vanishing)
and the following differential system is satisfied:
 \begin{align}
\label{eq:firsteqsystem}
D_xG      &=D_\nu F+(D_xK) (\p_{u_{x\tau}} F),\quad D_xH=D_\tau F-(D_xK)(\p_{u_{x\nu}}F),\\
\label{eq:secondeqsystem}
D_\nu G &=D_\tau H=0,\quad D_\tau G+D_\nu H=0,\quad D_\nu K=D_\tau K=0,\\
\label{eq:thirdeqsystem}
\partial_{u_{x\nu}} G &= \partial_{u_{xx}} F=\partial_{u_{x\tau}} H,\quad 
\partial_{u_{x\tau}} G = \partial_{u_{x\nu}} H =0,\quad \partial_{u_{x\nu}} K=\partial_{u_{x\tau}} K=0,\\
\label{eq:fourtheqsystem}
\partial_{u_{xx}} G &= (\partial_{u_{xx}} K) (\partial_{u_{x\tau}} F),\quad \partial_{u_{xx}} H = -(\partial_{u_{xx}} K)(\partial_{u_{x\nu}} F).
\end{align}
 \comm{
is satisfied and one of the following two alternatives hold:
\begin{itemize}
\item[(1)] If $\p_{u_{xx}} K$ is invertible (as a superfunction) then we require $\p_{u_{xx}}^2 F$ to be invertible. In this case the superdistribution is of SHC type;
\item[(2)] If $\p_{u_{xx}} K=0$, then we require $\p_{u_{xx}}^2 F=0$ and \FIXMEA{One condition to be determined}. In this case the superdistribution has symbol (M2) and \eqref{FGHK} is not 
the super-extension of a Monge equation.
\end{itemize}
If $\p_{u_{xx}} K$ is not as in (1) or (2), then the superdistribution is not strongly regular.
 }
\end{prop}

\begin{proof}
The supervector field $\bT= [\p_{u_{xx}},D_x]=\p_{u_x}+(\partial_{u_{xx}} F)\p_z$, generates $(\fg_{-2})_{\bar{0}}$ modulo $\cD$. On the other hand, the Lie brackets between the odd generators of $\cD$ are
 \begin{align*}
[\p_{u_{x\nu}},D_\nu]&=\p_{u_x}+(\partial_{u_{x\nu}} G) \p_z+(\partial_{u_{x\nu}} K) \p_{u_\tau}\;,
\quad
[\p_{u_{x\nu}},D_\tau]=(\partial_{u_{x\nu}} H) \p_z - (\partial_{u_{x\nu}} K) \p_{u_\nu}\;,\\
[\p_{u_{x\tau}},D_\tau]&=\p_{u_x}+ (\partial_{u_{x\tau}} H) \p_z - (\partial_{u_{x\tau}} K) \p_{u_\nu}\;,\quad
[\p_{u_{x\tau}},D_\nu]=(\partial_{u_{x\tau}}G) \p_z+ (\partial_{u_{x\tau}} K) \p_{u_\tau}\;,\\
\tfrac12[D_\nu,D_\nu]&=(D_\nu G)\,\p_z+(D_\nu K)\,\p_{u_\tau}\;,\quad
\tfrac12[D_\tau,D_\tau]=(D_\tau H)\,\p_z - (D_\tau K)\,\p_{u_\nu}\;,\\
[D_\nu,D_\tau]&=(D_\tau G+D_\nu H)\,\p_z+(D_\tau K)\,\p_{u_\tau}-(D_\nu K)\,\p_{u_\nu}\;,
 \end{align*}
and it is not difficult to see that these supervector fields are multiples of $\bT$ modulo $\cD$ if and only if  they are multiples of $\bT$. This immediately yields the equations \eqref{eq:secondeqsystem} and \eqref{eq:thirdeqsystem}.

Next, we calculate the Lie brackets between even and odd generators of $\cD$. We first set 
 \begin{gather*}
\bS_1= [\p_{u_{x\nu}},D_x]=  \p_{u_\nu}  + (\partial_{u_{x\nu}}F)\p_z,\quad
\bS_2= [\p_{u_{x\tau}},D_x]= \p_{u_\tau} + (\partial_{u_{x\tau}}F)\p_z,
 \end{gather*}
and note that the equivalence classes of $\bS_1$ and $\bS_2$ modulo $\cD$ generate 
$(\fg_{-2})_{\bar{1}}$. Again, it turns out that the supervector fields
 \begin{align*}
[D_x,D_\nu]&=(D_x G-D_\nu F)\,\p_z+(D_x K)\,\p_{u_\tau},\quad
[D_x,D_\tau]=(D_x H-D_\tau F)\,\p_z-(D_x K)\,\p_{u_\nu},\\
[\p_{u_{xx}},D_\nu]&=(\partial_{u_{xx}}G)\,\p_z+(\partial_{u_{xx}}K)\,\p_{u_\tau},\quad
[\p_{u_{xx}},D_\tau]=(\partial_{u_{xx}}H)\,\p_z-(\partial_{u_{xx}}K)\,\p_{u_\nu},
 \end{align*}
are linear combinations of $\bS_1$ and $\bS_2$ modulo $\cD$ precisely when they are 
linear combinations of $\bS_1$ and $\bS_2$.
This gives \eqref{eq:firsteqsystem} and \eqref{eq:fourtheqsystem}.


A closer look at the Lie brackets 
determined so far tells us that the maps $q$ and $\omega$ defining the models of Theorem \ref{RigThm} are both non-zero.
Hence the symbol of $\cD$ is isomorphic either to (M1) or (M2). Concerning the map $\Xi$, 
note that (the equivalence class of) a supervector field $E\in(\cA_M)_x\otimes(\fg_{-1})_{\bar 0}$ satisfies $\Xi(E,\cdot)=0$ if and only if 
$E$ is a multiple of $\partial_{u_{xx}}$ and $\partial_{u_{xx}}K=0$. The last condition follows from the identities 
\begin{align*}
[\p_{u_{xx}},D_\nu]&=(\partial_{u_{xx}}K)\bS_2,\quad
[\p_{u_{xx}},D_\tau]=-(\partial_{u_{xx}}K)\bS_1,\quad [\p_{u_{xx}},\partial_{u_{x\nu}}]=[\p_{u_{xx}},\partial_{u_{x\tau}}]=0.
\end{align*}
Finally, the map $\beta$
is given the following Lie brackets: 
 $$
[\partial_{u_{xx}},\bT] = (\p_{u_{xx}}^2 F)\p_z\;, \qquad
[D_x,\bT] \equiv -\partial_u\!\!\!\mod\langle\partial_z\rangle\;.
 $$
Thus the invertibility of $\p_{u_{xx}}^2 F$ is a necessary condition for a superdistribution of SHC type, but it is also sufficient
since differentiation by $u_{x\nu}$ of the first identity in \eqref{eq:fourtheqsystem}  yields
$\p_{u_{xx}}^2 F=(\partial_{u_{xx}}K)(\partial_{u_{x\nu}u_{x\tau}}F)$, and
$\partial_{u_{xx}}K$ is invertible too.
 \end{proof}
\begin{rem}
\label{rem:m2versusm1}
Superdistributions of SHC-type are strongly regular by definition. It can be shown that the superdistribution 
$\cD$ with generators \eqref{FGHK-deformeven} 
and \eqref{FGHK-deformodd} is strongly regular and with fundamental, non-degenerate symbol of growth $(2|4,1|2,2|0)$ that is {\em not} of SHC-type if and only if $\p_{u_{xx}} K=\p_{u_{xx}}^2 F=0$ and the symbol is (M2).
This case does not correspond to the super-extension of a Monge equation and it is characterized by additional differential conditions on $(F,G,H,K)$ that we omit due to their size. 

If $\p_{u_{xx}} K$ is neither zero nor invertible, then the superdistribution $\cD$ is not strongly regular.
\end{rem}
 \begin{cor}
The functional freedom of analytic superdistributions as in Proposition \ref{FGHK-constraints}
does not involve any ordinary function of 5 variables.
 \end{cor}

 \begin{proof}
For the sake of brevity, we will omit the dependence on even coordinates in this proof. From \eqref{eq:thirdeqsystem}, we obtain the following 
equations
 \begin{gather*}
K=K(\nu,\tau,u_\nu,u_\tau),\\
G=G_0(\nu,\tau,u_\nu,u_\tau)+L(\nu,\tau,u_\nu,u_\tau)u_{x\nu},\\
H=H_0(\nu,\tau,u_\nu,u_\tau)+L(\nu,\tau,u_\nu,u_\tau)u_{x\tau},
 \end{gather*}
where $L,K$ are even superfunctions and $G_0, H_0$ odd superfunctions.
 
If we set $F=F_0(\nu,\tau,u_\nu,u_\tau)+F_1(\nu,\tau,u_\nu,u_\tau)u_{x\nu}+F_2(\nu,\tau,u_\nu,u_\tau)u_{x\tau}
+F_{12}(\nu,\tau,u_\nu,u_\tau)u_{x\nu}u_{x\tau}$, for $F_0,F_{12}$ even and 
$F_1,F_2$ odd superfunctions, we see from \eqref{eq:thirdeqsystem} that 
 \begin{gather}
\label{eq:systemIdetermined}
 L=\partial_{u_{xx}} F_0\;,\quad
 \p_{u_{xx}} F_1 = \p_{u_{xx}} F_2 = \p_{u_{xx}} F_{12}=0\;,
 \end{gather}
and from \eqref{eq:fourtheqsystem} that
\begin{gather}
\label{eq:systemIIdetermined} 
\p_{u_{xx}} G_0 = -(\p_{u_{xx}}K)F_2,\quad 
\p_{u_{xx}} L = -(\p_{u_{xx}}K)F_{12},\quad
\p_{u_{xx}} H_0 = (\p_{u_{xx}}K)F_1.
\end{gather}
Finally, taking the $u_{x\nu}$-coefficient of the equation $D_\nu K=0$ 
in \eqref{eq:secondeqsystem} we get
\begin{gather}
\label{eq:systemIIIdetermined} 
 \p_{u_x} K = -L\,\p_z K.
\end{gather}

The system \eqref{eq:systemIdetermined}-\eqref{eq:systemIIIdetermined} of 8 equations on the 8 unknowns
$F_0,F_1,F_2, F_{12},G_0,H_0,L,K$ is a classical system of PDE: as soon as it is
expanded in the odd variables $\nu,\tau,u_\nu,u_\tau$ it becomes a system of 64 equations 
of the first order on 64 unknown ordinary functions (the coefficients of the expansion). 
We claim that this quasi-linear PDE system is determined.

Indeed, the symbol of the system is a $64\times 64$ matrix $A$ with linear functionals on
$T^*M_o$ as entries (i.e., functions linear in momenta $p_x,p_u,p_{u_x},p_{u_{xx}},p_z$),
and an easy computation shows that its determinant $P=\det A$ is a non-zero homogeneous 
polynomial in momenta. In fact, $P=p_{u_{xx}}^{56}(p_{u_x}+L|_{o} p_z)^8$, where $L|_{o}$ is 
evaluation on $M_o$ of the restriction of $L$ on the 0-jet of a solution.

The locus of $P$ in $\mathbb{P}T^*M_o$ is the characteristic variety of the system (depending on
the 0-jet of a solution) and a (4-dimensional) hypersurface $\Sigma\subset M_o$ is non-characteristic if at every point 
its annihilator is a non-characteristic covector, i.e., $P(\mathop{Ann}T\Sigma)\neq0$.
The Cauchy data are given by arbitrary values of the coefficients of the expansions of $F_0,F_1,\dots,L,K$
on $\Sigma$. 

If $\Sigma$ is analytic, non-characteristic and the Cauchy data are analytic then by
the Cauchy-Kovalevskaya theorem there exists a unique solution to the system. Thus 
(analytic) solutions are determined by 64 functions of 4 variables and the statement is proved.
 \end{proof}

 \begin{rk}
The above corollary also holds in the formal category (i.e., for power series). 
The result implies that the functional dimension drops by passing from classical Monge equations to
super-differential equations \eqref{FGHK} that are of SHC type. Informally, this can be understood on the basis of our finding that $H^{d,2}(\fm,\fg)\cong S^2\C^2$ for $d = 2$ (and trivial for other $d > 0$) for the SHC grading of $\fg=G(3)$. 

Indeed, the Cartan quartic of the underlying generic rank 2 distribution on a 5-dimensional space
should admit a square root, hence must be of Petrov type D  (a pair of double roots), N (a quadruple root) or O (identically zero). Put it differently, not all 
Monge equations are super-extendable to equations of SHC type. Details for achieving this correspondence in the framework 
of parabolic supergeometries will be given elsewhere. 
\end{rk}



\subsection{Integral submanifolds of the SHC distribution}
\label{subsec:integralsubmnfds}
In this section, we will consider solutions of the SHC equation \eqref{SHC-ODE}. More generally, 
we consider its {\em space of integral submanifolds}, i.e., the (set-theoretic) space consisting of 
all integral submanifolds of the associated Cartan superdistribution.
Namely, let $M=\C^{5|6}(x,u,u_x,u_{xx},z|\nu,\tau,u_\nu,u_\tau,u_{x\nu},u_{x\tau})$ be equipped 
with the Pfaffian system
$\Psi=\Psi_{\bar0}\oplus\Psi_{\bar1}$ given by
 \begin{align*}
\Psi_{\bar0} = \langle
& dz-dx\cdot(\tfrac12u_{xx}^2+u_{x\nu}u_{x\tau})-d\nu\cdot u_{xx}u_{x\nu}-d\tau\cdot u_{xx}u_{x\tau},\\
& du-dx\cdot u_x-d\nu\cdot u_\nu-d\tau\cdot u_\tau,\
du_x-dx\cdot u_{xx}-d\nu\cdot u_{x\nu}-d\tau\cdot u_{x\tau}\rangle,\\
\Psi_{\bar1} = \langle
& du_\nu-dx\cdot u_{x\nu}-d\tau\cdot u_{xx},\ du_\tau-dx\cdot u_{x\tau}+d\nu\cdot u_{xx}\rangle.
 \end{align*}
and consider morphisms $\iota:\C^{p|q}\to M$ such that $\iota^*\Psi=0$
and $\iota$ is an immersion almost everywhere, i.e., the pull-back $\iota^*:\cA_M\to\cA_{\C^{p|q}}$ on the sheaf of superfunctions is surjective at almost every stalk. We recall that  the pull-back between stalks is by definition an {\em even} morphism of superalgebras.

Our space of integral submanifolds has to be compared with the more general notion of {\em superspace of integral submanifolds} \cite{MR1701618}, usually introduced via the functor of points and for which the main r$\hat{\text o}$le is played by families of integral submanifolds parametrized by odd elements in an auxiliary algebra $\mathbb A$ (these are the super-points of such a superspace). Integral submanifolds in our sense are called ``bosonic solutions'' in the mathematical physics literature and it is customary to restrict the analysis to them. 

Before turning to integral submanifolds, we give the following preliminary result.
 \begin{proposition}\label{prop:solSHC}
Even superfunctions $u=u(x\,|\,\nu,\tau)$, $z=z(x\,|\,\nu,\tau)$ satisfy the SHC equation 
\eqref{SHC-ODE} if and only if
\begin{equation}\label{solSHC}
  \begin{aligned}
u& =c_0+c_1x+\tfrac12c_2x^2+\tfrac16c_3x^3+(c_2+c_3x)\nu\tau,\\
z& =c_4+\tfrac12c_2^2x+\tfrac12c_2c_3x^2+\tfrac16c_3^2x^3+c_3(c_2+c_3x)\nu\tau,
  \end{aligned}
\end{equation}
for some constants $c_0,\ldots,c_4\in\C$. 
\end{proposition}
\begin{proof}
Expand $u$ and $z$ in the odd coordinates $\nu,\tau$ and 
substitute into \eqref{SHC-ODE}. 
\end{proof}

We will consider integral submanifolds up to re-parametrization of the source $\C^{p|q}$
and shortly refer to them as ``cointegrals''. (This nomenclature stems from the fact that
``integrals'' are morphisms $\jmath:M\to\C^{s|t}$ that are constant on cointegrals.) 
We will denote generic variables of $M$ by $s\in\left\{x,u,u_x,u_{xx},z\right\}$ (even) and
$\xi\in\left\{\nu,\tau,u_\nu,u_\tau,u_{x\nu},u_{x\tau}\right\}$ (odd).

We consider various cases separately, depending on the dimension of the source $\C^{p|q}$.
When the space of $(p|q)$-cointegrals is reducible, we describe the regular stratum of biggest 
functional dimension. It is not difficult to obtain the complete stratification, but we will not pursue 
this here for simplicity of exposition.

\medskip

\noindent\underline{$(0|0)$-cointegrals}.
Just to start with: these are the points of the classical manifold $M_o=\C^5$,
so they are parametrized by 5 constants.

\medskip

\noindent\underline{$(1|0)$-cointegrals}.
In this case $\iota^*\xi=0$ for any odd variable, so cointegrals 
are just the integral curves $\iota:\C\to M_o$ of
the classical HC Pfaffian system. It is well-known that they are parametrized by 1 function of 1 variable. 

\medskip

\noindent\underline{$(0|1)$-cointegrals}.
Let $\theta$ be the odd coordinate of the source. In this case, 
$\iota^*\xi=c_{\xi}\theta$ for $c_\xi\in\C$
and there exists a point $o\in M_o$ such that
$\iota^*s=s|_o$ is evaluation at $o$. By reparametrization of the source, the constants $c_\xi$ have to be considered up to an overall scale.

The condition $\iota^*\Psi=0$ is then equivalent to the system
 $$
c_\nu c_{u_{x\nu}}+c_\tau c_{u_{x\tau}}=0,\quad c_{u_\nu}=(u_{xx}|_o)\,c_\tau,\quad 
c_{u_\tau}=-(u_{xx}|_o)\,c_\nu,
 $$
hence it is encoded in a projective quadric of dimension 2. Taking into account the choice of the point $o\in M_o$,
the space of $(0|1)$-cointegrals is parametrized by $5+2=7$ constants.


\medskip

\noindent\underline{$(1|1)$-cointegrals}.
The source has coordinates $(t|\theta)$ and we set
$\iota^*s=s(t)$ and $\iota^*\xi=c_\xi(t)\theta$. Note that the $(1|0)$-cointegrals $\gamma=\iota_o:\C\to M_o$
obtained by composing $\iota:\CC^{1|1}\to M$ with the natural embedding of $\C=\C^{1|0}$ in $\C^{1|1}$ are parametrized by 1 function of 1 variable. 

We restrict to the generic case and reparametrize the coordinate $t$ so that  $x(t)=t$ locally. Finally, reparametrizing the odd coordinate $\theta$ implies that the $6$ functions $c_{\xi}(t)$ have to be considered up to overall $t$-dependent scale.


Using the coefficients of the $1$-form $dt$ in the system of equations $\iota^*\Psi_{\bar0}=0$, we get the classical HC constraints
$$
u_x(t)=u'(t)\;,\quad u_{xx}(t)=u''(t)\;,\quad z'(t)=\frac12(u''(t))^2\;,
$$
and if
we use the coefficients 
of $d\theta$ in the systems $\iota^*\Psi_{\bar0}=0$ and $\iota^*\Psi_{\bar1}=0$, we get
a section of the bundle of quadrics over $\gamma(\C)\subset M_o$:
 $$
c_\nu(t)c_{u_{x\nu}}(t)+c_\tau(t)c_{u_{x\tau}}(t)=0,\quad
c_{u_\nu}(t)=u''(t)c_\tau(t),\quad c_{u_\tau}(t)=-u''(t)c_\nu(t).
 $$
Finally, if we combine these identities with the coefficients of $dt$ in $\iota^*\Psi_{\bar1}=0$, we arrive at the following equations
 $$
c_{u_{x\nu}}(t)=u'''(t)c_\tau(t),\quad c_{u_{x\tau}}(t)=-u'''(t)c_\nu(t).
 $$
In summary, the $(1|1)$-cointegrals are parametrized by the function $u(t)$ (together with the quadrature required to get $z(t)$)
and the function $[c_\nu(t):c_\tau(t)]\in\mathbb{P}^1_{|\gamma}$, therefore
by 2 functions of 1 variable.

\medskip

\noindent\underline{$(0|2)$-cointegrals}.
We let $\theta_1,\theta_2$ be the two odd coordinates of the source and note that the image of 
the underlying classical morphism $\iota_o:\C^{0|0}\to M_o$ is just a point $o\in M_o$. Then
$$\iota^*s=s|_o+c_s\,\theta_1\theta_2\qquad\text{and}\qquad
\iota^*\xi=c_\xi^k\theta_k=c_\xi^1\theta_1+c_\xi^2\theta_2,\quad(k=1,2),$$
where $c_s,c^k_\xi\in\C$ for all even and odd coordinates of $M$, respectively.
The system $\iota^*\Psi=0$ may be expanded in the $\theta_k$'s and $d\theta_k$'s, turning into a system of 20 equations on 22 unknowns constants.
The unknowns are the coordinates of $o\in M_o$ and the constants $c_s$ and $c^k_\xi$.

If $c_x\neq 0$, it turns out that this system is generated by the following 11 equations:
 \begin{gather*}
c_u=(u_x|_o)\,c_x+(u_{xx}|_o)\,(c_\nu^1c_\tau^2-c_\nu^2c_\tau^1),\quad
c_{u_x}=(u_{xx}|_o)\,c_x+\tfrac{c_{u_{xx}}}{c_x}\,(c_\nu^1c_\tau^2-c_\nu^2c_\tau^1),\\
c_z=\tfrac12(u_{xx}|_o)^2c_x+\tfrac{c_{u_{xx}}}{c_x}\,(u_{xx}|_o)\,(c_\nu^1c_\tau^2-c_\nu^2c_\tau^1),\\
c^k_{u_\nu}=(u_{xx}|_o)\,c^k_\tau,\ c^k_{u_\tau}=-(u_{xx}|_o)\,c^k_\nu,\
c^k_{u_{x\nu}}=\tfrac{c_{u_{xx}}}{c_x}\,c^k_\tau,\ c^k_{u_{x\tau}}=-\tfrac{c_{u_{xx}}}{c_x}\,c^k_\nu,\quad (k=1,2).
 \end{gather*}
If $c_x=0$, the number of independent equations increases, so this is not the generic case. 

Taking into account the reparametrization group $GL(2)$ associated to the source, we conclude that the number 
of independent constants is $7$. They are given by the coordinates of $o\in M_o$ and the constants $c_x$ and $c_{y_{xx}}$. 
Thus $(0|2)$-cointegrals depend on 7 constants.

\medskip

\noindent\underline{$(1|2)$-cointegrals}.
A straightforward computation says that the operation of
jet-prolongation sets up a bijective correspondence between the
(unparametrized) regular cointegrals and the
solutions of the SHC equation \eqref{SHC-ODE}. By Proposition \ref{prop:solSHC}, the space of cointegrals depends on 5 constants.

\medskip

\noindent\underline{$(p|q)$-cointegrals with $p>1$ or $q>2$}.
In all the remaining cases, the space of cointegrals is empty. If $p>1$, this follows 
just as in the classical case: the even part $\fm_{\bar 0}$ of the SHC symbol $\fm=\fg_-$ is non-degenerate, i.e., it has no central elements in $(\g_{-1})_{\bar 0}$. 
The explicit brackets of $\fm$ show that there are no Abelian 3-dimensional subspaces in $(\fg_{-1})_{\bar1}$, hence the claim for $q>2$.
\vskip0.3cm\par
\begin{rem}
We showed that solutions of the SHC equation \eqref{SHC-ODE} correspond to the integral submanifolds
of the largest possible dimension. On the other hand, the $(1|0)$-cointegrals correspond to the integral curves of the classical HC equation. 
This fact can be regarded as (an another) confirmation that \eqref{SHC-ODE} is a super-extension of the HC equation.
\end{rem}

Now we show that the functional dimension count persists in the curved setting.

 \begin{theorem}\label{thm:cointegrals}
The space of cointegrals associated to an analytic superdistribution of SHC type on a $(5|6)$-dimensional supermanifold $M=(M_o,\cA_M)$
has the same functional dimension as for the SHC equation \eqref{SHC-ODE}. Namely, the only non-trivial spaces of cointegrals
are as follows:
  \begin{center}
\begin{tabular}{c|cccccc}
 Type $(p|q)$ & $(0|0)$ & $(1|0)$ & $(0|1)$ & $(1|1)$ & $(0|2)$ & $(1|2)$ \\
 \hline
 Generators & 5 const & 1 funct & 7 const & 2 funct & 7 const & 5 const
\end{tabular}
 \end{center}
where ``const" means complex constant and ``func" means ordinary function of 1 variable.
 \end{theorem}

 \begin{proof}
Let us recall that the $-1$ degree component $\fg_{-1}=(\fg_{-1})_{\bar 0} \oplus(\fg_{-1})_{\bar 1}$  of the SHC symbol (M1) is the direct sum of two modules for $(\fg_0)_{\bar0}=\fsl(2)\oplus\fsp(2)$:
$$(\fg_{-1})_{\bar0}\cong\CC^2\boxtimes\CC\quad\text{and}\quad(\fg_{-1})_{\bar1}\cong\CC^2\boxtimes\CC^2\;.$$
The Lie brackets between the elements of $\fg_{-1}$ are encoded in the symplectic form $\omega$ and the maps $q,\Xi$ as in (M1). In particular, the quadratic form $q:\Lambda^2(\g_{-1})_{\bar1}\to(\g_{-2})_{\bar0}$ is given by $q(v_1\boxtimes w_1,v_2\boxtimes w_2)=\langle v_1,v_2\rangle \langle w_1,w_2\rangle$,
with $\langle - ,- \rangle$ the symplectic form on each factor of $(\fg_{-1})_{\bar 1}$.

The case of $(p|q)$-cointegrals with $p>1$ or $q>2$ is proved in the same way as for the SHC equation 
and the case $(p|q)=(0|0)$ is immediate. For $(p|q)=(1|0)$ we get classical integral curves of a generic 
rank 2 distribution on $M_o$. 

To treat the case of $(0|1)$-cointegrals, we need to consider the $q$-isotropic lines in 
$(\fg_{-1})_{\bar1}$. They form the 2-parametric family $[v]\boxtimes[w]$, which coupled with 
the 5 parameters needed to describe the point $o\in M_o$ gives 7 constants.

Next, $(1|1)$-cointegrals are given by a classical integral curve (parametrized by 1 function) and 
an odd supervector field along it. Using $\Xi$, it is not difficult to see that the $2$-dimensional 
subspace of $\fg_{-1}$ generated by $v\boxtimes\1\in(\fg_{-1})_{\bar 0}$ and 
$\vartheta\in(\fg_{-1})_{\bar 1}$ is Abelian if and only if $\vartheta$ belongs to the 
$\alpha$-plane $v\boxtimes\CC^2$ corresponding to $v$. (Recall that the isotropic planes in 
a 4-dimensional complex metric space $(\C^4,q)$ form two $\mathbb{P}^1$, i.e., 
the $\alpha$-planes $v\boxtimes\CC^2$ and the $\beta$-planes $\CC^2\boxtimes w$.) 
Thus $\vartheta$ is given by another 1 function and the claimed functional dimension follows.

Similarly, the space of $(0|2)$-cointegrals is parametrized by the 5 parameters required to describe 
$o\in M_o$ and an isotropic plane in $(\fg_{-1})_{\bar 1}$, which is parametrized by 2 constants.

Finally, $(1|2)$-cointegrals are parametrized by 5 constants. This counting does not depend 
on a particular superdistribution of SHC type, as 
the involutive prolongation of the system \eqref{FGHK} has the same type of equations as the SHC equation.
Involutivity implies that there exists a unique solution for any Cauchy data, 
which can be parametrized by the 5 constants $z,u,u_x,u_{xx},u_{xxx}$. 
 \end{proof}


 \subsection{Super-deformation and submaximally supersymmetric models}

Here we discuss locally homogeneous superdistributions of SHC-type
with submaximal symmetry dimension, along the lines of \cite{MR2949641}. 
The condition of local homogeneity can be relaxed,
but this involves the development of new techniques, which will be considered in a separate 
work.
\subsubsection{A bound on the dimension of transitive symmetry Lie superalgebras}

Model distributions with transitive symmetry superalgebras can be obtained by deformation theory as follows.

 \begin{prop}
\label{prop:homogeneousrealization}
Let $\fg = \bigoplus_{k \in \ZZ} \fg_k$ be a $\mathbb Z$-graded LSA 
such that $H^{d,1}(\g_-,\g)=0$ for all $d \geq 0$, where $\g_-$ is the
negatively graded part of $\fg$, and 
$\fh$ a filtered LSA whose associated graded $\opp{gr}(\fh)$ embeds into 
$\g$ (with filtration $\g^j=\oplus_{k\ge j}\g_k$) and $\opp{gr}(\fh)_-=\g_-$. 
Let $H$ be a connected Lie supergroup with LSA $\fh$ and 
$H^0$ the connected closed 
subgroup of $H$ with the algebra $\fh^0\subset \fh$. 

Then the homogeneous supermanifold $M=H/H^0$ equipped with the $H$-invariant distribution 
$\cD$ determined by $\cD|_o\cong\fh^{-1}\!\!\mod\fh^0$
has symmetry superalgebra $\mathfrak{inf}(M,\cD)\supset\fh$. Moreover: 
\begin{itemize}
\item[(i)] $\dim \mathfrak{inf}(M,\cD)\leq \dim\fg$, with equality
if and only if $\mathfrak{inf}(M,\cD)=\fg$,
\item[(ii)] if $\fh$ cannot be embedded into $\g$ as a filtered Lie algebra, then 
$\mathfrak{inf}(M,\cD)\neq\fg$. 
\end{itemize}
\end{prop}

This proposition allows a straightforward generalization for the case $H^{d,1}(\g_-,\g)=0$ for all $d > 0$, which
concerns distributions with a reduction of the structure group $G_0$.

Claim (i) follows from the fact that $\mathfrak{inf}(M,\cD)$ inherits a natural filtration such that the associated graded LSA embeds into $\fg$, in other words $\mathfrak{inf}(M,\cD)$ is a filtered 
deformation of a graded subalgebra of $\fg$. In addition, if a filtered deformation of a graded LSA 
includes the grading element then it is actually graded and this implies claim (ii). 

 \begin{theorem}\label{UpperBound}
Let $\cD$ be a superdistribution of SHC-type on a $(5|6)$-dimensional supermanifold $M$ such that
the symmetry superalgebra $\mathfrak{inf}(M,\cD)$ acts locally transitively on $M$.
If $\mathfrak{inf}(M,\cD)\neq G(3)$, then $\dim\mathfrak{inf}(M,\cD)\leq(10|8)$. 
 \end{theorem}

 \begin{proof}
We set $\fh=\mathfrak{inf}(M,\cD)$ and note that the graded Lie algebra $\fa=\opp{gr}(\fh)\subset\g$, where $\g=G(3)$ is equipped with the
SHC $\mathbb Z$-grading. We will tacitly identify $\fh$ and $\fa$ as $\mathbb Z_2$-graded vector spaces and denote the Lie brackets of $\fh$ by
$[-,-]=[-,-]_0+[-,-]_+:\fa\otimes\fa\to\fa$, where 
\begin{itemize}
\item $[-,-]_0$ is the fixed (zero-degree) Lie bracket of $\fa$, and
\item $[-,-]_+$ comprises the deformation components, which are of {\em positive} total degree.
\end{itemize}
By the transitivity assumption we have $\fa_-=\g_-$, which
implies $[\fa_{\ge0},\g_-]\subset\fa$. We will use this fact extensively and rely on the
root decomposition discussed in Section \ref{S:G3-gradings}. See also Table \ref{tab:p_2^IV} and the discussion before Lemma \ref{lem:centralizers}. 

Let us first note that if $\fa_3\neq0$ $\Rightarrow$ $(\fa_2)_{\bar0}\neq0$ $\Rightarrow$ $\fa_1=\g_1$. 
Also if $(\fa_2)_{\bar1}\neq0$ $\Rightarrow$ $(\fa_1)_{\bar0}=(\g_1)_{\bar0}$. Thus if $\fa_+\neq0$ we get $\fa_1\neq0$. One can similarly check that if $G(2)\subset\fa$ $\Rightarrow$ $\fa=\g$. 

\medskip\par\noindent
\underline{\it Step 1}: 
We first classify all graded subalgebras $\fa\subset\g$ with the indicated properties such that
(i) either $\dim\fa_{\bar 0}>10$ or $\dim\fa_{\bar 1}>8$, and (ii) $\fa$ does not include 
the grading element. 

\vskip0.1cm\par\noindent
{\underline{Case 1}.} 
If $(\fa_1)_{\bar0}\neq0$, we may conjugate by the Lie group 
$(G_{0})_{\bar 0}\cong SL(2)\times Sp(2) \times\C^\times$ so that $e_{\a_2+\a_3}\in(\fa_1)_{\bar0}$. 
Then $e_{-\a_1}\in(\fa_0)_{\bar0}$ and $e_{\pm\a_3}\in(\fa_0)_{\bar1}$, 
which in turn implies $e_{\a_2},e_{\a_2+2\a_3}\in(\fa_1)_{\bar1}$ and $\fsp(2)\subset(\fa_0)_{\bar0}$.   
Moreover $(\fa_0)_{\bar0}$ contains the 2-dimensional subspace 
of the Cartan subalgebra of $\fg$ generated by $h_{\alpha_2}$ and $h_{\alpha_3}$. 

This yields a subalgebra $\fa$ of $\fg$ of $\dim\fa=(11|10)$, whose non-negative part is displayed 
in \eqref{Tablecase(1)}. It is not difficult to see that $\fa$ cannot be extended to 
a proper graded subalgebra of $\fg$. 
 \begin{align}\label{Tablecase(1)}
\;\;\;\;\;\;\;\; \begin{array}{|c|c|c|} \hline
 k & \text{even part} &  \text{odd part} \\ \hline\hline
 1 & e_{\alpha_2 + \alpha_3} & e_{\alpha_2},\ e_{\alpha_2 + 2\alpha_3}\\ \hline
 0 & h_{\alpha_2},\ h_{\alpha_3},\ e_{-\alpha_1},\ e_{\pm 2\alpha_3} & e_{\pm \alpha_3}\\ \hline
\end{array}
 \end{align}
 
\vskip0.1cm\par\noindent
{\underline{Case 2}.} 
If $(\fa_1)_{\bar0}=0$ but $(\fa_1)_{\bar1}\neq0$, then $\fa_2=\fa_3=0$ and $(\fa_1)_{\bar1}$ is abelian. 
In this case, we may conjugate by $(G_0)_{\bar 0}$ so that either 
$e_{\alpha_2}+e_{\alpha_1+\alpha_2+2\alpha_3}\in(\fa_1)_{\bar1}$ or $e_{\a_2}\in(\fa_1)_{\bar1}$. 

If $e_{\alpha_2}+e_{\alpha_1+\alpha_2+2\alpha_3}\in(\fa_1)_{\bar1}$, then also
$e_{-\alpha_3}\in(\fa_0)_{\bar1}$ and hence $e_{\alpha_1+\alpha_2+\alpha_3}\in(\fa_1)_{\bar0}\neq0$, 
which contradicts our assumption. Therefore $e_{\a_2}\in(\fa_1)_{\bar1}$, whence 
$e_{-\a_1},e_{-2\a_3},h_{\a_2}\in(\fa_0)_{\bar0}$ and $e_{-\a_3}\in(\fa_0)_{\bar1}$, 
giving a subalgebra $\mathfrak b\subset\fa$ of $\dim\mathfrak b=(8|8)$. 

In the case $\fa$ satisfies (ii) and the first condition of (i), which in our case is 
$6=\dim(\fa_0)_{\bar0}>\dim(\mathfrak b_0)_{\bar0}+2$, then  
$(\fa_1)_{\bar1}\supsetneq (\mathfrak b_1)_{\bar1}$ implying the second condition of (i). 
Also in the case $2=\dim(\fa_0)_{\bar1}>\dim(\mathfrak b_0)_{\bar1}$ we get 
$(\fa_1)_{\bar1}\supsetneq (\mathfrak b_1)_{\bar1}$. Thus we can always assume 
$(\fa_1)_{\bar1}\supsetneq (\mathfrak b_1)_{\bar1}$, in which case the abelian condition on $(\fa_1)_{\bar1}$
gives a non-zero element $c_1e_{\alpha_1+\alpha_2}+c_2e_{\alpha_2+2\alpha_3}\in (\fa_1)_{\bar1}$.
If $c_2\neq0$, then $e_{\alpha_2+\alpha_3}\in (\fa_1)_{\bar0}\neq0$, which contradicts our assumption.

Consequently $e_{\alpha_1+\alpha_2}\in (\fa_1)_{\bar1}$, whence 
$e_{\a_1}, h_{\alpha_1}\in(\fa_0)_{\bar0}$. This yields a subalgebra $\fa$ with $\dim\fa = (10|9)$, 
which cannot be further extended to a proper graded subalgebra of $G(3)$ with $(\fa_1)_{\bar0}=0$.
Its non-negative part is:
 \begin{align}\label{Tablecase(2)}
\;\;\;\;\;\;\;\; \begin{array}{|c|c|c|} \hline
 k & \text{even part} &  \text{odd part} \\ \hline\hline
 1 & & e_{\a_2},\ e_{\alpha_1+\alpha_2}\\ \hline
 0 & e_{\pm\alpha_1},\ h_{\alpha_1},\ h_{\alpha_2},\ e_{-2\alpha_3} & e_{-\alpha_3}\\ \hline
\end{array}
 \end{align}

\vskip0.1cm\par\noindent
{\underline{Case 3}.} 
If $\fa_+=0$ then $\dim\fa\leq(11|8)$. In the case of equality
$(\fa_0)_{\bar0}=\fsl(2)\oplus\fsp(2)$ and $(\fa_0)_{\bar1}=(\g_0)_{\bar1}$ or $(\fa_0)_{\bar1}=0$. 
The two sub-cases have respectively $\dim\fa=(11|8)$ and $\dim\fa=(11|6)$. 

\medskip\par\noindent
\underline{\it Step 2}: 
We will now study filtered deformations of the above algebras $\fa$ with big dimensions. 
Note that in each case the subalgebra $\tilde\fa_{\bar0}=\fa_{\bar0}\cap G(2)$ 
has $\dim\tilde\fa_{\bar0}>7$, so by the classical sub-maximal bound \cite{MR1509120,MR3604980}
we conclude that $\tilde\fa_{\bar0}\simeq\fa_{\bar0}/(\fsp(2)\cap\fa_{\bar0})\subset G(2)$
has no filtered deformation as the quotient algebra, and so it will be fixed under super-deformation.

\vskip0.1cm\par\noindent
{\underline{Case 1}.} 
In this case, $\dim\fa=(11|10)$ and $\fa_{\bar0}=\tilde\fa_{\bar0}\op\fsp(2)$ as the direct sum of 
two ideals, where $\tilde\fa_{\bar0}=G(2)_{-}\oplus G(2)^{ss}_0\subset G(2)$ is the opposite parabolic 
of the contact $\mathbb Z$-grading of $G(2)$ with the ``reduced'' Levi factor 
$G(2)^{ss}_0=\langle e_{-\alpha_2-\alpha_3},h_{\alpha_2+\alpha_3},e_{\alpha_2+\alpha_3}
\rangle\cong \fsl(2)$. The Levi subalgebra of $\fa_{\bar0}$ is  
$\mathfrak{l}=G(2)^{ss}_0\oplus\fsp(2)\cong \fsl(2)\oplus\fsp(2)$. 

Since the semi-simple factor is rigid, the even subalgebra $\fa_{\bar0}$ will get no
deformation. Thus $\fh_{\bar 0}=\fa_{\bar0}$ remains graded. By the Whitehead lemma
$H^1(\mathfrak{l},\mathbb{M})=0$ for any semisimple Lie algebra $\mathfrak{l}$
and its finite-dimensional module $\mathbb{M}$. Applying this to $\mathfrak{l}=\fsl(2)\oplus\fsp(2)$ 
and $\mathbb{M}=\mathop{End}\fa_{\bar1}$ we conclude that the brackets of 
$\mathfrak{l}$ with $\fa_{\bar1}$ can be assumed non-deformed (graded).

In order to study the remaining Lie brackets of the filtered deformation $\fh$, we exploit 
$\mathfrak{l}$-equivariancy of the bracket $[-,-]_+:\fa\otimes\fa\to\fa$,
decompose into irreducible $\mathfrak{l}$-modules
 $$
\fa_{\bar0}=(\C\boxtimes\C)\oplus(S^3\C^2\boxtimes\C)\oplus
\underbrace{(S^2\C^2\boxtimes\C)\oplus(\C\boxtimes S^2\C^2)}_{\mathfrak{l}},\quad
\fa_{\bar1}=(\C^2\boxtimes\C^2)\oplus(S^2\C^2\boxtimes\C^2),
 $$
and note that these subspaces of $\fa$ are not graded (except for $\fsp(2)$). 

Let us first study filtered deformation of the map $[-,-]:
(S^3\C^2\boxtimes\C)\otimes\fa_{\bar1}\to\fa_{\bar1}$.
Observe that there are unique (up to constant) $\mathfrak{l}$-equivariant maps 
\begin{align*}
& (S^3\C^2\boxtimes\C)\otimes(S^2\C^2\boxtimes\C^2)
=(\C^2\boxtimes\C^2)\oplus(S^3\C^2\boxtimes\C^2)\oplus(S^5\C^2\boxtimes\C^2)
\longrightarrow(\C^2\boxtimes\C^2),\\
& (S^3\C^2\boxtimes\C)\otimes(\C^2\boxtimes\C^2)
=(S^2\C^2\boxtimes\C^2)\oplus(S^4\C^2\boxtimes\C^2)
\longrightarrow(S^2\C^2\boxtimes\C^2),
 \end{align*}
and that the first map coincides with a component of $[-,-]_0$, therefore it has total degree 0. 
This is not true for the second map, which is graded of degree 3,
as follows from the fact that with the insertion of $w\in S^3\C^2\boxtimes\C$ the first map
$[w,-]: S^2\C^2\boxtimes\C^2 \to \C^2\boxtimes\C^2$ is conjugate to the second map 
$[w,-]: \C^2\boxtimes\C^2 \to S^2\C^2\boxtimes\C^2$ (the modules are self-dual).

Because $[(S^3\C^2\boxtimes\C),(S^3\C^2\boxtimes\C)]=\C\boxtimes\C$ generates
the rest of $\fa_{\bar0}$, this gives an $\fa_{\bar0}$-equivariant map
$[-,-]_+:\fa_{\bar0}\otimes\fa_{\bar1}\to\fa_{\bar1}$ and hence an $\fa_{\bar0}$-equivariant
candidate $[-,-]=[-,-]_0+\epsilon[-,-]_+$ for the new bracket 
with deformation parameter $\epsilon\in\C$.

In a similar way, we consider the $\mathfrak{l}$-equivariant bracket $[-,-]_+:\Lambda^2\fa_{\bar1}\to\fa_{\bar0}$ and
decompose $\Lambda^2\fa_{\bar1}$ into $\mathfrak{l}$-irreducibles. 
The relevant $\mathfrak{l}$-irreducible modules appear 
with multiplicity 1,
 \begin{align*}
S^2(S^2\C^2\boxtimes\C^2)&\longrightarrow\mathfrak{l},\\
S^2(\C^2\boxtimes\C^2)&\longrightarrow \C\boxtimes\C,\\
(S^2\C^2\boxtimes\C^2)\otimes(\C^2\boxtimes\C^2)&\longrightarrow S^3\C^2\boxtimes\C,
 \end{align*}
and the unique $\mathfrak{l}$-equivariant projections have zero degree. Therefore 
$[-,-]_+|_{\Lambda^2\fa_{\bar1}}$ vanishes.

Now we verify the Jacobi identity with one even and two odd arguments, more precisely
investigate equivariance of the middle map above with respect to $w\in S^3\C^2\boxtimes\C$.
This gives $\epsilon=0$ at once, and we conclude that $\fa=\fh$ as LSA. 


\vskip0.1cm\par\noindent
{\underline{Case 3}.} 
Next we consider the third case, where $\dim\fa=(11|8)$ or $(11|6)$. 
Then $\fa_{\bar0}=\tilde\fa_{\bar0}\op\fsp(2)$ as the direct sum of two ideals, where 
$\tilde\fa_{\bar0}=G(2)_{-}\oplus G(2)^{ss}_0\subset G(2)$ is the opposite parabolic of the 
HC $\mathbb Z$-grading of $G(2)$ with ``reduced'' Levi subalgebra 
$G(2)^{ss}_0=\langle e_{-\alpha_1},h_{\alpha_1},e_{\alpha_1}\rangle\cong \fsl(2)$.

As in Case 1, $\fh_{\bar 0}=\fa_{\bar0}$ is non-deformed, as well as the Lie brackets of 
$\mathfrak{l}$ with $\fa_{\bar1}$. However here the Levi subalgebra 
$\mathfrak{l}=G(2)^{ss}_0\oplus\fsp(2)\cong \fsl(2)\oplus\fsp(2)=(\fa_0)_{\bar0}$ 
is graded in zero degree.

The following decompositions of compontents of $\fa$ into irreducible modules under the adjoint 
action of $\mathfrak{l}$ are compatible with the grading:
  \begin{equation}\label{eq:decom}
\begin{array}{l}
\fa_{\bar0}=(\C^2\boxtimes\C)_{-3}\oplus(\C\boxtimes\C)_{-2}\oplus(\C^2\boxtimes\C)_{-1}\oplus
\underbrace{(S^2\C^2\boxtimes\C)_0\oplus(\C\boxtimes S^2\C^2)_0}_{\mathfrak{l}},
\vspace{-7pt}\\
\fa_{\bar1}=(\C\boxtimes\C^2)_{-2}\oplus(\C^2\boxtimes\C^2)_{-1}\oplus(\C\boxtimes\C^2)_0,
 \end{array}
 \end{equation}
and the last (zero grading) term in $\fa_{\bar1}$ has to be omitted when $\dim\fa=(11|6)$. 
We note that, contrary to Case 1, this case exhibits non-trivial multiplicities, and we will get 
many candidates for the $\mathfrak l$-equivariant map $[-,-]_+:\fa\otimes\fa\to\fa$. 

First of all, we  remark that $(\C^2\boxtimes\C)_{-1}$ generates $(\fa_{\bar0})_{-}$, therefore
any $\mathfrak{l}$-equivariant map $[-,-]_+:\fa_{\bar0}\otimes\fa_{\bar1}\to\fa_{\bar1}$ is completely 
determined by its restriction $[-,-]_+:(\C^2\boxtimes\C)_{-1}\otimes\fa_{\bar1}\to\fa_{\bar1}$
due to the Jacobi identities. We then compute
 \begin{alignat*}{1}
(\C^2\boxtimes\C)_{-1}\otimes(\C\boxtimes\C^2)_0 &=  
 (\C^2\boxtimes\C^2)_{-1},\\
(\C^2\boxtimes\C)_{-1}\otimes(\C^2\boxtimes\C^2)_{-1} &=  
 \underline{(\C\boxtimes\C^2)}{}_{-2\,\vee\,0}\oplus(S^2\C^2\boxtimes\C^2)_{\times},\\
(\C^2\boxtimes\C)_{-1}\otimes(\C\boxtimes\C^2)_{-2} &=  
 \underline{(\C^2\boxtimes\C^2)}{}_{-1}.
 \end{alignat*}
Due to multiplicity the middle term has two possibilities: the bracket of degree 0
(values in degree -2) that restricts the non-deformed bracket $[-,-]_0$ of $\fa$, 
and the bracket $[-,-]_+$ of degree $2$ (values in degree 0) that is a deformation
(this case is vacuous when $(\fa_0)_{\bar1}=0$). 

Above we indicate with a cross the irreducible modules not relevant for our arguments 
(kernel of the projection: they do not arise in the decomposition of $\fa$) 
and underline those which may contribute to 
$[-,-]_+:(\C^2\boxtimes\C)_{-1}\otimes\fa_{\bar1}\to\fa_{\bar1}$, namely 
 \begin{align*}
\alpha=[-,-]_+&:(\C^2\boxtimes\C)_{-1}\otimes(\C^2\boxtimes\C^2)_{-1}\to(\C\boxtimes\C^2)_{0}\\
\beta=[-,-]_+&:(\C^2\boxtimes\C)_{-1}\otimes(\C\boxtimes\C^2)_{-2}\to(\C^2\boxtimes\C^2)_{-1}
 \end{align*}
By changing the $\mathfrak l$-module decompisition 
$\fa_{\bar1}= (\fa_{\bar1})_{-2}\oplus (\fa_{\bar1})_{-1}\oplus (\fa_{\bar1})_0$ in \eqref{eq:decom} to 
$\fa_{\bar1}= \mathop{Im}(\alpha)\oplus (\fa_{\bar1})_{-1}\oplus (\fa_{\bar1})_0$ 
we may ignore $\alpha$ and study the possibility of $\fa_{\bar1}$ being $\fa_{\bar0}$ module
as if we had $\alpha=0$; note that $\alpha$ plays no role when $\dim\fa=(11|6)$. 
With this trick we have only one deformation parameter $\epsilon\in\C$ in the next computation.
 
We consider the basis \eqref{m-basis1}-\eqref{m-basis2} of 
the negatively-graded part $\fh_-$ of $\fh$ and focus on its 
Lie brackets (only the non-trivial relations are shown):
\begin{gather*}
 \left.\begin{aligned}
{}[e_1,e_2]=h,\ [e_1,h]=f_1,\ [e_2,h]=f_2,\ [\theta_1'',\theta_2'']=h,\ [\theta_1',\theta_2']=h\\
[e_1,\theta_2']=\rho_1,\ [e_1,\theta_2'']=\rho_2,\ [e_2,\theta_1'']=\rho_1,\ [e_2,\theta_1']=-\rho_2\\
[\theta_1'',\rho_2]=f_1,\ [\theta_1',\rho_1]=f_1,\ [\theta_2',\rho_2]=-f_2,\ [\theta_2'',\rho_1]=f_2
\end{aligned}\right\}=\text{SHC symbol}
 \\
[e_1,\rho_1]=\epsilon\theta_1'',\ [e_1,\rho_2]=-\epsilon\theta_1',\ 
[e_2,\rho_1]=-\epsilon\theta_2',\ [e_2,\rho_2]=-\epsilon\theta_2'',\ 
[h,\rho_1]=-2\epsilon\rho_1,\ [h,\rho_2]=-2\epsilon\rho_2,\\
[h,\theta_1'']=\epsilon\theta_1'',\ [h,\theta_1']=\epsilon\theta_1',\ 
[h,\theta_2']=\epsilon\theta_2',\ [h,\theta_2'']=\epsilon\theta_2'',\\
[f_1,\theta_2']=3\epsilon\rho_1,\ [f_1,\theta_2'']=3\epsilon\rho_2,\ 
[f_2,\theta_1'']=3\epsilon\rho_1,\ [f_2,\theta_1']=-3\epsilon\rho_2,\\ 
[f_1,\rho_1]=-3\epsilon^2\theta_1'',\ [f_1,\rho_2]=3\epsilon^2\theta_1',\ 
[f_2,\rho_1]=3\epsilon^2\theta_2',\ [f_2,\rho_2]=3\epsilon^2\theta_2'',
 \end{gather*}
Computing $0=[[e_1,f_1],\theta_2'']=-6\epsilon^2\theta_1'$ via the Leibniz identity, 
we conclude that $\epsilon=0$. Thus $\beta=0$ in the splitting of $\fa_{\bar1}$ where $\alpha=0$,
so we get one deformation parameter $\epsilon_1$ entering $\alpha,\beta$. 
Returning to the original grading, we get explicit expressions of the new brackets 
$\fh_{\bar0}\otimes\fh_{\bar1}\to\fh_{\bar1}$ via $\epsilon_1$ (we indicate only non-trivial relations): 
 \begin{gather*}
[w_{-1},\theta_0] = c(w,\theta)_{-1},\
[w_{-1},\zeta_{-1}] = c(w,\zeta)_{-2}+\epsilon_1c(w,\zeta)_0,\
[w_{-1},\theta_{-2}] = -\epsilon_1c(w,\theta)_{-1},\\
[h_{-2},\theta_0] = c(h,\theta)_{-2}+\epsilon_1c(h,\theta)_0,\
[h_{-2},\theta_{-2}] = -\epsilon_1c(h,\theta)_{-2}-\epsilon_1c(h,\theta)_0.
 \end{gather*}
Here $w\in\C^2\boxtimes\C$, $h\in\C\boxtimes\C$, $\theta\in\C\boxtimes\C^2$
and $\zeta\in\C^2\boxtimes\C^2$ are elements of the modules entering decomposition
\eqref{eq:decom}, subscript indicating the grading to distinguish them. 
The bilinear map $c(-,-)_p$ denotes the contraction of the corresponding modules
taking values in the (odd) module of grading $p$. 
It reads off the graded bracket $[-,-]_0$ of $\fa$.

We now deal with the $\mathfrak l$-equivariant map $[-,-]_+:\Lambda^2\fa_{\bar1}\to\fa_{\bar0}$ 
in a similar way. Namely, we have the following decompositions into $\mathfrak{l}$-irreducible modules:
 \begin{alignat*}{1}
S^2(\C\boxtimes\C^2)_0 &= (\C\boxtimes S^2\C^2)_0,\\
(\C^2\boxtimes\C^2)_{-1}\otimes(\C\boxtimes\C^2)_0 &=  
 (\C^2\boxtimes\C)_{-1}\oplus(\C^2\boxtimes S^2\C^2)_{\times},\\
(\C\boxtimes\C^2)_{-2}\otimes(\C\boxtimes\C^2)_0 &=  
 (\C\boxtimes\C)_{-2}\oplus \underline{(\C\boxtimes S^2\C^2)}{}_{0},\\
S^2(\C^2\boxtimes\C^2)_{-1} &=  
 (\C\boxtimes\C)_{-2}\oplus(S^2\C^2\boxtimes S^2\C^2)_{\times},\\
(\C\boxtimes\C^2)_{-2}\otimes(\C^2\boxtimes\C^2)_{-1} &=  
 \underline{(\C^2\boxtimes\C)}{}_{-3\,\vee\,-1}\oplus(\C^2\boxtimes S^2\C^2)_{\times},\\
S^2(\C\boxtimes\C^2)_{-2} &=  
 \underline{(\C\boxtimes S^2\C^2)}{}_0.
 \end{alignat*} 
Again irrelevant terms are indicated by a cross and the candidate terms for deformation are underlined.
If $\dim\fa=(11|6)$, the first three lines disappear, so the last two underlined terms
are the only potential contributions to the deformation. Similarly to the above, we arrive at explicit 
expressions of the new brackets $\Lambda^2\fh_{\bar1}\to\fh_{\bar0}$ via deformation parameters 
$\epsilon_2,\epsilon_3,\epsilon_4$:
 \begin{gather*}
[\theta'_0,\theta''_0] = c(\theta',\theta'')_0,\
[\zeta_{-1},\theta_0] = c(\zeta,\theta)_{-1},\
[\theta'_{-2},\theta''_0] = c(\theta',\theta'')_{-2}+\epsilon_2c(\theta',\theta'')_0,\\
[\zeta'_{-1},\zeta''_{-1}] = c(\zeta',\zeta'')_{-2},\
[\theta_{-2},\zeta_{-1}] = c(\theta,\zeta)_{-3}+\epsilon_3c(\theta,\zeta)_{-1},\
[\theta'_{-2},\theta''_{-2}] = \epsilon_4c(\theta',\theta'')_0.
 \end{gather*}
For $\dim\fa=(11|6)$ the relations with $\theta_0$ and the deformation
parameters $\epsilon_1,\epsilon_2$ disappear.

Now we investigate $\fa_{\bar0}$-equivariance of the new brackets. In the case $\dim\fa=(11|8)$
the Leibniz rule for $\mathop{ad}_w$, $w\in(\fa_{\bar0})_{-1}\simeq\C^2\boxtimes\C$, 
applied to the first line of brackets implies $\epsilon_2=\epsilon_3=-\epsilon_1$.
Then applied to the second line of brackets it gives $\epsilon_1=0$, $\epsilon_4=\epsilon_1^2=0$
and hence $\epsilon_2=\epsilon_3=0$. In the case $\dim\fa=(11|8)$ we get only the second
line of brackets, and they similarly imply $\epsilon_3=0$ and then $\epsilon_4=0$.

Thus we again conclude that the deformation is trivial, so that $\fa=\fh$ as LSA. .

\vskip0.1cm\par\noindent
{\underline{Case 2}.} 
In the remaining second case $\fa_{\bar0}=\tilde\fa_{\bar0}\op\mathfrak{b}(2)$, where 
$\mathfrak{b}(2)=\langle h_{\alpha_2},e_{-2\alpha_3}\rangle$
and $\tilde\fa_{\bar0}=G(2)_{-}\oplus G(2)^{ss}_0\subset G(2)$ is the opposite parabolic of the HC 
$\mathbb Z$-grading of $G(2)$ as in Case 3, i.e., with the ``reduced'' Levi subalgebra 
$G(2)^{ss}_0=\langle e_{-\alpha_1},h_{\alpha_1},e_{\alpha_1}\rangle\cong \fsl(2)$. 
We stress that $\mathfrak{b}(2)$ is {\em abstractly isomorphic} to a Borel subalgebra of $\fsp(2)$, 
but it is actually {\it not} contained in $\fsp(2)$ because the coroot $h_{\alpha_2}$ sits diagonally 
w.r.t.\ the decomposition $\fg_{\bar0}=G(2)\oplus\fsp(2)$. 

This case is the most involved, as $[h_{\alpha_2},e_{\pm\alpha_1}]\neq 0$ and 
$\fa_{\bar0}=\tilde\fa_{\bar0}\op\mathfrak{b}(2)$ is not a decompositon into ideals. 
Moreover, the Levi factor of $\fa_{\bar0}$ is just  $\mathfrak l=G(2)^{ss}_0\cong\mathfrak\fsl(2)$ 
and its representation theory is less restrictive than in the previous cases.
Thus we exploit the representation theory of $\fsl(2)$ but also use a brute 
force computation. Those are done in Maple (available in the arXiv supplement)
and rely on linear algebra over $\mathbb{Q}$ only, so no rigor suffers. 

We will now summarize the computations.
Following the same strategy as before we show that the Lie brackets on $\fa_{\bar0}$ 
are rigid. Indeed, the quotient algebra $\tilde\fa_{\bar0}=\fa_{\bar0}\!\mod\mathfrak{b}(2)$ 
as well as the subalgebra $\mathfrak{b}(2)\subset(\fa_0)_{\bar0}$ are non-deformed, 
and the $\fsl(2)$ module structure 
 $$
\fa_{\bar0}=(\C^2)_{-3}\oplus(\C)_{-2}\oplus(\C^2)_{-1}\oplus(S^2\C^2)_0\oplus(\C\oplus\C)_0 
 $$
is rigid. 
Thus the only contribution to the filtered deformation  may arise from the positive degree brackets 
$(\C^2)_{-3}\otimes(\C^2)_{-1}\to(\C\oplus\C)_{-1}$ and $\Lambda^2(\C^2)_{-3}\to(\C\oplus\C)_{-1}$. 
These carry four parameters (by $\fsl(2)$-equivariance), 
which have to vanish due to the Jacobi identity.

Next we deform the even-odd brackets, i.e., the representation of $\fa_{\bar0}$ over $\fa_{\bar1}$.
As an $\fsl(2)$-module 
 $$
\fa_{\bar1}=(\C\oplus\C)_{-2}\oplus(\C^2\oplus\C^2)_{-1}\oplus(\C)_0\oplus(\C^2)_1
 $$
and we employ the $\fsl(2)$-equivariance of the brackets $\fa_{\bar0}\otimes\fa_{\bar1}\to\fa_{\bar1}$.
These satisfy the module structure constraints if and only if $\fa_{\bar0}\oplus\fa_{\bar1}$
with trivial brackets on $\fa_{\bar1}$ is a Lie algebra.
The Jacobi identity constrains the parameters of the semi-direct product 
$\fa_{\bar0}\ltimes\fa_{\bar1}$ as follows:
 \begin{gather*}
[e_1,e_2] = h,\ [e_1,h] = f_1,\ [e_2, h] = f_2,\ 
[u, e_2] = e_1,\ [l, e_1] = e_2,\ [s, e_1] = e_1,\ [s, e_2] = -e_2,\\
[u, f_2] = f_1,\ [l, f_1] = f_2,\ [s, f_1] = f_1,\ [s, f_2] = -f_2, 
[u, l] = s,\ [s, u] = 2u,\ [s, l] = -2l,\ [r, n] = 2n,\\
[r,e_1] = \tfrac14e_1,\ [r,e_2] = \tfrac14e_2,\ [r,h] = \tfrac12h,\ [r,f_1] = \tfrac34f_1, [r,f_2] = \tfrac34f_2,\
[u,\theta_2''] = \theta_1',\ [l,\theta_1'] = \theta_2'',\\ [s,\theta_1'] = \theta_1',\ [s,\theta_2''] = -\theta_2'',\ 
[u,\theta_2'] =-\theta_1'',\ [l,\theta_1''] =-\theta_2',\ [s,\theta_1''] =\theta_1'',\ [s, \theta_2'] = -\theta_2',\ 
[u, \xi_2] = \xi_1,\\ 
[l, \xi_1] = \xi_2,\ [s, \xi_1] = \xi_1,\ [s, \xi_2] = -\xi_2,\ 
[n, \theta_1'] = \epsilon_2\xi_1+\theta_1'',\ [n, \theta_2''] = \epsilon_2\xi_2-\theta_2',\\ 
[n, \rho_2] = -\rho_1+2\epsilon_1\zeta,\ 
[r, \rho_1] = \tfrac32\rho_1-\epsilon_1\zeta,\ [r, \rho_2] = -\tfrac12\rho_2+\epsilon_3\zeta,\ 
[r, \theta_1'] = -\tfrac34\theta_1'+\epsilon_4\xi_1,\\
[r, \theta_2''] = -\tfrac34\theta_2''+\epsilon_4\xi_2,\ [r, \theta_1''] = \tfrac54\theta_1''+\tfrac12\epsilon_2\xi_1,\ 
[r, \theta_2'] = \tfrac54\theta_2'-\tfrac12\epsilon_2\xi_2,\ [r, \zeta] = \zeta,\\ 
[r, \xi_1] = \tfrac34\xi_1,\ [r, \xi_2] = \tfrac34\xi_2,\ 
[e_1, \xi_2] = \zeta=-[e_2, \xi_1],\
[e_1, \zeta] = \epsilon_2\xi_1+\theta_1'',\ [e_2, \zeta] = \epsilon_2\xi_2-\theta_2',\\
[e_1, \theta_2''] = \rho_2+\tfrac23(\epsilon_4-\epsilon_3)\zeta=- [e_2, \theta_1'],\
[e_1, \theta_2'] = \rho_1-(2\epsilon_1-\epsilon_2)\zeta=[e_2, \theta_1''],\\
[e_1, \rho_1] = 2\epsilon_1\epsilon_2\xi_1+2\epsilon_1\theta_1'',\ [e_2, \rho_1] = 2\epsilon_1\epsilon_2\xi_2-2\epsilon_1\theta_2',\
[e_1, \rho_2] = \tfrac23\epsilon_3\theta_1'',\ [e_2, \rho_2] = -\tfrac23\epsilon_3\theta_2',\\
[h, \xi_1] = -\epsilon_2\xi_1-\theta_1'',\ [h, \xi_2] = -\epsilon_2\xi_2+\theta_2',\ [h, \zeta] = 4\epsilon_1\zeta-2\rho_1,\
[h, \theta_1''] = \epsilon_2\xi_1+\epsilon_2\theta_1'',\\ 
[h, \theta_1'] = -\tfrac23\epsilon_4\theta_1''-\tfrac13\epsilon_2(\epsilon_3+2\epsilon_4)\xi_1,\ 
[h, \theta_2''] = \tfrac23\epsilon_4\theta_2'-\tfrac13\epsilon_2(\epsilon_3+2\epsilon_4)\xi_2,\ 
[h, \theta_2'] = -\epsilon_2\xi_2+\epsilon_2\theta_2',\\ 
[h, \rho_1] = 8\epsilon_1^2\zeta-4\epsilon_1\rho_1,\ 
[h, \rho_2] = -\tfrac43\epsilon_3\rho_1+\tfrac83\epsilon_1\epsilon_3\zeta,\ 
[f_1, \xi_2] = 3\rho_1-6\epsilon_1\zeta=-[f_2, \xi_1],\\ 
[f_1, \theta_2''] = 2\epsilon_4\rho_1-4\epsilon_1\epsilon_4\zeta=-[f_2, \theta_1'],\
[f_1, \theta_2'] = 3\epsilon_2\rho_1-6\epsilon_1\epsilon_2\zeta=[f_2, \theta_1'']\;.
 \end{gather*}
Here $\fsl(2)=\langle u,s,l\rangle$, $\mathfrak{b}(2)=\langle r,n\rangle$,
$(\fa_0)_{\bar1}=\langle \zeta\rangle$, $(\fa_1)_{\bar1}=\langle \xi_1,\xi_2\rangle$,
and the other notations are as in (M1). The parameters $\epsilon_1,\epsilon_2,\epsilon_3,\epsilon_4$ 
satisfy the
following branching condition: either $\epsilon_2=0$ or else $\epsilon_2=1$, $\epsilon_4=0$. 
We now turn to the Lie superalgebra bracket between odd elements.



 In the second branch no $\fa_{\bar0}$-equivariant map $\Lambda^2\fa_{\bar1}\to\fa_{\bar0}$ exists.
In the first branch such a map 
exists if and only if $\epsilon_4=0$, in this case it is unique and survives 
the Jacobi identity for all odd elements. In other words, we get a two-parametric
filtered deformation of dimension $(10|9)$. However this deformation is trivial: 
the change of basis $\rho_1\mapsto\rho_1-2\epsilon_1\zeta$, $\rho_2\mapsto\rho_2-\frac23\epsilon_3\zeta$
eliminates the parameters from the expression for the brackets.

\vskip0.1cm\par\noindent
{\underline{\it Summary}.} 
No nontrivial filtered deformations of $\fa$ exist, and all cases give rise to the flat 
structure with maximal supersymmetry, contrary to the assumption. Hence the claim.
\end{proof}

 \subsubsection{Realization of the supersymmetry bound}
Consider the following system of PDE involving one arbitrary function $f$ of 1 variable:
\begin{equation}\label{sMonge}
z_x = f(u_{xx}) + u_{x\nu}u_{x\tau},\quad
z_{\tau} = f'(u_{xx}) u_{x\tau},\quad
z_{\nu}  = f'(u_{xx}) u_{x\nu},\quad
u_{\tau\nu} = f'(u_{xx}).
\end{equation}
According to Proposition \ref{FGHK-constraints}, the associated Cartan superdistribution $\cD$ on the superspace $M=\C^{5|6}(x,u,u_x,u_{xx},z|\tau,\nu,u_\tau,u_\nu,u_{x\tau},u_{x\nu})$ is of SHC type 
when $f''\neq 0$, and in this case 
it shall be considered as a super-extension of the classical family of generic rank 2 
distributions on a 5-dimensional space with Monge normal form $z_x = f(u_{xx})$.

We will now see that super-extensions of the classical submaximally symmetric models 
given by the choice $f(s)=\frac1m s^m$, namely
\begin{equation}\label{m-eqn}
 z_x=\tfrac1mu_{xx}^m+u_{x\nu}u_{x\tau},\quad
 z_{\tau}=u_{xx}^{m-1}u_{x\tau},\quad
 z_{\nu}=u_{xx}^{m-1}u_{x\nu},\quad
 u_{\tau\nu}=u_{xx}^{m-1},
 \end{equation}
do realize the upper bound of Theorem \ref{UpperBound}.

 \begin{theorem}\label{super-sub-max}
The internal symmetry superalgebra of the SHC type equation \eqref{sMonge}
for $f(s)=\frac1m s^m$ with $m\neq-1,0,\frac13,\frac23,1,2$,
as well as $f(s)=\ln s$ and $f(s)=\exp s$ has $\dim=(10|8)$.
 \end{theorem}

 \begin{proof}
It is obvious that for $m\neq0,1$ we have $f''\neq 0$, hence the distribution is of SHC type. 
We claim that for $m\neq -1,\frac13,\frac23,2$, the symmetry superalgebra $\mathfrak{inf}(M,\cD)$ 
is not $\fg=G(3)$. Otherwise, the superdistribution $\cD$ would be flat and the underlying classical 
distribution of HC type would be flat too. Indeed, the supersymmetry $G(3)$ of $\cD$ reduces 
to the symmetry $G(2)$ of the underlying distribution
(recall that $G(3)_{\bar0}=G(2)\oplus A(1)$ but the second factor belongs to the kernel of the action), 
which is possible only for 
$m= -1,\frac13,\frac23,2$, see \cite{MR3254311}. 


We will now give the explicit realization of $\mathfrak{inf}(M,\cD)$ for $m\neq -1,0,\frac13,\frac23,1,2$, 
and at the same time show that the local transitivity assumption required in Theorem \ref{T:G3-sym} is 
satisfied. In particular $\dim\mathfrak{inf}(M,\cD)\leq (10|8)$ by Theorem \ref{UpperBound}.

A supervector field on $M=\C^{5|6}(x,u,u_x,u_{xx},z|\tau,\nu,u_\tau,u_\nu,u_{x\tau},u_{x\nu})$ is in $\mathfrak{inf}(M,\cD)$ if and only if
its coefficients satisfy an appropriate system of differential equations in superspace.
As soon as it is
expanded in the odd variables, this gives a large overdetermined system of PDE on ordinary functions of 5 variables, 
which we solve bringing the system to involution.
To write down the explicit expression of the generators of $\mathfrak{inf}(M,\cD)$, we find convenient to relabel the odd coordinates
$\theta_1=\tau$, $\theta_2=\nu$, $\theta_3=u_\tau$, $\theta_4=u_\nu$, $\theta_5=u_{x\tau}$, $\theta_6=u_{x\nu}$ and set
$\theta_{ij}=\theta_i\theta_j$, $\theta_{ijk}=\theta_i\theta_j\theta_k$. Here are the even generators
\begin{align*}
V_1 =&\, \p_{x},\
V_2 = \p_{z},\
V_3 = \p_{u_x}+x\p_{u},\
V_4 = \p_{u},
 \\
V_5 =&\,
 x\p_{x}+2u\p_{u}+u_x\p_{u_x}+z\p_{z}+2\theta_{1}\p_{\tau}+2\theta_{4}\p_{u_{\nu}}-\theta_{5}\p_{u_{x\tau}}
+\theta_{6}\p_{u_{x\nu}},
 \\
V_6 =&\,
 u\p_{u}+u_x\p_{u_x}+u_{xx}\p_{u_{xx}}+mz\p_{z}
 +\theta_{1}\p_{\tau}-(m-1)\theta_{2}\p_{\nu}+m\theta_{4}\p_{u_{\nu}}+m\theta_{6}\p_{u_{x\nu}},
 \\
V_7 =&\,
 u_{xx}^{m-1}\p_{x}-(z-u_xu_{xx}^{m-1}+\theta_{36}-\theta_{45})\p_{u}
 +(\tfrac{m-1}{m}u_{xx}^m-\theta_{56})\p_{u_x}\hphantom{aaaaaaaaaaaaaaaaaaaa}
 \\
 & +u_{xx}^{m-1}(\tfrac{m-1}{m(2m-1)}u_{xx}^{m}-\theta_{56})\p_{z}
 +\theta_{6}\p_{\tau}-\theta_{5}\p_{\nu}+u_{xx}^{m-1}(\theta_{5}\p_{u_{\tau}}+\theta_{6}\p_{u_{\nu}}),
 \\
V_8 =&\,
 \theta_{1}\p_{\tau}-\theta_{2}\p_{\nu}-\theta_{3}\p_{u_{\tau}}+\theta_{4}\p_{u_{\nu}}
 -\theta_{5}\p_{u_{x\tau}}+\theta_{6}\p_{u_{x\nu}},
 \\
V_9 =&\,
 \theta_{2}\p_{\tau}-\theta_{3}\p_{u_{\nu}}-\theta_{5}\p_{u_{x\nu}},\
V_{10} = \theta_{1}\p_{\nu}-\theta_{4}\p_{u_{\tau}}-\theta_{6}\p_{u_{x\tau}},
 \end{align*}
and the odd generators
 \begin{align*}
U_1 =&\, \p_{u_{\tau}}-\theta_{1}\p_{u},\
U_2 = \p_{u_{\nu}}-\theta_{2}\p_{u},\
U_3 = \p_{\tau},\
U_4 = \p_{\nu},
 \\
U_5 =&\,
 \p_{u_{x\tau}}+x\p_{u_{\tau}}-x\theta_{1}\p_{u}-\theta_{1}\p_{u_x}-\theta_{4}\p_{z},\
U_6 = \p_{u_{x\nu}}+x\p_{u_{\nu}}-x\theta_{2}\p_{u}-\theta_{2}\p_{u_x}+\theta_{3}\p_{z},
 \\
U_7 =&\,
 (\theta_{4}-(2m-1)u_{xx}^{m-1}\theta_{1})\p_{x}
 +(u_x\theta_{4}+(2m-1)((z-u_xu_{xx}^{m-1})\theta_{1}+\theta_{136}-\theta_{145}))\p_{u}\\
 & -(2m-1)(\tfrac{m-1}{m}u_{xx}^m\theta_{1}-\theta_{156})\p_{u_x}-2u_{xx}\theta_{6}\p_{u_{xx}}
 -u_{xx}^{m-1}(\tfrac{m-1}{m}u_{xx}^{m}\theta_{1}-(2m-1)\theta_{156})\p_{z}\\
 & -(2m-1)\theta_{16}\p_{\tau}
 +(u_x+(2m-1)\theta_{15})\p_{\nu}-(\tfrac{m-1}{m}u_{xx}^m-2m\theta_{56})\p_{u_{x\tau}}\\
 & -(2m-1)(z+u_{xx}^{m-1}\theta_{15})\p_{u_{\tau}}-(2m-1)u_{xx}^{m-1}\theta_{16}\p_{u_{\nu}},\\
U_8 =&\,
 (\theta_{3}+(2m-1)u_{xx}^{m-1}\theta_{2})\p_{x}
 +(u_x\theta_{3}+(2m-1)((u_xu_{xx}^{m-1}-z)\theta_{2}+\theta_{245}-\theta_{236}))\p_{u}\\
 & +(2m-1)(\tfrac{m-1}{m}u_{xx}^m\theta_{2}-\theta_{256})\p_{u_x}-2u_{xx}\theta_{5}\p_{u_{xx}}
 +u_{xx}^{m-1}(\tfrac{m-1}{m}u_{xx}^{m}\theta_{2}-(2m-1)\theta_{256})\p_{z}\\
 & +(u_x+(2m-1)\theta_{26})\p_{\tau}-(2m-1)\theta_{25}\p_{\nu}+(\tfrac{m-1}{m}u_{xx}^m-2m\theta_{56})\p_{u_{x\nu}}\\
 & +(2m-1)u_{xx}^{m-1}\theta_{25}\p_{u_{\tau}}+(2m-1)(z+u_{xx}^{m-1}\theta_{26})\p_{u_{\nu}}.
 \end{align*}
We note that the even part $\mathfrak{inf}(M,\cD)_{\bar 0}$ is given by the direct sum of 
$\fsp(2)=\langle V_8,V_9,V_{10}\rangle$ and the $7$-dimensional complementary subalgebra 
$\langle V_1,\dots,V_7\rangle$, corresponding to the classical submaximal symmetry of a 
the Monge equation $z_x=\frac1mu_{xx}^m$. 

The above expressions hold for a generic value of the parameter $m$ and they only change a bit 
for $m=\tfrac12$, see \cite{MR3254311}. The case $f(s)=\ln s$ (as well as $f(s)=\exp s$, 
which classically corresponds to the case $m=\frac12$) 
is treated similarly and we omit the details.
\end{proof}

 \subsubsection{Special values of the parameter $m$}
The exceptional values of $m$ in Theorem \ref{super-sub-max} are the same as in the classical case.
The function $f(s)=\frac1ms^m$ in \eqref{sMonge} 
is not defined for $m=0$, but throughout this section we will replace it with the function $f(s)=s^m$.

For the underlying $(2,1,2)$ even distribution the values $m=0,1$ correspond to 
distributions of infinite type. Indeed, the distribution fails to be totally non-holonomic
and has isomorphic maximal leaves, whence infinite-dimensional symmetry.
The values $m=-1,\frac13,\frac23,2$ correspond to flat distributions with symmetry $G(2)$.
For all other values of $m$, as well as for $f(s)=\ln s$ and $f(s)=\exp s$, the symmetry dimension
is equal to 7.

In the super-case, the values $m=0,1$ again correspond to infinite-dimensional symmetry,
but for a different reason. In fact, the distribution is totally non-holonomic with the
growth vector is $(2|4,1|2,2|0)$, and its symbol superalgebra is isomorphic to (M2).
This has an Abelian even ideal $\langle e_2,h,f_1,f_2\rangle$, in coordinates corresponding 
to $\langle \p_u,\p_{u_x},\p_{u_{xx}},\p_z\rangle$. Thus the classical argument for infinite type
fails here, yet we can explicitly demonstrate the claim as follows.

For $m=1$, equation \eqref{m-eqn} becomes
 $$
 z_x= u_{xx}+u_{x\nu}u_{x\tau},\quad
 z_{\tau}=u_{x\tau},\quad
 z_{\nu}=u_{x\nu},\quad 
 u_{\tau\nu}=1
 $$
whose solutions are $u=\phi(x)+\nu\tau$, $z=\phi'(x)+c$, for an arbitrary function $\phi(x)$ and constant 
$c$. The vector field $\xi=\psi(x)\p_u+\psi'(x)\p_{u_x}+\psi''(x)\p_{u_{xx}}+\psi'(x)\p_{z}$
is a symmetry for any function $\psi(x)$, thus the symmetry superalgebra is infinite-dimensional.

For $m=0$, the corresponding equation is
 $$
 z_x= 1+u_{x\nu}u_{x\tau},\quad
 z_{\tau}=0,\quad
 z_{\nu}=0,\quad
 u_{\tau\nu}=0
 $$
whose solutions are $u=\phi(x)$, $z=x+c$, for an abitrary function $\phi(x)$ and constant $c$. 
The vector field $\xi=\psi(x)\p_u+\psi'(x)\p_{u_x}+\psi''(x)\p_{u_{xx}}$ is a symmetry for 
any function $\psi(x)$. 

We already know that the value $m=2$ corresponds to the SHC equation with symmetry superalgebra $G(3)$.
The other three values $m=-1,\frac13,\frac23$ are however special: a direct computation shows that
the symmetry superalgebra has dimension $(10|8)$ in each case! Hence, a maximally symmetric classical 
rank 2 distribution can be super-extended to one with the submaximal supersymmetry dimension.

However a maximally symmetric rank 2 classical distribution can also be extended to a superdistribution 
of SHC type with maximal supersymmetry $G(3)$. In fact, there are super-extensions of the Monge 
equations $z_x=\frac1mu_{xx}^m$ for $m=-1,\frac13,\frac23$ with supersymmetry $G(3)$. 
To get one, take a classical equivalence of any of these cases with the $m=2$ case \cite{MR3254311}
and apply its super-extension to the SHC equation. 

For instance, the Legendre transformation 
 $$
(x,\tau,\nu,u,u_x,u_\tau,u_\nu)\mapsto(u_x,\tau,\nu,u-xu_x,-x,u_\tau,u_\nu)$$
maps the SHC equation to the following system:
 $$
 z_x=u_{xx}^{-1}-u_{xx}^{-1}u_{x\nu}u_{x\tau},\quad
 z_\tau=-u_{xx}^{-2}u_{x\tau},\quad
 z_\nu=-u_{xx}^{-2}u_{x\nu},\quad
u_{\tau\nu}=2u_{xx}^{-1}-u_{xx}^{-1}u_{x\nu}u_{x\tau}.
 $$
This is a super-extension of the equation $z_x=-u_{xx}^{-1}$ with supersymmetry $G(3)$.


\appendix

\section{Parabolic $G(3)$-supergeometries and equivalences}\label{A0}

Here we discuss the 19 supergeometries associated to parabolic subgroups $P_\cA\subset G$, where
$G$ is a Lie supergroup corresponding to $\fg=G(3)$.  The flat models $M_\cA=G/P_\cA\stackrel{\text{loc}}\simeq\exp\fm$ are shown in 
Figure \ref{F:G3-supergeometries}, and in Table \ref{Table:19sgeo} below we specify the type of the left-invariant distribution $\cD$ given by 
$\fg_{-1}$. 
In particular, we indicate  its depth $\mu_\cD$ and growth vector. 

In the curved case, the geometry of type $(G,P_\cA)$ is given by a distribution $\cD$ on $M$ 
with symbol as in the flat case and with a possible reduction of structure group to $G_0$
(and perhaps even higher order reductions). The latter is
related to the computation of the cohomology group $H^1(\fm,\fg)$ and will not be discussed here.
Specific bracket relations are encoded by the roots associated to each Dynkin diagram,
and will not be indicated.
 \begin{table}[h]
  \[
 \begin{array}{|l|l|c|c|l|} \hline
\mbox{Type of }P_\cA & \mbox{Structure group }G_0 & \dim M_\cA\! & \mu_\cD & 
\mbox{Growth vector } \\ \hline\hline
P_1^{\rm I} & G(2)\times\C^\times &  (1|7) & 2 & (0|7,1|0) \\
P_3^{\rm I}=P_3^{\rm II}=P_2^{\rm III} & GL(2|1) &  (6|5) & 2 & (4|3,2|2) \\
P_1^{\rm III}=P_1^{\rm IV} & COSp(3|2) &  (5|4) & 2 & (4|4,1|0) \\ 
P_1^{\rm II}=P_3^{\rm III}=P_3^{\rm IV} & GL(2|1) &  (6|5) & 3 & (2|2,2|2,2|1) \\
P_{2}^{\rm IV} & GL(2) \times OSp(1|2) &  (5|6) & 3 & (2|4,1|2,2|0) \\ 
P_2^{\rm I}=P_2^{\rm II} & GL(2)\times SL(1|1) &  (6|6) & 4 & (2|2,1|1,2|2,1|1) \\ \hline
P_{13}^{\rm I} & GL(2)\times\C^\times &  (6|7) & 4 & (4|2,1|3,0|2,1|0) \\
P_{12}^{\rm III} & GL(2)\times\C^\times &  (6|7) & 4 & (0|5,5|0,0|2,1|0) \\ 
P_{13}^{\rm II}=P_{23}^{\rm III} & GL(1|1)\times\C^\times &  (7|6) & 5 & (2|2,1|1,1|1,2|1,1|1) \\
P_{13}^{\rm III}=P_{13}^{\rm IV} & GL(1|1)\times\C^\times &  (7|6) & 5 & (2|2,2|2,1|1,1|1,1|0) \\
P_{12}^{\rm IV} & COSp(1|2)\times\C^\times &  (6|6) & 5 & (2|2,1|2,1|2,1|0,1|0) \\ 
P_{12}^{\rm I} & GL(2)\times\C^\times &  (6|7) & 6 & (2|1,1|2,2|1,0|2,0|1,1|0) \\
P_{23}^{\rm I}=P_{23}^{\rm II} & GL(1|1)\times\C^\times &  (7|6) & 6 & (2|1,1|1,1|1,1|1,1|1,1|1) \\
P_{23}^{\rm IV} & GL(2)\times\C^\times &  (6|7) & 6 & (0|3,3|0,0|3,1|0,0|1,2|0) \\ 
P_{12}^{\rm II} & GL(2)\times\C^\times &  (6|7) & 7 & (0|3,2|0,0|1,1|0,0|2,3|0,0|1) \\ \hline
P_{123}^{\rm III} & \C^\times\times\C^\times\times\C^\times &  (7|7) & 7 & 
(1|2,1|2,1|1,2|0,1|1,0|1,1|0) \\ 
P_{123}^{\rm I} & \C^\times\times\C^\times\times\C^\times &  (7|7) & 8 & 
(2|1,1|1,1|1,1|1,1|1,0|1,0|1,1|0) \\ 
P_{123}^{\rm IV} & \C^\times\times\C^\times\times\C^\times &  (7|7) & 8 & 
(1|2,2|1,1|1,0|2,1|0,0|1,1|0,1|0)  \\ 
P_{123}^{\rm II} & \C^\times\times\C^\times\times\C^\times &  (7|7) & 9 & 
(1|2,1|1,1|0,0|1,1|0,0|1,1|1,2|0,0|1) \\ 
\hline
  \end{array}
 \]
 \caption{Distributions on generalized flag supermanifolds $M_\cA=G/P_\cA$}
 \label{Table:19sgeo}
 \end{table} 

For example, the first case represents a $(G(3),P_1^{\rm I})$ type geometry via a purely odd contact structure.  If $\sigma$ is a local defining 1-form for the rank $(0|7)$ contact distribution $\cC$, then $[d\sigma|_\cC]$ is a conformal metric (in the classical sense).  A reduction of the structure group to $G(2)\times\C^\times\subset CO(7)$ can be encoded by the additional choice of a conformal class of generic 3-forms on $\cC$.

Let us focus on the geometries highlighted on
Figure \ref{F:G3-supergeometries}. The geometries of types $M_1^{\rm IV}$ and 
$M_2^{\rm IV}$ are well studied in this paper. To understand the type $M_{12}^{\rm IV}$, we
consider the roots organized by parity and grading for the parabolic subalgebra $\fp_{12}^{\rm IV}$:
  \begin{align} \label{posroots-IV-12}
 \begin{array}{|c|c|c|} \hline
 k & \Delta_{\bar{0}}(k) &  \Delta_{\bar{1}}(k) \\ \hline\hline
 0 & \pm 2\alpha_3 & \pm \alpha_3\\\hline
 1 & \alpha_1,\ \alpha_2 + \alpha_3 & \alpha_2,\ \alpha_2 + 2\alpha_3, \\ \hline
 2 & \alpha_1 + \alpha_2 + \alpha_3 & \alpha_1 + \alpha_2,\ \alpha_1 + \alpha_2 + 2\alpha_3 \\ \hline
 3 & \alpha_1 + 2 \alpha_2 + 2 \alpha_3 & \alpha_1 + 2\alpha_2 + \alpha_3,\ 
       \alpha_1 + 2\alpha_2 + 3\alpha_3 \\ \hline
 4 & \alpha_1 + 3 \alpha_2 + 3\alpha_3 & \\ \hline
 5 & 2\alpha_1 + 3 \alpha_2 + 3\alpha_3 & \\ \hline
 \end{array}
 \end{align}

The bracket structure is given 
by the addition of roots. In particular, we note that $\fm_{\bar0}$ is trivial as a module over $\fsp(2)\subset(\fg_0)_{\bar0}$, 
while $\fm_{\bar1}$ consists of 3 standard representations. On a supermanifold of dimension $(6|6)$, a distribution with growth vector $(2|2,1|2,1|2,1|0,1|0)$ 
and the given symbol is said to be of type $M_{12}^{\rm IV}$.

Not all geometries in Table \ref{Table:19sgeo} are different. For example, we have the following:

 \begin{theorem}
There is an equivalence of categories between the germs of distributions of type 
$M_{12}^{\rm IV}$ (as above) and $M_2^{\rm IV}$ (i.e., of SHC-type).
 \end{theorem}

 \begin{proof}
This follows from the two mutually inverse constructions discussed next.

\smallskip

 \noindent{\bf From $M_2^{\rm IV}$ to $M_{12}^{\rm IV}$:} 
Consider a rank $(2|4)$ distribution $\cD$ of SHC-type on a supermanifold $M$ of dimension $(5|6)$.
We define its prolongation $\hat{M}=\mathbb{P}\cD_{\bar0}$ via the functor of points by the formula
$$\hat{M}=\{(x,\ell)\mid x=\text{super-point of}\;M,\;\ell=\text{rank $(1|0)$ free submodule of}\;\cD_{\bar0}|_x\}\;.$$ This is a supermanifold of dimension $(6|6)$.
Let $\pi:\hat{M}\to M$ be the natural projection. We define a distribution $\hat{\cD}$ on $\hat{M}$ 
by the formulae $\hat{\cD}_{\bar0}=\pi_*^{-1}(\ell)$ and 
$\hat{\cD}_{\bar1}=\mathop{\rm Ker}(\Xi(\ell,\cdot))$, where $\Xi:(\g_{-1})_{\bar0}\otimes (\g_{-1})_{\bar1}\to (\g_{-2})_{\bar1}$ is a component of the bracket. The rank of $\hat{\cD}$ is $(2|2)$ and it is straightforward to verify that 
$(\hat{M},\hat{\cD})$ is of type $M_{12}^{\rm IV}$.

\smallskip

 \noindent{\bf From $M_{12}^{\rm IV}$ to $M_2^{\rm IV}$:} 
Consider a rank $(2|2)$ distribution $\hat{\cD}$ of type $M_{12}^{\rm IV}$ on a manifold $\hat{M}$ 
of dimension $(6|6)$. From the symbol of this distribution (obtained from \eqref{posroots-IV-12}), we see that the derived distribution
$\hat{\cD}_2=[\hat{\cD},\hat{\cD}]$ has a Cauchy characteristic in $\hat{\cD}_{\bar0}$.
Let $\pi:\hat{M}\to M$ be a (local) quotient by it and define $\cD=\hat{\cD}_2/\mathop{\rm Ch}(\hat{\cD}_2)$.
The rank of $\cD$ is $(2|4)$ and it is straightforward to verify that $(M,\cD)$ is of SHC-type. 
 \end{proof}

 \begin{corollary}
For $\g=G(3)$ with $\fm = \fg_-$ associated to $\fp_{12}^{\rm IV}$, we have: $H^{d,1}(\fm,\g)=0$ for all $d \geq 0$.
 \end{corollary}
\begin{proof}
From the established equivalence of categories, we conclude  
$\mathfrak{inf}(M,\cD)=\mathfrak{inf}(\hat{M},\hat{\cD})$. Moreover, 
$G(3)$ is the maximal transitive symmetry algebra for distributions of type $M_{12}^{\rm IV}$.
We apply this to the flat distributions to conclude
that $G(3)=\mathfrak{inf}(\cD_\text{SHC})=\mathfrak{inf}(\hat{\cD}_\text{SHC})=pr(\fm)$, where 
the first equality follows from Theorem \ref{T:G3-SHC-sym} and the last equality from the fact that
Tanaka--Weisfeiler prolongations reproduce the symmetries of flat models. 
 \end{proof}

In other words, the geometry of type $M_{12}^{\rm IV}$ is given by a naked distribution
(that is without reduction of the structure group) on a supermanifold $\hat{M}$ of dimension $(6|6)$ 
with the given symbol. This explains and enhances the twistor correspondence from the 
introduction, see the right arrow in Figure \ref{F:G3-twistor}. 
 
We illustrate this correspondence using the equation \eqref{FGHK}, with the SHC-type 
constraints of Proposition \ref{FGHK-constraints} in force.
The distribution $\cD$ is given by \eqref{FGHK-deformeven}-\eqref{FGHK-deformodd} and the corresponding rank $(2|2)$ distribution $\hat{\cD}$
 on $\hat{M}$ of type $M_{12}^{\rm IV}$ is 
 $$
\hat{\cD} = \langle D_x+\lambda \partial_{u_{xx}},\ \partial_\lambda \,|\,
D_\nu+ (D_xK+\lambda\partial_{u_{xx}}K) \partial_{u_{x\nu}},\ 
D_\tau- (D_xK+\lambda\partial_{u_{xx}}K) \partial_{u_{x\tau}}\rangle.
 $$
Here $\lambda$ is the coordinate on the (projective) line bundle $\pi:\hat{M}\to M$. 
 
In particular, for the super-extension \eqref{sMonge} of the Monge equation, the
lifted distribution $\hat{\cD}$ has the same symmetries as $\cD$ (given via $f=f(u_{xx})$), 
and thus is subject to the submaximal result of Theorem \ref{super-sub-max}.

\section{Vanishing of the groups \texorpdfstring{$H^{d,2}(\fm_{\bar 1},\fg)_{\bar 0}$}{} for \texorpdfstring{$d\geq 3$}{}}\label{appendixA}

\begin{proposition}
\label{thm:techI}
The group $H^{d,2}(\fm_{\bar 1},\fg)_{\bar 0}=0$ for all $d\geq 5$.
\end{proposition}
\begin{proof}
The claim is immediate  by degree reasons if $d\geq 8$. We now consider the remaining degrees $d=5,6,7$ separately and
show that $Z^{d,2}(\fm_{\bar 1},\fg)_{\bar 0}=0$ in all of these cases.
\vskip0.05cm\par\noindent
{\underline{Case $d=7$}}
The components of an even $2$-cocycle $\varphi\in Z^{7,2}(\fm_{\bar 1},\fg)_{\bar 0}$ are given by
\begin{align*}
\varphi_{\alpha\beta}{}^{a}&:\Lambda^2 (\fg_{-2})_{\bar 1}\longrightarrow \fg_{3}\;,\\
\varphi_{a\alpha}{}^{\beta}&:\fg_{-3}\otimes (\fg_{-2})_{\bar 1}\longrightarrow (\fg_{2})_{\bar 1}\;,
\end{align*}
and we then note that
\begin{align*}
\partial\varphi|_{(\fg_{-1})_{\bar 0}\otimes(\fg_{-2})_{\bar 1}\otimes(\fg_{-2})_{\bar 1}}
&=0\Longrightarrow \varphi_{\a \b a}=0\;,\\
\partial\varphi|_{\fg_{-3}\otimes(\fg_{-2})_{\bar 1}\otimes(\fg_{-2})_{\bar 0}}
&=0\Longrightarrow 
\varphi_{a\alpha}{}^{\beta}=0\;.
\end{align*}
Hence $Z^{7,2}(\fm_{\bar 1},\fg)_{\bar 0}=0$.
\vskip0.05cm\par\noindent
{\underline{Case $d=6$}} The components of an even $2$-cocycle $\varphi\in Z^{6,2}(\fm_{\bar 1},\fg)_{\bar 0}$ are given by
\begin{align*}
\varphi_{\a b\b}{}^c&:(\fg_{-2})_{\bar 1}\otimes(\fg_{-1})_{\bar 1}\to\fg_{3}\;,\\
\varphi_{a b\beta}{}^{\a}&:\fg_{-3}\otimes (\fg_{-1})_{\bar 1}\longrightarrow (\fg_{2})_{\bar 1}\;,\\
\varphi_{\1\alpha}{}^{\beta}&:(\fg_{-2})_{\bar 0}\otimes (\fg_{-2})_{\bar 1}\longrightarrow (\fg_{2})_{\bar 1}\;,\\
\varphi_{\alpha\beta}{}^\1&:\Lambda^2 (\fg_{-2})_{\bar 1}\longrightarrow (\fg_{2})_{\bar 0}\;,\\
\varphi_{a\alpha}{}^{b\beta}&:\fg_{-3}\otimes (\fg_{-2})_{\bar 1}\longrightarrow (\fg_{1})_{\bar 1}\;,
\end{align*}
and they all vanish since
\begin{align*}
\partial\varphi|_{\fg_{-3}\otimes\fg_{-3}\otimes(\fg_{-1})_{\bar 1}}
&=0\Longrightarrow \varphi_{bc\b}{}^{\a}\bep_{a\a}-\varphi_{ac\b}{}^{\a}\bep_{b\a}=0\Longrightarrow \varphi_{a b\beta}{}^{\a}=0\;,\\
\partial\varphi|_{(\fg_{-2})_{\bar 1}\otimes(\fg_{-2})_{\bar 1}\otimes(\fg_{-2})_{\bar 1}}
&=0\Longrightarrow \varphi_{\a\b}{}^\1 \bep_{\gamma}+\varphi_{\gamma\a}{}^\1 \bep_\b+\varphi_{\b\gamma}{}^\1 \bep_\a=0\Longrightarrow \varphi_{\a\b}{}^\1=0\;,\\
\partial\varphi|_{\fg_{-3}\otimes(\fg_{-1})_{\bar 0}\otimes(\fg_{-1})_{\bar 1}}
&=0\Longrightarrow \omega_{cd}\varphi_{a\a}{}^{b\b}\bep_{b\b}=0\Longrightarrow \varphi_{a\alpha}{}^{b\beta}=0\;,\\
\partial\varphi|_{\fg_{-3}\otimes(\fg_{-2})_{\bar 0}\otimes(\fg_{-2})_{\bar 1}}
&=0\Longrightarrow  \varphi_{\1\a}{}^\b \bep_{c\b}=0\Longrightarrow \varphi_{\1\a}{}^\b=0
\;,\\
\partial\varphi|_{\fg_{-3}\otimes(\fg_{-2})_{\bar 1}\otimes(\fg_{-1})_{\bar 1}}
&=0\Longrightarrow \omega_{dc}\varphi_{\a b\b}{}^c=0\Longrightarrow \varphi_{\a b\b}{}^c=0\;.
\end{align*}
Hence $Z^{6,2}(\fm_{\bar 1},\fg)_{\bar 0}=0$.

\vskip0.05cm\par\noindent
{\underline{Case $d=5$}} We still work in cocycle components, writing
\begin{align*}
\varphi_{a\a b\b}{}^{c}&:\Lambda^2(\fg_{-1})_{\bar 1}\to\fg_{3}\;,\\
\varphi_{a\a\b}{}^{\1}&:(\fg_{-1})_{\bar 1}\otimes (\fg_{-2})_{\bar 1}\longrightarrow (\fg_{2})_{\bar 0}\;,\\
\varphi_{\1 a\a}{}^{\b}&:(\fg_{-2})_{\bar 0}\otimes (\fg_{-1})_{\bar 1}\longrightarrow (\fg_{2})_{\bar 1}\;,\\
\varphi_{a}{}^{\beta}{}_\alpha&:(\fg_{-1})_{\bar 0}\otimes(\fg_{-2})_{\bar 1} \longrightarrow (\fg_{2})_{\bar 1}\;,\\
\varphi_{a}{}^{c\gamma}{}_{b\b}&:\fg_{-3}\otimes (\fg_{-1})_{\bar 1}\longrightarrow (\fg_{1})_{\bar 1}\;,\\
\varphi_{\a\b}{}^{a}&:\Lambda^2 (\fg_{-2})_{\bar 1}\longrightarrow (\fg_{1})_{\bar 0}\;,\\
\varphi_{\1\a}{}^{b\b}&:(\fg_{-2})_{\bar 0}\otimes (\fg_{-2})_{\bar 1}\longrightarrow (\fg_{1})_{\bar 1}\;,\\
\varphi_{\a a}{}^{\b}&: (\fg_{-2})_{\bar 1}\otimes\fg_{-3}\longrightarrow (\fg_{0})_{\bar 1}\;,
\end{align*}
and show that they all vanish. We depart with the following chain of equations:
\begin{align*}
\partial\varphi|_{(\fg_{-2})_{\bar 1}\otimes(\fg_{-2})_{\bar 1}\otimes(\fg_{-2})_{\bar 1}}
&=0\Longrightarrow \varphi_{\a\b}{}^a\bep_{a\gamma}+\varphi_{\gamma\a}{}^a\bep_{a\b}+\varphi_{\b\gamma}{}^a\bep_{a\a}=0\\
&\;\;\;\;\;\;\,\Longrightarrow \varphi_{\a\b}{}^a=0\;,\\
\partial\varphi|_{(\fg_{-2})_{\bar 1}\otimes(\fg_{-2})_{\bar 1}\otimes(\fg_{-1})_{\bar 0}}
&=0\Longrightarrow \varphi_{a}{}^{\gamma}{}_\a (\bep_\gamma\otimes\bep_\b)+\varphi_{a}{}^{\gamma}{}_\b (\bep_\gamma\otimes\bep_\a)=0\\
&\;\;\;\;\;\;\,\Longrightarrow \varphi_{a}{}^{\gamma}{}_\a=0\;,\\
\partial\varphi|_{\fg_{-3}\otimes(\fg_{-2})_{\bar 1}\otimes(\fg_{-1})_{\bar 0}}
&=0\Longrightarrow \varphi_{\a a}{}^{\b}\bep_{c\b}=0\Longrightarrow \varphi_{\a a}{}^{\b}=0\;,\\
\partial\varphi|_{\fg_{-3}\otimes(\fg_{-1})_{\bar 1}\otimes(\fg_{-1})_{\bar 0}}
&=0\Longrightarrow \omega_{dc}\varphi_{a}{}^{c\gamma}{}_{b\b}\bep_{\gamma}=0\Longrightarrow \varphi_{a}{}^{c\gamma}{}_{b\b}=0\;,\\
\partial\varphi|_{\fg_{-3}\otimes(\fg_{-1})_{\bar 1}\otimes(\fg_{-1})_{\bar 1}}
&=0\Longrightarrow \varphi_{a\a b\b}{}^{c}(\be_c\otimes\be_d)=0\Longrightarrow \varphi_{a\a b\b}{}^{c}=0\;.
\end{align*}
It remains to deal with the three components with the index $\1$:
\begin{align*}
\partial\varphi|_{\fg_{-3}\otimes(\fg_{-2})_{\bar 0}\otimes(\fg_{-2})_{\bar 1}}
&=0\Longrightarrow \omega_{cb}\varphi_{\1\a}{}^{b\b}\bep_{\b}=0\Longrightarrow \varphi_{\1\a}{}^{b\b}=0\;,\\
\partial\varphi|_{\fg_{-3}\otimes(\fg_{-2})_{\bar 0}\otimes(\fg_{-1})_{\bar 1}}
&=0\Longrightarrow \varphi_{\1 a\a}{}^{\b}\bep_{c\b}=0\Longrightarrow \varphi_{\1 a\a}{}^{\b}=0\;,\\
\partial\varphi|_{\fg_{-3}\otimes(\fg_{-2})_{\bar 1}\otimes(\fg_{-1})_{\bar 1}}
&=0\Longrightarrow \varphi_{a\a\b}{}^{\1}\be_c=0\Longrightarrow \varphi_{a\a\b}{}^{\1}=0\;.
\end{align*}
Hence $Z^{5,2}(\fm_{\bar 1},\fg)_{\bar 0}=0$.
\end{proof}
\begin{proposition}
\label{thm:cohomology32}
The group $H^{3,2}(\fm_{\bar 1},\fg)_{\bar 0}=0$.
\end{proposition}
\begin{proof}
The components of a cocyle $\varphi\in Z^{3,2}(\fm_{\bar 1},\fg)_{\bar 0}$ are given by
\begin{align*}
\varphi_{a\a b\b}{}^{c}&:\Lambda^2(\fg_{-1})_{\bar 1}\to(\fg_{1})_{\bar 0}\;,\\
\varphi_{a\a b}{}^{c\gamma}&:(\fg_{-1})_{\bar 1}\otimes(\fg_{-1})_{\bar 0}\to(\fg_{1})_{\bar 1}\;,\\
\varphi_{a\a \b}{}^{Z}+\varphi_{a\a \b}{}^{\fsl(2)}+\varphi_{a\a \b}{}^{\fsp(2)}&:(\fg_{-1})_{\bar 1}\otimes (\fg_{-2})_{\bar 1}\longrightarrow (\fg_{0})_{\bar 0}=\mathbb {C}Z\oplus \fsl(2)\oplus\fsp(2)\;,\\
\varphi_{\1 a\a}{}^{\b}&:(\fg_{-2})_{\bar 0}\otimes (\fg_{-1})_{\bar 1}\longrightarrow (\fg_{0})_{\bar 1}\;,\\
\varphi_{\a a}{}^{\b}&:(\fg_{-2})_{\bar 1}\otimes (\fg_{-1})_{\bar 0}\longrightarrow (\fg_{0})_{\bar 1}\;,\\
\varphi_{a}{}^{c\gamma}{}_{b\b}&:\fg_{-3}\otimes (\fg_{-1})_{\bar 1}\longrightarrow (\fg_{-1})_{\bar 1}\;,\\
\varphi_{\a\b}{}^{a}&:\Lambda^2 (\fg_{-2})_{\bar 1}\longrightarrow (\fg_{-1})_{\bar 0}\;,\\
\varphi_{\1\a}{}^{b\b}&:(\fg_{-2})_{\bar 0}\otimes (\fg_{-2})_{\bar 1}\longrightarrow (\fg_{-1})_{\bar 1}\;,\\
\varphi_{a}{}^{\gamma}{}_{\b}&: \fg_{-3}\otimes(\fg_{-2})_{\bar 1}\longrightarrow (\fg_{-2})_{\bar 1}\;,
\end{align*}
and there is a non-trivial space
$(\fg_{2})_{\bar 1}\otimes (\fg_{-1})_{\bar 1}^*+(\fg_{1})_{\bar 1}\otimes (\fg_{-2})_{\bar 1}^*$
of $1$-cochains, which we may use to arrange for $\varphi_{\1 a\a}{}^{\b}=0$ and $\varphi_{\a a}{}^{\b}=0$.

We depart with
\begin{align*}
\partial\varphi|_{(\fg_{-1})_{\bar 0}\otimes(\fg_{-1})_{\bar 0}\otimes(\fg_{-2})_{\bar 1}}
&=0\Longrightarrow \varphi_{\1\a}{}^{b\b}=0\;,\\
\partial\varphi|_{(\fg_{-1})_{\bar 1}\otimes(\fg_{-1})_{\bar 1}\otimes(\fg_{-2})_{\bar 0}}
&=0\Longrightarrow \varphi_{a\a b\b}{}^{c}=0\;,\\
\partial\varphi|_{(\fg_{-1})_{\bar 0}\otimes(\fg_{-2})_{\bar 0}\otimes(\fg_{-2})_{\bar 1}}
&=0\Longrightarrow \varphi_{a}{}^{\gamma}{}_{\b}=0\;,\\
\partial\varphi|_{(\fg_{-1})_{\bar 0}\otimes(\fg_{-2})_{\bar 1}\otimes(\fg_{-2})_{\bar 1}}
&=0\Longrightarrow \omega_{ab}\varphi_{\a\b}{}^{a}\1=0\Longrightarrow \varphi_{\a\b}{}^{a}=0\;,\\
\partial\varphi|_{(\fg_{-1})_{\bar 1}\otimes(\fg_{-2})_{\bar 0}\otimes(\fg_{-2})_{\bar 1}}
&=0\Longrightarrow [\varphi_{a\a \b}{}^{Z},\1]=0\Longrightarrow \varphi_{a\a \b}{}^{Z}=0\;,
\end{align*}
and continue with
\begin{align*}
\partial\varphi|_{(\fg_{-1})_{\bar 1}\otimes(\fg_{-2})_{\bar 1}\otimes(\fg_{-2})_{\bar 1}}
&=0\Longrightarrow [\varphi_{a\a \b}{}^{\fsp(2)},\bep_\gamma]+[\varphi_{a\a \gamma}{}^{\fsp(2)},\bep_\b]=0\\
&\;\;\;\;\;\;\,\Longrightarrow (\varphi_{a\a \b}{}^{\delta}{}_\gamma+\varphi_{a\a \gamma}{}^{\delta}{}_\b)\bep_\delta=0\;,\\
&\;\;\;\;\;\;\,\Longrightarrow \varphi_{a\a \b\delta\gamma}+\varphi_{a\a \gamma\delta\b}=0\;,\\
\end{align*}
which implies
\begin{align*}
\varphi_{a\a \b\delta\gamma}&=-\varphi_{a\a \gamma\delta\b}=-\varphi_{a\a \gamma\b\delta}=\varphi_{a\a \delta\b\gamma}\\
&=\varphi_{a\a \delta\gamma\b}=-\varphi_{a\a \b\gamma\delta}=-\varphi_{a\a \b\delta\gamma}
\end{align*}
and $\varphi_{a\a \b}{}^{\fsp(2)}=0$. It remains to deal with the components $\varphi_{a\a b}{}^{c\gamma}$, $\varphi_{a\a \b}{}^{\fsl(2)}$ and $\varphi_{a}{}^{c\gamma}{}_{b\b}$.

We have:\footnote{We will use the symbol $\propto$ to denote ``proportional to'' with nontrivial constant of proportionality.}
\begin{align}
\partial\varphi|_{(\fg_{-1})_{\bar 1}\otimes(\fg_{-1})_{\bar 1}\otimes\fg_{-3}}
&\notag=0\Longrightarrow [\varphi_{d}{}^{c\gamma}{}_{b\b}\bep_{c\gamma},\bep_{a\a}]+[\varphi_{d}{}^{c\gamma}{}_{a\a}\bep_{c\gamma},\bep_{b\b}]=0\;,\\
&\notag\;\;\;\;\;\;\,\Longrightarrow \omega_{ac}\omega_{\a\gamma}\varphi_{d}{}^{c\gamma}{}_{b\b}+\omega_{bc}\omega_{\b\gamma}\varphi_{d}{}^{c\gamma}{}_{a\a}=0\;,\\
&\;\;\;\;\;\;\,\Longrightarrow \varphi_{d a\a b\b}=-\varphi_{db\b a\a}\;,\\
\partial\varphi|_{(\fg_{-1})_{\bar 1}\otimes(\fg_{-2})_{\bar 1}\otimes\fg_{-3}}
&\notag=0\Longrightarrow [\varphi_{c}{}^{d\delta}{}_{a\a}\bep_{d\delta},\bep_{\b}]+[\varphi_{a\a \b}{}^{\fsl(2)},\be_c]=0\;,\\
&\notag\;\;\;\;\;\;\,\Longrightarrow [\varphi_{a\a \b}{}^{\fsl(2)},\be_c]=\varphi_{c}{}^{d}{}_{\b a\a}\be_d\;,\\
&\label{eq:propI}\;\;\;\;\;\;\,\Longrightarrow \varphi_{c d\b a\a}=\varphi_{a\a \b d c}\;,\\
\partial\varphi|_{(\fg_{-1})_{\bar 0}\otimes(\fg_{-1})_{\bar 1}\otimes\fg_{-3}}
&=0\notag\Longrightarrow [\be_a, \varphi_{c}{}^{d\delta}{}_{b\b}\bep_{d\delta}]+[\be_c,\varphi_{b\b a}{}^{d\delta}\bep_{d\delta}]=0\;\\
&\label{eq:propII}\;\;\;\;\;\;\,\Longrightarrow \varphi_{b\b a c\delta}\propto\varphi_{c a\delta b\b}\;,
\end{align}
which  imply, in turn, strong symmetry properties on the component $\varphi_{a}{}^{c\gamma}{}_{b\b}$:
\begin{align*}
\varphi_{ac\gamma b\b}&=
-\varphi_{ac\b b\gamma}=\varphi_{ab\gamma c\beta}=\varphi_{ca\gamma b\beta}\;.
\end{align*}
In other words $\varphi_{a}{}^{c\gamma}{}_{b\b}$ is totally symmetric in the latin indices and skew in the greek ones;
the components $\varphi_{a\a \b}{}^{d}{}_{c}$ and $\varphi_{a\a b}{}^{c\gamma}$ satisfy the very same property due to \eqref{eq:propI} and \eqref{eq:propII}.

We shall write
$$
\varphi_{a\a b c\gamma}=\omega_{\a\gamma}\Psi_{abc}\;,
$$
for some totally symmetric tensor $\Psi\in S^3(\mathbb{C}^2)^*$. The component of $\partial\varphi|_{(\fg_{-1})_{\bar 0}\otimes(\fg_{-1})_{\bar 1}\otimes(\fg_{-1})_{\bar 1}}=0$ that takes values in $\fsp(2)\subset (\fg_{0})_{\bar 0}$ tells us that
\begin{align*}
&\varphi_{c\gamma a b}{}^{\delta}(\bep_{\b}\otimes\bep_{\delta}+\bep_{\delta}\otimes\bep_{\b})+\varphi_{b\b a c}{}^{\delta}
(\bep_{\gamma}\otimes\bep_{\delta}+\bep_{\delta}\otimes\bep_{\gamma})=0\\
&\Longrightarrow
\varphi_{c\gamma a b}{}^{\epsilon}\delta^\a_\beta+\varphi_{c\gamma a b}{}^{\a}\delta^\epsilon_\beta+
\varphi_{b\b a c}{}^{\epsilon}\delta^\a_\gamma+\varphi_{b\b a c}{}^{\alpha}\delta^\epsilon_\gamma=0
\end{align*}
and, summing over $\a=\b$ and multiplying by $\omega_{\theta\epsilon}$, we arrive at
\begin{align*}
2\varphi_{c\gamma a b\theta}+\varphi_{c\gamma a b\theta}+
\varphi_{b\gamma a c\theta}-\omega_{\gamma\theta}\varphi_{b\a a c}{}^{\alpha}=0\Longrightarrow 6\Psi_{abc}=0\;.
\end{align*}
Hence $\varphi_{a\a b}{}^{c\gamma}$ identically vanishes, as well as $\varphi_{a\a \b}{}^{\fsl(2)}$ and $\varphi_{a}{}^{c\gamma}{}_{b\b}$.
\end{proof}
\begin{proposition}
\label{thm:cohomology42}
The group $H^{4,2}(\fm_{\bar 1},\fg)_{\bar 0}=0$ and $H^{4,2}(\fm,\fg)_{\bar 0}$ does not contain any $(\fg_0)_{\bar 0}$-submodule isomorphic to $S^2\CC^2\boxtimes S^2\CC^2$.
\end{proposition}
\begin{proof} The proof is divided in three main steps. We show $H^{4,2}(\fm_{\bar 1},\fg)_{\bar 0}=0$ in the first two steps and conclude with the last claim of the proposition.
\vskip0.05cm\par\noindent
{\underline{\it First step}}
The components of a cocyle $\varphi\in Z^{4,2}(\fm_{\bar 1},\fg)_{\bar 0}$ are given by
\begin{equation}
\label{eq:componentsrestriction}
\begin{aligned}
\varphi_{a\a b\b}{}^{\1}&:\Lambda^2(\fg_{-1})_{\bar 1}\to(\fg_{2})_{\bar 0}\;,\\
\varphi_{a\a}{}^{\b}{}_b&:(\fg_{-1})_{\bar 1}\otimes(\fg_{-1})_{\bar 0}\to(\fg_{2})_{\bar 1}\;,\\
\varphi_{\a b\b}{}^{c}&:(\fg_{-2})_{\bar 1}\otimes(\fg_{-1})_{\bar 1} \longrightarrow (\fg_{1})_{\bar 0}\;,\\
\varphi_{\1 a\a}{}^{b\b}&:(\fg_{-2})_{\bar 0}\otimes (\fg_{-1})_{\bar 1}\longrightarrow (\fg_{1})_{\bar 1}\;,\\
\varphi_{\a a}{}^{b\b}&:(\fg_{-2})_{\bar 1}\otimes (\fg_{-1})_{\bar 0}\longrightarrow (\fg_{1})_{\bar 1}\;,\\
\varphi_{a b\b}{}^{\gamma}&:\fg_{-3}\otimes (\fg_{-1})_{\bar 1}\longrightarrow (\fg_{0})_{\bar 1}\;,\\
\varphi_{\a \b}{}^{Z}+\varphi_{\a \b}{}^{\fsl(2)}+\varphi_{\a \b}{}^{\fsp(2)}&:\Lambda^2 (\fg_{-2})_{\bar 1}\longrightarrow (\fg_{0})_{\bar 0}=\mathbb {C}Z\oplus \fsl(2)\oplus\fsp(2)\;,\\
\varphi_{\1 \a}{}^{\b}&:(\fg_{-2})_{\bar 0}\otimes (\fg_{-2})_{\bar 1}\longrightarrow (\fg_{0})_{\bar 1}\;,\\
\varphi_{a\a}{}^{b\b}&:\fg_{-3}\otimes(\fg_{-2})_{\bar 1} \longrightarrow (\fg_{-1})_{\bar 1}\;,
\end{aligned}
\end{equation}
and we may arrange for $\varphi_{\1 \a}{}^{\b}=0$ using the space of $1$-cochains $(\fg_{2})_{\bar 1}\otimes (\fg_{-2})_{\bar 1}^*$.
We then immediately note that
\begin{align}
\label{eq:step1}
\partial\varphi|_{(\fg_{-2})_{\bar 0}\otimes(\fg_{-2})_{\bar 1}\otimes(\fg_{-2})_{\bar 1}}
&=0\Longrightarrow \varphi_{\a \b}{}^{Z}=0\;.
\end{align}
We now turn to study the identity
\begin{equation}
\label{eq:step2}
\partial\varphi|_{(\fg_{-1})_{\bar 1}\otimes(\fg_{-1})_{\bar 1}\otimes(\fg_{-2})_{\bar 0}}=0\;.
\end{equation}
The component taking values in $\fsp(2)\subset (\fg_{0})_{\bar 0}$ says
\begin{align*}
&\omega_{ca}\varphi_{\1 b\b}{}^{c\gamma}(\bep_{\gamma}\otimes\bep_{\a}+\bep_{\a}\otimes\bep_{\gamma})
+\omega_{cb}\varphi_{\1 a\a}{}^{c\gamma}(\bep_{\gamma}\otimes\bep_{\b}+\bep_{\b}\otimes\bep_{\gamma})=0\\
&\Longrightarrow
\varphi_{\1 b\b a}{}^{\epsilon}\delta^{\theta}_\a
+\varphi_{\1 b\b a}{}^{\theta}\delta^{\epsilon}_\a
+\varphi_{\1 a\a b}{}^{\epsilon}\delta^{\theta}_\b
+\varphi_{\1 a\a b}{}^{\theta}\delta^{\epsilon}_{\b}
=0\\
&\Longrightarrow
3\varphi_{\1 b\b a \epsilon}
+\varphi_{\1 a\beta b\epsilon}
+\omega_{\epsilon\b}\varphi_{\1 a\a b}{}^{\a}
=0\;.
\end{align*}
Multiplying by $\omega^{\beta\epsilon}$ and summing over $\beta$ and $\epsilon$, we arrive at
$\varphi_{\1 b\b a \epsilon}=-\varphi_{\1 a\b b \epsilon}$.
An analogous argument for the component taking values in $\fsl(2)\subset (\fg_{0})_{\bar 0}$ yields
$\varphi_{\1 b\b a \epsilon}=-\varphi_{\1 b\epsilon a \b}$
while the component in $\mathbb CZ\subset (\fg_{0})_{\bar 0}$ says that $\varphi_{a\a b\b}{}^{\1}$ is proportional to $\varphi_{\1a\a b\b}$.
In particular $\varphi_{\1 b\b a \epsilon}$ and $\varphi_{a\a b\b}{}^{\1}$ are both skewsymmetric, in the latin and greek indices separately.

Now
\begin{align}
\label{eq:step3}
\partial\varphi|_{(\fg_{-1})_{\bar 1}\otimes(\fg_{-2})_{\bar 0}\otimes(\fg_{-2})_{\bar 1}}
&=0\Longrightarrow 
\varphi_{\b a\a c} \propto \varphi_{\1 a\a c\beta}\;,
\end{align}
so that $\varphi_{\b a\a c}$ is skewsymmetric, in the latin and greek indices. Furthermore
\begin{align}
\partial\varphi|_{(\fg_{-1})_{\bar 1}\otimes(\fg_{-1})_{\bar 1}\otimes(\fg_{-3})_{\bar 0}}
=0&\notag\Longrightarrow [\varphi_{c b\b}{}^{\gamma}\bep_\gamma,\bep_{a\a}]+[\varphi_{c a\a}{}^{\gamma}\bep_\gamma,\bep_{b\b}]
+[\varphi_{a\a b\b}{}^{\1}\1,\be_c]=0\\
&\label{eq:step4}\Longrightarrow \varphi_{c b\b\a}\delta_{a}^d+\varphi_{c a\a\b}\delta_{b}^d\propto\varphi_{a\a b\b}{}^{\1}\delta_{c}^d\;,
\end{align}
and taking $a=b$ implies that $\varphi_{c b\b\a}$ is skewsymmetric in the greek indices. If instead we multiply \eqref{eq:step4} by $\delta^a_d$ and sum over $a$ and $d$ we arrive at
\begin{align*}
2\varphi_{cb\b\a}+\varphi_{cb\a\b}\propto\varphi_{c\a b\b}{}^{\1}\Longrightarrow \varphi_{cb\b\a}\propto\varphi_{c\a b\b}{}^{\1}
\end{align*}
so that $\varphi_{c b\b\a}$ is skewsymmetric in the latin indices as well. Finally, we consider
\begin{align}
\label{eq:step5}
\partial\varphi|_{(\fg_{-1})_{\bar 0}\otimes(\fg_{-1})_{\bar 0}\otimes(\fg_{-2})_{\bar 1}}
&=0\Longrightarrow \varphi_{\alpha b a\beta}=\varphi_{\alpha a b\beta}
\end{align}
and also note
\begin{align}
\label{eq:step6}
\partial\varphi|_{(\fg_{-1})_{\bar 0}\otimes(\fg_{-2})_{\bar 0}\otimes(\fg_{-2})_{\bar 1}}
&=0\Longrightarrow \varphi_{a\a b\b}\propto\varphi_{\a a b\b}\;,
\end{align}
so that $\varphi_{a\a b\b}$ is symmetric in the latin indices.
\vskip0.05cm\par\noindent
{\underline{\it Second step}}
We already obtained $\varphi_{\1 \a}{}^{\b}=\varphi_{\a \b}{}^{Z}=0$ and
we now turn to prove that the other components in \eqref{eq:componentsrestriction} vanish. We start with
\begin{align}
\partial\varphi|_{(\fg_{-1})_{\bar 1}\otimes(\fg_{-2})_{\bar 1}\otimes\fg_{-3}}
=0&\notag\Longrightarrow \varphi_{c\b a\a}+\varphi_{c a\a\b}\\
&\notag\Longrightarrow \varphi_{c\b a\a}+\varphi_{a\b c\a}=0\\
&\label{eq:step7}\Longrightarrow \varphi_{c\b a\a}=0
\;,
\end{align}
where the next-to-last identity follows from symmetrization in $a$ and $c$. Hence
$\varphi_{\a a b\b}=0$
as well. The identity
\begin{equation}
\label{eq:step8}
\partial\varphi|_{(\fg_{-2})_{\bar 1}\otimes(\fg_{-2})_{\bar 1}\otimes(\fg_{-3})_{\bar 0}}=0
\end{equation}
now readily implies $\varphi_{\a \b}{}^{\fsl(2)}=0$.

We continue with
\begin{align}
\partial\varphi|_{(\fg_{-1})_{\bar 0}\otimes(\fg_{-1})_{\bar 1}\otimes(\fg_{-3})_{\bar 0}}
=0&\notag\Longrightarrow \varphi_{c b\b}{}^{\gamma}\bep_{a\gamma}\propto\varphi_{b\b}{}^{\gamma}{}_a \bep_{c\gamma}\\
&\label{eq:step9}\Longrightarrow \varphi_{b\b}{}^{\gamma}{}_a \bep_{b\gamma}=0
\;,
\end{align}
which yields $\varphi_{a\a}{}^{\b}{}_b=0$ and $\varphi_{a b\b}{}^{\gamma}=0$. Finally
\begin{equation}
\label{eq:step10}
\partial\varphi|_{(\fg_{-1})_{\bar 1}\otimes(\fg_{-2})_{\bar 0}\otimes(\fg_{-3})_{\bar 0}}=0\Longrightarrow \varphi_{\1 a\a}{}^{b\b}=0\;,
\end{equation}
whence $\varphi_{a\a b\b}{}^{\1}=\varphi_{\a b\b}{}^{c}=0$ too and
\begin{equation}
\label{eq:step11}
\partial\varphi|_{(\fg_{-1})_{\bar 0}\otimes(\fg_{-1})_{\bar 1}\otimes(\fg_{-2})_{\bar 1}}=0\Longrightarrow \varphi_{\a \b}{}^{\fsp(2)}=0\;.
\end{equation}
This concludes the proof of $H^{4,2}(\fm_{\bar 1},\fg)_{\bar 0}=0$.
\vskip0.05cm\par\noindent
{\underline{\it Last step}}
First of all, we note that $C^{4,2}(\fm,\fg)_{\bar 0}\cong C^{4,2}(\fm_{\bar 0},\fg)_{\bar 0}\bigoplus C^{4,2}(\fm_{\bar 1},\fg)_{\bar 0}$
as $(\fg_0)_{\bar 0}$-modules. The space $C^{4,2}(\fm_{\bar 0},\fg)_{\bar 0}$ has a unique $(\fg_0)_{\bar 0}$-irreducible submodule
\begin{equation*}
\label{eq:componentfixed}
\Big\{\psi_{ab}{}^{\beta}{}_{\alpha}:(\fg_{-1})_{\bar 0}\otimes (\fg_{-3})_{\bar 0}\to\fsp(2)\mid \psi_{ab}{}^{\beta}{}_{\alpha}=\psi_{(ab)}{}^{\beta}{}_{\alpha}\Big\}
\end{equation*}
of type $S^2\CC^2\boxtimes S^2\CC^2$.
On the other hand $C^{4,2}(\fm_{\bar 1},\fg)_{\bar 0}$ has $8$ such modules, formed by the maps in \eqref{eq:componentsrestriction} with $2$ latin and $2$ greek indices (including the maps $\varphi_{a\a b\b}{}^{\1}$, $\varphi_{\1 a\a}{}^{b\b}$ and $\varphi_{\a \b}{}^{\fsl(2)}$),  separately symmetric. Hence the $(\fg_0)_{\bar 0}$-isotypic component of type
$S^2\CC^2\boxtimes S^2\CC^2$ in $C^{4,2}(\fm,\fg)_{\bar 0}$ consists of the direct sum of $9$ irreducible submodules.

By the very same line of arguments running from \eqref{eq:step2} to \eqref{eq:step8}, 
one sees that any cocyle in the above $(\fg_0)_{\bar 0}$-isotypic component has trivial components, 
with the exception of $\varphi_{a\a}{}^{\b}{}_b$ and $\psi_{ab}{}^{\beta}{}_{\alpha}$. 
Identity \eqref{eq:step9} is now replaced by
 \begin{align}
\partial\varphi|_{(\fg_{-1})_{\bar 0}\otimes(\fg_{-1})_{\bar 1}\otimes(\fg_{-3})_{\bar 0}}
=0&\notag\Longrightarrow \psi_{ac}{}^{\gamma}{}_{\beta}\bep_{b\gamma}\propto
\varphi_{b\b}{}^{\gamma}{}_a\bep_{c\gamma}\\
&\notag\Longrightarrow \psi_{ac}{}^{\delta}{}_{\beta}\delta^{d}_b\propto
\varphi_{b\b}{}^{\delta}{}_a\delta^{d}_c\\
&\notag\Longrightarrow \psi_{ac}{}^{\delta}{}_{\beta}=\varphi_{b\b}{}^{\delta}{}_a=0\;,
 \end{align}
proving that $Z^{4,2}(\fm,\fg)_{\bar 0}$ does not contain any $(\fg_0)_{\bar 0}$-submodule isomorphic to $S^2\CC^2\boxtimes S^2\CC^2$. The same claim clearly holds for $H^{4,2}(\fm,\fg)_{\bar 0}$.
\end{proof}

 \section{Internal symmetries of the SHC equation}\label{S:SHC-sym}

The SHC equation written as the system of $(2|2)$ differential equations \eqref{SHC-ODE}
encodes as the following superdistribution $\cD$ of rank $(2|4)$ 
on $M=\C^{5|6}(x,u,u_x,u_{xx},z|\tau,\nu,u_\tau,u_\nu,u_{x\tau},u_{x\nu})$
with pure degree components
 \begin{align*}
\cD_{\bar0}=\langle D_x &= \p_x+u_x\p_u+u_{xx}\p_{u_x}+(\tfrac12u_{xx}^2+u_{x\nu}u_{x\tau})\p_z+u_{x\tau}\p_{u_\tau}
+u_{x\nu}\p_{u_\nu}, \p_{u_{xx}}\rangle,\\
\cD_{\bar1}=\langle D_\tau &= \p_\tau+u_\tau\p_u+u_{x\tau}\p_{u_x}+u_{xx}u_{x\tau}\p_z+u_{xx}\p_{u_\nu},\ \p_{u_{x\tau}},\\
D_\nu &= \p_{\nu}+u_{\nu}\p_u+u_{x\nu}\p_{u_x}+u_{xx}u_{x\nu}\p_z-u_{xx}\p_{u_\tau},\ \p_{u_{x\nu}}\rangle
 \end{align*}
 \comm{
The superdistribution $\overline\cH$ of rank $(2|4)$ on 
$\overline\cE=\C^{5|6}(x,u,u_x,u_{xx},z|\tau,\nu,u_\tau,u_\nu,u_{x\tau},u_{x\nu})$
that is associated to the SHC equation \eqref{SHC-ODE} is generated by two even supervector fields
 \begin{align*}
D_x&= \p_x+u_x\p_u+u_{xx}\p_{u_x}+(\tfrac12u_{xx}^2+u_{x\nu}u_{x\tau})\p_z+u_{x\tau}\p_{u_\tau}+u_{x\nu}\p_{u_\nu}\;,\qquad\p_{u_{xx}}\;,
\end{align*}
and four odd supervector fields
\begin{align*}
D_\tau &=\p_\tau+u_\tau\p_u+u_{x\tau}\p_{u_x}+u_{xx}u_{x\tau}\p_z+u_{xx}\p_{u_\nu}\;,\qquad\p_{u_{x\tau}}\;,\\
D_\nu &=\p_{\nu}+u_{\nu}\p_u+u_{x\nu}\p_{u_x}+u_{xx}u_{x\nu}\p_z-u_{xx}\p_{u_\tau}\;,\qquad\p_{u_{x\nu}}\;.
\end{align*}
 }
Its annihilator is given by
 \begin{align*}
\opp{Ann}\cD=\Big\langle
& dz-(dx)\left(\tfrac12u_{xx}^2+u_{x\nu}u_{x\tau}\right)-(d\tau)\left(u_{xx}u_{x\tau}\right)-(d\nu)\left(u_{xx}u_{x\nu}\right)\;,\\
& du-(dx)u_x-(d\tau)u_\tau-(d\nu)u_\nu\;,\
du_x-(dx)u_{xx}-(d\tau)u_{x\tau}-(d\nu)u_{x\nu}\;|\;\\
& du_\tau-(dx)u_{x\tau}+(d\nu)u_{xx}\;,\ du_\nu-(dx)u_{x\nu}-(d\tau)u_{xx}\Big\rangle\;.
 \end{align*}
The symmetry superalgebra 
 $
\mathfrak{inf}(M,\cD)= \{ \bX \in \mathfrak{X}(M) : \sigma(\cL_\bX V)=0\ \forall\, V\in\Gamma(\cD),
\sigma\in\Gamma(\opp{Ann}\cD)\}
 $
is isomorphic to $G(3)$. We computed its generators using symbolic packages of Maple.

To write down the explicit expression 
of the even generators $V_i$, $i=1,\ldots,17$, and the odd generators $U_j$, $j=1,\ldots, 14$, 
it is convenient to relabel the odd coordinates as follows:
$\theta_1=\tau$, $\theta_2=\nu$, $\theta_3=u_\tau$, $\theta_4=u_\nu$, $\theta_5=u_{x\tau}$, $\theta_6=u_{x\nu}$. We also set
$\theta_{ij}=\theta_i\theta_j$, $\theta_{ijk}=\theta_i\theta_j\theta_k$, etc.  
Here are the generators of $G(3)$:
 \begin{align*}
V_1 =&\, \p_{u},\
V_2 = \p_{z},\
V_3 = \p_{u_x}+x\p_{u},\
V_4 = \p_{x},
 \\
V_5 =&\,
\p_{u_{xx}}+(\tfrac12x^2-\theta_{12})\p_{u}+x\p_{u_x}+u_x\p_{z}-\theta_{2}\p_{u_{\tau}}+\theta_{1}\p_{u_{\nu}},
 \\
V_6 =&\,
 x\p_{x}+u\p_{u}-u_{xx}\p_{u_{xx}}-z\p_{z}+\theta_{1}\p_{\tau}+\theta_{2}\p_{\nu}
 -\theta_{5}\p_{u_{x\tau}}-\theta_{6}\p_{u_{x\nu}},
 \\
V_7 =&\,
u\p_{u}+u_x\p_{u_x}+u_{xx}\p_{u_{xx}}+2z\p_{z}+
\theta_{3}\p_{u_{\tau}}+\theta_{4}\p_{u_{\nu}}+\theta_{5}\p_{u_{x\tau}}+\theta_{6}\p_{u_{x\nu}},
 \\
V_8 =&\,
 \theta_{1}\p_{\tau}-\theta_{2}\p_{\nu}-\theta_{3}\p_{u_{\tau}}+\theta_{4}\p_{u_{\nu}}
 -\theta_{5}\p_{u_{x\tau}}+\theta_{6}\p_{u_{x\nu}},
 \\
V_9 =&\,
\theta_{2}\p_{\tau}-\theta_{3}\p_{u_{\nu}}-\theta_{5}\p_{u_{x\nu}},\
V_{10} =
\theta_{1}\p_{\nu}-\theta_{4}\p_{u_{\tau}}-\theta_{6}\p_{u_{x\tau}},
 \\
V_{11} =&\,
 u_{xx}\p_{x}+(u_x u_{xx}-z+\theta_{45}-\theta_{36})\p_{u}+(\tfrac12u_{xx}^2-\theta_{56})\p_{u_x}+
 (\tfrac16u_{xx}^3-u_{xx}\theta_{56})\p_{z}\\
 & +\theta_{6}\p_{\tau}-\theta_{5}\p_{\nu}+u_{xx}(\theta_{5}\p_{u_{\tau}}+\theta_{6}\p_{u_{\nu}}),
 \\
V_{12} =&\,
 (\tfrac16x^3-x\theta_{12})\p_{u}+(\tfrac12x^2-\theta_{12})\p_{u_x}+x\p_{u_{xx}}+
(x u_x-u+\theta_{13}+\theta_{24})\p_{z}\\
 & -x\theta_{2}\p_{u_{\tau}}+x\theta_{1}\p_{u_{\nu}}-\theta_{2}\p_{u_{x\tau}}+\theta_{1}\p_{u_{x\nu}},
 \\
V_{13} =&\,
(4u_x-3x u_{xx})\p_{x}+(3x(z-u_xu_{xx})+2u_x^2+3x(\theta_{36}-\theta_{45})-\theta_{34})\p_{u}\\
 & +(3z-\tfrac32xu_{xx}^2+3x\theta_{56})\p_{u_x}-
(u_{xx}^2+4\theta_{56})\p_{u_{xx}}-x u_{xx}(\tfrac12u_{xx}^2-3\theta_{56})\p_{z}\\
 & -(3\theta_{6}x-\theta_{4})\p_{\tau}+(3\theta_{5}x-\theta_{3})\p_{\nu}
-3x u_{xx}(\theta_{5}\p_{u_{\tau}}+\theta_{6}\p_{u_{\nu}})-2u_{xx}(\theta_{5}\p_{u_{x\tau}}+\theta_{6}\p_{u_{x\nu}}),
 \\
V_{14} =&\,
 (\tfrac12x^2+2\theta_{12})\p_{x}+\tfrac32xu\p_{u}+
(\tfrac12xu_x+\tfrac32u-\theta_{13}-\theta_{24})\p_{u_x}
+(u_x^2-\tfrac12\theta_{34})\p_{z}\\
 & +(2u_x-\tfrac12xu_{xx}-2\theta_{15}-2\theta_{26})\p_{u_{xx}}
+x\theta_{1}\p_{\tau}+x\theta_{2}\p_{\nu}
+(\tfrac12x\theta_{3}-2u_x\theta_{2})\p_{u_{\tau}}\\
 & +(\tfrac12x\theta_{4}+2u_x\theta_{1})\p_{u_{\nu}}
-(u_{xx}\theta_{2}-\tfrac12\theta_{3}+\tfrac12x\theta_{5})\p_{u_{x\tau}}
+(u_{xx}\theta_{1}+\tfrac12\theta_{4}-\tfrac12x\theta_{6})\p_{u_{x\nu}},\\
V_{15} =&\,
 (4xu_x-3u-\tfrac32x^2u_{xx}+3u_{xx}\theta_{12}+\theta_{13}+\theta_{24})\p_{x}
+(\tfrac32x^2(z-u_xu_{xx})+2xu_x^2\\
 & +3(u_xu_{xx}-z)\theta_{12}+u_x(\theta_{13}+\theta_{24})
 -x\theta_{34}+\tfrac32x^2(\theta_{36}-\theta_{45})-3\theta_{1236}+3\theta_{1245})\p_{u}\\
 & +(3xz+u_x^2-\tfrac34x^2u_{xx}^2+\tfrac32u_{xx}^2\theta_{12}+\theta_{34}+
\tfrac32x^2\theta_{56}-3\theta_{1256})\p_{u_x} +(3z+u_xu_{xx}-xu_{xx}^2\\
 & -2u_{xx}(\theta_{15}+\theta_{26})-4x\theta_{56}+2\theta_{36}-2\theta_{45})\p_{u_{xx}}
+(3zu_x-\tfrac14x^2u_{xx}^3+\tfrac12u_{xx}^3\theta_{12}\\
 &+\tfrac32x^2u_{xx}\theta_{56}-3u_{xx}\theta_{1256})\p_{z}
+(u_x\theta_{1}+x\theta_{4}-\tfrac32x^2\theta_{6}+3\theta_{126})\p_{\tau}
+(u_x\theta_{2}-x\theta_{3}\\
 & +\tfrac32x^2\theta_{5}-3\theta_{125})\p_{\nu}+
(2u_x\theta_{3}-3z\theta_{2}-\tfrac32x^2u_{xx}\theta_{5}+3u_{xx}\theta_{125})\p_{u_{\tau}}
+(3z\theta_{1}+2u_x\theta_{4}\\
 & -\tfrac32x^2u_{xx}\theta_{6} +3u_{xx}\theta_{126})\p_{u_{\nu}}
+(u_{xx}\theta_{3}-\tfrac12u_{xx}^2\theta_{2}+(u_x-2xu_{xx})\theta_{5}+4\theta_{256})\p_{u_{x\tau}}\\
 & +(u_{xx}\theta_{4}+\tfrac12u_{xx}^2\theta_{1}+(u_x-2xu_{xx})\theta_{6}-4\theta_{156})\p_{u_{x\nu}},
 \\
\! V_{16} =&\,
 (2 u_x^2-3u u_{xx}-\theta_{34})\p_{x}+(\tfrac43u_x^3-3uu_xu_{xx}+3uz-2u_x\theta_{34}+3u(\theta_{36}-\theta_{45}))\p_{u}\\
 & +(3u_xz-\tfrac32uu_{xx}^2+3u\theta_{56})\p_{u_x}+
(3u_{xx}z-u_xu_{xx}^2+2u_{xx}(\theta_{36}-\theta_{45})-4u_x\theta_{56})\p_{u_{xx}}\\
 & +(3z^2-\tfrac12uu_{xx}^3+3uu_{xx}\theta_{56})\p_{z}
 +(u_x\theta_{4}-3u\theta_{6})\p_{\tau}-(u_x\theta_{3}-3u\theta_{5})\p_{\nu}\\
 & +3(z\theta_{3}-uu_{xx}\theta_{5})\p_{u_{\tau}}
 +(\tfrac12u_{xx}^2\theta_{3}+(3z-2u_xu_{xx})\theta_{5}-4\theta_{356})\p_{u_{x\tau}}\\
 & +3(z\theta_{4}-uu_{xx}\theta_{6})\p_{u_{\nu}} +(\tfrac12u_{xx}^2\theta_{4}+(3z-2u_xu_{xx})\theta_{6}-4\theta_{456})\p_{u_{x\nu}},\\
V_{17} =&\,
(2x^2u_x-\tfrac12x^3u_{xx}-3xu+(3xu_{xx}-4u_x)\theta_{12}+x(\theta_{13}+\theta_{24}))\p_{x}+
(\tfrac12x^3(z-u_xu_{xx})\\
 & +x^2u_x^2-3u^2+(3x(u_xu_{xx}-z)-2u_x^2)\theta_{12}+xu_x(\theta_{13}+\theta_{24})
+\tfrac12x^3(\theta_{36}-\theta_{45})\\
 & -\tfrac12x^2\theta_{34}+3x(\theta_{1245}-\theta_{1236})+\theta_{1234})\p_{u}+
(\tfrac32x^2z-3uu_x+xu_x^2-\tfrac14x^3u_{xx}^2\\
 & +(\tfrac32xu_{xx}^2-3z)\theta_{12}+2u_x(\theta_{13}+\theta_{24})+
x\theta_{34}+\tfrac12x^3\theta_{56}-3x\theta_{1256})\p_{u_x}+(3xz-2u_x^2\\
 & +xu_xu_{xx}-\tfrac12x^2u_{xx}^2+u_{xx}^2\theta_{12}
+u_{xx}(\theta_{13}+\theta_{24})+2(2u_x-xu_{xx})(\theta_{15}+\theta_{26})+\theta_{34}\\
 & +2x(\theta_{36}-\theta_{45})-2x^2\theta_{56}+4\theta_{1256})\p_{u_{xx}}
+(3z(xu_x-u)-\tfrac23u_x^3-\tfrac1{12}x^3u_{xx}^3\\
 & +\tfrac12xu_{xx}^3\theta_{12}+3z(\theta_{13}+\theta_{24})+
u_x\theta_{34}+\tfrac12x^3u_{xx}\theta_{56}-3xu_{xx}\theta_{1256})\p_{z}
+((xu_x-3u)\theta_{1}\\
 & +\tfrac12x^2\theta_{4}-\tfrac12x^3\theta_{6}+3x\theta_{126}-\theta_{124})\p_{\tau}+((xu_x-3u)\theta_{2}
-\tfrac12x^2\theta_{3}+\tfrac12x^3\theta_{5}-3x\theta_{125}\\
 & +\theta_{123})\p_{\nu}+((2u_x^2-3xz)\theta_{2}+(2xu_x-3u)\theta_{3}-\tfrac12x^3u_{xx}\theta_{5}+
3xu_{xx}\theta_{125}-4\theta_{234})\p_{u_{\tau}}\\
 & +((3xz-2u_x^2)\theta_{1}+(2xu_x-3u)\theta_{4}-\tfrac12x^3u_{xx}\theta_{6}
+3xu_{xx}\theta_{126}+4\theta_{134})\p_{u_{\nu}}\\
 & +((2u_xu_{xx}-\tfrac12xu_{xx}^2-3z)\theta_{2}+(xu_{xx}-u_x)(\theta_{3}-x\theta_{5})
+2u_{xx}\theta_{125}+4x\theta_{256}-4\theta_{236}\\
 & -\theta_{135}+3\theta_{245})\p_{u_{x\tau}}+
((\tfrac12xu_{xx}^2-2u_xu_{xx}+3z)\theta_{1}+(xu_{xx}-u_x)(\theta_{4}-x\theta_{6})\\
 & +2u_{xx}\theta_{126}-4x\theta_{156}+3\theta_{136}-4\theta_{145}-\theta_{246})\p_{u_{x\nu}}
 \end{align*}
and
 \begin{align*}
U_1 =&\,
 \p_{u_{\tau}}-\theta_{1}\p_{u},\
U_2 = \p_{u_{\nu}}-\theta_{2}\p_{u},\
U_3 = \p_{\tau},\
U_4 = \p_{\nu},\\
U_5 =&\,
 \p_{u_{x\tau}}-\theta_{1}(\p_{u_x}+x\p_{u})-\theta_{4}\p_{z}+x\p_{u_{\tau}},\
U_6 = \p_{u_{x\nu}}-\theta_{2}(\p_{u_x}+x\p_{u})+\theta_{3}\p_{z}+x\p_{u_{\nu}},\\
U_7 =&\,
2\theta_{2}\p_{x}-\theta_{3}\p_{u_x}-2\theta_{5}\p_{u_{xx}}+x\p_{\tau}+2u_x\p_{u_{\nu}}+u_{xx}\p_{u_{x\nu}},\\
U_8 =&\,
 2\theta_{1}\p_{x}+\theta_{4}\p_{u_x}+2\theta_{6}\p_{u_{xx}}-x\p_{\nu}+2u_x\p_{u_{\tau}}+u_{xx}\p_{u_{x\tau}},\\
U_9 =&\,
 2x\theta_{1}\p_{x}+3u\theta_{1}\p_{u}+(u_x\theta_{1}+x\theta_{4})\p_{u_x}
 +(\theta_{4}+2x\theta_{6}-u_{xx}\theta_{1})\p_{u_{xx}}+u_x\theta_{4}\p_{z}\\
 & -(\tfrac12x^2-4\theta_{12})\p_{\nu}
 +(2xu_x-3u+3\theta_{13}+4\theta_{24})\p_{u_{\tau}}-\theta_{14}\p_{u_{\nu}}\\
 & +(xu_{xx}-u_x+\theta_{15}+4\theta_{26})\p_{u_{x\tau}}-3\theta_{16}\p_{u_{x\nu}},\\
U_{10} =&\,
 2x\theta_{2}\p_{x}+3u\theta_{2}\p_{u}+(u_x\theta_{2}-x\theta_{3})\p_{u_x}-(\theta_{3}+2x\theta_{5}+u_{xx}\theta_{2})\p_{u_{xx}}
 -u_x\theta_{3}\p_{z}\\
 & +(\tfrac12x^2-4\theta_{12})\p_{\tau}-\theta_{23}\p_{u_{\tau}}+(2xu_x-3u+4\theta_{13}+3\theta_{24})\p_{u_{\nu}}\\
 & -3\theta_{25}\p_{u_{x\tau}}+(xu_{xx}-u_x+4\theta_{15}+\theta_{26})\p_{u_{x\nu}},\\
U_{11} =&\,
 (3u_{xx}\theta_{2}+\theta_{3})\p_{x}+(3(u_xu_{xx}-z)\theta_{2}+u_x\theta_{3}-3(\theta_{236}-\theta_{245}))\p_{u}
 +(\tfrac32u_{xx}^2\theta_{2}-3\theta_{256})\p_{u_x}\\
 & -2u_{xx}\theta_{5}\p_{u_{xx}}+u_{xx}(\tfrac12u_{xx}^2\theta_{2}-3\theta_{256})\p_{z}
 +(u_x+3\theta_{26})\p_{\tau}-3\theta_{25}\p_{\nu}+3u_{xx}\theta_{25}\p_{u_{\tau}}\\
 & +3(z+u_{xx}\theta_{26})\p_{u_{\nu}}+(\tfrac12u_{xx}^2-4\theta_{56})\p_{u_{x\nu}},\\
U_{12} =&\,
 (\theta_{4}-3u_{xx}\theta_{1})\p_{x}+(3(z-u_xu_{xx})\theta_{1}+u_x\theta_{4}+3(\theta_{136}-\theta_{145}))\p_{u}\\
 & -(\tfrac32u_{xx}^2\theta_{1}-3\theta_{156})\p_{u_x}-2u_{xx}\theta_{6}\p_{u_{xx}}
 -u_{xx}(\tfrac12u_{xx}^2\theta_{1}-3\theta_{156})\p_{z}-3\theta_{16}\p_{\tau}\\
 & +(u_x+3\theta_{15})\p_{\nu}-3(z+u_{xx}\theta_{15})\p_{u_{\tau}}-3u_{xx}\theta_{16}\p_{u_{\nu}}
 -(\tfrac12u_{xx}^2-4\theta_{56})\p_{u_{x\tau}},\\
U_{13} =&\,
 ((3xu_{xx}-4u_x)\theta_{1}-x\theta_{4})\p_{x}+((3x(u_xu_{xx}-z)-2u_x^2)\theta_{1}-xu_x\theta_{4}+3x(\theta_{145}-\theta_{136})\\
 & +\theta_{134})\p_{u}+(3(\tfrac12xu_{xx}^2-z)\theta_{1}-2u_x\theta_{4}-3x\theta_{156})\p_{u_x}
 +(u_{xx}^2\theta_{1}-u_{xx}\theta_{4}+4\theta_{156}\\
 & +(2xu_{xx}-4u_x)\theta_{6})\p_{u_{xx}}+(\tfrac12xu_{xx}^3\theta_{1}-3z\theta_{4}-3xu_{xx}\theta_{156})\p_{z}
 +(3x\theta_{16}-\theta_{14})\p_{\tau}\\
 & +(3u-xu_x+\theta_{13}-3x\theta_{15})\p_{\nu}+(3xz-2u_x^2+4\theta_{34}+3xu_{xx}\theta_{15})\p_{u_{\tau}}
 +3xu_{xx}\theta_{16}\p_{u_{\nu}}\\
 & +(u_{xx}(\tfrac12xu_{xx}-2u_x+2\theta_{15})+3z-3\theta_{45}+4\theta_{36}-4x\theta_{56})\p_{u_{x\tau}}
 +(2u_{xx}\theta_{16}+\theta_{46})\p_{u_{x\nu}},\\
U_{14} =&\,
 ((4u_x-3xu_{xx})\theta_{2}-x\theta_{3})\p_{x}+((3xz+2u_x^2-3xu_xu_{xx})\theta_{2}-xu_x\theta_{3}
 +3x(\theta_{236}-\theta_{245})\\
 & -\theta_{234})\p_{u}+(3(z-\tfrac12xu_{xx}^2)\theta_{2}-2u_x\theta_{3}+3x\theta_{256})\p_{u_x}
 -(u_{xx}^2\theta_{2}+u_{xx}\theta_{3}+4\theta_{256}\\
 & +2(2u_x-xu_{xx})\theta_{5})\p_{u_{xx}}
 -(3z\theta_{3}+xu_{xx}(\tfrac12u_{xx}^2\theta_{2}-3\theta_{256}))\p_{z}
 +(3u-xu_x-3x\theta_{26}\\
 & +\theta_{24})\p_{\tau}
 +(3x\theta_{25}-\theta_{23})\p_{\nu}-3xu_{xx}\theta_{25}\p_{u_{\tau}}
 -(3xz-2u_x^2+4\theta_{34}+3xu_{xx}\theta_{26})\p_{u_{\nu}}\\
 & +(\theta_{35}-2u_{xx}\theta_{25})\p_{u_{x\tau}}
 -(3z+u_{xx}(\tfrac12xu_{xx}-2u_x+2\theta_{26})-4\theta_{45}+3\theta_{36}-4x\theta_{56})\p_{u_{x\nu}},
 \end{align*}

Note that $G(3)_{\bar0}$ contains the subalgebra $\mathfrak{sp}(2)=\langle V_8,V_9,V_{10}\rangle$,
with the complement $G(2)$ generated by the remaining $V_i$. 
The given generators are compatible with the SHC $\mathbb Z$-grading; more precisely:
$\opp{deg}(V_{1\text{-}2})=-3$, $\opp{deg}(V_3)=-2$, $\opp{deg}(V_{4\text{-}5})=-1$, $\opp{deg}(V_{6\text{-}12})=0$, $\opp{deg}(V_{13\text{-}14})=1$, $\opp{deg}(V_{15})=2$, $\opp{deg}(V_{16\text{-}17})=3$ for the even generators;
$\opp{deg}(U_{1\text{-}2})=-2$, $\opp{deg}(U_{3\text{-}6})=-1$, $\opp{deg}(U_{7\text{-}8})=0$, 
$\opp{deg}(U_{9\text{-}12})=1$, $\opp{deg}(U_{13\text{-}14})=2$ for the odd generators.

\section*{Acknowledgments}

This work was initiated during a visit by A. Santi to Troms\o{} in Fall 2018.  Support from the Troms\o{} Research Foundation -- Pure Mathematics in Norway is gratefully acknowledged.  The research of A. Santi is supported by the
project ``Supergravity backgrounds and Invariant theory'' of the University of Padova and partly supported
by the Project BIRD179758/17 ``Stratifications in algebraic groups, spherical varieties, Kac Moody algebras
and Kac Moody groups'' and Project DOR1717189/17 ``Algebraic, geometric and combinatorial properties of conjugacy classes''.
\bibliographystyle{abbrv}
\bibliography{G(3)-paper}

\begin{thebibliography}{10}

\bibitem{MR2840967}
C.~Carmeli, L.~Caston, and R.~Fioresi.
\newblock {\em Mathematical foundations of supersymmetry}.
\newblock EMS Series of Lectures in Mathematics. European Mathematical Society
  (EMS), Z\"{u}rich, 2011.

\bibitem{G2-Cartan}
E.~Cartan.
\newblock Sur la structure des groupes simples finis et continus.
\newblock {\em C.\ R.\ Acad.\ Sci.\ Paris}, 116:784--786, 1893.

\bibitem{MR1509120}
E.~Cartan.
\newblock Les syst\`emes de {P}faff \`a cinq variables et les \'{e}quations aux
  d\'{e}riv\'{e}es partielles du second ordre.
\newblock {\em Ann. Sci. \'{E}cole Norm. Sup. (3)}, 27:109--192, 1910.

\bibitem{MR2827478}
D.~Chapovalov, M.~Chapovalov, A.~Lebedev, and D.~Leites.
\newblock The classification of almost affine (hyperbolic) {L}ie superalgebras.
\newblock {\em J. Nonlinear Math. Phys.}, 17(suppl. 1):103--161, 2010.

\bibitem{MR1688484}
S.-J. Cheng and V.~G. Kac.
\newblock Generalized {S}pencer cohomology and filtered deformations of {${\bf
  Z}$}-graded {L}ie superalgebras.
\newblock {\em Adv. Theor. Math. Phys.}, 2(5):1141--1182, 1998.

\bibitem{MR0438925}
L.~Corwin, Y.~Ne'eman, and S.~Sternberg.
\newblock Graded {L}ie algebras in mathematics and physics ({B}ose-{F}ermi
  symmetry).
\newblock {\em Rev. Modern Phys.}, 47:573--603, 1975.

\bibitem{MR3531745}
K.~Coulembier.
\newblock Bott-{B}orel-{W}eil theory and {B}ernstein-{G}el'fand-{G}el'fand
  reciprocity for {L}ie superalgebras.
\newblock {\em Transform. Groups}, 21(3):681--723, 2016.

\bibitem{MR1701618}
P.~Deligne, P.~Etingof, D.~S. Freed, L.~C. Jeffrey, D.~Kazhdan, J.~W. Morgan,
  D.~R. Morrison, and E.~Witten, editors.
\newblock {\em Quantum fields and strings: a course for mathematicians. {V}ol.
  1, 2}.
\newblock American Mathematical Society, Providence, RI; Institute for Advanced
  Study (IAS), Princeton, NJ, 1999.
\newblock Material from the Special Year on Quantum Field Theory held at the
  Institute for Advanced Study, Princeton, NJ, 1996--1997.

\bibitem{MR0387363}
D.~v. Djokovi\'{c} and G.~Hochschild.
\newblock Semisimplicity of {$2$}-graded {L}ie algebras. {II}.
\newblock {\em Illinois J. Math.}, 20(1):134--143, 1976.

\bibitem{MR3254311}
B.~Doubrov and B.~Kruglikov.
\newblock On the models of submaximal symmetric rank 2 distributions in 5{D}.
\newblock {\em Differential Geom. Appl.}, 35(suppl.):314--322, 2014.

\bibitem{G2-Engel}
F.~Engel.
\newblock Sur un groupe simple \`a quatorze param\`etres.
\newblock {\em C.\ R.\ Acad.\ Sci.\ Paris}, 116:786--788, 1893.

\bibitem{MR3594366}
J.~Figueroa-O'Farrill and A.~Santi.
\newblock Spencer cohomology and 11-dimensional supergravity.
\newblock {\em Comm. Math. Phys.}, 349(2):627--660, 2017.

\bibitem{MR3328668}
R.~Fioresi and M.~A. Lled\'{o}.
\newblock {\em The {M}inkowski and conformal superspaces}.
\newblock World Scientific Publishing Co. Pte. Ltd., Hackensack, NJ, 2015.
\newblock The classical and quantum descriptions.

\bibitem{MR1773773}
L.~Frappat, A.~Sciarrino, and P.~Sorba.
\newblock {\em Dictionary on {L}ie algebras and superalgebras}.
\newblock Academic Press, Inc., San Diego, CA, 2000.
\newblock With 1 CD-ROM (Windows, Macintosh and UNIX).

\bibitem{MR874337}
D.~B. Fuks.
\newblock {\em Cohomology of infinite-dimensional {L}ie algebras}.
\newblock Contemporary Soviet Mathematics. Consultants Bureau, New York, 1986.
\newblock Translated from the Russian by A. B. Sosinski\u{\i}.

\bibitem{MR2541343}
A.~S. Galaev.
\newblock Holonomy of supermanifolds.
\newblock {\em Abh. Math. Semin. Univ. Hambg.}, 79(1):47--78, 2009.

\bibitem{MR2552682}
A.~S. Galaev.
\newblock Irreducible complex skew-{B}erger algebras.
\newblock {\em Differential Geom. Appl.}, 27(6):743--754, 2009.

\bibitem{Goncharov}
A.~Goncharov.
\newblock Generalized conformal structures on manifolds.
\newblock {\em Selecta Math. Sov.}, 6(4):306--340, 1987.

\bibitem{MR0203626}
V.~Guillemin.
\newblock The integrability problem for {$G$}-structures.
\newblock {\em Trans. Amer. Math. Soc.}, 116:544--560, 1965.

\bibitem{MR1511723}
D.~Hilbert.
\newblock \"{U}ber den begriff der klasse von differentialgleichungen.
\newblock {\em Math. Ann.}, 73(1):95--108, 1912.

\bibitem{MR0498755}
V.~G. Kac.
\newblock Classification of simple $\mathbb{Z}$-graded {L}ie superalgebras and
  simple {J}ordan superalgebras.
\newblock {\em Comm. Algebra}, 5(13):1375--1400, 1977.

\bibitem{MR0486011}
V.~G. Kac.
\newblock Lie superalgebras.
\newblock {\em Advances in Math.}, 26(1):8--96, 1977.

\bibitem{MR0142696}
B.~Kostant.
\newblock Lie algebra cohomology and the generalized {B}orel-{W}eil theorem.
\newblock {\em Ann. of Math. (2)}, 74:329--387, 1961.

\bibitem{MR2949641}
B.~Kruglikov.
\newblock The gap phenomenon in the dimension study of finite type systems.
\newblock {\em Cent. Eur. J. Math.}, 10(5):1605--1618, 2012.

\bibitem{MR3253556}
B.~Kruglikov.
\newblock Symmetries of filtered structures via filtered {L}ie equations.
\newblock {\em J. Geom. Phys.}, 85:164--170, 2014.

\bibitem{MR3604980}
B.~Kruglikov and D.~The.
\newblock The gap phenomenon in parabolic geometries.
\newblock {\em J. Reine Angew. Math.}, 723:153--215, 2017.

\bibitem{MR1931997}
D.~Leites, E.~Poletaeva, and V.~Serganova.
\newblock On {E}instein equations on manifolds and supermanifolds.
\newblock {\em J. Nonlinear Math. Phys.}, 9(4):394--425, 2002.

\bibitem{MR849339}
Y.~I. Manin.
\newblock Grassmannians and flags in supergeometry.
\newblock In {\em Some problems in modern analysis}, pages 83--101. Moskov.
  Gos. Univ., Mekh.-Mat. Fak., Moscow, 1984.

\bibitem{MR1632008}
Y.~I. Manin.
\newblock {\em Gauge field theory and complex geometry}, volume 289 of {\em
  Grundlehren der Mathematischen Wissenschaften [Fundamental Principles of
  Mathematical Sciences]}.
\newblock Springer-Verlag, Berlin, 1997.
\newblock Translated from the 1984 Russian original by N. Koblitz and J. R.
  King, With an appendix by Sergei Merkulov.

\bibitem{Rittenberg:1977eg}
V.~Rittenberg and M.~Scheunert.
\newblock {Elementary Construction of Graded Lie Groups}.
\newblock {\em J. Math. Phys.}, 19:709, 1978.

\bibitem{MR2640006}
A.~Santi.
\newblock Superization of homogeneous spin manifolds and geometry of
  homogeneous supermanifolds.
\newblock {\em Abh. Math. Semin. Univ. Hambg.}, 80(1):87--144, 2010.

\bibitem{MR1026702}
A.~Sciarrino and P.~Sorba.
\newblock Representations of the {L}ie superalgebra {$G(3)$}.
\newblock In {\em X{V} {I}nternational {C}olloquium on {G}roup {T}heoretical
  {M}ethods in {P}hysics ({P}hiladelphia, {PA}, 1986)}, pages 513--521. World
  Sci. Publ., Teaneck, NJ, 1987.

\bibitem{MR2743764}
V.~Serganova.
\newblock Kac-{M}oody superalgebras and integrability.
\newblock In {\em Developments and trends in infinite-dimensional {L}ie
  theory}, volume 288 of {\em Progr. Math.}, pages 169--218. Birkh\"{a}user
  Boston, Inc., Boston, MA, 2011.

\bibitem{MR0266258}
N.~Tanaka.
\newblock On differential systems, graded {L}ie algebras and pseudogroups.
\newblock {\em J. Math. Kyoto Univ.}, 10:1--82, 1970.

\bibitem{MR3759350}
D.~The.
\newblock Exceptionally simple {PDE}.
\newblock {\em Differential Geom. Appl.}, 56:13--41, 2018.

\bibitem{MR2069561}
V.~S. Varadarajan.
\newblock {\em Supersymmetry for mathematicians: an introduction}, volume~11 of
  {\em Courant Lecture Notes in Mathematics}.
\newblock New York University, Courant Institute of Mathematical Sciences, New
  York; American Mathematical Society, Providence, RI, 2004.

\bibitem{MR0232811}
B.~J. Weisfeiler.
\newblock Infinite dimensional filtered {L}ie algebras and their connection
  with graded {L}ie algebras.
\newblock {\em Funkcional. Anal. i Prilo\v{z}en}, 2(1):94--95, 1968.

\bibitem{MR0356830}
J.~Wess and B.~Zumino.
\newblock Supergauge transformations in four dimensions.
\newblock {\em Nuclear Phys.}, B70:39--50, 1974.

\bibitem{MR1699860}
K.~Yamaguchi.
\newblock {$G_2$}-geometry of overdetermined systems of second order.
\newblock In {\em Analysis and geometry in several complex variables ({K}atata,
  1997)}, Trends Math., pages 289--314. Birkh\"{a}user Boston, Boston, MA,
  1999.

\end{thebibliography}
\end{document}